\documentclass{amsart}
\usepackage{amsrefs}
\usepackage{color}
\usepackage{epsfig}
\usepackage{amssymb}
\usepackage{euscript}
\usepackage[all]{xy}
\usepackage{setspace}
\usepackage[margin=1.5in]{geometry}

\numberwithin{equation}{section}
\newcommand{\bC}{{\mathbb C}}
\newcommand{\bI}{{\mathbb I}}
\newcommand{\bP}{{\mathbb P}}
\newcommand{\bR}{{\mathbb R}}
\newcommand{\bZ}{{\mathbb Z}}

\newcommand{\cC}{{\mathcal C}}
\newcommand{\cE}{{\mathcal E}}
\newcommand{\cF}{{\mathcal F}}
\newcommand{\cL}{{\mathcal L}}
\newcommand{\cM}{{\mathcal M}}
\newcommand{\cN}{{\mathcal N}}
\newcommand{\cO}{{\mathcal O}}
\newcommand{\cP}{{\mathcal P}}
\newcommand{\cR}{{\mathcal R}}
\newcommand{\cS}{{\mathcal S}}
\newcommand{\cX}{{\mathcal X}}
\newcommand{\cY}{{\mathcal Y}}
\newcommand{\cZ}{{\mathcal Z}}

\newcommand{\sF}{{\EuScript F}}
\newcommand{\sJ}{{\EuScript J}}
\newcommand{\fn}{{\mathfrak n}}

\newcommand{\into}{\hookrightarrow}
\newcommand{\End}{\operatorname{End}}
\newcommand{\Aut}{\operatorname{Aut}}
\newcommand{\id}{\operatorname{id}}
\newcommand{\ev}{\operatorname{ev}}
\newcommand{\rk}{\operatorname{rk}}
\renewcommand{\dbar}{\bar{\partial}}
\newcommand{\alg}{\mathrm{alg}}
\newcommand{\I}[2]{\bI_{#1}^{#2}}
\newcommand{\gh}{\operatorname{gh}}

\def\co{\colon\thinspace}

\newcommand{\Maps}{\operatorname{Maps}}
\newcommand{\CP}{{\mathbb{CP}}}
\newcommand{\pt}{\operatorname{pt}}

\newcommand{\Pbar}{\overline{\cP}}
\newcommand{\Phat}{\hat{\cP}}

\newcommand{\Cbar}{\overline{\cC}}
\newcommand{\Chat}{\hat{\cC}}

\newcommand{\preG}{\operatorname{preG}}
\newcommand{\preGext}{\mathrm{pre\mathring{G}}}
\newcommand{\predGext}{\mathrm{pred\mathring{G}}}
\newcommand{\G}{\operatorname{G}}
\newcommand{\Gext}{{\operatorname{\mathring{G}}}}
\newcommand{\Gcod}[2]{\operatorname{G}^{#1 \to #2}}
\newcommand{\B}{\operatorname{B}}
\newcommand{\sol}{\operatorname{sol}}
\newcommand{\neck}{\mathrm{n}}
\renewcommand{\Im}{\mathrm{Im}}

\newcommand{\bfepsilon}{\boldsymbol \epsilon}
\newcommand{\bfJ}{\mathbf J}
\newcommand{\bfK}{\mathbf K}

\newcommand{\rO}{\mathrm{O}}

\newcommand{\semidirect}{\rtimes}
\newcommand{\dist}{dist}
\newcommand{\inte}{\operatorname{int}}
\newcommand{\aut}{\mathfrak{aut}}
\newcommand{\codC}[2]{\partial^{#1} \Cbar\left(#2\right)}
\newcommand{\codtC}[2]{\tilde{\partial}^{#1} \Cbar\left(#2\right)}
\newcommand{\codbC}[2]{\bar{\partial}^{#1} \Chat \left(#2 \right)}
\newcommand{\coduC}[2]{\partial^{#1}_{u} \Cbar \left(#2 \right)}
\newcommand{\codtuC}[2]{\tilde{\partial}^{#1}_{u} \Cbar \left(#2 \right)}
\newcommand{\codbuC}[2]{\bar{\partial}^{#1}_{u} \Chat \left(#2 \right)}
\newcommand{\cod}[2]{\partial^{#1} \Pbar\left(#2;0\right)}
\newcommand{\codu}[2]{\partial^{#1}_{u} \Pbar \left(#2;0 \right)}
\newcommand{\codb}[2]{\bar{\partial}^{#1} \Phat \left(#2;0 \right)}
\newcommand{\codt}[2]{\tilde{\partial}^{#1} \Pbar \left(#2;0 \right)}

\newcommand{\codbu}[2]{\bar{\partial}^{#1}_{u} \Phat \left(#2;0 \right)}

\newcommand{\Uthree}{U_3}
\newcommand{\N}{N}
\newcommand{\n}{\fn}
\newcommand{\Hopf}{\mathrm{Hopf}}

\newtheorem{thm}{Theorem}[section]

\newtheorem{addendum}[thm]{Addendum}
\newtheorem{cor}[thm]{Corollary}
\newtheorem{lem}[thm]{Lemma}
\newtheorem{prop}[thm]{Proposition}
\newtheorem{defin}[thm]{Definition}
\newtheorem{rem}[thm]{Remark}

\newcommand{\noproof}{\qed}

\newcommand{\noproofe}{ \vspace{-40pt}
\begin{flushright}
\qedsymbol
\end{flushright}}

\title[Exotic Spheres]{Framed bordism and Lagrangian embeddings \\ of exotic spheres}

\author[M.~Abouzaid]{Mohammed Abouzaid} 
\thanks{This research was conducted during the period the author served as a Clay Research Fellow. }

\begin{document}

\maketitle
\setcounter{tocdepth}{1}
\tableofcontents
\parskip1em

\section{Introduction}
Let $L$ be a  compact exact Lagrangian submanifold of $T^{*} S^{m}$ whose Maslov class vanishes.  By \cite{seidel-kronecker}, $L$ is a rational homology sphere which represents a generator of $H_m(T^* S^{m}, \bZ)$.  The classical nearby Lagrangian conjecture of Arnol'd would imply that $L$ is diffeomorphic to the standard sphere.

One approach to proving the diffeomorphism statement implied by Arnol'd's conjecture would be to first prove that $L$ is a homotopy sphere by establishing the vanishing of its fundamental group, then to use Kervaire and Milnor's classification of exotic spheres to exclude the remaining possibilities, see \cite{KM}.  In this paper, we begin the second part of the program and prove
\begin{thm} \label{main}
Every homotopy sphere which embeds as a Lagrangian in $T^* S^{4k+1}$ must bound a compact parallelisable manifold.
\end{thm}
Unfortunately, this result says nothing about the 28 exotic $7$-spheres originally constructed by Milnor.  In dimension $9$ (the first dimension for which exotic spheres exist and the theorem applies), Kervaire and Milnor proved that there are $8$ different smooth structures on the sphere, with only $2$ arising as boundaries of parallelisable manifolds. In particular, all but one of the exotic smooth spheres in dimension $9$ fail to embed as Lagrangian submanifolds of  the cotangent bundle of the standard sphere.  

By the Whitney trick, there are no obstructions to smooth embeddings of exotic spheres in the cotangent bundle of the standard sphere in dimensions greater than $4$.  In fact, the $h$-cobordism theorem implies the stronger statement that the unit disc bundles in the cotangent bundles of all homotopy spheres of the same dimension are diffeomorphic as manifolds with boundary.  Theorem \ref{main} implies that the symplectic structure on a cotangent bundle remembers some information about the smooth structure of the base that is forgotten by the mere diffeomorphism type of the unit disc bundle:
\begin{cor}
If $\Sigma^{4k+1}$ is an exotic sphere which does not bound a parallelisable manifold, then the cotangent bundles $T^* \Sigma^{4k+1}$ and $T^* S^{4k+1}$ are not symplectomorphic. \noproof
\end{cor}

The starting point in the proof of Theorem \ref{main} is the fact that the graph of the Hopf fibration embeds $S^{2n-1}$ as a Lagrangian in $\bC^{n} \times \bC \bP^{n-1}$; this fact was observed by Audin, Lalonde and Polterovich in \cite{ALP}, and used by Buhovsky in \cite{buh} to constrain the homology groups of Lagrangians embedding in the cotangent bundle of the sphere.   By Weinstein's neighbourhood theorem,  $L$ embeds in $\bC^{n} \times \bC \bP^{n-1}$ as well.  Because of the $\bC^n$ factor, there is a compactly supported Hamiltonian isotopy which displaces $L$ from itself.  Following a strategy which goes back to Gromov's introduction of pseudo-holomorphic curves to symplectic topology in \cite{gromov}, and whose applications to Lagrangian embeddings have been developed among others by Polterovich, Oh, Biran-Cieliebak, and Fukaya \cites{polterovich-monotone,oh-disjunction,BC,fuk}, the choice of such a displacing isotopy determines a one parameter family deformation of the Cauchy-Riemann equation.  Under generic conditions, the corresponding moduli space of solutions in the constant homotopy class becomes a manifold of dimension $2n$ with boundary diffeomorphic to $L$.

This moduli space is not compact, but admits a compactification which corresponds geometrically to considering exceptional solutions to the parametrised deformation of the  Cauchy-Riemann equation, along with a holomorphic disc or sphere bubble.  The strata corresponding to disc bubbles can be ``capped off'' by a different moduli space of holomorphic discs to obtain a compact manifold one of whose boundary components is diffeomorphic to $L$.  The remaining boundary components correspond to sphere bubbles, and  have explicit descriptions as bundles over $S^2$, which allows us to prove that they bound parallelisable manifolds.

To prove the parallelisability of the result of gluing these manifolds along their common boundaries, we rely on the fact that the $K$-theory class of the tangent space of the moduli space of $J$-holomorphic curves  agrees with the restriction of a class coming from the space of all smooth maps. An analysis of the homotopy type of this space (together with the fact that $\bC \bP^{n-1}$ has even first Chern class precisely when $n$ is even) yields the desired parallelisability of the bounding manifold we construct.  It is only in this very last step that the mod $4$ value of the dimension, rather than simply its parity, enters the argument.

\subsection*{Acknowledgments}
I would like to thank Paul Seidel both for the initial discussion that started my work on this project, as well as for many subsequent ones concerning the compatibility of Gromov-Floer compactifications with smooth structures.  I am also grateful to Katrin Wehrheim for answering my many questions about technical aspects of gluing pseudo-holomorphic curves, to Shmuel Weinberger for suggesting the argument of Lemma \ref{lem:framings_product_spheres}, and to Yasha Eliashberg for conversations which led to the explicit computations summarized in Addendum \ref{add:explicit_computation}.
\section{Construction of the bounding manifold}
\subsection{Deformation of the Cauchy-Riemann equation}
In \cite{oh-disjunction}, Oh studied displaceable Lagrangians by means of a family of inhomogeneous Cauchy-Riemann equations associated to a Hamiltonian function.  Instead of counting solutions passing through a generic point, we shall study the full moduli space.

As remarked in the introduction, the starting point is a Lagrangian embedding
\begin{equation} L \into \bC^{n} \times \bC \bP^{n-1} ,\end{equation}
where $ \bC \bP^{n-1} $ is equipped with the symplectic form coming from hamiltonian reduction with respect to the diagonal circle action on $\bC^{n} $ with weight $1$, and $ \bC^{n} $ with the negative of the standard symplectic form.  We shall moreover assume that the embedding of $L$  is real analytic (see Lemma \ref{eq:hamilt_perturb_L_regularity}).  Choose a ball in $\bC^n$ containing the projection of $L$ to $\bC^{n}$ whose diameter is thrice the diameter of this projection, and let $M$ denote the product of this ball with $\bC \bP^{n-1}$.  Given a compactly supported function
\begin{equation} H \co M \times [-1,1] \to \bR \end{equation}
we will write $H_{t}$ for the restriction of this function to $M \times \{ t \}$, $X_{H}$ for the corresponding time-dependent Hamiltonian vector field on $M$, and $\phi$ for the resulting Hamiltonian diffeomorphism obtained by integrating $X_{H}$ over the interval $[-1,1]$.  We shall assume that $H$ satisfies the following condition:
\begin{equation}
\parbox{35em}{The symplectomorphism $\phi$ displaces $L$ from itself.  Moreover,  we assume that $H_{t}$ vanishes in a neighbourhood of $t=\pm 1$.}
\end{equation}

The existence of such a function $H_{t}$ follows trivially from our assumption that one factor of $M$ is a ball of sufficiently large radius and the fact that any translation in $\bC^{n}$ is generated by a Hamiltonian flow; an appropriate cutoff of this Hamiltonian flow is supported on $M$ and satisfies the desired conditions.
\begin{rem}
The requirement that $\phi$ displace $L$ is essential, whereas the vanishing assumption on $H_t$ is only needed in order to simplify the proof of the transversality results in Section \ref{sec:transversality} and the gluing theorems in Section \ref{sec:codim1_gluing}.
\end{rem}

We shall write $\I{-\infty}{+\infty} = \bR \times [-1,1]$ for the bi-infinite strip equipped with coordinates $(s,t)$, and we fix the identification
 \begin{equation} \label{eq:identify_disc_strip} \xi \co \I{-\infty}{+\infty} \to D^{2}  - \{ \pm 1 \} \end{equation}
which takes the origin to itself.  Consider a smooth family $\gamma_{R}$ of  compactly supported $1$-forms on $D^{2}$  parametrized by $R \in [0,\infty)$, and satisfying the following properties 
\begin{itemize}
\item If $0 \leq R \leq 1$, then $\gamma_{R}= R \gamma_1$, and
\item If $1 \leq R$, then  $\xi^{*} \gamma_R$ is a multiple of $dt$ which depends only on $s$, and is constrained as follows
\begin{equation} \xi^{*} \gamma_R(s,t) =  \begin{cases} dt & \textrm{if }|s| \leq R \\ 0 & \textrm{if }|s|  \geq R+1 \end{cases} \end{equation}
\item the norm $| \xi^{*} \gamma_R(s,t)|$ is monotonically decreasing in the variable $s$ whenever $s \geq 0$, and monotonically increasing if $s \leq 0$
\end{itemize}

Let $\sJ$ denote the space of almost complex structures on $M$ which are compatible with the standard symplectic form and which agree with the complex structure $J_{\alg}$ on $\bC^{n} \times \bC \bP^{n-1} $ to infinite order on the boundary of $M$.   Here, $J_{\alg}$  is the direct sum of the complex conjugate of the standard complex structure on $\bC^{n}$ with the standard complex structure on $ \bC \bP^{n-1}$; we have to conjugate the complex structure on the first factor in order to maintain positivity with respect to the symplectic form.

Further, consider a family  $\bfJ = \{ \bfJ_{R} \}_{R \in [0,+\infty)} = \{ J_{z, R} \}_{(z,R) \in D^{2} \times [0,+\infty)}$ of almost complex structures parametrized by $(z,R) \in D^{2} \times [0,+\infty)$ such that $J_{z, R} = J_{\alg}$ whenever $R=0$ or $z \in D^{2}$ is sufficiently close to the boundary.

Having made the above choices, we define, for each $R \geq 0$, a moduli space which we denote $\cN(L;0,\bfJ_R,H, \gamma_{R})$, of finite energy maps
\[ u \co (D^{2},S^{1}) \to (M,L)  \]
in the homotopy class of the constant map, satisfying the equation
\begin{equation} \label{deformedCR} \left(du - \gamma_{R} \otimes X_{H}(z) \right)^{0,1} = 0 .\end{equation}
The $(0,1)$ part of the above $TM$-valued $1$-form is taken with respect to the family $\bfJ$ of almost complex structures.  Moreover, if $z \neq \pm 1$, then the vector field $X_{H}$ in the previous equation is the Hamiltonian vector field on $M$ of the function $H_{t}$ where $t \in [-1,1] $ is determined by the property that there is a value of $s$ such that $\xi(s,t)  = z$.  Note that $\gamma_{R}$ vanishes at $z = \pm 1$, so there is no ambiguity in Equation \eqref{deformedCR}.

  Away from $\pm 1$, the above equation can be explicitly written in $(s,t)$ coordinates as
\begin{equation} \label{deformedCR-explicity} \partial_s u + J_{\xi(s,t),R} \partial_t u = \xi^{*}(\gamma_{R})(\partial_{t}) \cdot  J_{\xi(s,t),R} X_H\left(u\circ \xi(s,t),t \right) \end{equation}
for each $s \in \bR$ and $t \in [-1,1]$ with the boundary conditions
\[ u(s,-1), u(s,1) \in L \]
for all $s \in \bR$.    The removal of singularities theorem proven in  \cite{oh-removal} implies that any finite energy solution of \eqref{deformedCR-explicity} extends smoothly to $D^{2}$.  In particular, there is a bijection between the solutions of \eqref{deformedCR} and  \eqref{deformedCR-explicity}.

Unless we explicitly need to discuss them, we will omit the complex structure $\bfJ$ and the Hamiltonian $H$ from the notation, and write $\cP(L;0)$ for the space
\begin{equation} \label{eq:define_parametrized_space} \cP(L;0,\bfJ,H) = \coprod_{R \in [0,+\infty)} \cN(L;0,\bfJ_{R},H, \gamma_{R}) .\end{equation}

To understand the topology of $\cP(L;0)$, we consider the product of $[0,+\infty)$ with $\cF(L)$, the space of smooth maps from $(D^2, S^1)$  to $(M,L)$:
\begin{equation}  \label{eq:parametrized_smooth_space} \cF_{\cP}(L) \equiv   [0,+\infty) \times  \cF(L) \equiv [0,+\infty) \times C^{\infty}\left((D^2,S^1),(M,L)\right). \end{equation}
We shall write $\cF_{\cP}(L; \beta_0)$ for the component of  $\cF_{\cP}(L)$ consisting of maps in a homotopy class $\beta_0$.  For the moment, we're interested in the homotopy class of the constant map.  Let $\cE_{\cP}$ be the bundle over $\cF_{\cP}(L)$ with fibres
\[ C^{\infty}\left(u^*(TM) \otimes_{\bC} \Omega^{0,1} D^{2}\right)\]
at a point $(R,u)$, where $\Omega^{0,1}$ is the bundle of complex anti-linear $1$-forms on $D^2$, and the tensor product is taken with respect to the $z$ and $R$ dependent almost complex structure $\bfJ$ on $TM$.   Note that the map
\begin{equation} \label{dbar-operator} \dbar_{\cP} \co (R,u) \mapsto \left(du - \gamma_{R} \otimes X_{H} \right)^{0,1} \end{equation}
defines a section of $\cE_{\cP}$ whose intersection with  the inclusion of $\cF_{\cP}(L;0)$ as the zero section  is equal to $\cP(L;0)$.  Our next assumption concerns the regularity of this moduli space:
\begin{equation} \label{req:0_moduli_reg}
\parbox{35em}{Choose the family $\bfJ$ such that the moduli space $\cP(L;0)$  is {\bf regular}, i.e. the graph of $\dbar_{\cP}$ is transverse to the zero section.}
\end{equation}

The existence of such a family is a standard transversality result appearing in various forms in \cite{polterovich-monotone} and \cite{oh-disjunction}, and regularity holds generically in an appropriate sense. We'll briefly discuss the proof in  Section \ref{sec:transversality}, culminating in Corollary \ref{cor:parametrized_moduli_manifold_boundary}, which specializes to the following result if we consider only the constant homotopy class:
\begin{cor} $\cP(L;0)$ is a smooth manifold with boundary 
\[ \partial \cP(L;0) =  \cN(L;0,\bfJ_0,H, \gamma_{0})  \cong L ,\]
consisting of constant holomorphic discs for the constant family $\bfJ_{0} \equiv J_{\alg}$.
\end{cor}

To recover tangential information, recall that, upon choosing a complex linear connection, the linearisation of $\dbar_{\cP}$ at its zeroes  extends to an operator
\begin{equation} \label{CR-operator} D_{\cP} \co T\cF_{\cP}(L)  \to \cE_{\cP},\end{equation}
which is  Fredholm in appropriate Sobolev space completions;  this essentially follows from the fact that Equation \eqref{deformedCR} is a compact perturbation of the usual Cauchy-Riemann operator (see Section \ref{sec:transversality}).    To compute the class of this operator in $K$-theory, we  shall use the fact that the homotopy type of $\cF_{\cP}(L;0)$ is extremely simple.

\begin{lem} \label{lem:homotopy_theory}
The evaluation map at $1 \in S^1$ and the relative homotopy class define a weak homotopy equivalence 
\begin{equation} \cF(L) \to L \times \bZ .\end{equation}
\end{lem}
\begin{proof}
Since $L$ and $S^{2n-1}$ are homeomorphic, and their inclusions in $M$ are homotopic, it suffices to prove the result whenever $L$ is the standard sphere embedded as the graph of the Hopf fibration.

Consider the projection to the second component of $M$, and the induced fibration
\begin{equation} \xymatrix{ \cF(S^{2n-1})  \ar[d] \\ \Maps\left((D^2,S^1), (\CP^{n-1}, \CP^{n-1})\right) .} \end{equation}
The base of this fibration is clearly homotopy equivalent to $\CP^{n-1}$ which includes as the set of constant maps.  To study the fiber, we fix a point $p$ on $\CP^{n-1}$, and consider the space of maps $(D^2,S^1) \to (\bC^{n}, S_p)$, where $S_p$ is the circle in $S^{2n-1}$ which lies over $p$.  The degree of the restriction $S^1 \to S_p$ and the image of $1$ in $S_p$ defines a fibration 
\begin{equation}  \Maps \left((D^2,S^1), (\bC^{n}, S_p ) \right) \to S^1 \times \bZ \end{equation}
which is easily seen to be a homotopy equivalence.  Note that we can identify the degree of the evaluation map to the circle $S_p$ with the relative homotopy class in $\pi_{2}\left(\bC^{n} \times \CP^{n-1}, S^{2n-1} \right)$.  Fixing such a class $\beta_0$,  the evaluation at the basepoint $1 \in S^1$ defines  a map of fibrations
\[ \xymatrix{\Maps \left((D^2,S^1), (\bC^{n}, S_p );\beta_0 \right) \ar[d] \ar[r] & S^1 \ar[d] \\
\cF(S^{2n-1}; \beta_0) \ar[r] \ar[d] & S^{2n-1} \ar[d] \\
  \Maps \left((D^2,S^1), (\CP^{n-1}, \CP^{n-1}) \right) \ar[r] & \CP^{n-1} .}\]
Since the top and bottom horizontal maps are homotopy equivalences, so is the middle one.
\end{proof}

In particular, the inclusion of $L$ into $\cF(L;0)$ as constant discs is a weak homotopy equivalence.  We use this to conclude:

\begin{lem} \label{lem:trivial_tangent_bundle_parametrized}
The tangent bundle of $\cP(L;0)$ is stably trivial, and its restriction to the component containing constant discs is trivial.
\end{lem}
\begin{proof}
It is well-known, see for example \cite{MS} for the case with no Hamiltonian perturbation, that the operator \eqref{CR-operator}  extends to a Fredholm map between Banach bundles over $\cF_{\cP}(L)$. Moreover, regularity implies that the class of the tangent bundle of $\cP(L;0)$ in $KO(\cP(L;0))$ agrees with the restriction of the class of the Cauchy-Riemann operator on  $\cF_{\cP}(L)$.  By the previous lemma, the inclusion of $L$ in $\cF_{\cP}(L)$ as the boundary of $\cP(L;0)$ is a homotopy equivalence.  In particular, the stable triviality of the tangent bundles of homotopy spheres (see \cite{KM}) implies that the class of the Cauchy-Riemann operator in  $KO(\cF_{\cP}(L) )$ is trivial.

Since the component of $\cP(L;0)$ containing $L$ is a manifold with boundary, it has the homotopy type of a $2n-1$ complex.  In particular, its tangent bundle is trivial if and only if it is trivial in reduced $K$-theory (we're in the stable range; see Lemma 3.5 of \cite{KM}).
\end{proof}

If $\cP(L;0)$ were compact, we would immediately be able to conclude that existence of exotic spheres which cannot embed as a Lagrangian in the cotangent bundle of the standard sphere.  Indeed, we would know that $L$ bounds a compact parallelisable manifold, and Kervaire and Milnor's results in \cite{KM} imply  the existence of exotic spheres which cannot bound such a manifold.  However, $\cP(L;0)$ admits an evaluation map to $L$ extending the identity on the boundary of $\cP(L;0)$, so that the compactness of $\cP(L;0)$ would imply that the fundamental class of $L$ vanishes in homology.

We must therefore study the Gromov-Floer compactification $\Pbar(L;0)$ with the hope of being able to construct a compact parallelisable bounding manifold for $L$.

\subsection{Compactification of the moduli spaces}  
The standard  Gromov-Floer compactification of  $\cP(L;0)$ is obtained by including cusp curves, see  \cite{gromov}.   These consist of a solution $u$ to Equation \eqref{deformedCR} in some (relative) homotopy class $\beta_{0}$, together with a collection of holomorphic discs or spheres $v_i$ (considered modulo their automorphisms) in homotopy classes $\beta_{i}$ such that
\[ \sum_{i=0}^{\nu} \beta_{i} = 0,\] 
which are arranged along a tree and satisfy a stability condition.   In order to  conclude that $\Pbar(L;0)$ is compact we shall require
\begin{lem}[Lemma 2.2 of \cite{oh-disjunction}]
There is no solution to Equation \eqref{deformedCR} for $R$ sufficiently large. \noproof
\end{lem}

We will presently see that there are in fact only three possible configurations to consider.

The simple connectivity of $L$ implies that we have an isomorphism
\begin{equation} \pi_{2}(M) \cong \pi_{2}(M,L) .\end{equation}
We will write $\alpha$ for the generator of the left hand side corresponding to the copy of $\bC \bP^{1} \subset \bC \bP^{n-1}$, and $\beta$ for its image in the right hand side.

Since the first Chern class  of $\bC \bP^{n-1}$ evaluates to $n$ on $\alpha$, the  moduli space of holomorphic maps from $\bC \bP^{n-1}$ to $M$ in class $k \alpha$ has virtual real dimension
\[  2(2n -1) + 2 k n \]
by a standard application of the Riemann-Roch formula.  There is an analogous formula for the virtual dimension of the space of holomorphic discs in class $k \beta$ before taking the quotient by the group of automorphisms, which is
\begin{equation} \label{eq:virtual_dimension_discs}  2n-1 + 2 k n .\end{equation}

The dimension of the moduli space of solutions to Equation \eqref{deformedCR} in $\cF_{\cP}(L;k\beta)$ is one higher because of the extra choice of the parameter $R$.  Writing $\cP(L;k \beta)$ for this moduli space, we conclude:
\begin{lem}
The expected dimension of $\cP(L;k \beta)$ is $ 2n(1+k)$.  \noproof
\end{lem}

In order to find the desired bounding manifold, we shall choose Equation \eqref{deformedCR} so that requirement \eqref{req:0_moduli_reg} extends to homotopy classes of maps which have non-positive energy:
\begin{equation} \label{ass:exceptional_discs_regular}
\parbox{35em}{All moduli spaces $\cP(L;k \beta)$ for $k \leq 0$ are regular.}\end{equation}
This implies in particular that all such moduli spaces for $k < -1$ are empty,  so that, a priori, the cusp curves in the compactification of  $\cP(L;0)$ have, as one of their components, a curve $u$ in $\cP(L;k \beta)$ where $k \geq -1$.  The proof that this requirement holds for a generic path of almost complex structures $\bfJ$ is done in Section \ref{sec:transversality}, with a precise statement given in Corollary \ref{cor:parametrized_moduli_manifold_boundary}.

Since we're considering the Gromov-Floer compactification of a moduli space of curves in the constant homotopy class, the sum of the homotopy classes of all irreducible components in our tree must vanish.  Positivity of energy implies that all (non-constant) holomorphic discs (and spheres) occur in classes $k \beta$  (respectively $k \alpha$) where $k$ is strictly positive.  The moduli spaces $\cP(L;k \beta)$ for $k \neq -1$ do not therefore contribute to the compactification of  $\cP(L;0)$.

Note that  $ \cP(L;- \beta)$, the only remaining moduli space that could appear in the compactification of $\cP(L;0)$, has dimension $0$.  In particular, $\cN(L;-\beta, \bfJ_{R}, H, \gamma_{R})$ is empty for generic values of $R$;  so we shall call the elements of $ \cP(L;- \beta) $  {\bf exceptional solutions}.  Having excluded all other possibilities, we obtain the following description of the corner strata of  $\Pbar(L;0)$:
\begin{lem}
The strata of $\Pbar(L;0) - \cP(L;0)$ consist of an exceptional solution $u$ solving Equation \eqref{dbar-operator} for some $R$ together with either \begin{itemize}
\item[(i)] a $J_{ z, R}$-holomorphic sphere in class $\alpha$  intersecting the image of $u$ at $u(z)$ with $z$ an interior point, or
\item[(ii)] a $J_{\alg}$-holomorphic disc in class $\beta$ with boundary on $L$ passing through a point on the boundary of $u$, or 
\item[(iii)] a $J_{\alg}$-holomorphic sphere in class $\alpha$ passing through a point of the boundary of $u$. 
\end{itemize}
\noproof
\end{lem}
Observe that a holomorphic disc in class $\beta$ can degenerate to a holomorphic sphere in class $\alpha$ passing through a point of $L$.  The Fredholm theory of such a problem is well understood, and is modeled by considering a constant (ghost) holomorphic disc mapping to the intersection of the sphere with $L$.  This, together with a standard gluing argument, implies that case (iii) in the above lemma has expected codimension $2$  within the Gromov-Floer compactification of case (ii).   Case (iii) has expected codimension $1$ in the closure of the stratum corresponding to case (i), as a sphere bubble passing through the interior of the disc can escape to the boundary.

To describe the topology of these strata, we introduce the notation 
\begin{equation} \cS_{i,j} (L; k\beta , J) \end{equation}
for the moduli space of $J$-holomorphic discs in class $k \beta$ with $i$ labeled interior marked points, and $j$ ordered boundary marked points, modulo automorphisms. 
We also have  the spaces
\begin{equation}\cN_{i,j}(L;k\beta,\bfJ_{R}, H, \gamma_{R}) \textrm{ and } \cP_{i,j}(L;k \beta) \end{equation}
whose meaning should be evident to the reader.  Slightly more generally, if $\bfJ$ is a family of almost complex structures parametrized by $D^2 \times [0,+\infty)$ as in the definition \eqref{eq:define_parametrized_space} of the moduli spaces $\cP(L;0)$, then we shall write
\begin{equation} \cM_{i}(M; k\alpha, \bfJ) = \coprod_{(R,z) \in D^2 \times [0,+\infty)} \cM_{i}(M; k\alpha, J_{ z, R})\end{equation}
for the moduli space, modulo automorphism, of spheres in homotopy class $k \alpha$ with $i$ marked points, which are holomorphic with respect to one of the almost complex structures $J_{z, R}$.  This disjoint union is topologised, in the same way as $\cP(L;0)$, by exhibiting it as the zero set of a Fredholm section of a Banach bundle. In particular, whenever this defining section is transverse to the zero section, the natural map from $\cM_{i}(M; k\alpha, \bfJ)$ to $D^2 \times [0,+\infty)$ is smooth.

In this language, the strata of $\Pbar(L;0) - \cP(L;0)$ can be described as fibre products
\begin{align} \cod{1}{L} & = \cS_{0,1}(L;\beta, J_{\alg}) \times_{L}  \cP_{0,1}(L;- \beta)  \\
\cod{2}{L} & =  \cM_{1}(M;\alpha, \bfJ) \times_{M \times [0,+\infty) \times D^2}  \cP_{1,0}(L;- \beta) \\
\cod{3}{L} & = \cM_{1}(M;\alpha, J_{\alg}) \times_{M} \cS_{1,1}(L; 0, J_{\alg}) \times_{L} \cP_{0,1}(L;- \beta).
 \end{align}

Here, the superscript on $ \cod{i}{L}$ is the expected codimension of the stratum, i.e. the difference between the dimension of $\cP(L;0)$ and the virtual dimension of the stratum.    We will now ensure that all boundary strata are smooth manifolds of the expected dimension, which is the first step towards constructing a manifold with corners.   Oh proved in \cite{oh-perturb-boundary}, that the next requirement holds after a small perturbation of $L$.
\begin{equation} \label{ass:moduli_discs_regular}
\parbox{35em}{The moduli space  $\cS(L; \beta,J_\alg )$ of $J_\alg$ holomorphic discs in class $\beta$ with boundary on $L$ is regular.}  \end{equation}

In particular, the moduli space of holomorphic discs with one boundary marked point is a smooth manifold equipped with an evaluation map
\begin{equation}  \cS_{0,1}(L; \beta) \to L .\end{equation}
We shall impose the following condition on the moduli space of exceptional solutions:
\begin{equation} \label{ass:transverse_evaluation}
\parbox{35em}{The evaluation map 
\begin{equation*} \cP_{0,1}(L; -\beta) \to L \end{equation*}  is transverse to the evaluation map from $\cS_{0,1}(L; \beta)$ . }\end{equation}
In  Lemma  \ref{lem:transverse_evaluation_parametrized}, we prove that this property holds for a generic family $\bfJ$.  From Equation \eqref{eq:virtual_dimension_discs}, we know that the moduli space $\cS_{0,1}(L; \beta)$  has real dimension $4n-3$, since the group of automorphisms of $D^2$ fixing a boundary point has real dimension $2$. Subtracting $2n-2$ for the codimension of the evaluation map \eqref{ass:transverse_evaluation}, we conclude:
\begin{cor}
The stratum $\cod{1}{L}$ is a smooth manifold of dimension $2n - 1$. \noproof
\end{cor}

Note that the moduli space of holomorphic spheres in class $\alpha$ with respect to the standard complex structure $J_\alg$ is regular and that the evaluation map from  $\cM_{1}(M ; \alpha ,J_{\alg})$ to  $M$ is a submersion.  Since the family $\bfJ$ has been assumed to agree with $J_\alg$ when $z$ lies on the boundary, we conclude 
\begin{lem}
The stratum $\cod{3}{L}$ is a smooth compact manifold of dimension $2n-3$ with components labeled by elements of $\cP(L;-\beta)$.  Each such component is  diffeomorphic to
\begin{equation} \CP^{n-2} \times S^1. \end{equation} \noproofe
\end{lem}
In Section \ref{sec:transversality}, see in particular Lemma \ref{lem:codim_2_strata_diffeo_type}, we shall explain how to use the stability of submersions under perturbations to extend this explicit description to the moduli space of spheres passing through the interior of an exceptional solution:
\begin{equation}  \begin{array}{c} 
\parbox{35em}{The stratum $\cod{2}{L}$ is a smooth manifold of dimension $2n-2$ with components labeled by elements of $\cP(L;-\beta)$.  The closure of every component is diffeomorphic to}  \\
 \label{eq:diffeo_type_codim_2}   \CP^{n-2} \times D^{2}. \end{array}\end{equation}

We have now ensured that all  strata of $\Pbar(L;0)$ are smooth manifolds.  Instead of proving that $\Pbar(L;0)$ is a manifold with boundary, we shall construct a manifold with corners $\Phat(L;0)$ one of whose boundary components is diffeomorphic to $L$, and whose remaining corner strata are in bijective correspondence to the strata of $\Pbar(L;0) - \cP(L;0)$.  To give a precise relationship between $\Phat(L;0)$ and $\Pbar(L;0)$, consider the trivial bundle over the parametrized moduli space of holomorphic spheres whose fibre is the tangent space of $\bC \bP^{1}$ at $\infty$ which we identify with the marked point.   Taking the quotient by $\Aut(\bC \bP^1,\infty) \cong \bC^{*} \semidirect \bC$ we obtain a possibly non-trivial complex line bundle which we denote $\cL \cM_{1}(M; \alpha) $. 
 This is the bundle over the moduli space of unparametrized spheres with one marked point $\cM_{1}(M; \alpha) $ whose fibre is the tangent line of $\bC \bP^1$ at the marked point.

Over $\cP_{1,0}(L;-\beta)$, we also have a complex line bundle $\cL \cP_{1,0}(L;-\beta)$ whose fiber at an element $(u,z)$ is the tangent space $T_{z} D^{2}$. Let $ \cL \cod{2}{L}$ denote the tensor product of the pullback of these two line bundles to $\cod{2}{L}$,  
\begin{equation} \label{eq:normal_bundle_codimension_2} \cL \cod{2}{L} = \cL \cM_{1}(M; \alpha) \boxtimes_{\bC} \cL \cP_{1,0}(L;-\beta).  \end{equation}
We will consider the unit circle bundle associated to $ \cL \cod{2}{L} $ which we denote by \begin{equation}\codt{2}{L}. \end{equation}
\begin{rem}
Were we to prove that $\Pbar(L;0)$ is a manifold with boundary, $\cL \cP_{1,0}(L;-\beta)$ would have been the normal bundle to $\cod{2}{L}$.  From the point of view of algebraic geometry, this follows from the well known fact that the infinitesimal smoothings of a nodal singularity are parametrized by the tensor product of the tangent spaces of the two branches.  From the point of view of gluing theory and hence closer to the discussion at hand, a gluing construction from $\cod{2}{L}$ to $\cP(L;0)$ requires choosing a gluing parameter in $[0,+\infty)$ as well as cylindrical ends on each branch, i.e. maps from $S^1 \times [0,+\infty)$ to each punctured surface.  The choice of such ends is equivalent up to homotopy  to the choice of a unit tangent vector at each branch.  Moreover, up to a decaying error term, the gluing construction is invariant under the simultaneous rotation of both cylindrical ends.
\end{rem}

The same construction performed on $\cM_{1}(M; \alpha, \bfJ)$ yields a circle bundle  $\codt{3}{L} $ over $\cod{3}{L}$.  We can now state the following result which will be proved in Section \ref{sec:construction_manifold_corner}.
\begin{lem} \label{lem:existence_manifold_corner}
The union
\begin{equation} \cP(L;0) \cup \codt{2}{L}\end{equation}
admits the structure of a smooth manifold with boundary.  Moreover, there exists a compact manifold with corners $\Phat(L;0)$ embedded in $ \cP(L;0) \cup \codt{2}{L} $ as a codimension $0$ submanifold,  one of whose boundary components is diffeomorphic to $L$, and whose remaining corner strata are diffeomorphic to the components of $\cod{1}{L}$, $\codt{2}{L}$, and $\codt{3}{L}$.
\end{lem}
By pushing the boundary strata of this manifold towards the interior, we obtain an embedding of $\Phat(L;0)$ into $ \cP(L;0) $.    This slightly different point of view shall be of use in Section \ref{sec:parallel_cobordism}.
\begin{rem}
We shall consistently use the notation $\tilde{\partial}^{i} \overline{\cR} $ when we consider circle bundles over codimension $i=2$ or $i=3$ corner strata of a moduli space $\cR$ corresponding to breakings of holomorphic spheres.  In all cases, the Gromov compactification which is denoted $\overline{\cR}$ can be replaced by a manifold with corners symbolised by $\hat{\cR}$ admitting $ \tilde{\partial}^{i}  \overline{\cR}$ as a corner stratum of codimension $i-1$. 
\end{rem}
Our goal in the next sections will be to  ``cap off'' the boundary components of $\Phat(L;0)$  (other than $L$ itself) in a controlled manner to ensure that the component containing constant discs is still parallelisable.  We will only succeed if the dimension is congruent to $1$ modulo $4$.

\begin{addendum} \label{add:explicit_computation}
In the case of  the standard inclusion of $S^{2n-1}$ in $\bC^{n} \times \bC \bP^{n-1}$ as the graph of the Hopf fibration, we claim that the complement of a small neighbourhood of
\begin{equation} \cod{1}{\Hopf\left( S^{2n-1} \right)} \cup \cod{2}{\Hopf\left( S^{2n-1} \right)}  \end{equation}
in $\Pbar\left(\Hopf\left( S^{2n-1} \right);0\right)$ has boundary which is diffeomorphic to a disjoint union of copies of $S^{2n-1}$.  In particular, we can smooth the corner strata of the boundary of $\Phat\left(\Hopf\left( S^{2n-1} \right);0\right)$ to obtain a cobordism between $S^{2n-1}$, and other copies of $S^{2n-1}$ indexed by the finite set $\cP\left(\Hopf\left( S^{2n-1} \right);-\beta\right)$.  We give the geometric arguments behind these statements, ignoring all analytic problems which can be resolved without great difficulty in this case.

Fixing an element $u \in \cP(\Hopf\left( S^{2n-1} \right);-\beta)$, our goal is to give an explicit description of the component of $\cod{1}{L}$ which consists of holomorphic discs passing through the boundary of $u$.  We shall write $\codu{1}{\Hopf\left( S^{2n-1} \right)}$ for this manifold.  The invariance of the Hopf map and the complex structure $J_{\alg}$ under the appropriate action of $U(n)$ implies that the projection map 
\begin{equation} \codu{1}{\Hopf\left( S^{2n-1} \right)}\to S^1 \end{equation}
is a fibration.  It suffices therefore to describe the moduli space of holomorphic discs passing though a point $(\pt,\pi_{\Hopf}(\pt))$ on the graph of the Hopf fibration.  We claim that this moduli space is diffeomorphic to a ball, and that its compactification is diffeomorphic to $\bC \bP^{n-1}$.    The precise statement is
\begin{equation}  \label{eq:explicit_description_holomorphic_discs} \parbox{35em}{The moduli space 
\begin{equation*} \cS_{0,1}\left(\Hopf \left( S^{2n-1} \right),\pt;\beta\right) \end{equation*}
of holomorphic discs in $  \bC^{n} \times \bC \bP^{n-1}$, with boundary on the graph of the Hopf fibration, with $1$ mapping to $(\pt,\pi_{\Hopf}(\pt))$, is naturally diffeomorphic to the space of complex lines in $\bC^{n}$ which are not tangent to $S^{2n-1}$ at $\pt$.  Moreover, its Gromov-Floer compactification is diffeomorphic to the space of all complex lines in $\bC^{n}$.}   \end{equation}
To describe the map in one direction, we observe that any complex line passing through $\pt$ and which is transverse to $S^{2n-1}$ can be parametrized so that the unit circle maps to $S^{2n-1}$.  The image of the unit disc under this map gives the first factor of an element of $ \cS_{0,1}(\Hopf \left( S^{2n-1} \right) ,\pt;\beta)$.  To obtain the second factor, we map $D^{2} - \{0 \}$ to $\bC$ by inverting about the unit circle, apply the previously fixed parametrization of our line in $\bC^{n}$, then project to $\bC \bP^{n-1}$ by the $\bC^{\star}$ action on $\bC^{n}$.  The resulting map extends continuously to the origin.  Since the complex structure on $\bC^{n}$ is conjugate to the usual one, while the one on $\bC \bP^{n-1}$ is standard, the resulting map is holomorphic, and it is trivial to check that the product map 
\begin{equation*} D^{2} \to \bC^{n} \times \bC \bP^{n-1} \end{equation*} 
has boundary lying on the graph of the Hopf map.  Note that applying essentially the same construction to the case where the plane is tangent to $ S^{2n-1}$ at $\pt$ yields a ``ghost bubble" on $S^{2n-1}$, together with a holomorphic sphere in $\bC \bP^{n-1}$ passing through $\pi_{\Hopf}(\pt)$, as we expect to find in the Gromov-Floer compactification of the moduli space of holomorphic discs.  We shall not prove that the smooth structure at this boundary stratum is compatible with the one induced from gluing pseudo-holomorphic curves, though this can be done along the lines of the arguments used in Section \ref{sec:more_gluing}.  

We give the barest sketch of the proof that the map we just described is surjective. Start by using the two factors of a holomorphic disc with boundary on $\Hopf(S^{2n-1})$ to construct  a holomorphic sphere in $\bC \bP^{n-1}$ which on the lower hemisphere agrees with the second factor, and on the upper hemisphere with the composition of conjugating the domain, applying the first factor, then projecting along the $\bC^{\star}$ action.  This first implies that the image of such a disc is contained in the product of a complex plane in $\bC^{n}$ with the corresponding line in $\bC \bP^{n-1}$, which reduces the problem to the case $n=2$.  In this case, we can interpret this doubled map as a bi-holomorphism of $\bC \bP^{1}$, and the image of the boundary of the disc under such a map must agree with the unit circle in $\bC$ under stereographic projection after composition with an appropriate translation and dilation.  Geometrically, fixing the image of the boundary to agree with such a circle implies that the boundary of the holomorphic disc in the product $\bC^{2} \times \bC \bP^{1}$ projects to a totally real torus $T^{2} \to \bC^{2}$ which is isotopic, through complex linear maps,  to the product of a circle in each factor of $\bC$ (this torus is usually called the Clifford torus).  Up to re-parametrisation, there is exactly one holomorphic disc (in the appropriate homotopy class) passing through a given point of the Clifford torus, which proves surjectivity.

We conclude that we have a diffeomorphism
\begin{equation} \codu{1}{L} \cong S^1 \times D^{2n-2} \end{equation}
Moreover, while the codimension $2$ stratum is diffeomorphic to $  \CP^{n-2} \times D^{2}$ as in equation \eqref{eq:diffeo_type_codim_2}, a neighbourhood thereof is diffeomorphic to 
\begin{equation} S^{2n-3}\times D^{2} \end{equation}
as we shall see in Equation \eqref{eq:diffeo_type_circle_bundle_codim_2}.  This implies that the boundary of a small neighbourhood of the union of these two strata is obtained by gluing
\begin{equation} S^1 \times D^{2n-2} \cup  S^{2n-3} \times D^{2} \end{equation}
along their common boundary.  The description we have given is sufficiently explicit that the reader is invited to show that this is the standard decomposition of $S^{2n-1}$ into neighbourhoods of an equatorial circle and the orthogonal $S^{2n-3}$.
\end{addendum} 
\subsection{Capping moduli spaces}
The  moduli spaces defined in this section will be constructed starting with moduli spaces of holomorphic discs with boundary on $L$.  Since the family of almost complex structures $\bfJ$ is constantly equal to $J_{\alg}$ near the boundary of $D^2$, we shall drop the almost complex structure from the notation.

Given that $L$ is simply connected, we may choose,  for each exceptional solution $u$, a smooth map 
\begin{equation} \label{eq:choose_capping_disc} c_{u} \co D^2 \to L \end{equation}
such that $c_{u}|S^1 = u|S^1$.  Requirement \eqref{ass:transverse_evaluation} is equivalent to the statement that for each such map, $u|S^1$ is transverse to the evaluation map from $\cS_{0,1}(L; \beta)$, so we may ensure that the same is true for $c_{u}$.   In particular, we can define a smooth manifold whose boundary is diffeomorphic to $\cod{1}{L}$ by considering the fibre product
\begin{equation} \label{eq:definition_capping_moduli_space} \cC(L) = \cS_{0,1}(L; \beta) \times_{L} \left(\cP(L;-\beta) \times D^2\right)   \end{equation}
where the evaluation map on the second factor is
\begin{equation} (u,z) \mapsto c_{u}(z) .\end{equation}

The space $\cC(L)$ is again not compact, but it admits a Gromov-Floer compactification $\Cbar(L)$.  The analysis performed in the previous section implies that the only stratum that needs to be added to $\cC(L)$ is the fibre product
\begin{equation} \label{eq:definition_boundary_stratum_capping_moduli} \codC{2}{L} = \cM_{1}(M;\alpha) \times_{L} (\cP(L;-\beta) \times D^2). \end{equation}

As a slight variation on the construction performed in the previous section, we introduce a line bundle
\begin{equation} \label{eq:normal_bundle_virtual_codim_2} \cL \codC{2}{L} \end{equation}
which is the pullback of $ \cL \cM_{1}(M;\alpha)$ to $ \codC{2}{L}$.  This allows us to define a
circle bundle
\begin{equation} \codtC{2}{L} \end{equation}
whose fibre at a given equivalence class of maps is the unit tangent space to $\bC \bP^{1}$ at the marked point.    We shall presently introduce, in Lemma \ref{lem:compactification_capping_moduli_spaces_corners}, a manifold with corners in which this appears as a codimension $1$ stratum. Note that the boundary of $ \codtC{2}{L}$ is $\codt{3}{L}$.  
\begin{rem} \label{rem:gluing_sphere_at_virtual_disc}
The only difference between  \eqref{eq:normal_bundle_virtual_codim_2} and \eqref{eq:normal_bundle_codimension_2}, is that the former does not involve the tangent space of the disc.  The reason for this is the fact that the Fredholm description of a gluing theorem on $\codC{2}{L}$ corresponds to gluing of a sphere to a ghost disc bubble with one interior marked point and one boundary marked point.  In analogy with \eqref{eq:normal_bundle_codimension_2}, we should strictly speaking define \eqref{eq:normal_bundle_virtual_codim_2}  to be the tensor product of $ \cL \cM_{1}(M;\alpha)$ with the tangent space to this ghost bubble at its interior marked point.  But this second line bundle is canonically trivial.
\end{rem}

In Section \ref{sec:compactify_capping_mod_space}, we shall sketch the proof of
\begin{lem} \label{lem:compactification_capping_moduli_spaces_corners}
There exists a compact manifold with corners $\Chat(L)$ containing  $\cC(L)$ as a complement of a union of strata contained in the boundary, and whose corner strata are respectively diffeomorphic to the components of $\cod{1}{L}$, $\codtC{2}{L}$, and $\codt{3}{L}$.
\end{lem}
We denote the closure of $\cod{1}{L}$ in $\Chat(L)$ by
\begin{equation} \codbC{1}{L} = \cod{1}{L} \cup \codt{3}{L} ,\end{equation} which is a manifold with boundary equipped with a smooth structure by virtue of  its inclusion in $\Chat(L)$.   This space is also the closure $\codb{1}{L}$ of $  \cod{1}{L}  $ in $\Phat(L;0)$.  In Lemma \ref{lem:boundary_smooth_structure_diffeo}, we prove that the two induced smooth structures are diffeomorphic via a diffeomorphism which is the identity on the boundary $\codt{3}{L}$.  We write $\hat{W}(L)$ for the union of $\Chat(L)$ and $\Phat(L;0)$  along this common codimension $1$ boundary stratum.  

As an immediate consequence of Lemmas \ref{lem:existence_manifold_corner} and \ref{lem:compactification_capping_moduli_spaces_corners}  we conclude:
\begin{lem}
$\hat{W}(L)$ is a compact smooth manifold whose boundary is the disjoint union of $L$ with the smooth manifold \begin{equation} \codtC{2}{L}\cup_{\codt{3}{L}} \codt{2}{L}. \end{equation} \noproofe
\end{lem}

We shall now focus on explicitly describing the boundary components of $\hat{W}(L)$.  As in the previous section, the fact that the evaluation map from $\cM_{1}(M;\alpha)$ to $M$ is a submersion with fibres $\bC \bP^{n-2}$ implies that $\codC{2}{L}$ is a disjoint union of copies of 
\begin{equation} D^{2} \times \bC \bP^{n-2} \end{equation}
indexed by exceptional solutions to \eqref{deformedCR}; i.e. by elements of $\cP(L;-\beta)$.  In particular, the components of $\partial \hat{W}(L) - L$  are $S^1$ bundles over $\bC \bP^{n-2}$ bundles over $S^2$.  As we shall see in Proposition \ref{prop:parallel_cobordism}, the parallelisability of the component of  $\hat{W}(L)$ containing constant discs is equivalent to the parallelisability of the total spaces of these bundles.

\subsection{The neighbourhood of the space of sphere bubbles}

In this section, we prove
\begin{prop} \label{parity-product}
The manifold
 \begin{equation} \label{eq:union_spheres_bubbles} \codtC{2}{L}\cup_{\codt{3}{L}} \codt{2}{L}. \end{equation}
is a disjoint union of $S^{2n-3}$ bundles over $S^2$.  If $n$ is odd, its components are diffeomorphic to $S^2 \times S^{2n-3}$, and they are otherwise diffeomorphic to the non-trivial $S^{2n-3}$ bundle over $S^2$.   In particular, if $n$ is odd, the diffeomorphism type of the boundary of $\hat{W}(L)$ is given by 
\begin{equation}  \partial \hat{W}(L) \cong L \amalg (S^2 \times S^{2n-3} \times \cP(L;-\beta)) \end{equation}
\end{prop}

This result will follow easily once we explicitly identify $ \cL \cod{2}{L}$,  $\cL \codC{2}{L}$, and the isomorphism which is used to glue them along $ \cL \cod{3}{L} $.  For simplicity, fix an exceptional curve $u$, and write $ \cL \codu{2}{L}$ and $\cL \coduC{2}{L}$ respectively for the components labeled by $u$.  The complement of the zero-section in $\cL \coduC{2}{L}$  can be naturally identified, as a bundle over $D^2$ with the complement of the zero section of the pullback of the tangent space of $\bC \bP^{n-1}$ under the map $c_u$
\begin{equation} \cL \coduC{2}{L} -  \coduC{2}{L} \cong c_{u}^{*} T \bC \bP^{n-1} - D^{2}. \end{equation} 
Indeed, it is easiest to construct a map from the right hand side to the left hand side, taking a non-zero vector in $ c_{u}^{*} T \bC \bP^{n-1}$ at $c_u(z)$ to the pair consisting of the unique complex line passing through $z$ that it spans, together with the corresponding tangent vector to this complex line at $z$.  In particular, we have an identification between the unit circle bundle on one side, and the unit sphere bundle on the other
\begin{equation} \codtuC{2}{L}  \cong c_{u}^{*} S T \bC \bP^{n-1}. \end{equation}

Similarly, we have an identification between the complements of the zero sections
\begin{equation} \cL \codu{2}{L} - \codu{2}{L} \cong  \left( u^{*} T \bC \bP^{n-1} \otimes_{\bC} TD^{2}  \right) - D^{2},
\end{equation}
inducing a diffeomorphism of sphere bundles
\begin{equation} \label{eq:diffeo_type_circle_bundle_codim_2} \codbu{2}{L}  \cong S \left( u^{*} T \bC \bP^{n-1} \otimes_{\bC} TD^{2}  \right)  \cong S^{2n-3} \times D^{2} \end{equation}
In particular, using the fact that $u$ and $c_u$ agree on the boundary circle, the geometric identification of $ \codt{3}{L} $ as the common boundary of $ \codbu{2}{L}$ and $ \codbuC{2}{L}$ gives a diffeomorphism of sphere bundles over the circle
\begin{equation} \label{eq:glue_boundary_complement_zero_section}c_{u}^{*} S T \bC \bP^{n-1}|_{S^1}  \cong S \left(u^* T \bC \bP^{n-1} \otimes_{\bC} TD^{2}  \right)|_{S^{1}}.
\end{equation}

\begin{figure} 
   \centering
   \includegraphics[width=2in]{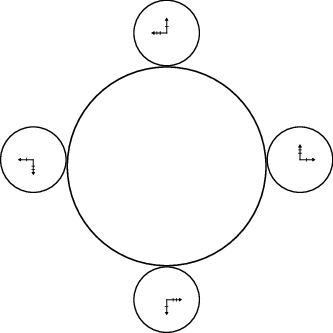} 
   \caption{}
   \label{fig:ghost_bubble_trivialise}
\end{figure}

\begin{lem}
The diffeomorphism \eqref{eq:glue_boundary_complement_zero_section} is induced by an  isomorphism of vector bundles over $S^1$:
\begin{equation} c_{u}^{*} T\CP^{n-1}|_{S^1}  \to u^* T\CP^{n-1}|_{S^1} \otimes_{\bC} TD^2|_{S^1} \end{equation}
which, for a vector $x$ lying over a point $\theta \in S^1$ is given by
\[ x \mapsto x \otimes \partial_{\theta} \]
where $\partial_\theta$ is the tangent vector to the circle.
\end{lem}
\begin{proof}
As discussed in Remark \ref{rem:gluing_sphere_at_virtual_disc}, we should really think of the fibres of $ \cL \codC{2}{L}$  as the tensor product of the tangent space of the holomorphic sphere bubble at the marked point, with the tangent space of a ghost disc bubble at the origin; this second factor is canonically trivial.  The operation of gluing the ghost disc bubble to the exceptional solution $u$ gives a trivialization of the tangent space of $D^2$ near its boundary.  It is evident that this is precisely the trivialization described above (see Figure \ref{fig:ghost_bubble_trivialise}, which shows the trivialization of the tangent space of a ghost bubble attached at various points). 
\end{proof}
Introducing the notation 
\begin{equation} c_u \cup u \co S^{2} \to M \end{equation}
for the map on $S^2$ which is respectively equal to $c_u$ and $u$ on the top and bottom hemisphere, we can now prove the main result of this section

\begin{proof}[Proof of Proposition \ref{parity-product}]
 By the previous Lemma, the component of \eqref{eq:union_spheres_bubbles}  labeled by $u$ is diffeomorphic to the unit sphere bundle in
\begin{equation} \left( c_u \cup u \right)^{*}  T \bC \bP^{n-1} \otimes_{\bC} \cO(1) ,\end{equation}
where $\cO(1)$ is the tautological bundle on $\bC \bP^{1}$, thought of as a complex linear vector bundle (the holomorphic structure is irrelevant).  By construction, $c_u \cup u$ represents the class $- \alpha \in \pi_2(M)$, so that $ \left( c_u \cup u \right)^{*} T\CP^{n-1}$ is the non-trivial bundle precisely if $n$ is odd.  In particular, the unit sphere bundle of $\left( c_u \cup u \right)^{*}  T \bC \bP^{n-1} \otimes \cO(1) $ is the trivial $S^{2n-3}$ bundle over $S^2$ if $n$ is odd, and the non-trivial bundle if $n$ is even.
\end{proof}

\subsection{Parallelizability of the cobordism} \label{sec:parallel_cobordism}
In the previous section, we established that  all the boundary components of  $\hat{W}(L)$ are stably parallelisable whenever $n$ is odd.  We now prove that this is the only obstruction to the stable parallelisability of $\hat{W}(L)$.  

  Following the proof of Lemma \ref{lem:trivial_tangent_bundle_parametrized} it is easy to show the triviality of the tangent space of any component of $\cC(L)$ whose boundary is not empty. Indeed,  let $\cF_{\cC}(L)$ denote the fibred product
\begin{equation}\cF_{\cC}(L) = \cF(L;\beta) \times_{L}  (\cP(L;-\beta) \times D^2)  \end{equation} 
where the first factor maps to $L$ by evaluation at $1$, and the second map is given by
\begin{equation*} (u,z) \mapsto c_{u}(z) .\end{equation*} 

Since $\cC(L)$ consists of holomorphic discs modulo parametrizations, it does  not naturally embed in $\cF_{\cC}(L)$.  Nonetheless, the natural $\Aut(D^2,-1)$ bundle over $\cC(L)$ embeds in $\cF_{\cC}(L)$ as the vanishing set of the $\dbar$ operator
\begin{equation} \dbar_{\cC} \co  \cF_{\cC}(L) \to \cE_{\cC}  \end{equation}
where  $\cE_{\cC}$ is the pullback of $\cE$ from $\cF(L;\beta)$. Since $\Aut(D^2,-1)$ is contractible, we can lift $\cC(L)$ to this space.    As with Lemma \ref{lem:trivial_tangent_bundle_parametrized}, there exists an elliptic operator on $\cF_{\cC}(L)$ 
\begin{equation} \label{eq:dbar_complex_capping} D_{\cC} \co T  \cF_{\cC}(L) \to \cE_{\cC}. \end{equation}
whose index in reduced $K$-theory agrees with the class of tangent space of $\cC(L)$.   Note that the description of $ \cF_{\cC}(L)  $ as a fibre product implies that its tangent space consists of all vector fields along a disc with boundary on $L$ which, when evaluated at $1$, point in a direction tangent to the capping disc.

Note that Lemma \ref{lem:homotopy_theory} implies that the evaluation map
\begin{equation} \cF(L;\beta) \to L \end{equation}
is a fibration with contractible fibres.  Since $D^2$ is contractible, we conclude that $\cF_{\cC}(L)$ is homotopy equivalent to $\cP(L;-\beta)$, i.e. to a finite union of points.  It follows that the tangent space of $\cC(L)$ must vanish in reduced $K$-theory, and hence that any component of $\Chat(L)$ which is not closed is in fact parallelisable.

The main result of this section is the following proposition whose proof we will explain after setting up the notation, and stating some preliminary lemmas:
\begin{prop} \label{prop:parallel_cobordism}
If $n$ is odd, then $\hat{W}(L)$ has stably trivial tangent bundle.  In particular, the component whose boundary contains $L$ is parallelisable.
\end{prop}
\begin{rem}
The idea behind the proof is that $\cF_{\cC}(L)$ is an infinite dimensional manifold with boundary
\begin{equation}
   \partial \cF_{\cC}(L) =  \cF(L;\beta) \times_{L}  (\cP(L;-\beta) \times \partial D^2) 
\end{equation}
and that the map $\cod{1}{L} \to \Phat(L;0)$ given by the inclusion of the boundary can be extended to a map
\begin{equation} \partial \cF_{\cC}(L) \to \cF_{\cP}(L;0) .\end{equation}
allowing us to form a union
\begin{equation} \label{eq:union_along_boundary_banach_spaces} \cF_{\cC}(L) \cup_{\partial \cF_{\cC}(L)}  \cF_{\cP}(L;0) \end{equation}
in which we can embed $\inte \hat{W}(L)$.  Elementary algebraic topology implies that this space has the homotopy type of the wedge of $L$ with a finite collection of $2$-spheres labeled by the exceptional solutions.  If there were a Fredholm complex over this space whose restriction to $\inte \hat{W}(L)$ agreed with the class of the tangent space, then stable triviality of the tangent space would follow from the stable triviality of the tangent space of $L$, together with the triviality of the restriction of the index class of this Fredholm complex to the $2$-spheres.  On the other hand, a choice of section of the $S^2$ bundles described in Lemma \ref{parity-product} represents these generating $2$-spheres, so the triviality of their normal bundle whenever $n$ is odd implies the desired result.

The problem with this naive approach is that the natural Banach bundles over $\cF_{\cC}(L)$ studied for the purpose of the gluing theorem  are ``much larger" than their analogues over $\cF_{\cP}(L;0)$, and hence it seems unlikely that a Banach bundle could be defined on the union \eqref{eq:union_along_boundary_banach_spaces} which would restrict to the original object on both sides.  This problem is already familiar from Morse theory, where the Banach bundles over boundary strata of the moduli space of gradient trajectories are modeled after two copies of maps from $\bR$ to $\bR^{n}$, whereas the interior strata are modeled after only one such copy.  It is reasonable to expect that the theory of polyfolds (see \cite{HZW}), specifically the notion of a filler which is concretely given by anti-gluing in this situation,  would provide a resolution to this problem that would turn the hypothetical discussion of the previous remark into an honest proof.  We prefer to give an alternative proof, less general though more elementary by adding a (trivial) finite dimensional vector bundle.  Such an idea has a long history, going back at least to Atiyah and Singer's proof of the family index theorem (see Proposition 2.2 of \cite{ASIV}), and can be thought of as an infinitesimal version of the obstruction bundle used in the construction of Kuranishi spaces.

\end{rem}

We introduce a triple of (compact finite-dimensional) CW-complexes $ \cX(L)$, $\partial \cX(L) $, and $\cY(L) $  which fit in a diagram
\begin{equation} \label{eq:big_diagram_spaces} \xymatrix{
\Chat(L) \ar[d] & \codbC{1}{L} \cong \codb{1}{L} \ar[d] \ar[r] \ar[l]   & \Phat(L;0) \ar[d]  \ar[dr] & \\
\cX(L) \ar[d] & \partial \cX(L) \ar[l] \ar[d] \ar[r]   &  \cY(L) \ar[d] & \cP(L;0) \ar[ld] \\
\cF_{\cC}(L) & \partial \cF_{\cC}(L) \ar[l]  \ar[r] & \cF_{\cP}(L;0). & }  \end{equation}
Note that we're working with the conventions of the paragraph following Lemma \ref{lem:existence_manifold_corner}, so that the inclusion of $ \Phat(L;0)  $ into $\cF_{\cP}(L;0)$ factors through the inclusion of $\cP(L;0)$ into the space of maps defining it.  Similarly, by Lemma \ref{lem:compactification_capping_moduli_spaces_corners}, we may fix an embedding $\Chat(L)  \to \cC(L)$ which is the identity away from a collar of $\codtC{2}{L}$; this uses the fact that a manifold with corners can be embedded in the complement of some of its corner strata such that the embedding is the identity away from the removed strata (this mildly generalises the well-known fact that a manifold with boundary can be embedded in its interior).

In Sections \ref{sec:pre-gluing_smooth} and \ref{sec:construct_Y},   we shall prove that our construction satisfies the following property (see, in particular, Lemma \ref{lem:construct_Z}): 
\begin{equation} \label{eq:triple_CW_complexes_exists} \parbox{35em}{We can choose the triple of CW complexes above such that 
\begin{equation*} \cZ(L) \equiv \cX(L) \cup_{ \partial \cX(L)}  \cY(L)\end{equation*}
is homotopy equivalent to $ \vee_{u} S^{2} \vee L$.}
\end{equation}

For each exceptional disc $u$, the linearisation of $\dbar_{\cP}$ gives an isomorphism
\begin{equation} \label{eq:linearised_dbar_P} D_{\cP}|u \co T  \cF_{\cP}(L;-\beta)| u \to \cE_{\cP}|u.  \end{equation}
Over the corresponding component of $\partial \cF_{\cC}(L) $, we consider the direct sum
\begin{equation}
 T  \cF_{\cP}(L;-\beta) \oplus   T  \cF_{\cC}(L)|u
\end{equation}
which we call the extended tangent space and denote
\begin{equation}
T^{ext} \cF_{\cC}(L). 
\end{equation}
With this in mind, we define the \emph{extended $\dbar$ complex} of $\cF_{\cC}(L)$ 
\begin{equation} \label{eq:extended_CR} D_{\cC} \co T^{ext} \cF_{\cC}(L)\to  \cE^{ext}_{\cC}. \end{equation}
to be the direct sum of the above complex with \eqref{eq:dbar_complex_capping}, where $   \cE^{ext}_{\cC}$, on the component labelled by $u$, is the direct sum
\begin{equation}
 \cE_{\cC} \oplus   \cE_{\cP}|u .
\end{equation}

If we restrict this complex to the tangent space of $\partial \cF_{\cC}(L)$, we obtain the extended $\dbar$ complex of $\partial  \cF_{\cC}(L)$
\begin{equation} \label{eq:extended_CR-boundary} D_{\partial \cC} \co T^{ext} \partial \cF_{\cC}(L)\to  \cE^{ext}_{\cC}|\partial  \cF_{\cC}(L) . \end{equation}

The index of this operator  gives a class  in reduced $K$-theory that restricts to the tangent space of $ \partial \cC(L)$.  More precisely, since the inclusion
\begin{equation} \codbC{1}{L} \to \partial \cF_{\cC}(L) \end{equation}
is defined by composing a section of an $\Aut(D^2,-1)$ bundle with the inclusion of the zero locus of $\dbar_{\cP}$,  the restriction of the kernel of the operator $D_{\partial \cC}$ to $\codbC{1}{L}$ is isomorphic to a direct sum
\begin{equation} \label{eq:decomposition_kernel} \aut(D^2,-1) \oplus T\codbC{1}{L}  \tilde{\to} \ker D_{\partial \cC}| \codbC{1}{L}\end{equation}
where $\aut(D^2,-1)$ is the rank $2$ vector space of holomorphic vector fields on $D^2$ which vanish at $-1$.  The above map takes such a vector field to its image under the differential of the holomorphic map.  We will consider the decomposition
\begin{equation} \aut(D^2,-1) \cong \langle \partial_{s} \rangle \oplus \langle \partial_{p} \rangle \end{equation}
where the first factor is the $\bR$-summand in $\aut(D^2,-1)$ coming from translation of the strip $\I{-\infty}{+\infty}$ after identification with $D^{2} - {\pm 1}$ by the map $\xi$, and the second is generated by the unique vector field with a zero at $-1$, whose value is $\partial_{\theta}$ at $1$.

On the other side,  $\cF_{\cP}(L;0)$ is also equipped with a linearisation of its $\dbar$ operator, which we shall extend in Equation \eqref{eq:add_rotation_to_linearisation_dbar_operator_parametrized}  to an operator
\begin{equation} \label{eq:add_rotations} D_{\cP}^{\langle \partial_{\theta} \rangle} \co \langle \partial_{\theta} \rangle \oplus T \cF_{\cP}(L;0) \to \cE_{\cP}, \end{equation}
where $\langle \partial_{\theta} \rangle$ is a rank $1$ vector space.  This vector space is to be thought of as generated by infinitesimal rotations of the disc.  The appearance of $\langle \partial_{\theta} \rangle$ is an indication of a difficulty we shall face proving the gluing theorem for  $\cod{1}{L}$.   Our proof relies on applying the implicit function theorem in $\cF_{\cP_{0,1}}(L;0)$ (see Corollary \ref{cor:surjectivity_extended_map}), rather than working directly with $\cF_{\cP}(L;0)$ as one might expect.   The vector space $\langle \partial_{\theta} \rangle$ is a small remnant of this fact.  We state a key property which shall be obvious from the construction, and which is needed for the statements below to make sense (see Lemma \ref{lem:add_rotation_direction_to_P})
\begin{equation}
\parbox{35em}{The restriction of $D_{\cP}^{\langle \partial_{\theta} \rangle}$ to $\cP(L;0)$ vanishes on the factor $ \langle \partial_{\theta} \rangle$.  In particular, we have an isomorphism
\begin{equation*}  \ker D_{\cP}^{\langle \partial_{\theta} \rangle} \cong   \langle \partial_{\theta} \rangle \oplus  T \cP(L;0) . \end{equation*} }
\end{equation}

Since the tangent bundles of $\Phat(L;0)$ and $  \Chat(L)  $ are both stably trivial, the isomorphism
\begin{equation}\label{eq:iso_tangent_spaces_gluing} T \codbC{1}{L} \to T \codb{1}{L} .\end{equation}
together with the identification of the normal bundles of these boundaries determines whether the union of $\Phat(L;0)$ and $\Chat(L)$ along $\codbC{1}{L} \cong \codb{1}{L}$ has stably trivial tangent bundle.  

\begin{lem} \label{lem:iso_stabilizations_bundles}
There exists a finite dimensional vector space $V_{\cZ}$ and an extension of the operators $D_{\partial \cC} $ of Equation \eqref{eq:extended_CR-boundary} and $D_{\cP}^{\langle \partial_{\theta} \rangle} $ of Equation \eqref{eq:add_rotations} to surjections
\begin{align} \label{eq:dbar_capping_surjective} D_{\cX} \co V_{\cZ} \oplus T^{ext} \cF_{\cC}(L)  | \cX(L) & \to \cE^{ext}_{\cC} | \cX(L)   \\ 
D_{\partial \cX} \co V_{\cZ} \oplus T^{ext} \partial \cF_{\cC}(L)  | \partial \cX(L) & \to \cE^{ext}_{\cC} | \partial \cX(L)    \\ 
\label{eq:dbar_para_surjective} D_{\cY} \co V_{\cZ} \oplus  \langle \partial_{\theta} \rangle \oplus T \cF_{\cP}(L;0)| \cY(L) & \to \cE_{\cP} | \cY(L).
\end{align}
Moreover, there exists an isomorphism \begin{equation} \label{eq:map_kernels} \ker(D_{\partial \cX} )|\partial   \cX(L)  \to \ker(D_{\cY})|\partial   \cX(L) \end{equation}
between the kernels of the Cauchy-Riemann operators  over $\partial   \cX(L)$.  If we restrict this isomorphism further to $\codbC{1}{L}$, then  \eqref{eq:map_kernels} decomposes as a direct sum of the identity on $V_{\cZ}$ with a map
\begin{equation}  \aut(D^2,-1) \oplus T\codbC{1}{L}  \to  \langle \partial_{\theta} \rangle \oplus  T \cP(L;0) \end{equation} 
which is isotopic through isomorphisms to a map satisfying the following three properties:
\begin{equation} \label{eq:tangent_spaces_boundary_agree} \parbox{35em}{The image of $T\codbC{1}{L}$ lies in $T \cP(L;0)$, and agrees with the map induced  by the inclusions \begin{equation*} \codbC{1}{L}\to \Phat(L;0) \to  \cP(L;0) .\end{equation*}} \end{equation}
\begin{equation} \label{eq:superfluous_directions_agree} \parbox{35em}{The vector field $\partial_{p}$ maps to $  \partial_{\theta}$.}   \end{equation}
\begin{equation} \label{eq:normal_directions_agree} \parbox{35em}{The image of the vector field $\partial_{s}$ lies in $T \cP(L;0)$, and is an outwards pointing vector on the boundary of  $ \Phat(L;0)$.}   \end{equation}
 \end{lem}

The proof is postponed until Section \ref{sec:triviality_tangent_space}, but we can now proceed with the proof of the main result of this section.
\begin{proof}[Proof of Proposition \ref{prop:parallel_cobordism}]
The kernel of \eqref{eq:dbar_capping_surjective} is a vector bundle over $ \cX(L)$ whose restriction to $\partial \cX(L)$ is isomorphic to the direct sum of $ \ker(D_{\partial \cX}) $ with the normal bundle of the inclusion $S^1 \to D^2$, which is a trivial bundle.  Fixing such a trivialization, the isomorphism of Equation \eqref{eq:map_kernels} extends to an isomorphism
\begin{equation} \label{eq:isomorphism_kernels_boundary_with_normal} \ker(D_{\cX})|\partial   \cX(L)  \cong \ker(D_{\cY})|\partial   \cX(L)  \oplus \bR. \end{equation}
Extending $\bR$ trivially to $\cY(L)$, this isomorphism defines a vector bundle $E_{\cZ}$ on $\cZ(L)$ which restricts on $\cX(L)$ and $\cY(L)$ to the two vector bundles in the above equation.

We claim that the restriction of this vector bundle to $\hat{W}(L)$ is isomorphic to the direct sum of the tangent space of $\hat{W}(L)$ with a trivial vector bundle.   To see this, decompose $\ker(D_{\cX})|\codbC{1}{L}$ as
\begin{equation}  V_{\cZ} \oplus \langle \partial_{s} \rangle \oplus \langle \partial_{p} \rangle    \oplus T\codbC{1}{L} \oplus N_{\Chat(L)} \codbC{1}{L} \end{equation}
and consider the isomorphism which acts trivially on all the factors except the rank $2$ vector bundle $\langle \partial_{s} \rangle  \oplus  N_{\Chat(L)} \codbC{1}{L}$ on which it acts by rotation.  Note that this isomorphism is isotopic to the identity, so the result of gluing $\ker(D_{\cX})| \codbC{1}{L}$ and $\ker(D_{\cY}) \oplus \bR| \codbC{1}{L}$ by composing \eqref{eq:isomorphism_kernels_boundary_with_normal} with this rotation is a vector bundle $E_{\hat{W}(L)}$ which is isomorphic to the restriction of $E_{\cZ}$.

The bundle $E_{\hat{W}(L)}$ contains a sub-bundle of rank $2 + \rk V_{\cZ} $ which, upon restriction to $\Chat(L)$, decomposes as a direct sum
\begin{equation} V_{\cZ} \oplus \aut(D^2,-1)  \end{equation}
and on $\Phat(L;0)$ as a direct sum
\begin{equation} V_{\cZ} \oplus  \langle \partial_{\theta} \rangle \oplus \bR .\end{equation}
By Equations \eqref{eq:superfluous_directions_agree}, \eqref{eq:normal_directions_agree} and \eqref{eq:isomorphism_kernels_boundary_with_normal}, this bundle is trivial.  Moreover, the quotient is isomorphic to the tangent space of $\hat{W} (L) $ by Equation  \eqref{eq:tangent_spaces_boundary_agree}.

Since $\cZ(L)$ is homotopy equivalent to ${\bigvee}_{u} S^{2} \vee L $ by property \eqref{eq:triple_CW_complexes_exists}, the stable triviality of $T\hat{W}(L)$ therefore follows from the triviality of its restriction to the spheres representing $S^2$.  This was already proved  in Proposition \ref{parity-product}.
\end{proof}

\subsection{Constructing a compact parallelisable manifold}

In this section, we prove that if $n$ is odd, there exists a compact parallelisable manifold $W(L)$ with boundary $L$.  By considering the component of $\hat{W}(L)$ containing the constant discs, we obtain from the previous section a parallelisable cobordism from $L$ to a disjoint union of copies of $S^2 \times S^{2n-3}$.  In this section, we prove that up to a connect sum with a framed standard sphere, every stable framing on $S^2 \times S^{2n-3}$ arises as the induced framing on the boundary of a parallelisable $2n$-dimensional manifold.  By choosing such a filling, we shall prove the existence of a manifold $W(L)$ satisfying the desired conditions.

The first step is a computation in homotopy theory.
\begin{lem} \label{lem:framings_product_spheres}
The space of trivializations of the direct sum of the tangent bundle of $S^2 \times S^{2n-3}$ with a trivial bundle of rank $1$ can be identified as a set with the product
\begin{equation} \label{eq:direct_sum_decomposition_trivialisation}  \pi_{2n-3}(O(2n)) \times \pi_{2n-1}(O(2n)).\end{equation}
\end{lem}
\begin{proof}
Recall that the inclusion of $O(m)$  into $O(m+1)$ is $m+1$-connected, so there is an identification between the sets of homotopy classes of maps
\begin{equation}[S^2 \times S^{2n-3}, O(2n)]  \cong   [S^2 \times S^{2n-3}, O] ,\end{equation}
where $O$ is the direct limit of the orthogonal groups $O(m)$.  Using the fact that $O$ is an infinite loop space, a classical computation (see Theorem 4.4 of Chapter X in \cite{whitehead}) shows that $[S^2 \times S^{2n-3}, O]$ admits a natural group structure with respect to which it is isomorphic to an extension of the direct sum $\pi_{2} ( O )\oplus \pi_{2n-3} (O)$ by $\pi_{2n-1} (O)$.  Since $\pi_2 (O)$ vanishes, we obtain the desired result.
\end{proof}
If we fix the stable framing of $S^2 \times S^{2n-3}$ coming from the product inclusion in $\bR^{3} \times \bR^{2n+2}$ as a basepoint,  we may write every trivialization of the tangent bundle of $S^2 \times S^{2n-3}$ as $(a, b)$, with $a \in \pi_{2n-3}(O)$ and $b \in \pi_{2n-1}(O)$. Since the group of stable framings of a sphere is non-trivial, taking the connect sum with a framed sphere acts on the set of framings on any stably parallelisable manifold. In general, this action is far from transitive, and the next result describes the cosets of the action when the manifold is $S^2 \times S^{2n-3}  $:
\begin{lem} \label{connect-sum}
The set of trivializations of the stable tangent bundle of $S^{2n-1}$ acts on the set of trivializations of the stable tangent bundle of $S^2 \times S^{2n-3}$ by the connect sum operation. Choosing the standard trivializations coming from the inclusion in $\bR^{2n}$ as a basepoint on the first set, this action is given by 
\begin{align}\pi_{2n-1}(O) \times   \pi_{2n-3}(O) \times \pi_{2n-1}(O) & \to  \pi_{2n-3}(O) \times \pi_{2n-1}(O) \\
(c,(a,b)) & \mapsto (a,b+c)\end{align}
in the notation of Equation \eqref{eq:direct_sum_decomposition_trivialisation}. \noproof
\end{lem}
Among the framings of $S^2 \times S^{2n-3}$, those of the form $(a,0)$ are pulled back from $S^{2n-3}$.  
\begin{lem}  \label{lem:extend_to_disc}
Every trivialization of the form $(a,0)$ arises as the restriction to the boundary of a framing of the tangent space of $D^3 \times S^{2n-3}$. \noproof
\end{lem}
This suggests that we consider the manifold
\begin{equation}  D^3 \times S^{2n-3} \times \cP(L;-\beta) \end{equation}
whose boundary is  $S^2 \times S^{2n-3} \times \cP(L;-\beta)$. Let us write $W(L)$  for the component of  
\begin{equation}\hat{W}(L) \cup_{ S^2 \times S^{2n-3} \times \cP(L;-\beta)} \left(   D^3 \times S^{2n-3}  \times \cP(L;-\beta) \right)\end{equation}
containing constant loops, and  prove the main Theorem of this paper:
\begin{thm}
If $n$ is odd, there exists a choice of trivialization of the tangent space of the component of $\hat{W}(L)$ containing $L$ which extends to $W(L)$.
\end{thm}
\begin{proof}
By Lemma \ref{lem:extend_to_disc}, it suffices to produce a trivialization of the tangent space of the desired component of $\hat{W}(L)$ which, upon restriction to each boundary component diffeomorphic $S^2 \times S^{2n-3}$ corresponds to a trivialization of the form $(a,0)$.  In order to achieve this, we pick an arbitrary trivialization, and a collections of paths $\gamma$ in $\hat{W}(L) $ starting at  $L$ and ending on each different boundary component.   The connect sum of  $\hat{W}(L)$  and  copies of $[0,1] \times S^{2n-1}$  along the paths $\gamma$ and $(t,1)$ is again diffeomorphic to  $\hat{W}(L)$.

By assumption, all boundary components other than $L$ are diffeomorphic to $ S^2 \times S^{2n-3} $.  If the induced trivialization on a boundary component $S^2 \times S^{2n-3}$ is given by a homotopy class $(a,b)$, then we equip the corresponding copy of $[0,1] \times S^{2n-1}$ with a trivialization of its stable tangent bundle induced by $-b \in \pi_{2n-1}(O)$.   After connect sum, Lemma \ref{connect-sum} implies that the relevant boundary component $S^2 \times S^{2n-3}$ will have a trivialization of the form $(a,0)$, which can be extended to $D^3 \times S^{2n-3}$ by Lemma \ref{lem:extend_to_disc}.
\end{proof}

\section{Transversality} \label{sec:transversality}

In this section, we discuss the proof of some of the transversality results used in the construction of the moduli space $\cP(L;0)$.   Following the usual technical approach, this entails replacing all spaces of smooth maps or sections with appropriate Banach spaces, which will be Sobolev space completions except for the spaces of perturbations of the almost complex structure for which we shall use the Banach spaces of smooth functions introduced by Floer in Section 5 of  \cite{floer}.  We begin by checking that our requirements on the Lagrangian $L$ can be realized after appropriate perturbations.

\subsection{Detail of setup}
We fix the embedding of  $L$ as follows:
\begin{lem}  \label{eq:hamilt_perturb_L_regularity}
After an arbitrarily small perturbation of the Lagrangian embedding $L \subset M$, the following properties can be assumed to hold:
\begin{enumerate}
\item  all the moduli spaces of $J_{\alg}$-holomorphic discs with boundary on $L$ in class $k \beta$ are regular if $k \leq 1$, and
\item $L$ is a real analytic submanifold. \noproof
\end{enumerate}
\end{lem}

In fact, the moduli space of $J_{\alg}$ holomorphic discs in class $k \beta$ is automatically empty if $k < 0$ since $\omega$ integrates to a negative number on curves of this class; the existence of holomorphic curves in such classes would contradict positivity of energy.  The same argument shows that $J_{\alg}$ holomorphic discs  in the trivial homotopy class are constant, and constant discs are automatically regular.

Since $\beta$ is not divisible, all holomorphic discs in class $\beta$ are necessarily simple (see \cites{OK,lazzarini}).  In \cite{oh-perturb-boundary}, Oh proved that the moduli space of simple holomorphic discs is regular for a generic choice of $L$.  Since the space of real analytic embeddings of a Lagrangian submanifold is dense in the space of smooth embeddings, the second property can be simultaneously achieved.

We've assumed real analyticity only in order to conclude the following result, which can likely be proved under the weaker conditions (see Lemma 4.3.3 in \cite{MS})
\begin{cor} \label{cor:totally_geodesic_hermitian_metric}
There exists a metric $g$ on $M$ satisfying the following conditions 
\begin{enumerate}
\item $L$ is totally geodesic with respect to $g$, and
\item $g$ is a hermitian metric for the standard complex structure $J_{\alg}$, and
\item There exists a neighbourhood of $L$ which admits an isometric anti-holomorphic involution preserving $L$.
\end{enumerate}
\end{cor}
\begin{proof}
First, note that $M$ has holomorphic charts near every point of $L$ which map real points to $L$ (See, for example, Lemma 2.2 in \cite{wells}), so that $M$ admits a holomorphic anti-involution defined near $L$.  Consider a real analytic embedding of $L$ in $\bR^{N}$ for some sufficiently large $N$, which automatically extends uniquely to a holomorphic embedding of a neighbourhood of $L$.  The uniqueness of the extension implies that the embedding of this neighbourhood  is equivariant with respect to the $\bZ/2\bZ$ action which acts anti-holomorphically on the source and the target. The pull-back of the standard metric gives a metric $g$ in a neighbourhood of $L$ satisfying the desired conditions.  Indeed, local geodesics must be invariant under the holomorphic anti-involution, so must lie on $L$, which implies that $L$ is totally geodesic.  One can then extend $g$ arbitrarily to a hermitian metric on the rest of $M$.\end{proof}
\begin{rem}
Let us briefly explain the reasons behind this choice of metric:  Consider a map $u \co D^2 \to M$ which takes the boundary to $L$, and $X$ a $TM$-valued vector field along $u$ with boundary conditions on $TL$.  The first property satisfied by $g$ implies that the exponential map applied to $X$ gives a new map with boundary on $L$.  This will be used to construct charts for a Banach manifold  $\cF^{1,p}(L)$ of maps from $D^2$ to $M$ with $L$ as boundary conditions, replacing $\cF(L)$.

To study holomorphic discs, we will also have to construct $\cE^p(L)$, a Banach bundle over $\cF^{1,p}(L)$ with fibres $L^p$ sections of $u^{*}(TM) \otimes_{\bC} \Omega^{0,1} D^2$ and eventually a  linearisation of the $\dbar$ operator
\begin{equation}D \co  T\cF^{1,p}(L) \to \cE^p(L).\end{equation}
This entails choosing a connection which preserves the almost complex structure.  Since the Levi-Civita connection of an almost hermitian connection preserves the almost complex structure, the choice of metric above determines a choice of linearization $D$ which can be expressed in a relatively simple form in term of the Levi-Civita connection.

Finally, the fact that a neighbourhood of $L$ admits an anti-holomorphic involution implies that we may double maps with boundary on $L$, whose image is contained in this neighbourhood, in order to obtain closed holomorphic curves.  In particular, decay estimates for the $C^1$-norm of holomorphic cylinders of small energy will easily imply the analogous result for holomorphic strips with Lagrangian boundary conditions.
\end{rem}

\subsection{Varying almost complex structures.}
Recall that $\sJ$ is the space of almost complex structures compatible with the symplectic form $\omega$, and which agree to infinite order with $J_{\alg}$ on the boundary of $M$.  Let $\bfepsilon = \{ \epsilon_k \}_{k =0}^{\infty}$ be a sequence of positive numbers, and define $\sJ_{\bfepsilon}(D^2)$ to be the space of maps from $D^{2}$ to $\sJ$ which can be written as
\begin{equation} \label{eq:complex_deformation} J(z) = J_{\alg} \exp(-J_{\alg} K(z))\end{equation}
where $J_{\alg}$ is the standard complex structure on $M$ and $K$ is a map from $D^{2}$ to the space of smooth endomorphisms of the tangent bundle of $M$
\begin{equation} K \co D^{2} \to End(TM) \end{equation}
satisfying the following conditions
\begin{align} \label{eq:disc_support} & \parbox{35em}{$K(z)$ vanishes if $|z| \geq 1/2$} \\
\label{eq:boundary_condition} & \parbox{35em}{$K(z)$ vanishes to infinite order on the boundary of $M$} \\
\label{eq:complex_structure} &\parbox{35em}{$K(z)$ anti-commutes with $J_{\alg}$ and $\omega(K(z)v,w) = \omega(v,K(z)w)$} \\
\label{eq:boundedness_varying} &\parbox{35em}{$K$ is bounded  in Floer's $\| \_ \|_{\bfepsilon}$ norm}  \\
\label{eq:tameness} &\parbox{35em}{$\omega$ is positive on every $J(z)$ complex plane.}
\end{align}
For the penultimate condition, we now think of $K$ as a section of a bundle over $D^2 \times M$, and  require the convergence of the sum
\begin{equation} \sum_{k =0}^{\infty} \epsilon_{k} \left|D^{k}  K\right|,\end{equation}
where $ \left|D^{k}  K \right|$ is the supremum of the norm of the higher co-variant derivatives of $K$ thought of as sections of bundles over $D^2 \times M$. These derivatives can be taken, for example, with respect to the Levi Civita connection of the metric $g$ on $M$ introduced in the previous section. 
\begin{rem}
The requirement that $K(z)$ vanish for points whose norm is larger than $1/2$ is completely arbitrary; vanishing outside any open pre-compact subset of the open disc would have served us just as well for our later purposes.
\end{rem}

The space of sections $K$ satisfying Conditions  \eqref{eq:disc_support}-\eqref{eq:boundedness_varying}  is a vector space on which $\| \_ \|_{\bfepsilon}$ defines a complete norm.  Since $\| \_ \|_{\bfepsilon}$ is stronger than the pointwise $C^0$-norm, the remaining condition \eqref{eq:tameness} is an open condition, proving the fact that $\sJ_{\bfepsilon}(D^2)$ is a Banach manifold.

A slight generalization of Floer's Lemma 5.1 in \cite{floer} implies:

\begin{lem} \label{lem:floer_banach_manifold_dense_L2} If $\bfepsilon$ decays sufficiently fast, then $\sJ_{\bfepsilon}(D^2)$ is a Banach manifold (in fact, an open subset of a Banach space) which is dense  in an $L^2$ neighbourhood of the constant almost complex structure $J_{\alg} \in \sJ(D^{2})$. \noproof
\end{lem}

For any class $k\beta \in \pi_{2}(M,L)$, we replace $\cF(L;k \beta)$ with $\cF^{1,p}(L;k \beta)$, the space of $W^{1,p}$ maps in  homotopy class $k \beta$ from $D^2$ to $M$ with boundary values on $L$ for an integer $p$ which is sufficiently large.   The Sobolev embedding theorem implies that this class of maps is continuous, so that the requirement that the boundary lie on $L$ makes sense.  The tangent space is naturally a Banach bundle whose fibre at a map $u$ are $W^{1,p}$ sections of the pullback of $TM$ whose values are restricted to $TL$ on the boundary, i.e
\begin{equation} T_{u} \cF^{1,p}(L;k \beta) = W^{1,p} \left((D^2,S^1), (u^{*}TM, u^{*} TL) \right) .\end{equation}
As in Equation \eqref{eq:parametrized_smooth_space}, we write $\cF^{1,p}_{\cP}(L;k \beta) $ for the product of $\cF^{1,p}(L;k \beta)$ with $[0,+\infty)$.

We consider the Banach bundle  
\begin{equation} \label{universal-space} \xymatrix{ \cE^{p}_{\sJ_{\bfepsilon}(D^2) \times \cP}  \ar[d] \\   \sJ_{\bfepsilon}(D^2) \times \cF^{1,p}_{\cP}(L;k \beta).} \end{equation}
whose fibre at a map $(\bfJ,R,u)$ are $L^p$ sections of the space of complex anti-linear $1$-forms on $D^{2}$ valued in $TM$ equipped with varying almost complex structure $J(z)$:
\begin{equation} \label{eq:L-p-sections} L^{p} \left(u^{*}(TM) \otimes_{\bC, J} \Omega^{0,1} D^2 \right). \end{equation}
To describe $\cE^{p}_{\sJ_{\bfepsilon}(D^2) \times \cP}$  as a Banach bundle, we must identify the vector spaces \eqref{eq:L-p-sections} for nearby maps and nearby points of $ \sJ_{\bfepsilon}(D^2)$.  In the simplest case, starting with a deformation $J(z)$ of $J_{\alg}$ as in Equation \eqref{eq:complex_deformation} and  a vector field $X$ along a map $u \in \cF^{1,p}(L;k \beta)$, we define 
\begin{align} \label{eq:trivalisation_L_p_anti-holomorphic}
L^{p} \left(u^{*}(TM) \otimes_{\bC, J_{\alg}} \Omega^{0,1} D^2 \right) & \tilde{\to} L^{p} \left(u_{X}^{*}(TM) \otimes_{\bC,J} \Omega^{0,1} D^2 \right) \\
Y & \mapsto (\tilde{\Pi}_{u}^{u_X} Y) \exp(J_{\alg} K(z))
\end{align}
Here, $\tilde{\Pi}$ denotes parallel transport,  from $u$ to $u_X$, with respect to the $J_{\alg}$-complex linear connection determined by $g$.   We multiply by $\exp(J_{\alg} K(z))$ because the complex tensor product in the right hand side is performed with respect to $J(z)$, rather than $J_{\alg}$.

The $\dbar_{\cP}$ operator of Equation \eqref{dbar-operator} extends to a section $  \dbar_{\sJ_{\bfepsilon}(D^2) \times \cP}$ of \eqref{universal-space} where the $(0,1)$ part of $du - \gamma_{R} \otimes X_H$ is taken with respect to the varying almost complex structure $J(z)$.  By using the connection $\nabla$ induced by the metric $g$,  we may write  the linearisation of this operator at a zero of the form $( J_{\alg}, R, u)$ as an operator $D_{\sJ_{\bfepsilon}(D^2) \times \cP  } $ given by:
\begin{align} \label{eq:linearisation_dbar_perturb_complex} 
T \sJ_{\bfepsilon}(D^2)  \times  T \cF_{\cP}^{1,p}(L)  & \to \cE^{p}_{\sJ_{\bfepsilon}(D^2) \times \cP} \\ \notag
(K,0,0) & \mapsto K \circ \left(du - \gamma_{R} \otimes X_{H}\right) \circ j \\ \notag
(0,\partial_{R},0) & \mapsto \left( \frac{d \gamma_{R}}{dR} \otimes X_{H} \right)^{0,1} \\ \notag
(0,0,X) & \mapsto \left( \nabla X - \gamma_{R} \otimes \nabla_{X} X_{H} \right)^{0,1} - \frac{1}{2} J_{\alg} \nabla_{X} J_{\alg} \partial_{\cP} (R,u),
\end{align}
where the operator $ \partial_{\cP} $ is
\begin{equation}  \partial_{\cP} (R,u) = \left( du - \gamma_{R} \otimes X_{H} \right)^{1,0}.\end{equation}

Given a point $z \in D^{2}$ such that $|z|< 1/2$ and a positive number $\rho \leq 1/2-|z|$, we introduce the subspace $\sJ_{\bfepsilon}(B_{\rho}(z))  \subset \sJ_{\bfepsilon}(D^2)$ consisting of deformations supported in a disc of radius $\rho$ about $z$.  

\begin{lem} \label{lem:surjectivity_parametrized_moduli_space}
The restriction of $D_{\sJ_{\bfepsilon}(D^2) \times \cP  } $ to
 \begin{equation}  T \sJ_{\bfepsilon}(B_{\rho}(z)) \times  T \cF_{\cP}^{1,p}(L)  \end{equation}
is surjective at every zero of $ \dbar_{\sJ_{\bfepsilon}(D^2) \times \cP}$.  In particular, the universal moduli space
\begin{equation}  \cP_{k \beta}(L, \sJ_{\bfepsilon}(D^2) ) \end{equation}
is a smooth Banach manifold. \noproof
\end{lem}
\begin{cor} \label{cor:parametrized_moduli_manifold_boundary}
If $\cS(L; k \beta,  J_{\alg})$ is regular, then a generic family $\bfJ = \{\bfJ_{R} \in \sJ_{\bfepsilon}(D^2) \}_{R \in [0,+\infty)}$ starting at $\bfJ_{0} \equiv J_{\alg}$ defines a parametrized moduli space
\begin{equation}  \cP(L; k \beta, \bfJ) \end{equation}
which is a smooth manifold of the expected dimension with boundary $\cS(L; k \beta, J_{\alg})$. 
\end{cor} 
\begin{proof} Thinking of $\bfJ$ as a smooth map $[0,+\infty) \to \sJ_{\bfepsilon}(D^2)$, we see that $\cP(L; k \beta, \bfJ)$ is a fibre product
\begin{equation} [0,+\infty) \times_{ \sJ_{\bfepsilon}(D^2) \times [0,+\infty)}  \cP(L; k \beta,  \sJ_{\bfepsilon}(D^2) ) .\end{equation}
The regularity assumption on $\cS(L;k \beta,  J_{\alg})$ is equivalent to the fact that $J_{\alg}$ is a regular point of the projection $\cP(L; k \beta,  \sJ_{\bfepsilon}(D^2) ) \to \sJ_{\bfepsilon}(D^2) \times [0,+\infty) $, so the result is a standard application of Sard-Smale. 
\end{proof}

In preparation for the proof that $\Phat(L;0, \bfJ)$  is a smooth manifold with corners, we must know that the  other strata of $\Pbar(L; 0, \bfJ)$ are also smooth manifolds  for a generic choice of $\bfJ$, the.  Again, the results we need are standard, and some form thereof already appears in the literature (see, e.g. \cite{FHS} or Lemma 2.5 of \cite{seidel-LOS}):

\begin{lem} \label{lem:transverse_evaluation_parametrized}\label{lem:one_exceptional_solution_at_a_time}
For a generic path $\bfJ$ starting at the constant almost complex structure $J_{\alg}$, the evaluation map
\begin{equation} \cP_{0,1}(L ;-\beta,\bfJ) \to L  \end{equation}
 is transverse to $ \cS_{0,1}(L;\beta, J_{\alg}) \to L$ and the maps
\begin{align} \cP(L; - \beta , \bfJ)  & \to [0,+\infty) \\
\cP(L; - \beta , \bfJ)  & \to \cF^{1,p}(L; -\beta) \end{align} are injective. \noproof
\end{lem}
Note that this result says that an exceptional solution is uniquely determined by the underlying map, or by the parameter $R$ at which it occurs.

The last transversality result concerns sphere bubbles for which we will need more control  in later arguments.  We begin with the elementary result:
\begin{lem} \label{lem:moduli_space_alg_discs_submersion}
The moduli space of $J_{\alg}$ holomorphic spheres in classes $k \alpha$ is regular if $k \leq 1$.  Moreover, the evaluation map
\begin{equation} \cM_{1}(M; \beta , J_{\alg}) \to M \end{equation}
is a submersion with fibres diffeomorphic to $\bC \bP^{n-2}$.   \noproof
\end{lem}

 Over the Banach manifold with boundary 
\begin{equation} D^2 \times \sJ_{\bfepsilon}(D^2) \end{equation}
there is a Banach manifold which we denote $\cM_{1}\left(M;  \beta , D^2 \times \sJ_{\bfepsilon}(D^2) \right)$ whose fibres at a fixed $\bfJ_{R} \in \sJ_{\bfepsilon}(D^2)$, is the union
\begin{equation} \coprod_{z \in D^{2}} \cM_{1}(M; \beta,  J_{z, R}) \end{equation}
of the moduli spaces of holomorphic spheres with one marked point for one of the almost complex structures $J_{ z, R}$.  The space $\cM_{1}\left(M;  \beta , D^2 \times \sJ_{\bfepsilon}(D^2) \right)$ carries an evaluation map 
\begin{equation} \label{eq:evaluation_map_spheres} \cM_{1}\left(M;  \beta , D^2 \times \sJ_{\bfepsilon}(D^2) \right) \to D^2 \times \sJ_{\bfepsilon}(D^2) \times M. \end{equation}
\begin{lem} \label{lem:codim_2_strata_diffeo_type}
The evaluation map \eqref{eq:evaluation_map_spheres}  is a  proper submersion in a neighbourhood of $D^{2} \times \{ J_{\alg} \} \times M$.  In particular, if we choose $\bfJ$ to be sufficiently close to the constant family  $J_{\alg}$, then for any $R \in [0, +\infty)$, and any map $u \co D^2 \to M$, we have a diffeomorphism 
\begin{equation} \cM_{1}(M; \beta , J_{\alg}) \times_{M} D^{2} \cong \coprod_{z \in D^{2}} \cM_{1}(M; \beta , J_{z, R}) \times_{M \times D^2} D^{2} \end{equation}
which is canonical up to diffeotopy.
\end{lem}
\begin{proof}
Consider a sequence $v_{i}$ of $\bfJ_i$-holomorphic curves in class $\beta$  where $\bfJ_{i}$ converges to $J_{\alg}$ in $\sJ_{\bfepsilon}$, hence in the $C^{\infty}$-topology.  By the Gromov compactness theorem, a subsequence must converge to a configuration of $J_{\alg}$-holomorphic spheres the sum of whose homology classes is $\beta$.  Since the moduli spaces of holomorphic spheres are empty for homology classes $k \beta$ with $k <0$, the only possible limit is an honest $J_{\alg}$-holomorphic sphere in class $\beta$, which proves properness.  The fact that the map is a submersion in a neighbourhood of $\{ J_{\alg}  \} \times M$  is now a simple consequence of the implicit function theorem and  Lemma \ref{lem:moduli_space_alg_discs_submersion}. 
\end{proof}

\section{Preliminaries for gluing}

In the remainder of this paper, we construct manifolds with corners from the Gromov-Floer compactifications of various moduli spaces.  The analytic ingredients are gluing theorems, and our approach follows as closely as possible the one used in Fukaya, Oh, Ohta, and Ono's book \cite{FOOO}; the only significant difference is that they use obstruction bundles to reduce their gluing theorem to one where all marked points are fixed.  Having to allow a varying marked point is the main reason why we have not sought to prove that the Gromov-Floer compactification itself is a smooth manifold with corners. 

In order not to overwhelm the reader with long proofs, we have opted for one of three different approaches in justifying results:  Whenever we believe that our point of view differs from the one usually taken in the literature, or that the proof is helpful to understand the overall structure of our argument, we present a proof in the text proper.  If we believe that a ``standard technique'' can be adapted to our setting, then the proof is  relegated to Section \ref{sec:proof-results-pre-gluing}, and usually is focused on showing that these techniques are not affected by the differences between our setup and the one used, for example in \cite{FOOO}.  Finally, when a result follows from previously explained ideas, we leave the proof to the reader.

\subsection{Sobolev spaces with exponential weights}
Fix a Riemann surface $\Sigma$ with $k$ interior marked point $p_i$ and $l$ boundary marked points $q_j$.  We choose positive strip-like or cylindrical ends at the marked points, i.e. embeddings
\begin{align} \xi_{p_i} \co [0,+\infty) \times [-1,+1] & \to \Sigma - \{ p_i \}   \\
\xi_{q_j} \co [0,+\infty) \times S^1 & \to \Sigma  - \{ q_j \}
\end{align}
which at infinity converge to the labeling marked points.  When discussing gluing, it is useful to allow negative ends, but since the analytic treatment is the same, we restrict to positive ends for now.  We begin by introducing Sobolev spaces of functions with exponential decay along these ends.

\begin{defin}[Definition 7.1.3 of \cite{FOOO}]
Given a constant $\delta >0$,  let 
\[  \cF^{1,p,\delta}_{\left(\Sigma, \{p_i \}, \{ q_j\}\right) }(L)  \]
denote the space of continuous maps $u$ from $\Sigma$ to $M$ mapping $\partial \Sigma$ to $L$ which are locally in $W^{1,p}$ and such that the integral
\begin{equation} \label{eq:W_1_p-bounded_function} \int \left(|d(u \circ \xi_r)|^p + \dist\left(u(1), u(\xi_r(s,t)) \right)^p \right) e^{\delta p|s|} ds dt \end{equation}
 is finite for each end $\xi_r$ converging to an interior or boundary marked point $r$.
\end{defin}

\begin{rem} Note that the class of maps satisfying \eqref{eq:W_1_p-bounded_function}  is independent of the chosen metric on the target.  This follows from the fact that the image of any such map lies in a compact subset of $M$, and that the distortion between two metrics in a compact domain is bounded.
\end{rem}

Our first goal is to prove that this is a Banach manifold.  If $\Sigma$ is a strip or a cylinder (or more generally a half-strip or half-cylinder), we consider the norm
\begin{equation}|X|_{1,p,\delta}^{p} =  \int|X|^{p} e^{\delta p|s|} ds dt + \int|\nabla X|^{p} e^{\delta p| s|} ds dt  \end{equation}
on vector fields which vanish sufficiently fast at infinity. 

More generally, we pick a metric on the complement of the marked points 
\begin{equation} \Sigma - \coprod \{ p_i \}  - \coprod \{ q_i \} \end{equation} which on the image of each end $\xi$ agrees with the push-forward of the standard metric on the half-strip or the half-cylinder. 

\begin{defin}[See Lemma 7.1.5 of \cite{FOOO}] \label{def:exponential_decay_sobolev_space_vector_fields}

If $u \in \cF^{1,p,\delta}_{\left(\Sigma, \{p_i \}, \{ q_j\}\right)}(L)$, we define 
\begin{equation} W^{1,p,\delta}_{\left(\Sigma, \{p_i \}, \{ q_j\}\right)} (u^*TM, u^*TL) \end{equation}
to be the space of locally $W^{1,p}$ sections of $u^*TM$ whose values on $\partial \Sigma$ lie in $u^* TL$ and which are bounded with respect to the norm
\begin{equation} \label{eq:definition_weighted_sobolev_ends}| X|_{1,p,\delta}^{p} = \left| X{|}_{\Sigma - \cup_{r}\Im(\xi_{r})} \right|_{1,p}^{p} + \sum_{r \in \{p_i\} \cup \{q_j\}} |X(r)|^{p}+ \left| X \circ \xi_{r} - \Pi_{u(r)}^{u \circ \xi_{r}} X(r)  \right|_{1,p,\delta}^{p} \end{equation}
where $\Pi_{u(r)}^{u \circ \xi_{r}}$ is the parallel transport map with respect to the Levi-Civita connection of the metric on $M$ from the constant map with value $u(r)$ to $u(\xi_r(s,t))$ along the image of horizontal lines on the strip or the cylinder.
\end{defin}

In upcoming discussions, it will be useful to have a shorthand notation $\I{S}{S'}$ for the domains $[S,S') \times [-1,+1]$.  The proof of the next result is given in Section \ref{sec:proof-results-pre-gluing}.

\begin{lem} \label{lem:weighted_1_p_implies_c_0_exponential_decay}
Vector fields in $W^{1,p,\delta}_{\left(\Sigma, \{p_i \}, \{ q_j\}\right)} \left(u^*TM, u^*TL\right)$ satisfy a uniform decay condition towards their value at the marked points, e.g. if $r$ is a boundary marked point with a positive strip-like end, there exists a constant $C$ independent of $X$ and $S$ such that
\begin{equation} \sup_{\I{S}{+\infty}}  \left| X \circ \xi_{r} - \Pi_{u(r)}^{u \circ \xi_{r}} X(r)  \right| \leq C e^{-\delta S} | X|_{1,p,\delta} \end{equation}
\end{lem}

The closed finite codimension subspace of $W^{1,p,\delta}_{\left(\Sigma, \{p_i \}, \{ q_j\}\right)} (u^*TM, u^*TL)$ consisting of vector fields which vanish at all marked points is likely to be more familiar to the reader, and is what most people would refer to as the ``Sobolev space with exponential weights."  Such spaces are used to model moduli spaces of holomorphic maps which take the marked points $\{p_i\}$ and $\{q_j\}$ to fixed points in $M$ and $L$ respectively.

It is easy to use the finite codimensionality of this inclusion to study the vector space $W^{1,p,\delta}_{\left(\Sigma, \{p_i \}, \{ q_j\}\right)} (u^*TM, u^*TL)$.  In particular, as with the usual Sobolev spaces with exponential weights, $W^{1,p,\delta}_{\left(\Sigma, \{p_i \}, \{ q_j\}\right)}(u^*TM, u^*TL)$, as a topological vector space, is independent of the choice of metric.  More precisely, the norms coming from two different metrics are equivalent.  This is particularly useful for the next result, in which we assume that we've chosen a metric for which $L$ is a totally geodesic submanifold.  A sketch of the proof  is given in Section \ref{sec:proof-results-pre-gluing}.
\begin{lem} \label{lem:1_p_delta_maps_Banach_manifold}
$\cF^{1,p,\delta}_{\left(\Sigma, \{p_i \}, \{ q_j\}\right)}(L)$ is a Banach manifold locally modeled after the Banach space $W^{1,p,\delta}_{\left(\Sigma, \{p_i \}, \{ q_j\}\right)}(u^*TM, u^*TL)$.
\end{lem}

Given a family of almost complex structures $\bfJ = \{ J_q \}_{q \in \Sigma}$ on $M$ parametrized by point in $\Sigma$, together with a map $u \in \cF^{1,p,\delta}_{\left(\Sigma, \{p_i \}, \{ q_j\}\right)}(L)$, we let 
\begin{equation} L^{p,\delta}_{\left(\Sigma, \{p_i \}, \{ q_j\}\right)}(u^*TM \otimes_{\bC} \Omega^{0,1} D^2) \end{equation}
denote the space of section of $u^*TM \otimes_{\bC} \Omega^{0,1} D^2$ which are bounded in an $L^{p,\delta}$ norm which we will define presently.  It is convenient to first introduce the function
\begin{equation} \kappa_{\Sigma, \delta} \co \Sigma - \coprod \{ p_i \}  - \coprod \{ q_i \} \to [1,+\infty) \end{equation}
which is identically equal to $1$ on the complement of the ends, and which is given on each end by
 \begin{equation} \kappa_{\Sigma, \delta} \circ \xi_{r}(s,t) =  e^{\delta p|s|}. \end{equation}
 With this bit of notation, we set
\begin{equation}| Y|_{p,\delta}^{p} =  \int_{\Sigma }|Y|^p  \kappa_{\Sigma, \delta}    = \left| Y {|}_{\Sigma - \cup_{r} \Im(\xi_{r})} \right|^{p}_{p} + \sum_{r}|Y \circ \xi_{r}|_{p,\delta}^{p} \end{equation}
where the integral is performed with respect to the previously chosen metric on $\Sigma$ that is cylindrical on the ends, and the pointwise norm of $Y$ is induced from the metric on $M$ and (again) the metric on $\Sigma$.

Choose a family of metrics $g_{q}$ on $M$ parametrized by $q \in \Sigma$ which are almost hermitian for the almost complex structures $J_q$, and for which $L$, whenever $q \in \partial \Sigma$, is totally geodesic.  Given $X \in W^{1,p,\delta}_{\left(\Sigma, \{p_i \}, \{ q_j\}\right)}\left(u^*TM, u^*TL\right)$, we obtain an identification
\begin{align} \label{eq:complex_parallel_transport}
L^{p,\delta}_{\left(\Sigma, \{p_i \}, \{ q_j\}\right)}(u^{*}(TM) \otimes_{\bC} \Omega^{0,1} D^2) & \to L^{p,\delta}(u_{X}^{*}(TM) \otimes_{\bC} \Omega^{0,1} D^2) \\
Y & \mapsto \tilde{\Pi}_{u}^{u_X} Y,
\end{align}
where $\tilde{\Pi}_{u}^{u_X}$ is, at each point $q \in \Sigma$, given by parallel transport along the image of the exponential map $u(q)$ to $u_{X}(q)$, with respect to the complex linear connection
\begin{equation} \label{eq:formula_complex_linear_connection} \tilde{\nabla} = \nabla - \frac{1}{2} J_{q} \nabla J_{q} \end{equation}
where $\nabla$ is the Levi-Civita connection associated to $g_q$.  We omit the proof that $\tilde{\Pi}_{u}^{u_X}$ respects the exponential decay condition, as well as that of the following result: 
 \begin{lem} \label{lem:banach_bundle_0_1_forms}
The isomorphisms \eqref{eq:complex_parallel_transport} define a  Banach bundle
\begin{equation} \label{eq:weighted_sobolev_curve} \cE^{p,\delta}_{\left(\Sigma, \{p_i \}, \{ q_j\}\right)}(M) \to \cF^{1,p,\delta}_{\left(\Sigma, \{p_i \}, \{ q_j\}\right)}(L) \end{equation}
whose fibre at $u$ is \begin{equation} L^{p,\delta}_{\left(\Sigma, \{p_i \}, \{ q_j\}\right)}(u^*TM \otimes_{\bC} \Omega^{0,1} \Sigma). \end{equation}  \noproofe
\end{lem}

Given any holomorphic map 
\begin{equation} u \co (\Sigma, \partial \Sigma) \to (M,L) \end{equation}
there exists some $\delta$ such that $u \in  \cF^{1,p,\delta}_{\left(\Sigma, \{p_i \}, \{ q_j\}\right)}(L)$.  In fact, since $u$ is smooth on the boundary (see Theorem B.1 of \cite{lazzarini}), any $\delta$ smaller than $1$ will do.  The $\dbar$ operator defines a section of \eqref{eq:weighted_sobolev_curve} which is always Fredholm, and whose zero set is the moduli space of parametrized maps.    Moreover, if $\delta <1$, we have an identification between the kernels and cokernels of the linearisation of $\dbar$ as an operator on the weighted Sobolev spaces \eqref{eq:weighted_sobolev_curve} and on the standard unweighted Sobolev spaces.   For the kernel, this follows from the inclusion
\begin{equation} W^{1,p,\delta}_{\left(\Sigma, \{p_i \}, \{ q_j\}\right)}(u^*TM, u^*TL)  \subset W^{1,p}_{\Sigma}(u^*TM, u^*TL)  \end{equation}
and the fact that the kernel of $D_{\dbar}$ on $W^{1,p}_{\Sigma}(u^*TM, u^*TL)$ consists only of smooth functions; i.e. functions which decay in any $C^k$-norm at least as fast as $e^{-s}$ along the strip-like end.  The same argument, applied to the formal adjoint, proves the result for the cokernel.   We conclude:

\begin{lem} \label{lem:surjectivity_survives} 
If $D_{\dbar}$ is surjective on the unweighted Sobolev spaces, then it is surjective as an operator on \eqref{eq:weighted_sobolev_curve} for every $\delta <1$. \noproof
\end{lem}
From now on, we shall assume that we're working with a fixed $\delta$ which is extremely small; say smaller than $1/4$. 
\subsection{Pre-gluing near the codimension $1$ stratum}
Recall that the codimension $1$ strata of $\Pbar(L;0)$ are the components of
\begin{equation}  \cP_{0,1}(L; - \beta) \times_{L} \cS_{0,1}(L; \beta) .\end{equation}
We will respectively write $\pi_{\cP}$, $\pi_{\cS}$, and $\pi_L$ for the projections to $ \cP_{0,1}(L; - \beta)$, $ \cS_{0,1}(L; \beta)$, and $L$.

We first show that there are embeddings
\begin{align} \label{eq:include_discs}  \cS_{0,1}(L; \beta) & \subset   \cF^{1,p,\delta}_{(D^2,-1)}(L) \\
\label{eq:include_parametrized_moduli_banach} \cP_{0,1}(L;-\beta) & \subset \cF^{1,p,\delta}_{\cP_{0,1}}(L) \equiv [0,+\infty) \times  S^1 \times \cF^{1,p,\delta}_{(D^2,1)}(L).
\end{align}
In order to do this, we must pick strip-like ends.  For concreteness, we fix 
\begin{equation} \xi_{-1} \co \I{-\infty}{0} \to D^{2} \end{equation}
such that $\xi_{-1} (0,0) = -1/2$, which determines the map uniquely if we require it to extend to an isomorphism
\begin{equation} \label{eq:rational_strip} \I{-\infty}{+\infty} \to D^2 - \{\pm 1 \}.\end{equation}
We also fix a positive strip-like end  at $1$ 
\begin{equation} \label{eq:positive_strip_at_1} \xi_{1} \co  \I{0}{+\infty} \to D^{2} ,\end{equation}
which we will require to map the origin to $1/2$, and again to extend as in \eqref{eq:rational_strip}.

The first embedding can be constructed as follows: the $\dbar$ operator for $J_{\alg}$ defines a section of 
\begin{equation} \label{eq:banach_bundle_L_p_domain_fixed} \cE^{p,\delta}_{(D^2,-1) }(M) \to  \cF^{1,p,\delta}_{(D^2,-1)}(L) .\end{equation}  
By Lemma \ref{lem:surjectivity_survives}, the zero locus of $\dbar$ is the set of (parametrized) holomorphic maps from the disc; those in homotopy class $\beta$ form an $\Aut(D^2,-1)$ bundle over $\cS_{0,1}(L; \beta) $.  Since $\Aut(D^2,-1)$ is contractible, this principal bundle admits a section.  In Section \ref{sec:degree_gluing}, we shall specify the behaviour of this section away from a compact subset of $\cS_{0,1}(L; \beta)$.  The composition gives the desired embedding claimed in \eqref{eq:include_discs}.

Concerning \eqref{eq:include_parametrized_moduli_banach}, we have a Banach bundle  \begin{equation} \label{eq:banach_bundle_L_p_domain_vary} \cE^{p,\delta}_{\cP_{0,1}}(M) \to  \cF^{1,p,\delta}_{\cP_{0,1}}(L) ,\end{equation}
 by pullback from $\cF^{1,p,\delta}_{(D^2,1)}(L)$, whose fibre at $(R,\theta,u)$ consists of $TM$-valued $1$-forms on $D^2$ which are anti-holomorphic with respect to the family of almost complex structures
\begin{equation} J_{\theta, R} \equiv  \{ J_{r_{\theta} z,R} \}_{z \in D^{2}} \end{equation} 
where $r_{\theta}$ denotes the map which rotates $D^2$ by angle $\theta$.

We construct a section of \eqref{eq:banach_bundle_L_p_domain_vary}  by composing $\dbar_{\cP}$ with rotation by $\theta$:
\begin{equation} \label{eq:section_theta_depend} \dbar_{\cP_{0,1}} \co (R, \theta, u)  \mapsto (du - \gamma_{\theta,R} \otimes X_{H})^{0,1} ,\end{equation}
where the $(0,1)$ part is taken with respect to $J_{\theta, R}$ and $\gamma_{\theta,R}$ is  the pullback of $\gamma_{R}$ by $r_{\theta}$.   

We have a bijection
\begin{align} \label{eq:rotate_by_theta} \cP_{0,1}(L;-\beta) & \cong \cP(L;-\beta) \times S^1 \\
\notag (R,\theta,u) & \mapsto (R, u \circ r_{-\theta}, \theta) . \end{align}
where $\cP_{0,1}(L;-\beta) $ is the zero set of $\dbar_{\cP_{0,1}}$.  Embedding the source and the target within the space of all smooth maps by elliptic regularity, we conclude that this map must be a diffeomorphism whenever both spaces are regular.  Since the bundle \eqref{eq:banach_bundle_L_p_domain_vary} is defined by pullback from the factor $\cF^{1,p,\delta}_{(D^2,-1)}(L)$, and the section $\dbar_{\cP_{0,1}}$ differs from the pullback of $\dbar_{\cP}$ only by reparametrization,  we conclude:
\begin{lem} If $\cP(L;-\beta)$ is regular, then $\dbar_{\cP_{0,1}}$ is a transverse Fredholm section of \eqref{eq:banach_bundle_L_p_domain_vary} and \eqref{eq:rotate_by_theta} is a diffeomorphism. \noproof \end{lem}

Combining the linearisations of the two Cauchy-Riemann operators with the evaluation map to $L$, we define a Fredholm map of Banach bundles
\begin{equation} \label{eq:product_dbar_gluing} \xymatrix{ \pi_{\cS}^{*} T \cF^{1,p,\delta}_{(D^2,-1)}(L) \oplus \pi_{\cP}^* T \cF^{1,p,\delta}_{\cP_{0,1}}(L) \ar[dd]^{\left(
\begin{array}{cc}
D_{\dbar} &  0 \\
-d \ev_{-}& d \ev_{+}  \\
 0 & D_{\cP_{0,1}}
\end{array}
\right)} \\
\\
\pi_{\cS}^{*} \cE^{p}_{(D^2,-1) }(M) \oplus \pi_{L}^* TL \oplus  \pi_{\cP}^* \cE^{p}_{\cP_{0,1}}(M).}  \end{equation}
Here, $D_{\dbar} $ and $D_{\cP_{0,1}}$ stand respectively for the linearisations of $\dbar$ and $\dbar_{\cP_{0,1}}$, and the evaluation maps $\ev_{\pm}$ take place at the marked points $\pm 1 \in D^2$.
\begin{lem}  \label{lem:surjectivity_differential_map}
The requirements imposed in Equations \eqref{ass:exceptional_discs_regular}, \eqref{ass:moduli_discs_regular}, and \eqref{ass:transverse_evaluation}, imply that the map \eqref{eq:product_dbar_gluing}  is surjective. 
\end{lem} 
\begin{proof}
Conditions \eqref{ass:exceptional_discs_regular} and \eqref{ass:moduli_discs_regular} imply that $D_{\dbar}$ and $D_{\cP_{0,1}}$ are surjective, while Condition \eqref{ass:transverse_evaluation} implies that the images of $d\ev_{\pm}$, restricted to the kernels of $D_{\dbar}$ and $D_{\cP_{0,1}}$, span $TL$.
\end{proof}

Restricting \eqref{eq:product_dbar_gluing} to the kernel of the map to $TL$, i.e. to those pairs of vector fields whose values agree at the marked points, we have a surjective Fredholm operator
\begin{equation} \label{eq:fibre_product_dbar_gluing}  \xymatrix{ \pi_{\cS}^{*} T \cF^{1,p,\delta}_{(D^2,-1)}(L) \oplus_{TL} \pi_{\cP}^* T \cF^{1,p,\delta}_{\cP_{0,1}}(L) \ar[dd]^{ D_{\cod{1}{L}} = \left(
\begin{array}{cc}
D_{\dbar} &  0 \\
 0 & D_{\cP_{0,1}}  
\end{array}
\right)  }\\
\\
\pi_{\cS}^{*} \cE^{p}_{(D^2,-1) }(M) \oplus  \pi_{\cP}^* \cE^{p}_{\cP_{0,1}}(M)} \end{equation}

So simplify the notation, we write $T \cF^{1,p,\delta}_{\cod{1}{L}}(L)$ for the source of \eqref{eq:fibre_product_dbar_gluing}, as well as $ \cE^{p,\delta}_{\cod{1}{L}}(M)$  for the target.

For the purpose of proving the gluing theorem, we shall consider a bounded right inverse
\begin{equation}   \label{eq:fibre_product_right_inverse} Q_{\cod{1}{L}} \co \cE^{p,\delta}_{\cod{1}{L}}(M) \to T \cF^{1,p,\delta}_{\cod{1}{L}}(L) \end{equation}
to \eqref{eq:fibre_product_dbar_gluing}.  Such a right inverse exists because \eqref{eq:fibre_product_dbar_gluing} is a Fredholm map.

\begin{rem}
In Equations \eqref{eq:constrain_right_inverse} and \eqref{eq:right_inverse_constraint_basepoints} we constrain our choice of right inverse $Q_{\cod{1}{L}}$ on certain subsets of $\cod{1}{L}$, but for now, our choice of a right inverse is only required to be smooth.
\end{rem}

\subsection{Pre-gluing maps} \label{sec:pre-gluing_maps}
\begin{figure} 
   \centering
 \input{disc_pre_gluing.pstex_t}
   \caption{}
   \label{fig:disc_pre_gluing}
\end{figure}

Given $S \in [0, \infty)$, we consider the surface (see Figure \ref{fig:disc_pre_gluing}) 
\begin{equation} \label{eq:definition_Sigma_S} \Sigma_{S} =  (D^2,1) \#_{S}  (D^2,-1)\end{equation}
obtained by removing  $\xi_{1}(\I{4S}{+\infty})$ and $\xi_{-1}(\I{-\infty}{-4S})$ from two different copies of the disc, and identifying the remaining parts of the strip-like ends using the map
\begin{equation} \label{eq:glue_surfaces_isometry} (s,t) \mapsto (s-4S , t) .\end{equation}
For each $S_0 \leq 4S$, we shall write
\begin{equation} \label{eq:notation_inclusion_two_halves_gluing} \iota^{\cP,S}_{S_0} \co D^2  -  \xi_{1}(\I{S_0}{+\infty}) \to \Sigma_{S} \textrm{ and }  \iota^{\cS,S}_{S_0} \co D^2  -  \xi_{-1}(\I{-\infty}{S_0}) \to \Sigma_{S}  \end{equation}
for the two inclusions into $\Sigma_S$.  The images of $ \iota^{\cP,S}_{4S}$ and  $\iota^{\cS,S}_{4S}$ cover $\Sigma_{S}$.

Note that $ \iota^{\cP,S}_{4S}$ extends to a unique biholomorphism from $D^2$ to $\Sigma_{S}$, whose inverse we denote by
\begin{equation} \label{eq:identify_glued_surface_disc} \phi_{S} \co  \Sigma_{S} \to D^2. \end{equation}
Assuming that $S_0 \leq 4S$, we restrict $\xi_{1}$ to $\I{0}{S_0}$, and compose with $ \iota^{\cP,S}_{4S}$, to obtain an inclusion
\begin{equation} \xi_{1}^{S_0} \co \I{0}{S_0} \to \Sigma_{S}  .\end{equation}
Symmetrically, we shall also consider the composition of $ \iota^{\cS,S}_{S_0}$ with the restriction of $\xi_{-1}$ to  $\I{-S_0}{0}$ which we denote by
\begin{equation} \xi_{-1}^{S_0} \co \I{-S_0}{0} \to \Sigma_{S}  . \end{equation}
The most important case occurs when $S_0 = 4S$, and it will be convenient to shift the domain of $ \iota^{\cP,S}_{4S} \circ \xi_{1}^{4S}$  by the translation
\begin{equation} \tau_{2S}(s,t) = (s+2S,t) \end{equation} to obtain a map
\begin{equation} \xi_{S,\neck} \equiv  \xi_{1}^{4S} \circ \tau_{2S} \co \I{-2S}{2S} \to \Sigma_{S} \end{equation}
whose image we shall refer to as {\bf the neck} of $\Sigma_{S}$.

Having glued the domains, we now define a ``pre-gluing" of maps.  First, we fix a smooth cutoff function 
\begin{equation} \chi \co   \bR \to [0,1] \end{equation} which vanishes for $s>1$ and equals $1$ for $s < -1$.  Given real numbers $S_+$ and $S_-$ and $S_0 \in [S_-,S_+]$, we shall slightly abuse notation and write $\chi_{S_0}$ for the function
\begin{align} \chi_{S_0} & \co \I{S_-}{S_+} \to [0,1]  \\
\chi_{S_0}(s,t) & = \chi(s-S_0)
\end{align}
without keeping track of the domain.

For a compact subset 
\begin{equation} \label{eq:compact_subset_codim_1} K \subset \cod{1}{L} ,\end{equation}
 there exists a positive real number $S_K$ sufficiently large so that for any pair $(u,\theta,v) \in K$, the images of $ u\circ \xi_1(\I{S_K}{+\infty} ) $ and $v \circ \xi_{-1}(\I{-\infty}{S_K} )$
are both contained in a geodesically convex neighbourhood of $u(1) = v(-1)$.  Moreover, writing $R_u$ for the image of $u$ under the projection of $\cP_{0,1}(L;-\beta)$ to $[0,+\infty)$, we may also assume that $S_K$ is so large that the image of $\xi_1( \I{S_K}{+\infty} )$ does not intersect the union of the disc of radius $1/2$ about the origin  with the support of $\gamma_{\theta,R_u}$ for any $\theta \in S^1$.  In particular, the almost complex structure and the metric on $M$ are independent of the point on the neck, and agree respectively with $J_{\alg}$ and $g$. 
\begin{lem} \label{lem:large_S_K}
The $\dbar_{\cP_{0,1}}$ equation on the image of $\xi_1( \I{S_K}{+\infty} )$ reduces to the usual $\dbar$ equation with respect to the complex structure $J_{\alg}$. \noproof
\end{lem}

If $S>S_K$ we define a map
\begin{equation}u \#_{S} v \co \Sigma_{S} \to M \end{equation}
as follows
\begin{itemize}
\item The compositions $ u \#_{S} v \circ \iota^{\cP,S}_{2S-1} $ and $ u \#_{S} v \circ \iota^{\cS,S}_{2S-1} $ respectively agree with the restrictions of $u$ and $v$.
\item On the image of $\xi_{1}^{2S}$, we define $u \#_{S} v$ by considering its composition with $\xi_{1}^{2S}$:
\begin{equation} u \#_{S} v \circ \xi_{1}^{2S} \equiv \exp_{u(1)} \left(  \chi_{2S-1} \cdot \exp^{-1}_{u(1)} (u \circ \xi_{1} ) \right), \end{equation}
where $\exp_{u(1)}$ is the exponential map from the tangent space of $L$ at $u(1)$.  Note, in particular, that the image of $\{ 2S \} \times  [-1,1]$ under this map is constant, and agrees with $u(1)$. In other words, we restrict $u \circ \xi_{1}$ to $\I{0}{2S}$,  use the exponential map to pull back to the tangent space of $u(1)$, then rescale the corresponding image by the cutoff function $\chi(s-2S+1)$, and finally push back using the exponential map. 
\item On the image of $\xi^{2S}_{-1}$, we analogously define $u \#_{S} v$ as follows:
\begin{equation} u \#_{S} v \circ \xi_{-1}^{S}\equiv \exp_{v(-1)} \left( ( 1- \chi_{-2S+1} ) \cdot \exp^{-1}_{v(-1)} (v \circ \xi_{-1}) \right), \end{equation}
remembering, of course that $v(-1) = u(1)$.
\end{itemize}

Defining
\begin{equation} \label{eq:W^1,p-maps-parametrized_circle} \cF^{1,p}_{\cP_{0,1}} (L) \equiv  S^1 \times \cF^{1,p}_{\cP} (L) \equiv [0,+\infty) \times S^1 \times  \cF^{1,p}(L) ,\end{equation}
we consider a pre-gluing map which records the angle $\theta \in S^1$
\begin{align} \label{eq:defin_pre-gluing_extended}  \preGext \co K \times [S_K, +\infty) & \to  \cF^{1,p}_{\cP_{0,1}} (L) \\
\preGext_{S} (u,\theta,v) & = \left(R_u, \theta ,   (u \#_{S} v)  \circ \phi_{S}^{-1} \right) \end{align}

By composing this map with rotation by $\theta$ and projecting to $\cF^{1,p}(L)$ in the product decomposition of Equation \eqref{eq:W^1,p-maps-parametrized_circle},  we obtain a map
\begin{align} \label{eq:defin_pre-gluing}  \preG \co K \times [S_K, +\infty) & \to  \cF^{1,p}(L) \\
\preG_{S} (u,\theta,v) & = \left(R_u ,   (u \#_{S} v)  \circ \phi_{S}^{-1}  \circ r_{-\theta}\right) .\end{align}

\begin{lem} \label{lem:pre-gluing_injective}
The pre-gluing map $\preG$ is a smooth embedding.
\end{lem}
\begin{proof}
Our  construction using smooth cut off functions, and the fact that a holomorphic vector field cannot vanish on any open set, readily imply that $\preG$ is an immersion.  Given an element $(u, \theta,v) \in K$, the unique continuation theorem for holomorphic maps implies that, in the complement of the disc of radius $\frac{1}{2}$, $\preG_{S}(u, \theta,v)$ has isolated critical points away from a unique semi-circle with endpoints on the boundary, which is the image of the arc $\{ 2S \} \times  [-1,1]$ under the composition of the strip $\xi_{S,\neck}$ with the bi-holomorphism $\phi_{S}$ and rotation by angle $r_{-\theta}$.  The position of this semi-circle uniquely determines $S$ and $\theta$.

Consider another element $(u',\theta,v') \in K$ such that $\preG_{S}(u',\theta,v') = \preG_{S}(u, \theta,v)$.  In particular, the pairs $\{ u, u' \}$ agree on an open set which includes the support of  the perturbation $\bfK$ of $J_{\alg}$ and of the forms $\gamma_{\theta,R}$ and  $\gamma_{\theta,R'}$, so by the unique continuation theorem applied to the complement of this support $u$ and $u'$ agree.

Note that the images of the pairs $\{ v, v' \}$ also agree on an open set, so by the unique continuation theorem, the two holomorphic maps must agree up to parametrization. Moreover, the previous result implies that $v(-1) = u(1) = v'(-1)$.  Having chosen unique representatives under the equivalence class induced by the action of $\Aut(D^2, -1)$, $v$ and $v'$ must in fact agree on the nose.
\end{proof}

The fact that $u$ and $v$ are smooth on $D^2$ implies that the distances to $v(-1) = u(1)$, as well as the norms of $d u$ and $dv$, decay exponentially along the strip-like ends of $u$ and $v$.  The definition of $u \#_{S} v$ using cutoff functions therefore implies 
\begin{lem} \label{lem:exponential_decay_bounds}
The $C^1$-norm of the restriction of $u \#_{S} v$ to the neck decays exponentially: 
\begin{equation}
\left| u \#_{S} v  \circ \xi_{S,\neck} | \I{-S-1}{S+1} \right|_{C^1}  = \rO \left( e^{-S} \right) .
\end{equation} \noproofe
\end{lem}

\section{Construction of an extended gluing map} \label{sec:codim1_gluing}
Let $K$ denote a compact subset of  $\cod{1}{L}$ as in Equation \eqref{eq:compact_subset_codim_1}. In this section, we prove the existence of a gluing map from $K \times [S,+\infty)$ to $\cP_{0,1}(L;0)$ for $S$ sufficiently large, which will be shown in later sections to satisfy the required surjectivity and injectivity properties.  More precisely, we shall embed $K \times [S,+\infty)$ as the zero section of the pull back of the kernel of $D_{\cP_{0,1}}$ under the pre-gluing map $\preGext$.  After restricting to vectors of norm bounded by a constant $\epsilon$ we shall construct a map
\begin{equation}   \Gext_{\epsilon} \co \preGext^{*}  \ker_{\epsilon} D_{\cP_{0,1}} \to \cP_{0,1}(L;0),\end{equation}
by applying an implicit function theorem to the operator $D_{\cP_{0,1}}$ on $\cF^{1,p}_{\cP_{0,1}}(L;0)$ equipped with appropriate norms.  The usual procedure of constructing a gluing map is performed chart by chart; this is essentially equivalent to restricting to one of the fibres of $\preGext^{*}  \ker_{\epsilon} D_{\cP_{0,1}}$.  The new problems to consider concern the study of the behaviour of this map with respect to the base variable.   Since this is usually not addressed in the literature, we shall give an essentially complete proof of the gluing theorem.  It turns out that $ \Gext_{\epsilon}$ is evidently smooth with respect to directions along the manifold $K$.  However, we have only proved that the map is continuous in the $S$ direction.
\begin{rem}
It is unlikely that smoothness in the $S$ direction is an essential problem; we expect that $C^1$-differentiability would follow if one were willing to bound one more derivative.  However, our failure to define a gluing map by applying an implicit function theorem in a Banach space completion of $\cF_{\cP}(L;0)$  rather than a completion of the product of this space with $S^1$  seems to point to a more fundamental problem in gluing holomorphic curves at varying marked points.  This might be an artefact of a poor choice of metric, but we suspect that there is a deeper problem caused by the well known fact that quotients of Sobolev spaces by diffeomorphisms in the source are not naturally equipped with the structure of a smooth Banach manifold.  We hope that the polyfold theory of \cite{HZW}, applied to this relatively simple setting, would give a clean way of circumventing this problem.
\end{rem}

Equation \eqref{eq:section_theta_depend} defines a Fredholm section
\begin{equation}  \dbar_{\cP_{0,1}} \co \cF^{1,p}_{\cP_{0,1}} (L) \to \cE^{p}_{\cP_{0,1}}(M) .\end{equation}
For the purposes of the gluing theorem, we need to explicitly write down an operator
\begin{equation} \label{eq:extension_linearisation_dbar_parametrized} D_{\cP_{0,1}}  \co T\cF^{1,p}_{\cP_{0,1}} (L) \to \cE^{p}_{\cP_{0,1}}(M)\end{equation}
which extends the linearisation of $\dbar_{\cP_{0,1}}$ away from the zero section.  Recall that  the fibres of $\cE^{p}_{\cP_{0,1}}(M)$ at $( R,\theta, u)$ are $J_{\theta,R}$-anti-holomorphic $TM$ valued $1$-forms, where $J_{\theta,R} $ is the family of almost complex structures $J_{ r_{\theta} z, R}$ depending on the point $z \in D^2$.   We shall write $x^{\natural} = (R, \theta, x)$ for an element of 
\begin{equation} M^{\natural} \equiv  [0,+\infty) \times  S^1 \times M \end{equation}
to simplify matters.  For each pair $(z,R) \in D^2 \times [0,+\infty)$, choose a metric $g_{z,R}$ which is (almost) hermitian for the almost complex structure $J_{z, R}$, and agrees with the metric $g$ of Corollary \ref{cor:totally_geodesic_hermitian_metric} whenever $z$ is close to the boundary.  We equip $D^{2} \times M^{\natural}$ with a metric which is given pointwise by the direct sum of the flat metric in the $D^{2} \times S^1 \times [0,+\infty)$ direction, and the metric $g_{\theta,R} \equiv \{ g_{r_{\theta}z,R} \}_{z \in D^{2}}$ along $M$.  Note that the tangent space of $M$ is preserved by parallel transport with respect the Levi-Civita connection of this metric on $D^2 \times M^{\natural}$. We shall write $\nabla$ for the restriction of this connection to $TM$ as a bundle over $D^2 \times  M^{\natural}$, and $\tilde{\nabla}$ for the corresponding complex linear connection given by the formula that already appeared in Equation \eqref{eq:formula_complex_linear_connection}.   The key property of this connection is that the family of almost complex structures $J_{r_{\theta} z, R}$ is flat with respect to it.

In the product decomposition \eqref{eq:include_parametrized_moduli_banach}  of $T\cF^{1,p}_{\cP_{0,1}} (L)  $ at $(R,\theta, u)$, an elementary computation shows that the linearisation $D_{\cP_{0,1}}$ of $\dbar_{\cP_{0,1}}$ with respect to the connection $\tilde{\nabla}$ can be written in terms the Levi-Civita connection $\nabla$  as follows (compare with Equation \eqref{eq:linearisation_dbar_perturb_complex}):
\begin{align}   \label{eq:expression_linearisation_circle_varying_parametrized}
\partial_{R} &  \mapsto  \left( \frac{1}{2} J_{R, \theta} \left( \nabla_{\partial_{R}}  J_{R, \theta} \right) \left( du -  \gamma_{\theta,R} \otimes X_{H}  \right) \right)^{0,1} + \\  \notag & \qquad \left( \nabla_{\partial_R} du  - \frac{d\gamma_{R, \theta}}{dR}   \otimes X_{H}  - \gamma_{\theta,R} \otimes \nabla_{\partial_R} X_{H}  \right)^{0,1} \\ \notag
\partial_{\theta} &   \mapsto    \left( \frac{1}{2} J_{r_{\theta} z, R} \left( \nabla_{\partial_\theta}J_{R, \theta} \right) \left( du -  \gamma_{\theta,R} \otimes X_{H}  \right) \right)^{0,1}   \\ \notag
& \qquad \left( \nabla_{\partial_\theta} du  - \frac{d  \gamma_{R, \theta} }{d\theta} \otimes X_{H}  - \gamma_{\theta,R} \otimes \nabla_{\partial_\theta} X_{H}   \right)^{0,1} 
\\ \notag
X & \mapsto \left( \nabla X - \gamma_{\theta,R} \otimes \nabla_{X} X_{H}  \right)^{0,1} - \frac{1}{2} J_{R, \theta} \left( \nabla_{X}J_{R, \theta}  \right) \partial_{\cP_{0,1}} (R, \theta,u)  
\end{align}
 In the last equation, 
\begin{equation} \partial_{\cP_{0,1}}  ( R,\theta, u)  = \left( du - \gamma_{\theta, R} \otimes X_{H} \right)^{1,0} .\end{equation}

Given an element $w^{\natural} = (R, \theta, w)$ of $\cF^{1,p}_{\cP_{0,1}} (L) $, and a tangent vector  $X^{\natural} =(r,\lambda,X)$, we write
\begin{equation} w^{\natural}_{X^{\natural}} = (R + r, \theta+ \lambda, \exp_{w}(X)) ,\end{equation}
keeping in mind that the exponential map in the third component is performed with respect to a metric which depends on $(R,\theta)$ as well as the point $z \in D^2$.

Having introduced this additional notation, we define a non-linear map $ \sF $ from  $T\cF^{1,p}_{\cP_{0,1}} (L)$  to  $\cE^{p}_{\cP_{0,1}}(M)$
\begin{equation} \sF_{w^{\natural}}(X^{\natural}) =  \tilde{\Pi}_{w^{\natural}_{X^{\natural}}}^{w^{\natural}}\left(\dbar_{\cP_{0,1}}  w^{\natural}_{X^{\natural}}\right)  \end{equation}
where $ \tilde{\Pi}_{w^{\natural}_{X^{\natural}}}^{w^{\natural}}$ is the parallel transport map with respect to the connection $\tilde{\nabla}$ along the image of the exponential map.  
\begin{lem} \label{lem:bounds_for_implicit_function_theorem}
The right inverse $Q_{\cod{1}{L}}$ to the operator \eqref{eq:fibre_product_right_inverse} determines a right inverse 
\begin{equation} Q_{\preGext} \co \preGext^{*} \cE^p_{\cP_{0,1}} (M)  \to  \preGext^{*}  T\cF^{1,p}_{\cP_{0,1}} (L) , \end{equation}
to the restriction of $D_{\cP_{0,1}}$ to the image of $\preGext$.  Moreover, there are Sobolev norms $| \_|_{1,p, S}$  and $| \_|_{p,S}$ on the Banach bundles $\preGext^{*}( T\cF^{1,p})$  and $\preGext^{*}( \cE^p)$, as well a positive real number $C$ which is independent of $S$ such that the following properties hold:
\begin{equation}  \label{eq:norm_dbar_exponentially_decay}  \parbox{35em}{ The norm of $\dbar_{\cP_{0,1}} \preGext_{S}(u,\theta,v)$ decays exponentially with $S$
\begin{equation*} \left|\dbar_{\cP_{0,1}} \preGext_{S}(u,\theta,v) \right|_{p, S}  = \rO \left( e^{-2(1-\delta)S} \right) .\end{equation*}}    \end{equation}  
\begin{equation}  \label{eq:quadratic_inequality} \parbox{35em}{If $\left|X^{\natural}_i\right| < 1/C$, for $i=1,2$, then we have a uniform bound} \end{equation}
\begin{equation*} \left| \sF_{\preGext_{S} ( u , \theta ,v)}(X_1^{\natural}) - \sF_{ \preGext_{S} ( u , \theta ,v)}(X_2^{\natural}) - D_{\cP_{0,1}}(X_1^{\natural} - X_2^{\natural})  \right| \leq C  \left|X_1^{\natural} + X_2^{\natural} \right| \left|X_1^{\natural} - X_2^{\natural}\right| .\end{equation*}
\begin{equation} \label{eq:Q_bounded} \parbox{35em}{ The norm of $Q_{\preGext}$ is uniformly bounded:
\begin{equation*}   \left\|Q_{\preGext}\right\| \leq C. \end{equation*}}  \end{equation}
\begin{equation} \label{eq:continuity_properties_right_inverse} \parbox{35em}{ $Q_{\preGext}$ is a strongly continuous map of Banach bundles.  Moreover, the restriction of $Q_{\preGext}$ to a slice $K \times \{S\}$ is a smooth uniformly continuous map of Banach bundles.}   \end{equation}
\end{lem}
\begin{rem}
Recall that the topology of point-wise convergences on the space of bounded linear operators is called the strong topology.  In particular, a map between trivial Banach bundles over a parameter space $X$
\begin{equation} F \co X \times B_1 \to X \times B_2 \end{equation}
is called strongly continuous if the family of operators $F_{x} \co B_1 \to B_2$ satisfy the property that $F_{x_i}(b)$ converges to $F_{x}(b)$ whenever $x_i$ converges to $x$ and $b \in B_1$.   Such a map is uniformly continuous if the rate of convergence depends linearly in the norm of $b$.
\end{rem}
\begin{proof}[Proof of Lemma \ref{lem:bounds_for_implicit_function_theorem}]
The existence of $Q_{\preGext}$, the bound \eqref{eq:Q_bounded}, and the continuity properties \eqref{eq:continuity_properties_right_inverse} are stated in Corollary \ref{cor:uniformly_bounded_inverse}, and the quadratic inequality \eqref{eq:quadratic_inequality} is proved in Section  \ref{sec:proof_quadratic}.  To prove the bound \eqref{eq:norm_dbar_exponentially_decay}, we observe that the support of $\dbar_{\cP_{0,1}} \preGext_{S}(u,\theta,v)$ is contained in the image of $\xi_{S,\neck}(\I{-2}{2})$ where Lemma \ref{lem:exponential_decay_bounds} implies that the $C^1$-norm of $\dbar_{\cP_{0,1}} \preGext_{S}(u,\theta,v)$ is bounded by a constant multiple of $e^{-2S}$.  Since the weight for the $| \_|_{p,S}$-norm, which we will define in Equation \eqref{eq:weighted_norms_differential_S},  is bounded by $e^{2Sp\delta}$ the result follows.
\end{proof}
Estimate \eqref{eq:quadratic_inequality} is a version of the quadratic inequality.  It is sufficient for the proof of the implicit function theorem, but we shall need the following stronger version in later arguments:
\begin{prop}[c.f. Proposition 3.5.3 of \cite{MS}] \label{prop:quadratic_inequality} There exists a constant $c >0$ independent of $S$ with the following property: given a tangent vector
\begin{equation} Z^{\natural} = (r, \lambda, Z) \in T \cF_{\cP_{0,1}}(L) \end{equation} which is bounded by $c$ in the $C^0$-norm, 
\begin{equation} \label{eq:bound_C_0-for_quadratic}|Z^{\natural}|_{\infty} = |Z|_{\infty} +|\lambda| +|r| < c \end{equation} the difference between the linearisation of $\dbar_{\cP_{0,1}}$ at the origin and at $Z^{\natural}$ is bounded by a constant multiple of $|Z^{\natural}|_{1,p,S}  $ :
\begin{equation} \label{eq:quadratic_inequality_first_statement} \left\| d \sF_{\preGext_{S}(u,\theta,v)} (Z^{\natural}) - D_{\cP_{0,1}} \right\| \leq c |Z^{\natural}|_{1,p,S}  = c \left(|\lambda| +|r|+|Z|_{1,p,S}  \right) .\end{equation}
\end{prop}
The proof of this result is postponed until Section \ref{sec:proof_quadratic}.

As a consequence of these estimates, we can apply a quantitative version of the implicit function theorem.  In the statement of the next Corollary, 
\begin{equation}\preGext^{*}  \ker_{\epsilon} D_{\cP_{0,1}} \end{equation}
refers to the open subset of $\preGext^{*}  \ker D _{\cP_{0,1}}  \subset \preGext^{*}    T\cF^{1,p}_{\cP_{0,1}} (L)$ consisting of vectors whose norm is smaller than a constant $\epsilon$, which is assumed to be smaller than  $\frac{1}{4C^2}$.  We continue writing $X^{\natural}$ for such a tangent vector. 
\begin{prop} \label{prop:statement_implicit_function_theorem}
If $S_K$ is sufficiently large, there exists a map
\begin{equation}  \sol \co \preGext^{*}  \ker_{\epsilon} D_{\cP_{0,1}} \to  \preGext^{*} T\cF^{1,p}_{\cP_{0,1}} (L) \end{equation}
which is uniquely determined by the following two conditions
\begin{equation} \label{eq:sol_in_image_inverse} \parbox{35em}{$\sol$ factors through the right inverse $Q_{\preGext}$.} \end{equation}
\begin{equation} \label{eq:adding_sol_gives_holomorphic}  \parbox{35em}{If $X^{\natural}$ is an element  of $ \ker_{\epsilon} D_{\cP_{0,1}}$, then 
\begin{equation*} \dbar_{\cP_{0,1}}\left( \exp_{\preGext_{S}(u,\theta,v)} \left( X^{\natural} + \sol_{(u,\theta, v,S)} X^{\natural}  \right) \right) = 0 .\end{equation*}  } \end{equation}

Moreover, the following properties are satisfied 
\begin{equation} \label{eq:image_of_implicit_function_thm_surjective}
\parbox{35em}{The map \\
\begin{equation*} \id + \sol \co   \preGext^{*} \ker_{\epsilon} D_{\cP_{0,1}} \to T\cF^{1,p}_{\cP_{0,1}} (L)   \end{equation*} 
is a surjection onto the set of tangent vectors of norm bounded by $\epsilon$ whose image under the exponential map is a zero of $\dbar_{\cP_{0,1}}$.} \end{equation}
\begin{equation}  \label{eq:exponential_decay_solution} \parbox{35em}{The $C^1$-norm of the restriction of $\sol$ to the fibres over a point $(u,\theta,v,S)$ decays exponentially,  
\begin{equation*} \left| \sol|_{(u,\theta,v,S)} \right|_{C^1}  = \rO\left(e^{-2(1-\delta)S}\right). \end{equation*}}  \end{equation}
\begin{equation}  \parbox{35em}{The map $\sol$ is continuous, and its restriction $\sol_{S}$ to  $ \preGext_{*}  \ker_{\epsilon} D_{\cP_{0,1}} | K \times \{S \}$  is smooth for any $S$.}  \end{equation}
\end{prop}
\begin{proof}
The existence and uniqueness of $\sol$, the bound  \eqref{eq:exponential_decay_solution}, as well as the surjectivity of \eqref{eq:image_of_implicit_function_thm_surjective} follow from a quantitative version of the implicit function theorem.  The standard reference in symplectic geometry is Floer's version of the Picard Lemma which is proved in Proposition 24 in  \cite{floer-monopole}.    The reader is invited to check that the boundedness of $Q$ stated in \eqref{eq:Q_bounded}  and the quadratic inequality \eqref{eq:quadratic_inequality} together imply Estimate (156) of Floer's paper.

Smoothness in the  $K \times \{ S \}$ direction follows from the fact that $Q_{\preGext}$ is smooth in that direction. To prove continuity in the gluing parameter direction, we fix maps $u$ and $v$.  Given a gluing parameter $S$, and a tangent vector $X^{\natural}$ to $\preGext_{S}(u, \theta,v)$, we consider sequences converging to these data
\begin{equation}  \lim_{i \to +\infty} S_i =  S \textrm{ and } \lim_{i \to +\infty} \Pi_{i}  X^{\natural}_i = X^{\natural}, \end{equation}  where we're writing $\Pi_{i}$ for the parallel transport map from $\preGext_{S_i}(u,\theta,v)$  to $\preGext_{S}(u,\theta,v)$. Our goal is to prove that 
\begin{equation} \label{eq:convergence_sequence_solutions} \lim_{i \to +\infty}  \Pi_{i} \sol_{S_i} X_i^{\natural} = \sol_{S} X^{\natural}, \end{equation}
where we've dropped most of the subscripts from $\sol$ since the meaning should be clear.
From property \eqref{eq:adding_sol_gives_holomorphic}, we know that
\begin{equation} \lim_{i \to +\infty} \sF\left( \Pi_{i}  \left( X_{i}^{\natural} + \sol_{S_i} X_i^{\natural} \right) \right) = 0. \end{equation}
Let us assume by contradiction that \eqref{eq:convergence_sequence_solutions}  does not converge.  Since $X_{i}^{\natural} + \sol_{S_i} X_i^{\natural}$ is bounded, and $D_{\cP_{0,1}}$ is Fredholm, the projection of this sequence to the kernel $D_{\cP_{0,1}}$ along the right inverse $Q_{\preGext}$
\begin{equation}  \left(\id - Q_{\preGext} \circ  D_{\cP_{0,1}} \right) \left( X_{i}^{\natural} + \sol_{S_i} X_i^{\natural} \right) \end{equation} 
admits a convergent subsequence. Upon passing to this convergent subsequence, the fact  that $\sF$ is $C^1$-close to $D_{\cP_{0,1}}$ implies that its tangent space is transverse at every point to the kernel of $D_{\cP_{0,1}}$, so that 
\begin{equation} \lim_{i \to +\infty} \Pi_{i}  \left( X_{i}^{\natural} + \sol_{S_i} X_i^{\natural} \right) = X^{\natural} + X' \textrm{ for some } X'. \end{equation}
We shall show that $X' = \sol_{S} X^{\natural}$.  Since $\sF(X^{\natural} + X')$ vanishes, this follows immediately from the uniqueness statement of the first part of the Lemma once we know that $X'$ is in the image of $Q_{\preGext}$.

By property \eqref{eq:sol_in_image_inverse}, we may write $\sol_{S_i} X_i^{\natural} = Q_{\preGext} Y_i$ for a unique sequence $Y_i$.  Let $\tilde{\Pi}_{i}$ denote the map
\begin{equation}  L^{p} \left( u \#_{S_i} v ^{*}(TM) \otimes \Omega^{0,1} D^2 \right) \to  L^{p} \left( u \#_{S} v ^{*}(TM) \otimes \Omega^{0,1} D^2 \right)  \end{equation}
coming from a trivialization of the bundle $\cE^{p,\delta}_{\cP_{0,1}}(M)$ using parallel transport with respect to the connection $\tilde{\nabla}$, and observe that the uniform continuity of $D_{\cP_{0,1}} $ implies that
\begin{align*}  \lim_{i \to +\infty}  \tilde{\Pi}_{i}  Y_i  & =   \lim_{i \to +\infty}  \tilde{\Pi}_{i}  D_{\cP_{0,1}} \sol_{S_i} X_i^{\natural} \\
 & = \lim_{i \to +\infty} D_{\cP_{0,1}}  \Pi_{i} \sol_{S_i} X_i^{\natural} \\
 & =  D_{\cP_{0,1}} X' \equiv Y  .  \end{align*} 
Using the fact that the operators $ \Pi_{i}  Q_{\preGext} \Pi^{-1}_{i}$ are uniformly bounded to see that
\begin{equation} \left| \Pi_{i}  Q_{\preGext} \Pi^{-1}_{i}  Y_{i} -  \Pi_{i}  Q_{\preGext} \Pi^{-1}_{i}  Y \right| = \rO(|Y_i -Y|),  \end{equation}and applying the pointwise convergence at $Y$ of $ \Pi_{i}  Q_{\preGext} \Pi^{-1}_{i}$   to $Q$, we conclude that
\begin{equation} \lim_{i \to +\infty} \Pi_{i}  Q_{\preGext}  Y_{i} =  Q_{\preGext} Y, \end{equation}
which implies that the limit of $ \Pi_{i}  \sol_{S_i} X_i^{\natural}$ lies in the image of $Q_{\preGext}$.  
\end{proof}

In particular, we obtain the gluing map
\begin{equation} \label{eq:extended_gluing_map_with_epsilon} \Gext_{\epsilon} \co \preGext^{*}  \ker_{\epsilon} D_{\cP_{0,1}} \to \cP_{0,1}(L;0)\end{equation}
promised at the beginning of this section by exponentiating $\id +\sol$.  The proof of the following Corollary is given in Section \ref{sec:proof_surjectivity_extended_gluing}
\begin{cor} \label{cor:surjectivity_extended_map}
For any $\epsilon \geq 0$, the composition of $ \Gext_{\epsilon}$ with the projection to $\cP(L;0)$ is a surjection onto a neighbourhood of $K$.\end{cor}

\subsection{A family of Sobolev norms}
We begin by fixing metrics on $D^{2} - \{1\}$ and $D^2 - \{-1\}$ which agree with the pullback of the standard metrics on the strip-like ends.  Since the identification of Equation \eqref{eq:glue_surfaces_isometry} used to define $\Sigma_{S}$ is an isometry of the metric on the strip-like ends, $\Sigma_S$ carries a metric with respect to whose volume form the integrals below will be taken.

Next, we consider the function
\begin{equation} \label{eq:delta_weight_function} \kappa_{\delta, S} \co \Sigma_{S} \to [1,+\infty) \end{equation}
defined piece-wise as follows
\begin{enumerate}
\item $\kappa \equiv 1$ on the complement of the neck of $\Sigma_{S}$
\item On $\Im (\xi_{S,\neck})$, we have
\begin{equation} \kappa_{\delta, S} \circ  \xi_{S,\neck}  = e^{(2S-|s|) p \delta} .\end{equation}
\end{enumerate}

We define a norm $| \_|_{p,S}$ on $L^{p} \left( (u \#_{S} v)^{*} TM \otimes \Omega^*  D^2 \right) $ by weighting  the usual  $L^p$-norm  using $\kappa_{\delta, S}$
\begin{equation} \label{eq:weighted_norms_differential_S} |Y|^p_{p,S} = \int_{\Sigma_S}|Y|^{p} \kappa_{\delta, S} .\end{equation}
The same formula defines norms on the vector spaces $L^{p} \left( (u \#_{S} v)^{*} TM \otimes_{\bC} \Omega^{0,1} D^2 \right)$ and $L^{p} \left( (u \#_{S} v)^{*} TM \otimes_{\bC} \Omega^{1,0} D^2 \right)$.

Following Definition \ref{def:exponential_decay_sobolev_space_vector_fields}, we next define a norm $| \_|_{1,p,S,\delta}$ on vector fields along on the strip $\I{-2S}{2S}$
\begin{equation} \label{eq:weight_S_varying_sobolev_strip}| X_\bI|^p_{1,p,S,\delta} = \int_{\I{-2S}{2S}} \left(|X_\bI|^p +|\nabla X_\bI|^p \right) e^{(2S-|s|)p \delta} ds dt \end{equation}
With this notation, we define a norm on $W^{1,p}\left( (u \#_{S} v)^{*}(TM),  (u \#_{S} v)^{*}(TL) \right)$:
\begin{equation} \label{eq:weighted_norms_tangent_vectors_S}| X|^p_{1,p,S}  = \left| X|_{\Sigma_{S} - \Im(\xi_{S,\neck} )} \right|^p_{1,p} + |X(\xi_{S,\neck}(0,0))|^p +| X \circ  \xi_{S,\neck} - \Pi_{u(1)}^{u \#_{S} v} (X(\xi_{S,\neck}(0,0))|^{p}_{1,p,S,\delta}\end{equation}
where $\Pi_{u(1)}^{u \#_{S} v}$ is the parallel transport with respect to the Levi-Civita connection $\nabla$ from $u \#_{S} v \circ \xi_{S,\neck}(s,t)$ to $u(1)=u \#_{S} v \circ \xi_{S,\neck}(0,t)$ along the image of the horizontal segment on the strip (we shall later prune the subcript and superscript from $\Pi$ to simplify the notation).

\begin{rem}
We write $| \_|_{1,p,\theta,S}$ for the norm obtained on $D^{2}$ by identifying it with $\Sigma_{S}$ using the bi-holomorphism $r_{\theta} \circ \phi_{S}$.  This is the metric one would use were one to study a gluing problem occurring at the boundary point $e^{i\theta}$.  We shall abuse notation and write $| \_|_{1,p,S}$ for $| \_|_{1,p,0,S}$.
\end{rem}
The proof of the next result is discussed in Section \ref{sec:proof-results-gluing}:
\begin{lem} \label{lem:Sobolev_complete_norm}
The norm $| \_|_{1,p,S}$ defines a complete norm on the Sobolev space of vector fields $W^{1,p} \left( (u \#_{S} v)^{*} TM ,  (u \#_{S} v)^{*} TL \right)$.
\end{lem}
\begin{rem} \label{rem:L_p_norm_tangent_vectors} 
In Section \ref{sec:injectivity_derivative} we shall have occasion to use an $L^{p}$ norm on tangent vector fields.  For a vector field on a finite strip, it is defined by
\begin{equation} \label{eq:weight_S_varying_sobolev_strip_L_p}\left|X_\bI\right|^p_{p,S,\delta} = \int_{\I{-2S}{2S}} \left|X_\bI\right|^p e^{(2S-|s|) p \delta} ds dt \end{equation}
which we extend to a norm $| \_|_{p,S}$ on vector fields  on $\Sigma_{S}$ valued in $TM$ in exactly the same way as in Equation  \eqref{eq:weighted_norms_tangent_vectors_S}.  In particular, the appearance of the parallel transport of the value at the origin distinguishes this norm from the weighted $L^p$ norm of Equation  \eqref{eq:weighted_norms_differential_S} (of course, the two norms are on different Sobolev spaces, but it is easy to forget this since the tangent bundle of $D^2$ is trivial).
\end{rem}

\subsection{Construction of the right inverse} \label{sec:right_inverse}
Recall that the Banach bundle $T \cF^{1,p,\delta}_{\cod{1}{L}}(L)$, which controls the tangent space of the codimension $1$ boundary strata of $\Pbar(L,0)$, was defined following Equation \eqref{eq:fibre_product_dbar_gluing}. We begin by defining maps 
\begin{equation} \label{eq:pre-gluing_tangents} \predGext_{S} \co T \cF^{1,p,\delta}_{\cod{1}{L}}(L) |K \to \preGext^{*}_{S}  T\cF^{1,p}_{\cP_{0,1}} (L),  \end{equation}
which we think of as pre-gluings of tangent vector fields.  We write $  \predGext $ for the union of these maps over all $S$,  formally yielding a map of Banach bundles over $K \times [0,+\infty)$.

Explicitely, the fibre of $T \cF^{1,p,\delta}_{\cod{1}{L}}(L)$ at $(u,\theta, v)$ consists of deformation of $R$ and $\theta$, as well as vector fields $X_u$ and $X_v$ along $u$ and $v$ which agree at the evaluation point
\begin{equation}  X_{u}(1) - X_{v}(-1) =0  . \end{equation}

The first step is to define a vector field $X_{u} \#_{S} X_{v}$ along $u \#_{S} v$.  As usual, this will be defined using the decomposition
\begin{equation} \Sigma_{S} = \Im(\iota^{\cP,S}_{0})  \cup \Im(\xi_{S,\neck}) \cup \Im(\iota^{\cS,S}_{0}) . \end{equation}
On the two components of the complement of the neck, $X_{u} \#_{S} X_{v}$ agrees respectively with $X_{v}$ and $X_{u}$:
\begin{align} X_{u} \#_{S} X_{v} \circ \iota^{\cP,S}_{0} = X_{u}  & \text{ on $\Im(\iota^{\cP,S}_{0})$} \\
 X_{u} \#_{S} X_{v} \circ \iota^{\cS,S}_{0} = X_{v} & \text{ on $ \Im(\iota^{\cS,S}_{0})$.}  \end{align}
Along the neck, $X_{u} \#_{S} X_{v}$ interpolates between these two vector fields:
\begin{align} \notag
X_{u} \#_{S} X_{v} \circ \xi_{S,\neck} & = \Pi_{u(1)}^{u \#_{S} v \circ \xi_{S,\neck} } X_{u}(1) + \chi_{S}  \left( \Pi_{u \circ \xi_{1} \circ \tau_{2S} }^{u \#_{S} v \circ \xi_{S,\neck} } X_{u} \circ  \xi_{1} \circ \tau_{2S} - \Pi_{u(1)}^{u \#_{S} v \circ \xi_{S,\neck} } X_{u}(1)  \right)   \\  
\label{eq:formula_gluing_vector_fields} & 
+ (1-\chi_{-S}) \left(  \Pi_{v \circ \xi_{-1} \circ \tau_{-2S} }^{u \#_{S} v \circ \xi_{S,\neck} }  X_{v} \circ  \xi_{-1} \circ \tau_{-2S}  - \Pi_{u(1)}^{u \#_{S} v \circ \xi_{S,\neck} } X_{u}(1) \right) 
.\end{align}
Here, $\Pi$ stands either for parallel transport from $u(1) = v(-1)$ along the images of horizontal lines under $u \#_{S} v$, or for parallel transport along minimal geodesics from the image of $u$ (or $v$) to the image of $u \#_{S} v$.    We are following the conventions of \cite{FOOO} (see in particular Equation (7.1.23)), so that the cutoff for the vector fields $X_u$ and $X_v$ takes place in a region where $u \#_{S} v$ agrees with one of the maps $u$ and $v$.   In particular, if $s < -S-1$, $\chi_{S}(s,t) = \chi(s-S)$ is equal to $1$, whereas $1-\chi_{-S}$ vanishes; in this region, we recover $X_{u}$.  Similarly, the vector field we constructed agrees with $X_{v}$ when $s>S+1$.

We now state our first estimate concerning the behaviour of $\predGext_{S}$ as $S$ grows.  The proof is given in Section \ref{sec:proof-results-gluing}:
\begin{lem} \label{lem:pre-gluing_tangent_vectors_bounded} 
The pre-gluing map \eqref{eq:pre-gluing_tangents}, given for fixed $S$ by \begin{equation} \label{eq:definition_pre-gluing_vector_fields} \predGext_{S} (X_{v}, r \partial_{R}, \lambda \partial_{\theta},  X_{u})  = \left(r \partial_{R} , \lambda \partial_{\theta}, X_{u} \#_{S} X_{v} \circ \phi^{-1}_S\right) ,\end{equation}
is a strongly continuous map of Banach bundles over $K \times [S_K , +\infty)$, whose restriction to any fibre is uniformly bounded, and whose restriction to a slice $K \times \{ S \}$ is uniformly continuous and smooth.   Moreover if $X_u$ and $X_v$ extend to smooth vector fields on $D^2$, then 
\begin{equation} \label{eq:pre-gluing_converges_toisometry} \lim_{S \to +\infty} \left|  \predGext_{S}(X_{v}, r \partial_{R}, \lambda \partial_{\theta},  X_{u}) \right|_{1,p,S} = \left| \left(X_{v}, r \partial_{R}, \lambda \partial_{\theta},  X_{u}\right)  \right|_{1,p,\delta} .\end{equation}
\end{lem}

We now define an approximate inverse
\begin{equation}  \label{eq:approximate_inverse} \tilde{Q} \co \preGext^{*} \cE^p_{\cP_{0,1}} (M)  \to  \preGext^{*}  T\cF^{1,p}_{\cP_{0,1}} (L) .\end{equation}

Given a $TM$-valued anti-holomorphic $1$-form along the image of $u \#_{S} v$,  
\begin{equation*} Y \in L^{p}\left((u \#_{S} v)^{*} TM \otimes \Omega^{0,1} D^2\right), \end{equation*}  
we use parallel transport to construct such a $1$-form along the image of $v$:
\begin{align} Y_{v} &   \in  L^{p}\left(v^{*}(TM) \otimes \Omega^{0,1} D^2 \right) \\ Y_{v} (p)& = \begin{cases} \tilde{\Pi}_{u \#_{S} v \circ  \iota^{\cS,S}_{2S}}^{v}(Y(p)) & \text{on $D^{2} -  \xi_{-1}(\I{-\infty}{-2S})$} \\
0 & \text{otherwise.} \end{cases} \end{align}
Here, the parallel transport goes from $u \#_{S} v \circ  \iota^{\cS,S}_{2S}  (p)$ to $v(p)$ along the minimal geodesic between them (recall that $S$ has been chosen large enough for this to make sense). Similarly, we define 
\begin{align}
Y_{u} & \in  L^{p}(v^{*}(TM) \otimes \Omega^{0,1}) \\
Y_{u} (p) & =  \begin{cases} \tilde{\Pi}_{u \#_{S} v \circ  \iota^{\cP,S}_{2S} }^{v}(Y(p)) & \text{on $D^{2} -  \xi_{-1}(\I{2S}{+\infty})$} \\
0 & \text{otherwise.} \end{cases} 
\end{align}
Composing with $\phi_{S}$, these two maps define a map of bundles over $K \times [S, +\infty)$ 
\begin{equation} \label{eq:breaking_1_forms_holomorphic} \B \co  \preGext^{*} \cE^p_{\cP_{0,1}} (M)  \to  \cE^{p,\delta}_{\cod{1}{L}}(M)  .\end{equation}
which we call the {\bf breaking map}.  Recall that outside of $\xi_{S,\neck}(\I{-1}{1})$, the map  $u \#_{S} v$ agrees with $u$ or $v$, so the parallel transport maps appearing in the definition of $B$ are the identity away from this region.  Since we've assumed $S$ to be sufficiently large, and the metric $g$ to be hermitian, $\tilde{\Pi}$ preserves the norm of vector fields whenever $z$ lies in the image of $\xi_{S,\neck}| \I{-1}{1}$. 
\begin{lem}
The breaking map is a strongly continuous map of Banach bundles over $K \times [S_K,+\infty)$ whose restriction to a slice is uniformly continuous and smooth.  Moreover, the restriction of the breaking map $B$ to any fibre is an isometric embedding, i.e.
\begin{equation}|Y_u|_{p,\delta} + |Y_v|_{p,\delta} =|Y|_{p,S}. \end{equation}
 \noproofe
\end{lem}

Together with Lemma \ref{lem:pre-gluing_tangent_vectors_bounded}, this immediately implies
\begin{lem}
The map $\tilde{Q}$ of \eqref{eq:approximate_inverse} which is defined by
\begin{equation} \tilde{Q} = \predGext \circ Q_{\cod{1}{L}} \circ B  \end{equation}
is a strongly continuous map of Banach bundles over $K \times [S_K , +\infty)$, whose restriction to any fibre is uniformly bounded, and whose restriction to a slice $K \times \{ S \}$ is uniformly continuous and smooth.\noproof
\end{lem}

To justify our designation of $\tilde{Q}$ as an approximate inverse, we shall prove the next result in Section \ref{sec:proof-results-gluing}:
\begin{lem}[Compare Lemma 7.1.32 of \cite{FOOO}] \label{lem:approximate_right_inverse}
The norm of the operator
\begin{equation} D_{\cP_{0,1}} \circ \tilde{Q} - \id \co \preGext^{*} \cE^p_{\cP_{0,1}} (M) \to \preGext^{*} \cE^p_{\cP_{0,1}} \end{equation}
is bounded by a constant multiple of $e^{-2 \delta S}$. 
\end{lem}
\begin{cor} \label{cor:uniformly_bounded_inverse}
The composition $D_{\cP_{0,1}} \circ \tilde{Q}$ defines an isomorphism of Banach bundles
\begin{equation} \preGext^{*} \cE^{p}_{\cP_{0,1}}(M) \to \preGext^{*}\cE^{p}_{\cP_{0,1}}(M)  \end{equation}
with respect to the strong topology.  Its restriction to a slice $K \times \{ S \}$ is a diffeomorphism of Banach bundles such that the composite
\begin{equation} Q_{\preGext} = \tilde{Q}  \circ (D_{\cP_{0,1}}  \circ \tilde{Q} )^{-1}  \end{equation}  
is a uniformly bounded right inverse to $D_{\cP_{0,1}} $. \noproof
\end{cor}

\section{Local injectivity of the gluing map}
In Equation \eqref{eq:extended_gluing_map_with_epsilon}, we defined $ \Gext_{\epsilon} $ to be the gluing map resulting from applying the version of the implicit function theorem stated in Proposition \ref{prop:statement_implicit_function_theorem}. Let $\Gext$ denote the restriction of $\Gext_{\epsilon}$ to the zero-section, and define the gluing map 
\begin{equation} \label{eq:gluing_map} \G \co K \times [S,+\infty) \to \cP(L;0) \end{equation}
to be the composition of $\Gext$ with the fibre bundle 
\begin{align} \label{eq:projection_of_boundary_marked_point}  \cP_{0,1}(L;0) & \to  \cP(L;0)  \\
(R, \theta,w) & \mapsto (R, w \circ r_{-\theta}) \end{align}
with fibres $S^1$.

The main result of this section is
\begin{prop} \label{prop:local_injectivity_gluing}
For $S$ sufficiently large, the restriction  $\G_{S}$ of \eqref{eq:gluing_map} to a slice $K \times \{ S \}$ is a smooth immersion.  Moreover, each point of $\cod{1}{L}$ has a neighbourhood independent of $S$ on which the restriction of $\G_{S}$ is injective.
\end{prop}
We shall find a vector in $\ker D_{\cod{1}{L}}$ but not in the tangent space to $K$ whose image under $\predGext_{S}$ becomes arbitrarily close to the tangent space of the fibres of \eqref{eq:projection_of_boundary_marked_point}.  Moreover, we shall prove that the difference between the restriction of $\predGext_{S}$ to $TK$ and the differential of $\Gext_{S}$ decays with $S$.  Using Equation \eqref{eq:pre-gluing_converges_toisometry}, we conclude that the differential of $\Gext_{S}$ is injective for large $S$.  This will imply that the image of a tangent vector to $K$ under the differential of $\Gext_{S}$ cannot lie in the tangent space to the fibres of \eqref{eq:projection_of_boundary_marked_point}, and hence that $d\G_{S}$ is also injective.

In order to be more precise, let us introduce the notation $u^{\sharp} = (u,\theta,v)$ for an element of $\cod{1}{L}$, and write
\begin{align}
w^{\natural}_{S} & \equiv \preGext_{S}(u^{\sharp})  \\
w^{\flat}_{S} & \equiv \Gext_{S}(u^{\sharp}) 
\end{align}
At $  w^{\flat}_{S}  $, the tangent space to the fibres of \eqref{eq:gluing_map} is spanned by the vector
\begin{equation} \label{eq:tangent_vector_circle} dw^{\flat}_{S} \left( \partial_{\theta} \right)   = \left(0, \partial_{\theta}, dw_{S} \left( \partial_{\theta} \right)\right).\end{equation}
where we use $ \partial_{\theta}$ to denote the infinitesimal generator of rotation on both $S^1$ and $D^2$.
Recall that $\partial_{p}$  stands for  the unique holomorphic vector field on $D^2$ which vanishes at $-1$ and agrees with $\partial_{\theta}$ at $1$, and that the kernel of $D_{\cod{1}{L}}$ at any point $(R,u,\theta,v)$ as a direct sum decomposition as in Equation \eqref{eq:decomposition_kernel}
\begin{equation} \label{eq:decomposition_kernel_again} \aut(D^2,-1) \oplus T \cod{1}{L}   \tilde{\to} \ker D_{\cod{1}{L}}, \end{equation}
where $\aut(D^2,-1)$ is spanned by the vector fields $\partial_{p}$ and $\partial_{s}$, whose image in the right hand side is obtained by applying the differential of $v$.  The next result, proved in Section \ref{sec:injectivity_derivative}, identifies the vector field in the image of $\predGext_{S}$ which is close to the kernel of the projection map to $\cP(L;0)$:
\begin{lem} \label{lem:pre-gluing_parabolic_close_to_rotation}
There is an exponential decay bound
\begin{equation}\label{eq:decay_difference_rotation_parabolic} 
\left| e^{-4S} dw^{\flat}_{S} \left( \partial_{\theta} \right) - \Pi_{w^{\natural}_{S}}^{w^{\flat}_{S}}  \predGext_{S} \left( dv\left(  \partial_{p} \right) \right)  \right|_{p,S} = \rO(e^{- S}) \end{equation}
\end{lem}

In order to prove the injectivity of the derivative of the gluing map, we shall also have to control its behaviour in the other directions of the tangent space of the source.  This can be done by comparing it with the map $ \predGext_S $, although we have to apply the appropriate parallel transport maps (and projections) in order to do this.  Note that the quadratic inequality implies that the composition
\begin{equation} \Pi_{w^{\natural}_{S}}^{w^{\flat}_{S}}   \circ Q_{\preGext_{S}} \circ    \Pi_{w^{\flat}_{S}}^{w^{\natural}_{S}}   \end{equation}
is arbitrarily close to a right inverse to $D_{\cP_{0,1}}$ along the image of $\Gext_{S}$.    In particular, there is a unique right inverse $Q_{\Gext_{S}} $ whose image agrees with the image of this map, and which is moreover uniformly bounded.  We write
\begin{equation} \kappa_{\Gext_{S}}  \equiv \id - Q_{\Gext_{S}} \circ D_{\cP_{0,1}}  \end{equation}
for the projection to $\ker D_{\cP_{0,1}}$ along $Q_{\Gext_{S}}$.    By composing with  $  \Pi_{w^{\natural}_{S}}^{w^{\flat}_{S}} \circ \predGext_S$, we obtain a map
\begin{equation} \label{eq:map_gluing_kernels} \kappa_{\Gext_{S}} \circ  \Pi_{w^{\natural}_{S}}^{w^{\flat}_{S}}  \circ  \predGext_S \co   \ker D_{\cod{1}{L}}|K \to \Gext_{S}^{*} T \cP_{0,1} (L;0). \end{equation}
Our main bound on the derivative of $ \Gext_{S}$ follows from this result whose proof also appears in Section \ref{sec:injectivity_derivative}:
\begin{lem} \label{lem:image_differential_gluing_transverse_circle}
The restriction of \eqref{eq:map_gluing_kernels} to $ TK \oplus \langle dv \left( \partial_{s} \right) \rangle  $ is transverse to the subspace spanned by $ d w^{\flat}_{S}  \left( \partial_{\theta} \right)$. Moreover, its restriction to the tangent space of $K$ satisfies
\begin{equation} \left\| d\Gext_{S} - \kappa_{\Gext_{S}}  \circ  \Pi_{w^{\natural}_{S}}^{w^{\flat}_{S}}  \circ   \predGext_S \right\| = \rO\left( e^{-2\delta S} \right). \end{equation}
\end{lem}

None of the proofs in this section require the use of an implicit function theorem except in quoting results  from the previous section.  Moreover, all the important operators we need to bound have sources which are finite dimensional vector spaces of smooth vector fields.  As we shall see, this implies that we can often work with the $L^p$-norm in the target, and a stronger norm on the source, which simplifies computations.

\subsection{The derivative pre-gluing of tangent vectors}
Our first observation concerns the restriction of pre-gluing to the kernel of the operator \eqref{eq:fibre_product_dbar_gluing}.
\begin{lem} \label{lem:pre-gluing_uniform_bilipschitz}
The restriction of $\predGext_{S}$ to 
\begin{equation} \ker D_{\cod{1}{L}}  \subset T \cF^{1,p,\delta}_{\cod{1}{L}}(L) \end{equation}
is  bi-lipschitz (uniformly in $S$) with respect to both (i) the $W^{1,p}$ norms $| \_ |_{1,p,\delta}$ on the source and  $| \_ |_{1,p,S}$ on the target and (ii) the $L^p$ norms $| \_ |_{p,\delta}$ and $| \_|_{p,S}$.  Moreover, the bi-lipschitz constant converges to $1$.
\end{lem}
\begin{proof}
This follows immediately from elliptic regularity, Equation \eqref{eq:pre-gluing_converges_toisometry}, and the fact that $\ker D_{\cod{1}{L}}$  is finite-dimensional.
\end{proof}

Given such a vector field $X^{\sharp} \in  \ker D_{\cod{1}{L}}$, we observe that $D_{\cP_{0,1}} \predGext_{S}(X^{\sharp} )$ is supported on the image of $\xi_{S,\neck} \left(  \I{-S-1}{S+1} \right)$.  Expressing  $\predGext_{S}(X^{\sharp} )$ as the sum of the two vector fields coming from the two sides of the gluing, it is easy to check that whenever one of the vector fields is not holomorphic, its $C^1$-norm is bounded by $\rO\left( e^{-2S} \right)$.    Since the weight is bounded by $e^{-2 \delta S}$ we conclude an  analogue of Equation \eqref{eq:norm_dbar_exponentially_decay}  for the pre-gluing of tangent vectors.
\begin{lem}  \label{lem:tangent_pre-gluing_almost_holomorphic}
The composition of $D_{\cP_{0,1}}$ with $\predGext_{S}$ is an operator
\begin{equation}  \ker D_{\cod{1}{L}}  \to \preGext_{S}^{*} \cE_{\cP_{0,1}}^{p}(M)\end{equation}
whose operator norm with respect to $| \_ |_{1,p,S}$  is bounded by a constant multiple of $e^{-2 (1-\delta) S}$.   \noproof
\end{lem}

Note that we have an inclusion
\begin{equation} TK \subset \ker D_{\cod{1}{L}}|K \end{equation}
which we extend to
\begin{equation} TK \times [S, +\infty) \hookrightarrow \ker D_{\cod{1}{L}}|K \end{equation}
by identifying the tangent space to the gluing parameter $S$ with the vector field in the kernel of $D_{\cod{1}{L}}$  spanned by translation along the strip.  The next Lemma compares the derivative of pre-gluing curves with the result of pre-gluing vector fields.

\begin{lem} \label{eq:differential_pre-gluing_preluing_tangents_almost_equal}
The difference
\begin{equation} d \preGext_{S} - \predGext_{S}  \co  TK \oplus \langle \partial_{s} \rangle \to \preGext_{S}^{*} T \cF^{1,p}_{\cP_{0,1}}(L) \end{equation}
is bounded by a constant multiple of $e^{-2(1-\delta)S}$ in the $| \_ |_{1,p,S}$  norm.
\end{lem}
\begin{proof}
Given such a tangent vector field $X$, the difference between $ d \preGext_{S} (X)$ and $\predGext_{S}(X)$ is supported in $\xi_{S,\neck}(\I{-S-1}{S+1})$.  However, since $X$ is smooth, the pointwise $C^1$-norm of this difference is bounded above by a constant multiple of $e^{-2S}$. Since the integral of the weight is bounded by $e^{2\delta S}$, the result follows.
\end{proof}
This Lemma, together with Lemma \ref{lem:tangent_pre-gluing_almost_holomorphic}, immediately implies:
\begin{cor} \label{eq:image_derivative_pre_gluing_dbar_bounded}
The restriction of $D_{\cP_{0,1}}$ to $\Im(d \preGext_{S})$ is uniformly bounded in $| \_ |_{1,p,S}$   by a constant multiple of $e^{-2 (1-\delta) S}$. \noproof  
\end{cor}

\subsection{The derivative of the gluing map} \label{sec:derivative_gluing}
Consider the composition of $\predGext$ with $ \Pi_{w^{\natural}_{S}}^{w^{\flat}_{S}} $, the parallel transport map from the image of $\preGext$ to the image of the gluing map.  The goal of this section is to prove:
\begin{lem} \label{lem:derivative_gluing_close_to_parallel_transport_pre-gluing}
The operator
\begin{equation} d \Gext_{S}  -  \Pi_{w^{\natural}_{S}}^{w^{\flat}_{S}}  \circ \predGext_{S}|  TK \end{equation}
is bounded in the $| \_|_{p,S}$-norm by a constant multiple of $e^{-2(1-\delta)S}$.
\end{lem}
Since $\Gext$ is defined as the composition of $\preGext$ with $\sol$, using Lemma \ref{eq:differential_pre-gluing_preluing_tangents_almost_equal}, we shall see that  it suffices to prove that the derivative of $\sol$ is bounded.  A straightforward application of Floer's version of the Picard Lemma as stated in Proposition \ref{prop:statement_implicit_function_theorem} yields $C^0$-bounds on $\sol$, and on its derivatives in the direction of vectors in the kernel of $D_{\cP_{0,1}}$, but we're missing a bound on derivatives in directions tangent to $K$.

In particular, let $u^{\sharp} = (u,\theta,v)$ and $u_{X}^{\sharp}$ denote of elements of $K$, and assume that 
\begin{equation} u_{X}^{\sharp} = \exp_{u^{\sharp}} X^{\sharp} \end{equation}
for an element of $X^{\sharp} \in T \cF^{1,p,\delta}_{\cod{1}{L}}(L)$.   We write $w_{S}^{\natural}$ and $w_{S,X}^{\natural}$ for the images of $u^{\sharp}$ and $u_{X}^{\sharp}$ under $\preGext$ for some gluing parameter $S$, while  $ w^{\flat}_{S}$  and $ w^{\flat}_{S,X}$ stand for the images under $\Gext_{S}$.  If $S$ is sufficiently large, we have a vector field $X^{\natural}$ along $w^{\natural}_{S}$ such that  $w^{\natural}_{S,X} = \exp_{w^{\natural}_{S}} X^{\natural}$.  

Combining  Lemmas \ref{lem:pre-gluing_uniform_bilipschitz} and Equation \ref{eq:differential_pre-gluing_preluing_tangents_almost_equal} we find that we may choose $| X^{\sharp}|_{1,p,\delta}$ small enough so  that $| X^{\natural}|_{1,p,S}$ is uniformly bounded by a constant multiple of  $| X^{\sharp}|_{1,p,\delta}$.  In particular, we may assume that it is smaller than the constant $\epsilon$ of Proposition \ref{prop:statement_implicit_function_theorem}, so that we can work in a chart about $ w^{\natural}_{S}$ where the quadratic inequality holds.  In such a chart, Lemma \ref{lem:derivative_gluing_close_to_parallel_transport_pre-gluing} is equivalent to the bound
\begin{equation} \left|  \Pi_{w^{\flat}_{S}}^{w^{\natural}_{S}}\exp^{-1}_{ w^{\flat}_{S}} \left( w^{\flat}_{S,X}  \right) - \predGext_{S} X^{\sharp} \right|_{p,S} = \rO(e^{-2(1-\delta)S})|X^{\sharp}| + \rO(1)|X^{\sharp}|^2 \end{equation}

By applying Lemma \ref{lem:error_add_vector_fields_L_p_S} twice, we find that the first term can be approximated by a sum of vector fields:
\begin{equation} \label{eq:replace_inverse_exponential} \left|   \Pi_{w^{\flat}_{S}}^{w^{\natural}_{S}} \exp^{-1}_{ w^{\flat}_{S}} \left( w^{\flat}_{S,X}  \right)  -  \exp^{-1}_{w^{\natural}_{S}} \left( w^{\flat}_{S,X}    \right) - \sol_{w^{\natural}_{S}}(0) \right|_{p,S} = \rO(1)|X^{\sharp}|^2 \end{equation}
According to Lemma \ref{eq:differential_pre-gluing_preluing_tangents_almost_equal}, we have
\begin{equation} \left|  \predGext_{S} X^{\sharp}  - X^{\natural} \right|_{p,S}  =\rO \left(e^{-2(1- \delta)S} \right)|X^{\sharp}| +  \rO(1)|X^{\sharp}|^2 \end{equation}
These two bounds imply that we have to prove
\begin{equation} \label{eq:estimate_inverse_image_nearby_glued_map_original_inverse} \left|\exp^{-1}_{ w^{\natural}_{S}} \left( w^{\flat}_{S,X}  \right) - \left( X^{\natural} + \sol_{w^{\natural}_{S}}(0) \right) \right|_{p,S} =\rO \left(e^{-2(1- \delta)S} \right) |X^{\sharp}| + \rO(1) |X^{\sharp}|^{2}
\end{equation}

Because the vector field $\sol$ is constructed implicitly, we shall obtain \eqref{eq:estimate_inverse_image_nearby_glued_map_original_inverse} in a round about way.  Recall that we have a direct sum decomposition
\begin{equation}{w^{\natural}_{S}}^*  T \cF^{1,p}_{\cP_{0,1}}(L) \cong \Im Q_{\preGext} \oplus \ker D_{\cP_{0,1}},\end{equation}
which induces a uniformly bounded  projection map
\begin{equation} \kappa_{w^{\natural}_{S}}  \co  \preGext_{S}^{*} T \cF^{1,p}_{\cP_{0,1}}(L) \to \preGext_{S}^{*}  \ker D_{\cP_{0,1}} .\end{equation}
By the defining properties of $\sol$, Equations \eqref{eq:sol_in_image_inverse} and \eqref{eq:derivative_gluing_close_to_parallel_transport_pre-gluing_chart}, the restriction of $ \kappa_{w^{\natural}_{S}}$ to the inverse image of $\cP_{0,1}(L;0)$ under the exponential map is the inverse diffeomorphism to $\sol + \id$.   This will allow us to prove:
\begin{lem}
The restriction of  $\sol_{w^{\natural}_{S}} \circ \kappa_{w^{\natural}_{S}} $ to  $\exp_{w^{\natural}_{S}}^{-1}( \Im(\preGext_{S}))$, satisfies a uniform bound:
 \begin{equation}  \left| X^{\natural}  + \sol_{w^{\natural}_{S}}(0) -
 \left( \kappa_{w^{\natural}_{S}}  X^{\natural} + \sol_{w^{\natural}_{S}}( \kappa_{w^{\natural}_{S}}  X^{\natural}) \right)  \right|_{p,S}  =  \rO(e^{-2(1- \delta) S})|X^{\natural}| + \rO(1)|X^{\natural}|^2  \end{equation}
In particular, the restriction of $(\id + \sol_{w^{\natural}_{S}}) \circ \kappa_{w^{\natural}_S} $ to $\exp_{w^{\natural}_{S}}^{-1}( \Im(\preGext_{S}))$ is a diffeomorphism near the origin, and the derivative has uniformly bounded inverse.
\end{lem}
\begin{proof}
Note that the distance between $\Im( d \preGext_{w^{\natural}_{S}})$ and $\exp_{w^{\natural}_{S}}^{-1}( \Im(\preGext_{S}))$ is bounded by a constant multiple of $|X^{\natural}|^2$.  The proof that this constant can be chosen independently of $S$ follows from the exponential decay estimates of Lemma \ref{lem:exponential_decay_bounds}, and the fact that the tangent space of $K$ consists of smooth vector fields whose point-wise norm decays exponentially along the neck.

It suffices therefore to prove that
\begin{equation} \label{eq:bound_extend_sol_pre-gluing_chart} \left|X^{\natural} - 
\left( \kappa_{w^{\natural}_S}  X^{\natural} + \sol_{w^{\natural}_{S}}( \kappa_{w^{\natural}_S}  X^{\natural})  - \sol_{w^{\natural}_{S}}(0) \right) \right| =  \rO(e^{-2(1-\delta) S})|X^{\natural}|  \end{equation}
whenever $X^{\natural}  \in \Im( d \preGext_{w^{\natural}_{S}})$.  To prove this, we observe that Corollary \ref{eq:image_derivative_pre_gluing_dbar_bounded} implies that
\begin{equation}|X^{\natural} - \kappa_{w^{\natural}_S}  X^{\natural}| = \left| Q_{\preGext} \circ D_{\cP_{0,1}} X^{\natural}  \right| =  \rO\left(e^{-2(1- \delta) S} \right)|X^{\natural}| \end{equation}

In addition, by Equation \eqref{eq:exponential_decay_solution},  the $C^1$-norm of $\sol$ is  bounded by $\rO(e^{-2(1-\delta)S})$.  Composing this inequality with $\id + \sol$, we obtain the desired result.
\end{proof}

Returning to Equation \eqref{eq:estimate_inverse_image_nearby_glued_map_original_inverse}, we find that the missing estimate is
\begin{equation} \label{eq:derivative_gluing_close_to_parallel_transport_pre-gluing_chart}   \left| \kappa_{w^{\natural}_S}  X^{\natural} + \sol_{w^{\natural}_{S}}( \kappa_{w^{\natural}_{S}}  X^{\natural})   -  \Pi^{w^{\natural}_{S}}_{w^{\flat}_{S}}   \exp^{-1}_{w^{\flat}_{S}}\left( w^{\flat}_{S,X} \right) \right|_{p,S} \\ =\rO \left(e^{-2(1- \delta)S} \right)|X^{\sharp}| +  \rO(1)|X^{\sharp}|^2 \end{equation}

We shall obtain such a bound by bounding the derivatives of the right inverse $Q_{\preGext}$. First, since $K$ is compact, all reasonable norms on $TK$ are comparable. In particular, we have a universal bound
\begin{equation}| X^{\sharp}|_{C^1} \leq c |X^{\sharp}|_{1,p,\delta}, \end{equation}
where we shall use the $C^1$-norm coming from the standard metric on $D^2$.  Moreover, since $\dbar_{\cod{1}{L}} u^{\sharp}_{X^{\sharp}} = 0$, there is a constant $c_0$ independent of $X^{\sharp}$ such that
\begin{equation}| D_{\cod{1}{L}} X^{\sharp} |_{p,\delta} \leq  c_0|X^{\sharp}|_{1,p,\delta}^{2}.\end{equation}
Also, since $Q_{\cod{1}{L}}$ is a smooth family of right inverses on a compact region, we have a uniform bound 
 \begin{equation} \label{eq:Q_source_bounded_derivatives} \left\| Q_{\cod{1}{L}}  - \Pi^{u^{\sharp}}_{u_{X}^{\sharp}}  \circ Q_{\cod{1}{L}} \circ  \tilde{\Pi}_{u^{\sharp}}^{u_{X}^{\sharp}} \right\| \leq c_1|X^{\sharp}|  \end{equation}
where $\tilde{\Pi}$ is parallel transport with respect to the connection preserving the complex structure.  This operator is acting on the fibre of $\cE^{p,\delta}_{\cod{1}{L}}(M)$ at $u^{\sharp}$ with target $T\cF^{1,p,\delta}_{\cod{1}{L}}(L)$.

Our goal is to obtain a bound similar to \eqref{eq:Q_source_bounded_derivatives} for the right inverse $Q_{\preGext}$ evaluated at $w^{\natural}_{S}$ and $w_{S,X}^{\natural}$:
\begin{lem} \label{lem:uniform_bound_difference_nearby_right_inverse}
If both the source and the target are equipped with the $L^p$-norms $| \_ |_{p,S}$, there is a uniform constant $C$ such that 
\begin{equation} \left\| Q_{\preGext, w^{\natural}_{S}} - \Pi_{w_{S,X}^{\natural}}^{w^{\natural}_{S}} Q_{\preGext, w_{X}^{\natural}} \tilde{\Pi}^{w_{S,X}^{\natural}}_{w^{\natural}_{S}} \right\| \leq C|X^{\sharp}|. \end{equation}
\end{lem}
\begin{proof}
Lemma  \ref{lem:approximate_right_inverse} implies that the difference between  $Q_{\preGext}$ and the approximate inverse $\tilde{Q}$ is uniformly bounded by a constant multiple of $e^{-2\delta S}$.  In particular it shall suffice to prove analogue of \eqref{eq:Q_source_bounded_derivatives}  for the approximate inverses.

Now $\tilde{Q}$ is defined as the composition $\predGext \circ Q_{\cod{1}{L}} \circ B$, so  the desired result follows from estimates
\begin{align} \left\| \tilde{\Pi}_{u^{\sharp}}^{u_{X}^{\sharp}} \circ B  -   B \circ \tilde{\Pi}^{w_{S,X}^{\natural}}_{w^{\natural}_{S}} \right\| & = \rO\left(e^{-2(1-\delta)S} \right)|X^{\sharp}| \\ \label{eq:error_parallel_transport_pre-gluing}
\left\| \predGext \circ \Pi^{u^{\sharp}}_{u_{X}^{\sharp}}  -   \Pi_{w_{S,X}^{\natural}}^{w^{\natural}_{S}}  \circ \predGext  \right\| & = \rO\left(e^{-2(1-\delta)S} \right)|X^{\sharp}|.
 \end{align} 
To prove the first estimate, consider $Y \in L^{p}(w^{*} TM \otimes \Omega^{0,1} D^2)$, and write $(Y_u,Y_v)$ for $B(Y)$.  The main point that $\left( \tilde{\Pi}^{w_{S,X}^{\natural}}_{w^{\natural}_{S}} Y \right)_{u}$ differs from $ \tilde{\Pi}_{u}^{u_{X}}  Y_{u}$ only in $\xi_{1} ( \I{2S-2}{2S})$.  In this region, the parallel transport  between $u$ and $w$ takes place along paths of length bounded by $e^{-2S}$.  The first estimate therefore follows by applying Equation \eqref{eq:parallel_transport_bound_error_family},  and using the fact that the weight is bounded by $e^{-2\delta p S}$.

The second estimate is of a similar nature.  The terms whose difference we're bounding differ only on $\xi_{S,\neck} \left( \I{-S-1}{S+1} \right)$, where the distance between $w$ and $u$ or $v$, as well as the norm of the differential of these functions is bounded by $\rO \left( e^{-S} \right)$. Moreover, the $C^0$-norm of $\nabla X^{\sharp}$  and $\nabla X^{\natural}$ are also bounded by  $ \rO\left(e^{-S} \right)|X^{\sharp}|$.    Applying   Equation \eqref{eq:parallel_transport_bound_error_family} yields the desired inequality.
\end{proof}

In particular there is a unique right inverse $Q^{X}_{w^{\natural}_{S}}$ whose image agrees with the image of $\Pi_{w_{S,X}^{\natural}}^{w_{S}^{\natural}} Q_{\preGext, w_{S,X}^{\natural}} \tilde{\Pi}^{w_{S,X}^{\natural}}_{w^{\natural}_{S}} $.  Moreover, we have a uniform bound
\begin{equation} \label{eq:two_right_inverses_close} \left\| Q_{\preGext, w^{\natural}_{S}} - Q^{X}_{w^{\natural}_{S}} \right\| \leq c_3|X^{\sharp}| \end{equation}
To deduce the desired $C^{1}$ bound on $\sol$, let  $\kappa^{X}_{w^{\natural}_{S}}$ denote the projection to $\ker D_{\cP_{0,1}}$ along the image of $Q^{X}_{\preGext, w^{\natural}_{S}}$, and let
\begin{equation}  \sol^{X}_{w^{\natural}_{S}} \co \preGext_{*}  \ker_{\epsilon} D_{\cP_{0,1}} \to T_{ w^{\natural}_{S}}\cF^{1,p}_{\cP_{0,1}}(L) \end{equation}
denote the map given by Proposition \ref{prop:statement_implicit_function_theorem} if we use the right inverse $Q^{X}_{w^{\natural}_{S}}$ rather than $ Q_{\preGext, w^{\natural}_{S}}$. 

\begin{lem}
There is a uniform bound
\begin{equation} \left| \kappa_{w^{\natural}_{S}} X^{\natural} +  \sol_{w^{\natural}_{S}}(\kappa_{w^{\natural}_{S}} X^{\natural})  - \left( \kappa^{X}_{w^{\natural}_{S}} X^{\natural} +  \sol^{X}_{w^{\natural}_{S}}(\kappa^{X}_{w^{\natural}_{S}} X^{\natural})  \right) \right|_{p,S} = \rO\left(1\right) \left|X^{\natural}\right|^2 \end{equation}
\end{lem}
\begin{proof}
Note that the restrictions of $\kappa_{w^{\natural}_{S}}$ and $\kappa^{X}_{w^{\natural}_{S}}$ to $\exp^{-1}_{w^{\natural}_{S}}$ are respectively the inverses of $\id + \sol_{w^{\natural}_{S}}$ and $\id + \sol^{X}_{w^{\natural}_{S}}$.  Since the previous Lemma implies that
\begin{equation}  \left| \kappa_{w^{\natural}_{S}}  - \kappa^{X}_{w^{\natural}_{S}}   \right| = \rO \left(1\right) \left|X^{\natural}\right|, \end{equation}
the desired result follows easily from the fact that the $C^1$ norms of $\sol_{S}$ and $\sol^{X}_{S}$ are both bounded by $\rO\left(e^{-2(1-\delta)S}\right)$.
\end{proof}

We have established the Lemmas necessary to prove the main result of this subsection 

\begin{proof}[Proof of Lemma \ref{lem:derivative_gluing_close_to_parallel_transport_pre-gluing}]
Note that we've already reduced the proof of the Lemma to the proof of Equation \eqref{eq:estimate_inverse_image_nearby_glued_map_original_inverse}.  The previous Lemma shows that this is equivalent to
\begin{align} \left| \Pi_{w^{\flat}_{S}}^{w^{\natural}_{S}} \exp^{-1}_{ w^{\flat}_{S}}   \left( w_{S,X}^{\flat} \right) - \kappa'_{w^{\natural}_{S}} X^{\natural} +  \sol'_{w^{\natural}_{S}}(\kappa'_{w^{\natural}_{S}} X^{\natural}) \right| & = \rO\left(e^{-2(1-\delta)S} \right)|X^{\natural}|
\end{align}

Applying $\kappa'_{w^{\natural}_{S}}$, this is equivalent to proving
\begin{align} \label{eq:apply_kappa_prime_gluing_compare_kappa_prime_pre-gluing}
\left| \kappa'_{w^{\natural}_{S}}  \exp^{-1}_{w^{\natural}_{S}}\left( w^{\flat}_{S,X} \right) -\kappa'_{w^{\natural}_{S}} X^{\natural}  \right| & = \rO\left(e^{-2(1-\delta)S} \right) \left|X^{\natural}\right|,
\end{align}
which we shall do presently.

Since $ w^{\flat}_{S,X}$ is obtained by exponentiating $\sol_{w^{\natural}_{S,X}}(0) $ at $w^{\natural}_{S,X}$, Lemma \ref{lem:error_add_vector_fields_L_p_S} implies the existence of an $L^p$-bound
\begin{equation}  \left| X^{\natural} + \Pi_{w^{\natural}_{S,X}}^{w^{\natural}_{S}} \sol_{w^{\natural}_{S,X}}(0)   - \exp^{-1}_{w^{\natural}_{S}} \left( w^{\flat}_{S,X} \right)   \right| =\rO\left(e^{-2(1-\delta)S} \right)|X^{\natural}|  .\end{equation}
In particular, Equation \eqref{eq:two_right_inverses_close}  implies
\begin{equation}\label{eq:apply_kappa_prime_failure_commutativity_exponential}
\left| \kappa'_{w^{\natural}_{S}} \left( X^{\natural} + \Pi_{w^{\natural}_{S,X}}^{w^{\natural}_{S}} \sol_{w^{\natural}_{S,X}}(0)   - \exp^{-1}_{w^{\natural}_{S}} \left( w^{\flat}_{S,X} \right) \right)  \right| = \rO\left(e^{-2(1-\delta)S} \right)|X^{\natural}| .  \end{equation}
From the definition of $  \kappa'_{w^{\natural}_{S}}$, we also have a uniform bound
\begin{equation} \label{eq:difference_kappa_prime_parallel_transport_kappa_bounded}
\left\| \kappa'_{w^{\natural}_{S}}\Pi_{w_{S,X}^{\natural}}^{w_{S}^{\natural}}   - \Pi_{w_{S,X}^{\natural}}^{w_{S}^{\natural}}  \kappa_{w^{\natural}_{S,X}} \right\| = \rO(1)|X^{\natural}|  .\end{equation}
Using the fact that $ \kappa_{w^{\natural}_{S,X}}\sol_{w^{\natural}_{S,X}}(0)  $, vanishes, and that $| \sol_{w^{\natural}_{S,X}}(0)| = \rO(e^{-2(1-\delta)S})$, we see that Equations \eqref{eq:apply_kappa_prime_failure_commutativity_exponential} and \eqref{eq:difference_kappa_prime_parallel_transport_kappa_bounded} imply \eqref{eq:apply_kappa_prime_gluing_compare_kappa_prime_pre-gluing}, which proves the Lemma.
\end{proof}

\subsection{Injectivity of the derivative of the gluing map} \label{sec:injectivity_derivative}
Lemma \ref{lem:derivative_gluing_close_to_parallel_transport_pre-gluing} implies that the derivative of $\Gext_{S}$ is injective.  Our goal in this subsection is to show that the derivative of $\G_{S}$ is also injective. As explained at the beginning of this section, it suffices to prove that the image of $d \G$ cannot be tangent to the fibres of \eqref{eq:gluing_map}.

\begin{lem} \label{lem:norm_rotation_vector_field}
Given $\mu <1$, the $| \_|_{p,S}$-norm of the pull back of  $e^{-4S} dw_{S} \left( \partial_{\theta} \right)$ by $\iota^{\cP,S}_{3S}$ is bounded by a constant multiple of $e^{-S}$.
\end{lem}
\begin{proof}
The norm of $\partial_{\theta} $ on a strip-like end is bounded above by a constant multiple of $e^{s}$, which is why we're rescaling by $e^{-4S}$ to obtain a vector field whose point wise norm under pullback by $\xi_{S,\neck}$ is bounded by $\rO(e^{-4S-s})$.  Since $\sol_{S}$ is bounded by $\rO\left( e^{-2(1-\delta)S} \right)$ in the $W^{1,p}$ norm $| \_ |_{1,p,S}$, the norm of $dw_{S}$ is uniformly bounded in the $L^p$ norm $| \_ |_{p,S}$. We conclude that the $L^p$ norm of the restriction of $e^{-4S} dw_{S} \left( \partial_{\theta} \right)$ to  $\iota^{\cP,S}_{3S}$ decays like $\rO(e^{-S})$.

It remains to bound the contribution of the parallel transport of the value of this vector field at $\xi_{S,\neck}(0,0)$.  Observe that, for $S_0$ large enough, the energy of the restriction of $w_{S}$ to $\xi_{S,\neck} \left( \I{-2S+S_0}{2S+S_0} \right)$ is arbitrarily small, so the exponential decay estimate on long strips of small energy implies that 
\begin{equation} \left| dw_{S} |\xi_{S,\neck}(0,0)\right| = \rO\left(e^{-2\mu S} \right). \end{equation}  As the norm of $e^{-4S} \partial_{\theta}$ at this point is bounded by $\rO\left( e^{-2S} \right)$, the result easily follows.
\end{proof}

In order to proceed, we need an additional estimate on $\predGext$:
\begin{lem} \label{eq:pre-gluing_tangent_vectors_almost_commutes_dbar}
The norm of the composition of $D_{\cP_{0,1}}$ with $ \Pi_{w^{\natural}_{S}}^{w^{\flat}_{S}} \circ \predGext_S$ is bounded by a constant multiple of $e^{-2 (1-\delta) S}$.  
\end{lem}
\begin{proof}
This is a consequence of Lemma \ref{lem:tangent_pre-gluing_almost_holomorphic} and the quadratic inequality \eqref{eq:quadratic_inequality_first_statement}, applied to $Z^{\flat} = Q \circ \sol_{S}(u,\theta,v)$, whose norm by a constant multiple of $e^{-2(1-\delta)S}$ as in Equation \eqref{eq:exponential_decay_solution}, which is smaller than $e^{-2 \delta S}$ for $S$ large enough.
\end{proof}

We are ready to give the proof of the two Lemmas stated at the beginning of the section:
\begin{proof}[Proof of Lemma \ref{lem:pre-gluing_parabolic_close_to_rotation}]
Note that $\partial_{p}$ extends to a vector field on $S^{2} = \bC \cup \{ \infty \}$ with a unique zero at $-1$, so that the vector field $\partial_{p}$ vanishes to first order at $-1$.  In particular, its pointwise norm along the strip-like end $\xi_{-1}$ is bounded by $\rO(e^{-s})$.

Moreover, given a sequence of points in $D^2$ converging to $-1$, the corresponding sequence of $1$-periodic holomorphic vector fields which vanish at these points converges in the complement of any compact subset of $D^2 - \{ -1 \}$, after rescaling and reparametrisation, to the vector field $\partial_{p}$.  The relevant estimate is
\begin{equation} \label{eq:exponential_decay_difference_rotation_parabolic} \left| e^{-4S} (\iota^{\cS,S}_{S})^{*} \partial_{\theta}  - \partial_{p}| D^{2} - \xi_{-1} \I{-\infty}{-S} \right|_{\infty} = \rO(e^{-S}) ,\end{equation}
where $\iota^{\cS,S}_{S}$ is the inclusion map of  $D^{2} - \xi_{-1} \I{-\infty}{-S}$ into $\Sigma_S$.   The proof is an elementary computation left to the reader.  

From the proof of Lemma \eqref{lem:norm_rotation_vector_field}, we know that the contribution to  the $| \_|_{p,S}$-norm of the parallel transport of the value of $  dw^{\flat}_{S}\left(  \partial_{\theta}\right) $  at the image of the origin under $\xi_{S,\neck}$ is bounded by $\rO(e^{-2(1+\mu-\delta)S} ) $.  Moreover, in $\xi_{S,\neck} \I{-S}{0}$ both vector fields are respectively bounded pointwise by $ e^{-(1+\mu)(|s|-2S)}$, so the contribution of $\xi_{S,\neck} \I{-S}{0}$ to \eqref{eq:decay_difference_rotation_parabolic}  is bounded by $ \rO( e^{-(1+\mu - \delta)S} ) $.

It remains to bound the difference between these two vector fields on the image of $ \iota^{\cS,S}_{S}$.  Since $w^{\flat}_{S} \circ \iota^{\cS,S}_{S}$ and $v| D^{2} - \xi_{-1} \I{-\infty}{-S}$  differ by a vector field whose $| \_|_{1,p,S}$-norm is bounded by a constant multiple  of $e^{-2(1-\delta) S}$, the pointwise bound \eqref{eq:exponential_decay_difference_rotation_parabolic} implies that difference between $ \Pi_{w^{\natural}_{S}}^{w^{\flat}_{S}}  \predGext_{S} \left( dv \left( \partial_{p} \right) \right) $ and $e^{-4S} dw^{\flat}_{S} \left(  \partial_{\theta} \right)$ is bounded, in this subset of $\Sigma_{S}$, by $\rO(e^{-S})$ in the $|\_|_{p,S}$-norm. 
\end{proof}

\begin{proof}[Proof of Lemma \ref{lem:image_differential_gluing_transverse_circle}]
Assuming that there is an element $X^{\sharp}$ of $ TK \oplus \langle dv \left(  \partial_{s} \right) \rangle $ whose image under  \eqref{eq:map_gluing_kernels} is $ e^{-4S} dw^{\flat}_{S} \left( \partial_{\theta} \right)$, Lemmas \ref{lem:pre-gluing_parabolic_close_to_rotation} and \ref{eq:pre-gluing_tangent_vectors_almost_commutes_dbar}  imply that \begin{equation}  \left|   \Pi^{w^{\flat}_{S}}_{ w^{\natural}_{S}}  \predGext_{S} \left( X^{\sharp} -  dv \left( \partial_{p} \right) \right)  \right| =\rO( e^{-(1-\delta)S}  ) \end{equation}
For $S$ large enough, this contradicts Lemma \ref{lem:pre-gluing_uniform_bilipschitz} since the norm of  $X^{\sharp} -  dv \left( \partial_{p} \right)$ is bounded above by some multiple of $dv \left( \partial_{p} \right)$ whenever  $X^{\sharp}$ lies in $ TK \oplus \langle dv \left(  \partial_{s} \right)  \rangle $.   Closeness to $d\Gext_{S}$ is implied by Lemmas \ref{lem:derivative_gluing_close_to_parallel_transport_pre-gluing} and \ref{eq:pre-gluing_tangent_vectors_almost_commutes_dbar} \end{proof}

We also need a bound on the distance between the parallel transports of the images of the tangent spaces at various points under $d\G_{S}$.   First, we  bound the distance between the values of $ \predGext_{S}$ at nearby points.  The proof of the following Lemma runs along the same arguments are Lemma \ref{lem:uniform_bound_difference_nearby_right_inverse}, especially Equation \eqref{eq:error_parallel_transport_pre-gluing} and is omitted:
\begin{lem}
There is a uniform bound
\begin{equation} \left\| \Pi_{w^{\natural}_{S}}^{w^{\flat}_{S}}\circ  \predGext_{S}|u^{\sharp} - \Pi_{w^{\natural}_{S,X}}^{w^{\flat}_{S}}  \circ \predGext_{S}|u^{\sharp}_X \circ \Pi_{u^{\sharp}}^{u^{\sharp}_X}  \right\| = \rO(1)|X|^{\sharp} \end{equation} \noproof
\end{lem}

From this Lemma and Lemma \ref{lem:derivative_gluing_close_to_parallel_transport_pre-gluing}, we conclude
\begin{equation} \label{eq:bound_difference_parallel_transport_gluing} \left\| d\Gext_{S}|u^{\sharp} -  \Pi_{w^{\flat}_{S,X}}^{w^{\flat}_{S}}   \circ d\Gext_{S}|u^{\sharp}_X \circ \Pi_{u^{\sharp}}^{u^{\sharp}_X}  \right\| = \rO\left( e^{-2(1- \delta) S}\right) + \rO(1)|X|^{\sharp} \end{equation}

\begin{proof}[Proof of Proposition  \ref{prop:local_injectivity_gluing}]
The fact that $\G_{S}$ is an immersion is immediate from the proof of Lemma \ref{lem:image_differential_gluing_transverse_circle} and Lemma \ref{lem:derivative_gluing_close_to_parallel_transport_pre-gluing}, since they imply that $d \Gext_{S}$ is injective and does not intersect the kernel of the projection map to $\cP(L;0)$. The second statement is simply a quantitative version implied by the following argument:  Given $(u,\theta,v) \in K$, we can fix a neighbourhood about $(u,\theta,v)$ whose image under $\Gext_{S}$ is contained in ball of radius $\epsilon/2$ of $w^{\flat}_{S} = \Gext_{S}(u,\theta,v) \in \cP_{0,1}(L;0)$ with respect to our chosen norm on $T \cF^{1,p}_{\cP_{0,1}}(L;0)$, i.e. can be obtained by exponentiating vector fields $X^{\flat}$ of $| \_|_{1,p,S}$-norm bounded by $\epsilon/2$.

Let us write
\begin{equation} w^{\flat}_{S} = (R(w^{\flat}_{S}), \theta(w^{\flat}_{S}), w_{S}) ,\end{equation}
with $R(w^{\flat}_{S})-R_u$ and $\theta(w^{\flat}_{S})-\theta$ decaying exponentially with $S$.  By elliptic regularity, we may choose $S$ sufficiently large so that the restriction of $w_{S}$  to the images of $\iota_{S}^{S,\cS}$ and $\iota_{S}^{S,\cP}$ are arbitrarily close in $C^1$-norm to $u$ and $v$ in the respective domains, while the restriction to $\xi_{S,\neck} \I{-S}{S}$ is bounded point-wise in $C^1$-norm by $\rO(e^{-\mu S})$.    We conclude that for $S$ large enough, if $|\lambda| \leq \epsilon/2$, every point of the form
\begin{equation} \label{eq:orbit_w_flat_rotation} \left( R(w^{\flat}_{S}), \theta+\lambda, w_{S} \circ r_{\lambda} \right)   \end{equation}
lies in the ball of radius $\epsilon$ in the $| \_|_{1,p,S}$-norm  about $w^{\flat}_{S}$.  In other words, a ball of radius $\epsilon$ about $w^{\flat}_{S}$ contains a segment of the fibre of the projection map to $\cP(L;0)$ which passes through $w^{\flat}_{S}$,  whose size is uniform.

Note that the norm of $e^{-4S}  w^{\flat}_{S} \left( \partial_{\theta} \right)$ is uniformly bounded above and below by constants independent of $S$.  In particular, we can choose $\epsilon$ small enough so that 
\begin{equation} e^{-4S}   \left( \Pi_{w^{\flat}_{S}}^{w^{\flat}_{S,X}}dw^{\flat}_{S} \left( \partial_{\theta} \right)   - dw^{\flat}_{S,X} \left( \partial_{\theta}  \right) \right) \end{equation}
is arbitrarily small in the $L^p$-norm whenever $|X^{\flat}|_{1,p,S} \leq \epsilon$.

Since $\exp^{-1}_{w^{\flat}_{S}} \Im  \Gext_{S}$ is $C^1$-close in the $| \_|_{p,S}$-norm to a linear subspace by Equation \eqref{eq:bound_difference_parallel_transport_gluing},   we conclude that no element of $\exp_{w^{\flat}_{S}}^{-1} \Gext_{S}(u',\theta',v')$ can lie on a path of the form \eqref{eq:orbit_w_flat_rotation}, whenever $ (u',\theta',v') $ is in some fixed neighbourhood of $(u,\theta,v)$.  This proves injectivity is a small neighbourhood of uniform size about every point. 
\end{proof}

\section{Gluing near higher codimension strata} \label{sec:more_gluing}

Having proved the existence of a gluing map from a compact subset of the codimension $1$ stratum of $\Pbar(L;0)$ to the interior of $\cP(L;0)$, we shall proceed to describe the analogous results for the other boundary strata, which requires gluing sphere bubbles.  It turns out that we're working in a special situation where the moduli space of sphere bubbles passing through a fixed point is regular.  This implies that gluing can be done without ``moving the attaching point," thereby reducing the analytic results to standard ones.   We shall therefore concentrate on describing the geometric parts of the construction, including the choice of right inverse, leaving all estimates to the reader.

\begin{rem}
In the discussion that follows, we shall often use the expression ``a neighbourhood of $X$ in $Y$," even though $X$ is not a subset of $Y$.  In all such situations, the Gromov-Floer compactification $Y$ will contain $X$ as a stratum, so such a neighbourhood is defined to be the subset of $Y$ obtained by removing all strata of virtual codimension greater than $0$ from a neighbourhood of $X$ in the Gromov-Floer compactification of $Y$.
\end{rem}

For the codimension $3$ stratum, our main result, discussed in Section \ref{sec:codim_3_to_0_gluing} is:
\begin{lem} \label{lem:codim_3_to_parametrized_space_diffeo_gluing}
There exists a gluing map
\begin{equation}  \label{eq:gluing_codim_3} \Gcod{3}{0} \co \codt{3}{L} \times [S, +\infty)^{2} \to \cP(L;0) \end{equation}
which is a diffeomorphism onto a neighbourhood of $\codt{3}{L}$. \noproof
\end{lem}
In Section \ref{sec:corner-structure-at-codim-3}, we explicitly define a gluing map
\begin{equation} \Gcod{3}{2} \co \codt{3}{L} \times [S, +\infty) \to  \codt{2}{L}. \end{equation}
The main result for the codimension $2$ stratum, discussed in Section \ref{sec:codim-2-strat} is:
\begin{lem} \label{lem:gluing_commute_codim_3_codim_2_top_stratum}
There exists a gluing map
\begin{equation} \label{eq:codimension_2_gluing_map} \Gcod{2}{0} \co \codt{2}{L} \times [S, +\infty) \to \cP(L;0) \end{equation}
which is a diffeomorphism onto a neighbourhood of $  \codt{2}{L}$ in $ \cP(L;0)  $ and such that the diagram
\begin{equation} \label{eq:commutative_gluing_codim_2} \xymatrix{ \codt{3}{L} \times [S,+\infty)^2 \ar[dr] \ar[r] & \codt{2}{L} \times [S,+\infty) \ar[d] \\ & \cP(L;0) } \end{equation}
 commutes for $S$ large enough. \noproof
\end{lem}
\subsection{Gluing the codimension $3$ stratum into the top stratum} \label{sec:codim_3_to_0_gluing}

Recall that $\cod{3}{L}$ consists of complex lines passing through the boundary of an exceptional disc.  For each $u \in \cP(L;-\beta)$ and $\theta \in S^1$, choose a complex hyperplane $ \N_{(u,\theta)}$ in $\bC \bP^{n-1}$ which does not intersect the projection of $u(\theta)$ to $\bC \bP^{n-1}$ and take the product with $\bC^{n}$ to obtain a hypersurface
\begin{equation}  \N^{\bC^{n}}_{(u,\theta)}  \equiv  \bC^{n} \times \N_{(u,\theta)}  \cong \bC^{n} \times \bC \bP^{n-2}.\end{equation}
We shall assume that these hyperplanes satisfy the following property
\begin{equation} \label{eq:condition_hyperplanes} \parbox{35em}{ The hyperplanes $ \N_{(u,\theta)}$ vary smoothly and are independent of the angular parameter $\theta$ if it is sufficiently close to $0$.} \end{equation}

Such a family of hyperplanes may be easily obtained by first choosing the complex hyperplane in $\bC \bP^{n-1}$ to be the unique one orthogonal to the projection of $u(\theta)$ to this factor, then modifying this choice near $\theta =0$ using cutoff functions.

 Note that every $J_{\alg}$ holomorphic sphere passing through $u(\theta)$ intersects  $\N^{\bC^{n}}_{(u,\theta)}$ at a unique point and is moreover determined by the projection of this point to $\N_{(u,\theta)}$.  The data of this intersection point gives a diffeomorphism from the moduli space of holomorphic spheres passing through $u(\theta)$ to $\N_{(u,\theta)}$, which yields a diffeomorphism
\begin{equation} \cod{3}{L} \to \N(\cP_{0,1}(L;-\beta)). \end{equation}
where $ \N(\cP_{0,1}(L;-\beta))$ is the bundle over $ \cP_{0,1}(L;-\beta)$ whose fibre at $(u,\theta)$ is $\N_{(u,\theta)}$.  Note that this is a sub-bundle of the product $ \cP_{0,1}(L;-\beta)) \times \bC \bP^{n-1}$.  If $\tilde{\N}_{(u,\theta)}$ is the unit normal bundle of $\N_{(u,\theta)}$, and $ \tilde{\N}(\cP_{0,1}(L;-\beta))$ is the corresponding bundle over $ \cP_{0,1}(L;-\beta)$, we can extend this map to a diffeomorphism 
\begin{align} \label{eq:diffeomorphism_codim_3_stratum_circle_bundle}  \codt{3}{L} & \to \tilde{\N}(\cP_{0,1}(L;-\beta)) \\
(u, \theta, [\tilde{v}]) & \mapsto (u, \theta, \tilde{n}(u, \theta, [\tilde{v}])
 \end{align}
whose construction we explain presently.  It is convenient to first give an alternative description of  $\cL \cM_{1}(M; \alpha) $: Consider the subgroup $ \Aut_{\bR^{+}} (\bC \bP^1,\infty)$ of  $\Aut(\bC \bP^1,\infty)$ consisting of automorphisms of  $\bC \bP^1$ which fix $\infty$ and act on the tangent space of $\bC \bP^1$ at $\infty$ by multiplication with a positive real number.  In terms of Mobius transformations, this is simply the group of automorphism
\begin{equation}
  z \mapsto a z + b \quad a \in \bR^{+}, b \in \bC. 
\end{equation}
The quotient of the moduli space of parametrized holomorphic spheres by $ \Aut_{\bR^{+}} (\bC \bP^1,\infty)$ is naturally diffeomorphic to $\cL \cM_{1}(M; \alpha) $.  In particular, given an element $[v]$ of $\cM_{1}(M; \alpha) $, the choice of a parametrization determines a lift $[\tilde{v}]$ to $\cL \cM_{1}(M; \alpha)$.

With this in mind, we write down the inverse diffeomorphism to \eqref{eq:diffeomorphism_codim_3_stratum_circle_bundle} as follows:  The normal bundle of $\N_{(u,\theta)}$ is $\cO(1)$ so there is a unique section of the normal bundle, passing through every element $\tilde{\n}$ of $\tilde{\N}_{(u,\theta)}$ lying over $\n \in \N_{(u,\theta)}$, which vanishes exactly on the orthogonal complement of $\n$.  Using the exponential map to identify a neighbourhood of the zero section in the normal bundle of $\N_{(u,\theta)}$ with a neighbourhood in $\bC \bP^{n-1}$, and rescaling the section determined by $\tilde{n}$ by a uniform constant so that it is contained in such a neighbourhood,  we obtain a hyperplane in $\bC \bP^{n-1}$ which we denote by $\N_{(u,\theta,\tilde{n})}$.   Given an element $( u, \theta, \tilde{n})$ in the right hand side of \eqref{eq:diffeomorphism_codim_3_stratum_circle_bundle}, its image under the inverse diffeomorphism is the triple $(u,\theta, [\tilde{v}]) \in \codt{3}{L}$ where $[\tilde{v}]$ is the lift determined by the parametrization 
\begin{equation} \label{eq:condition_parametrization_sphere_passing_hyperplane} v(\infty) = u(\theta) \textrm{, } v(0) \in \bC^{n} \times \{ \n \} \textrm{ and } v(1) \in  \N^{\bC^{n}}_{(u,\theta,\tilde{n})}. \end{equation}
From the data of the diffeomorphism of \eqref{eq:diffeomorphism_codim_3_stratum_circle_bundle}, we associate to each element $(u, \theta, [\tilde{v}]) \in \codt{3}{L}$ hyperplanes
\begin{equation} \N_{ (u, \theta, [\tilde{v}]) } \equiv \N_{(u, \theta, \tilde{n}(u, \theta, [\tilde{v}])} \subset \bC \bP^{n-1} \textrm{ and }  \N^{\bC^{n}}_{ (u, \theta, [\tilde{v}]) } \equiv  \bC^{n} \times \N_{ (u, \theta, [\tilde{v}]) } . \end{equation}

\begin{rem}
When we were proving the existence of a gluing map with source $\cod{1}{L}$, we considered the family of equations obtained by rotation by the parameter $\theta$, so that our matching condition was  on $u(1)$ (in other words, $u$ here corresponds to $u \circ r_{-\theta}$ in the previous section).  The reason for introducing a family of equations parametrized by the circle was that the moduli space of discs passing through a fixed point in $L$ may not be regular.  As regularity holds for the moduli space of spheres passing through a point in $L$, we shall use the simpler setup here. 
\end{rem}

Setting aside issues of notation, the conceptual part of the proof is simpler than the gluing result introduced in Section \ref{sec:codim1_gluing}  because we shall be able to construct an explicit inverse to the gluing map after introducing an additional parameter corresponding to the tangent space of $\cP_{0,1}(L;-\beta)$.  

The first step to proving Lemma \ref{lem:codim_3_to_parametrized_space_diffeo_gluing}, is setting up a Banach complex controlling the tangent space of $\codt{3}{L}$ at each point.  Because of the way we're treating the variable $\theta$, this will not be a complex of Banach bundles (in the uniform topology) over $\codt{3}{L}$, but our arguments will be insensitive to this problem since we will prove smoothness by constructing an inverse which will be proved to be a diffeomorphism.

Equip $(D^2, e^{i\theta})$ with the strip-like end $\xi_{\theta}$ at $e^{i \theta}$ obtained by composing the strip-like end $\xi_{1}$ of \eqref{eq:positive_strip_at_1} with rotation by angle $\theta$.   The linearisation of  the operator $\dbar_{\cP}$ on the Sobolev space of $W^{1,p}$ maps with an exponential weight $\delta$ along the strip $\xi_{\theta}$ is a Fredholm map
\begin{equation} \label{eq:linearisation_dbar_parametrized_dbar_fix_theta} D_{\cP} \co T [0,+\infty) \oplus W^{1,p,\delta}_{(D^2,e^{i\theta})}\left(u^{*} TM, u^{*} TL\right) \tilde{\to} L^{p,\delta}_{(D^2,e^{i\theta})}\left( u^{*} TM \otimes \Omega^{0,1} D^2\right) \end{equation}
which is an isomorphism since $\cP(L;-\beta)$ is regular and zero-dimensional.  Moreover, the factor $W^{1,p,\delta}_{(D^2,e^{i\theta})}(u^{*} TM, u^{*} TL) $ is also equipped with an evaluation map to $TL$ which is the value of the vector field at $e^{i\theta}$.

Next, we equip $\bC \bP^{1}$ with a cylindrical end near $\infty$ coming from the unique isomorphism
\begin{equation} \label{eq:cylindrical_end_infinity_sphere} S^1 \times (-\infty, +\infty) \to \bC \bP^{1} - \{0,\infty\}  \end{equation} 
taking $(1,0)$ to $1$, which we use to define a Banach manifold $\cF_{\cM_{1}}^{1,p,\delta}(M)$ modeled after
\begin{equation} W^{1,p,\delta}_{(\bC \bP^{1},\infty)}(v^{*} TM)  \end{equation}
where $\bC \bP^{1}$ equipped with a cylindrical metric near $\infty$.  The direct sum of the linearisation of the $\dbar$ operator and the evaluation map to $TM$ at $\infty$ is a Fredholm map
\begin{equation} \label{eq:linearisation_dbar_sphere_bubble_and_evaluation_to_L} W^{1,p,\delta}_{(\bC \bP^{1},\infty)}(v^{*} TM) \tilde{\to} L^{p,\delta}_{(\bC \bP^{1},\infty)}(v^{*} TM) \oplus TM. \end{equation}
The above discussion implies that for each $(u,\theta, [\tilde{v}]) \in \codt{3}{L}$, the restriction of \eqref{eq:linearisation_dbar_sphere_bubble_and_evaluation_to_L} to those $W^{1,p}$ section of $v^{*}(TM)$ whose projection to $\bC \bP^{n-1}$ vanishes at $0$ and takes value in $T\N_{(u, \theta, [\tilde{v}]) }$ at $1$ is an isomorphism: 
\begin{equation} \label{eq:linearisation_dbar_sphere_bubble_and_evaluation_to_L_restrict_tangents} W^{1,p,\delta}_{(\bC \bP^{1},\infty)}\left((\bC \bP^{1},0,1),\left(v^{*} TM, T \bC^{n} \times \{ \n \}, T  \N^{\bC^{n}}_{(u, \theta, [\tilde{v}]) }\right) \right) \tilde{\to} L^{p,\delta}(v^{*} TM) \oplus TM. \end{equation}

Finally, the moduli space of holomorphic discs in the trivial homotopy class mapping $-1$ to $u(1)$ and $0$ to $v(0)$ is rigid and consists of a unique ``ghost bubble" $\gh(u,v)$. Consider the cylindrical end at $0$ 
\begin{align} \label{eq:varying_cylindrical_end_ghost_bubble} \xi_{0,\theta} \co S^1 \times (-\infty,0] & \to D^2 - \{0\} \\
(\phi, S) & \mapsto e^{S+ i (\phi - \theta)}, 
\end{align}
which varies according to the parameter $\theta$, and fix the strip-like end $\xi_{-1}$ at $-1$.  The fact that the moduli space of constant discs is regular implies that the direct sum of the linearisation of the $\dbar$ operator with the evaluation map at the boundary marked point defines an isomorphism
\begin{equation} \label{eq:linearisation_dbar_ghost_bubble_and_evaluation_to_L_and_M} W^{1,p,\delta}_{(D^2,0,-1)}\left(\gh(u,v)^{*} TM, \gh(u,v)^{*} TL\right) \tilde{\to} L^{p,\delta} \left( \gh(u,v)^{*} TM \right)  \oplus TL. \end{equation}

The tangent space of $\codt{3}{L}$ is controlled by the operators \eqref{eq:linearisation_dbar_parametrized_dbar_fix_theta}, \eqref{eq:linearisation_dbar_sphere_bubble_and_evaluation_to_L}, and \eqref{eq:linearisation_dbar_ghost_bubble_and_evaluation_to_L_and_M}.  Taking a fibre product over the evaluation maps to $TL$ and $TM$ at the marked points, we obtain an isomorphism
\begin{equation} \label{eq:banach_bundle_codim_three_stratum}
\begin{array}{ccc} T [0,+\infty) \oplus  W^{1,p,\delta}_{(D^2,e^{i\theta}))}\left(u^* TM, u^{*} TL \right)  & & \hspace{-20pt} L^{p,\delta}_{(D^2,e^{i\theta}))}\left(u^* TM \otimes \Omega^{0,1}D^2\right) \\
 \oplus_{TL}  & &  \hspace{-20pt} \oplus  \\
 W^{1,p,\delta}_{(D^2,0,-1)}\left(\gh(u,v)^{*} TM, \gh(u,v)^{*} TL\right) & \hspace{-40pt} \tilde{\longrightarrow}    & \hspace{-20pt} L^{p,\delta}_{(D^2,0,1)}\left(\gh(u,v)^* TM \otimes \Omega^{0,1} D^2\right) \\
  \oplus_{TM}  & & \hspace{-20pt} \oplus  \\
 W^{1,p,\delta}_{(\bC \bP^{1},\infty)}\left((\bC \bP^{1},0,1),(v^{*} TM, T \bC^{n} \times \{ \n \},v^{*} T \N^{\bC^{n}}_{(u, \theta, [\tilde{v}]) })\right) & &  \hspace{-20pt} L^{p,\delta}_{(\bC \bP^{1},\infty)}\left(v^* TM \otimes \Omega^{0,1}\bC \bP^{1}\right) .\end{array}
\end{equation}

As in Equation \eqref{eq:definition_Sigma_S}, a choice of parameters $(S_1,S_2)$ determines a surface $\Sigma_{S_1,\theta, S_2}$ defined by gluing $\bC \bP^{1}$ to $\Sigma_{S_1, \theta}$ (itself obtained by gluing $(D^2, e^{i\theta})$ to $(D^2,-1)$) along the chosen cylindrical ends, which moreover carries a distinguished bi-holomorphism to $D^2$ coming from the marked disc $(D^2,\theta)$.  If we choose $S_2$ large enough, the points $0$ and $1$ on $\bC \bP^{1}$ lie away from the gluing region, and survive as distinguished points on $\Sigma_{S_1,\theta, S_2}$ whose images in $D^2$ are denoted by
\begin{equation} z_{0}(S_1, \theta) \textrm{ and }  z_1(S_1,S_2,\theta).\end{equation}
\begin{lem}
Under the unique automorphism $ \tau_{S_1}(\theta)$ of $D^2$ fixing $e^{i\theta}$ and mapping $z_{0}(S_1, \theta)$ to the origin,  the image of $ z_1(S_1,S_2,\theta)$ lies on the positive real axis. \noproof
\end{lem}

Pre-gluing the maps $u$ and $v$ together with the ghost bubble gives a map
\begin{equation} \preGext_{S_1,S_2}( u ,\theta, [\tilde{v}])  \co D^2 \to M \end{equation}
which is an approximate solution  to Equation \eqref{dbar-operator} in the trivial homotopy class.  To prove that there is a solution near $\preGext_{S_1,S_2}( u ,\theta, [\tilde{v}]) $, we  shall consider Sobolev spaces on $D^2$ equipped with the metric $g_{S_1, \theta , S_2}$ coming from $\Sigma_{S_1,\theta, S_2}$.

Using breaking of $TM$-valued $1$-forms and pre-gluing of vector fields, the inverse to \eqref{eq:banach_bundle_codim_three_stratum} defines an approximate right inverse to $D_{\dbar_{\cP}}$ at  $\preGext_{S_1,S_2}( u ,\theta, [\tilde{v}]) $.  As in Corollary \ref{cor:uniformly_bounded_inverse}, we conclude the existence of a right inverse whose image consists of vector fields which are obtained by parallel transporting (and cutting off) vector fields in the left hand side to Equation \eqref{eq:banach_bundle_codim_three_stratum}. For $S_1$ and $S_2$ large enough our two marked points $z_{0}(S_1, \theta) $ and $z_1(S_1,S_2,\theta)$ lie in the region where the cutoff functions vanish, and the parallel transport is the identity.  In this case, the image of the right inverse consists of vector fields along the disc whose component in the $\bC \bP^{n-1}$-direction vanishes at $z_{0}(S_1, \theta)$ and which take value in the tangent space to $\N^{\bC^{n}}_{(u, \theta, [\tilde{v}]) }$ at  $z_1(S_1,S_2,\theta)$.  Abusing notation a little bit by dropping the pullbacks from the notation, we conclude
\begin{lem}
The operator
\begin{equation}  \label{eq:restriction_dbar_vanish_1_point_hyperplane_1_point}  \begin{array}{c} T [0,+\infty) \oplus W^{1,p}\left( (D^2, S^1,  z_{0}(S_1, \theta) ,  z_1(S_1,S_2,\theta)), ( TM,   TL,  T \bC^{n} \times \{ \n \}, T \N^{\bC^{n}}_{(u, \theta, [\tilde{v}]) } ) \right) \\ 
\downarrow \\ 
L^{p}( TM \otimes \Omega^{0,1} D^2) \end{array} \end{equation}
is an isomorphism with uniformly bounded inverse. \noproof
\end{lem}

After checking the quadratic inequality with respect to the metrics on $\Sigma_{S_1,\theta,S_2}$, and proving an estimate on the norm of $\dbar_{\cP} \preGext_{S_1,S_2}( u ,\theta, [\tilde{v}]) $, we may apply the implicit function theorem at the origin of a chart centered on $ \preGext_{S_1,S_2}( u ,\theta, [\tilde{v}]) $ to conclude:

\begin{lem} There exists a constant $\epsilon$ such that if $S_1$ and $S_2$ are sufficiently large, then there exists a unique element $\sol_{S_1,S_2}$ of the source of \eqref{eq:restriction_dbar_vanish_1_point_hyperplane_1_point} at $\preGext_{S_1,S_2}( u ,\theta, [\tilde{v}])$ whose norm is bounded by $\epsilon$ and whose image under the exponential map is a solution to \eqref{dbar-operator}. \noproof
 \end{lem}

We define the codimension $3$ gluing map $\Gcod{3}{0}$ to be the composition of pre-gluing with the exponentiation of the vector field $\sol_{S_1,S_2}$.

It remains to prove that \eqref{eq:gluing_codim_3} is a diffeomorphism onto a neighbourhood of $\cod{3}{L}$ in $ \cP(L;0)$.  Our arguments rely on extending the gluing map to a map:
\begin{align} \label{eq:extended_gluing_codim_3} \Gext^{\leq \eta} \co \codt{3}{L} \times [0,+\infty)^{2} \times [-2\eta,2\eta] & \to \cP(L;0) \times S^1 \\
(u,\theta,[\tilde{v}],S_1,S_2,\lambda) & \mapsto \left( \G_{S_1,S_2}^{\lambda}(u,\theta,[\tilde{v}]), \theta + \lambda\right).
\end{align}
Here $\eta$ is a sufficiently small real number, and $ \G_{S_1,S_2}^{\lambda}(u,\theta,[\tilde{v}])$ is the gluing map constructed with respect to the right inverse induced by the hyperplanes $\N_{(u,\theta+\lambda)}$ and $\N_{(u,\theta+\lambda,[\tilde{v}] )}$ rather than $\N_{(u,\theta)}$ and $\N_{(u,\theta,[\tilde{v}] )}$.  In particular, the image of this right inverse consists of vector fields which are tangent to the $\bC^{n}$ factor at $z_0(S_1,\theta)$ and take value in $T\N^{\bC^{n}}_{(u,\theta+\lambda,[\tilde{v}] )}$ at $z_1(S_1, \theta, S_2) $.

For $\lambda$ small enough, the hyperplane $\N^{\bC^{n}}_{(u,\theta+\lambda)}$ still intersects all spheres passing through $u(\theta)$ transversely, and this is the only property we used.    We shall write $\Gext^{\lambda}$ for the restriction of \eqref{eq:extended_gluing_codim_3} to a fixed $\lambda$. Note that the composition of  $\Gext^{\lambda}$ for $\lambda=0$ with the projection to $\cP(L;0)$ is the previously constructed gluing map $\Gcod{3}{0}$.

We shall now define an a priori left inverse
\begin{equation} \label{eq:inverse_gluing_codim_3}   \Uthree(\cP(L;0) \times S^1) \to \codt{3}{L} \times [0,+\infty)^{2} \times (-\eta,\eta) \end{equation}
from a neighbourhood $\Uthree(\cP(L;0) \times S^1)$ of the image of $\Gext^{\leq \eta}$  in $\cP(L;0)$.

The neighbourhood $\Uthree(\cP(L;0) \times S^1)$ consists  of pairs $(w,\phi)$ such that $w$ is  $C^1$ close to a unique solution $u$ to \eqref{dbar-operator}  in homotopy class $-\beta$, away from a small disc centered on the boundary, and  a re-parametrization of $w$ is $C^1$ close (again, away from  a small disc) to a holomorphic sphere $v'$ passing through some point $u(\theta')$. In addition, the angular parameter is required so satisfy $| \phi - \theta'| \leq 2\eta$.  Choosing $\eta$ to be sufficiently small, these properties  are sufficient to ensure that $w^{-1}\left(\N^{\bC^{n}}_{(u,\phi)}\right)$ contains a unique point lying in a neighbourhood of $\phi$.  This point can be uniquely written as $z_0(S_1, \theta)$, determining the parameters $S_1$ and $\theta$, and hence defining a  map
 \begin{align}  \Uthree(\cP(L;0)\times S^1) & \to \N(\cP_{0,1}(L;-\beta)) \times [0,+\infty) \times (-\eta,\eta) \\
 (w,\phi) & \mapsto \left(w(z_0(S_1, \theta)), u,\theta, S_1, \phi - \theta\right) .
  \end{align}
The point $w(z_0(S_1, \theta))$ already determines a holomorphic sphere $v$ passing through $u(\theta)$ up to parametrization.  We shall now explain how to recover the lift $[\tilde{v}]$ and the second gluing parameter.

Since the composition of $w$ with $\tau_{S_1}(\theta)$ is $C^1$-close, away from a neighbourhood of the boundary of $D^2$, to a re-parametrization of some holomorphic sphere $v$, the image of the positive real axis under this map intersects a unique hyperplane of the form $\N^{\bC^{n}}_{(u, \phi, [\tilde{v}])}$.  The image of this intersection point under the automorphism $\tau_{S_1}(\theta)^{-1}$ can be uniquely written as $z_1(S_1, \theta, S_2)$, so  we obtain a map
\begin{equation} \label{eq:inverse_codim_3_gluing}  \Uthree(\cP(L;0)\times S^1) \to \tilde{\N}(\cP_{0,1}(L;-\beta)) \times [0,+\infty)^{2} \times (-\eta,\eta) ,\end{equation}
which completes the construction of \eqref{eq:inverse_gluing_codim_3}.

The fact that the map described in Equation \eqref{eq:inverse_gluing_codim_3} is a left inverse to the gluing map follows straightforwardly from the fact that, for a given $\lambda$, our chosen right inverse along a pre-glued curve consists of vector fields whose projection to $T \bC \bP^{n-1}$ vanishes at $z_0(S_1, \theta)$, and which take value along $T \N^{\bC^{n}}_{(u, \theta+\lambda,[\tilde{v}])}$ at $z_1(S_1, \theta, S_2)$.  Since hyperplanes are geodesically convex for the standard metric on $\bC \bP^{n-1}$, applying the implicit function theorem yields a glued curve that has exactly the same intersection properties with $\N^{\bC^{n}}_{(u,\theta)}$ and $ \N^{\bC^{n}}_{(u,\theta+\lambda,[\tilde{v}])}$;  i.e, the pair of points of $D^2$ mapping to $\bC^{n} \times \{ \n \}$ and $ \N^{\bC^{n}}_{(u,\theta+\lambda,[\tilde{v}])}$ do not change.

Note that the construction of a left inverse automatically implies that the extended gluing map \eqref{eq:extended_gluing_codim_3} is injective.  As to surjectivity onto the open set $\Uthree(\cP(L;0) \times S^1)$, observe that it suffices to prove surjectivity onto the fibers of \eqref{eq:inverse_gluing_codim_3}, so long that the gluing parameter in $[0,+\infty)^2$ is sufficiently large.  The idea, as implemented in Lemma \ref{lem:surjectivity_codim_1_gluing} for the gluing result near the codimension $1$ stratum, is that as the gluing parameters go to infinity, the distance  (in an appropriate Sobolev norm) between any two solutions in the same fiber of \eqref{eq:inverse_gluing_codim_3} must go to zero (this uses the definition of Gromov compactness and an exponential estimate for the $C^1$ norm of small energy annuli and strips).  However, the implicit function theorem implies that there is a unique such solution in a uniformly sized neighbourhood of the image of the gluing map.  We conclude:
\begin{lem}
The composition of the extended gluing map \eqref{eq:extended_gluing_codim_3} with the projection to $\cP(L;0)$ is a proper surjection onto a neighbourhood of $\cod{3}{L}$.
\noproof
\end{lem}

Next, we show that the extended gluing map  \eqref{eq:inverse_gluing_codim_3} is a diffeomorphism onto its image for $S_1$ and $S_2$ sufficiently large. First, we prove that the map
\begin{equation} \label{eq:smooth_projection_to_hyperplane}  \Uthree(\cP(L;0) \times S^1) \to \N(\cP_{0,1}(L;-\beta)) \end{equation}
is smooth with no critical point.  To see this, consider the restriction of $D_{\cP}$ at $\G^{\lambda}_{S_1,S_2}(u, \theta, [\tilde{v}])$ to those vector fields whose projection to $\bC \bP^{n-1}$ vanishes at $z_0(S_1,\theta)$.  This operator is surjective by parallel transport of the analogous result from $\preGext_{S_1,S_2}(u, \theta , [\tilde{v}])$.  In particular, the evaluation map
\begin{align} \Uthree(\cP(L;0) \times S^1)  \times D^2 &  \to M  \times \cP_{0,1}(L;0) \\
(w,\phi,z) & \mapsto (w(z), u,\phi)
\end{align}
is a submersion at 
\begin{equation} \left(\G^{\lambda}_{S_1,S_2}(u, \theta, [\tilde{v}]),  \theta +\lambda , z_0(S_1,\theta) \right), \end{equation}
 so that the component of the inverse image of 
 \begin{equation} \N^{\bC^{n}} (\cP_{0,1}(L;-\beta))  \equiv  \bC^{n} \times  \N (\cP_{0,1}(L;-\beta))  \end{equation}
 passing through this point is a smooth submanifold.  Using Gromov compactness, we know that for $S_1$ and $S_2$ sufficiently large, $\G^{\lambda}_{S_1,S_2}(u, \theta, [\tilde{v}])$ is $C^1$-close to $v$, which implies that the tangent space of the image of $\G^{\lambda}_{S_1,S_2}(u, \theta, [\tilde{v}])$ at $z_0(S_1,\theta)$ is transverse to $T\N^{\bC^{n}}_{(u,\theta+\lambda)}$.  In particular, this component of the inverse image of $\N^{\bC^{n}}_{(u,\theta+\lambda)}$ in $\Uthree(\cP(L;0) \times S^1)  \times D^2 $ is transverse to the $D^2$ fibres, and is therefore defined by a smooth section.  The reader may easily check that the map defined in Equation \eqref{eq:smooth_projection_to_hyperplane} is the composition of this section with the evaluation map to $M$, which establishes its smoothness.

By composing this smooth section with the projection to the $D^2$-direction, we conclude the  smoothness of the map
\begin{equation}  \Uthree(\cP(L;0) \times S^1) \to  \N(\cP_{0,1}(L;-\beta))  \times [0,+\infty) \times S^1 \end{equation}
where the last two coordinates are the pair $(\theta,S_1)$ determined by the intersection point $z_0(S_1,\theta)$.  An argument along the same lines  shows that the inverse to the gluing map described in Equation \eqref{eq:inverse_gluing_codim_3} is a bijective smooth map with no critical points.  This implies that the extended gluing map $\Gext^{\leq \eta}$ also satisfies these properties, i.e. that it is a diffeomorphism.

We shall need to know only one quantitative property of the differential of  \eqref{eq:inverse_gluing_codim_3}.   Consider a pair of sequences  $w_i$ converging to a point $(u',\theta', v')$ in $\cod{3}{L}$, and $\phi_i$ converging to $\theta'$:
\begin{lem} \label{lem:derivative_inverse_gluing_preserve_extra_factor}
The $\phi$-derivative at $(w_i, \phi_i)$ of the third component of the inverse \eqref{eq:inverse_gluing_codim_3} of the extended gluing map $\Gext^{\leq \eta}$  converges to $1$.
\end{lem}
\begin{proof}
The third component is a difference $\theta - \phi_i$, where the angle $\theta$ depends on $\phi$ as the angular parameter of the point of intersection between $w_i$ and $\N^{\bC^{n}}_{(u,\phi_i)}$.  Assuming $i$ to be large enough, we may choose a gluing parameter $S_i$ very large such that $w_i \circ \tau_{S_i}(\phi)$ is $C^2$-close to a parametrization $v_i$ of a holomorphic sphere mapping $0$ to $\N^{\bC^{n}}_{(u,\phi_i)}$ and $1$ to a hyperplane $\N^{\bC^{n}}_{(u,\phi_i, \tilde{n})}$.  In particular, the $C^2$-norm of  the restriction of $w_i \circ \tau_{S_i}(\phi)$ to this region is uniformly bounded for $i$ large enough, so the inverse image of a nearby hyperplane $\N^{\bC^{n}}_{(u,\phi)}$ under $w_i \circ \tau_{S_i}(\phi)$ is a point whose norm is bounded by a constant multiple of $|\phi - \phi_i|$.  Reparametrising the disc by the inverse to $\tau_{S_i}(\theta_i)$, whose derivative near the origin decays exponentially with $S_i$, we reach the desired conclusion.
\end{proof}

We shall now prove surjectivity of the gluing map by proving that the gluing map has degree $1$ at infinity.  Let $\G^{\leq \eta}$ denote the composition of $\Gext^{\leq \eta}$ with the projection to $\cP(L;0)$.  We pick an open neighbourhood $\End_{3}(\cP(L;0))$ of $\cod{3}{L}$ in $\cP(L;0)$ lying in the image of $\G^{\leq \eta}$, and with the additional property that it does not intersect the image of the set
\begin{equation}  \codt{3}{L} \times \{S\} \times [S,+\infty) \times [-2\eta, 2\eta] \cup \codt{3}{L} \times [S,+\infty) \times \{S\} \times [-2\eta, 2\eta].\end{equation}
Note that this choice is possible because the definition of Gromov's topology  implies that the images of the above sets, as $S_1$ or $S_2$ in $[S,+\infty)$ go to infinity, do not converge to $\cod{3}{L}$.  To prove that this map has a well-defined degree, consider any two points of  $\End_{3}(\cP(L;0))$  which lie in the same component.  Choosing a generic path between them, we find that the inverse image under $\G^{\leq \eta}$ is a smooth $2$-dimensional submanifold of $ \codt{3}{L} \times [S,+\infty)^{2} \times [-2\eta, 2\eta] $ with boundary on the ``horizontal boundary"
\begin{equation} \codt{3}{L} \times (S,+\infty)^{2} \times \{ -2\eta \} \cup \codt{3}{L} \times (S,+\infty)^{2} \times \{ 2\eta \}  .\end{equation}
In particular, the intersection of this surface with
\begin{equation} \codt{3}{L} \times (S,+\infty)^{2} \times \{0  \} \end{equation}
is a compact $1$-manifold with boundary for generic paths.  Interpreting this manifold as a cobordism between the inverse images of the two points under the gluing map, we conclude
\begin{lem}  The number of inverse images of a point in $\End_{3}(\cP(L;0))$ under $\Gcod{3}{0}$, counted with signs, depends only on the connected component  in $\End_{3}(\cP(L;0))$.  \noproof \end{lem}
We denote this integer count of inverse images the {\bf degree at infinity} of the gluing map.

We shall now use Condition \eqref{eq:condition_hyperplanes} to prove that this degree is $1$, thereby proving surjectivity of the gluing map.  Let us choose a constant $\rho$ such that the hyperplane $\N_{(u,\theta)}$ agrees with  $\N_{(u,0)}$ whenever $|\theta|\leq 2\rho$, and assume that $\eta \leq \rho$.   Using Gromov compactness, we note that if the gluing parameters are large enough
\begin{equation} \label{eq:gluing_near_1_isolated} \G^{\lambda}_{S_1,S_2}(u,1,[v]) = \G^{\lambda'}_{S'_1,S'_2}(u',\theta',[v']) \implies u=u' \textrm{ and }|\theta'| \leq \rho. \end{equation}

\begin{lem}
The gluing map has degree $1$ at infinity.
\end{lem}
\begin{proof}
By the Equation \eqref{eq:gluing_near_1_isolated}, it suffices to prove the gluing map is injective when restricted to a neighbourhood of $\theta=0$.  In such a neighbourhood, the extended gluing map is essentially independent of $\lambda$:
\begin{equation} \Gext^{\leq \eta} \left( u, \theta, [\tilde{v}], \lambda \right) = \left( \Gcod{0}{3} ( u, \theta, [\tilde{v}]), \theta - \lambda \right) .\end{equation} 
The proof that $ \Gext^{\leq \eta}$ is a diffeomorphism by constructing a left inverse implies that the gluing map $ \Gcod{0}{3}$ is also a diffeomorphism, and hence is injective.
\end{proof}

\begin{lem} \label{lem:diffeo_end_near_codim_3}
The gluing map \eqref{eq:gluing_codim_3} is a diffeomorphism onto a neighbourhood of $\cod{3}{L}$ in $\cP(L;0)$.
\end{lem}
\begin{proof}
Since a degree $1$ map with no critical point is a local diffeomorphism, it remains to show that $\G$ is an immersion.  Since the extended gluing map is a diffeomorphism, it suffices to show that its restriction to $\lambda =0$ is transverse to the fibers of the projection map to $\cP(L;0)$, which is itself equivalent to the statement that the image of $\partial_{\phi}$ (the tangent direction to the circle fibre in $\cP(L;0) \times S^1$) under the differential of the inverse gluing map, does not lie in the tangent space of $\codt{3}{L} \times [S,+\infty)^2$.  Lemma \eqref{lem:derivative_inverse_gluing_preserve_extra_factor} implies precisely this.
\end{proof}

\subsection{The corner structure at the codimension $3$ stratum} \label{sec:corner-structure-at-codim-3}
Next, we consider the moduli space $\cod{2}{L}$, with its universal circle bundle $\codt{2}{L}$ which can be compactified to a manifold with boundary $\codt{2}{L}$.  In particular, we may define a gluing map
\begin{equation} \label{eq:gluing_codim_3_to_codim_2} \Gcod{3}{2} \co \codt{3}{L} \times [S, +\infty) \to  \codt{2}{L} \end{equation}
which is a diffeomorphism onto a neighbourhood of $ \codt{3}{L}$.   We can define such a map using the implicit function theorem as before, but it is also possible to do so explicitly: consider $S$ large enough so that holomorphic spheres passing through $u(z_0(S_1, \theta))$ for $S \leq S_1$ are transverse to the hyperplanes $\N^{\bC^{n}}_{(u,\theta)}$ and $\N^{\bC^{n}}_{(u,\theta,\tilde{n})}$ introduced in the previous section.  We have a smooth map
\begin{align} \tilde{\N}(\cP_{0,1}(L;-\beta)) \times   [S, +\infty) & \to  \codt{2}{L}  \\
(u,  \theta, \tilde{\n}, S_1) & \mapsto \left(u, z_0(S_1, \theta), [\tilde{v}] \right), 
\end{align}
where the parametrization on $[\tilde{v}]$ is fixed by the properties
\begin{equation}  v(\infty) = u(z_0(S_1, \theta) ) \textrm{, } v(0) \in \bC^{n} \times \{ \n \} \textrm{ and } v(1) \in  \N^{\bC^{n}}_{(u,\theta,\tilde{n})}. \end{equation} 
It is not hard to check that this is a diffeomorphism onto a neighbourhood of $\codt{3}{L}$ in $ \codt{2}{L} $.  After composition with the diffeomorphism of Equation \eqref{eq:diffeomorphism_codim_3_stratum_circle_bundle}, we obtain the desired gluing map.

Let $U_{S}^{S'}( \codt{2}{L})$ denote the image of the restriction of $\Gcod{3}{2}$ to $\codt{3}{L} \times [S, S')$.  By inverting $\Gcod{3}{2}$  in this domain, and composing with $\Gcod{3}{0}$,  we obtain a map
\begin{equation} \label{eq:compose_inverse_gluing_codim_3_gluing} U_{S}^{S'}( \codt{2}{L}) \times [S, +\infty) \to   \codt{3}{L} \times [S, +\infty)^2 \to \cP(L;0). \end{equation}
As a prelude to defining a gluing map on $\codt{2}{L}$, we shall realize this composition by applying the implicit function theorem to an appropriate right inverse along pre-glued curves.  More precisely, consider a pair $(u,z,[\tilde{v}])$, where $z = z_0(S_1,\theta)$. We have a cylindrical end at $z_0(S_1,\theta)$ induced by pre-gluing from the cylindrical end  $\xi_{0,\theta}$ of Equation \eqref{eq:varying_cylindrical_end_ghost_bubble} on the ghost bubble.  Explicitly, the cylindrical end at $z_0(S_1,\theta)$ is \begin{align} \notag  S^1 \times (-\infty,0] & \to D^2 - \{0\} \\
(\phi, S) & \mapsto \left( \tau_{S_1}(\theta) \right)^{-1} \left( e^{S+ i \phi} \right), \label{eq:varying_cylindrical_end}
\end{align}
Next, we consider the parametrization of $[\tilde{v}]$ which maps $0$ to $\N^{\bC^{n}}_{(u,\theta)}$ and $1$ to $\N^{\bC^{n}}_{(u,\theta,\tilde{n})}$.  A choice of parameter $S_2$ gives an approximate solution to $\dbar_{\cP}$ in the trivial homotopy class obtained by pre-gluing $[\tilde{v}]$ and $u$ along their chosen cylindrical ends.  Note that $z_0(S_1,\theta)$ survives as a distinguished point on the domain of the pre-gluing, which is also equipped with a marked point $z_1(S_1,\theta,S_2)$ coming from $1 \in \bC \bP^1$.  If $S_1$ and $S_2$ are sufficiently large,  the operator $D_{\cP}$ at this pre-glued curve becomes an isomorphism, with an inverse that is uniformly bounded with respect to the gluing parameters, when restricted to those vector fields whose projection to $\bC \bP^{n-1}$ vanishes at $z_0(S_1,\theta)$ and which take value in the tangent space of the appropriate hyperplane $T \N^{\bC^{n}}_{(u,\theta,\tilde{n})}$ at $z_1(S_1,\theta,S_2)$.

Using this right inverse, we can define a gluing map
\begin{equation} \label{eq:glue_compact_region_codim_2_near_codim_3}  U_{S}^{S'}( \codt{2}{L})  \times [S, +\infty) \to \cP(L;0) \end{equation}
for $S$ large enough.  The image of $\left( (u,z_0(S_1,\theta),[\tilde{v}]), S_2 \right)$ is a solution to \eqref{dbar-operator} such that $z_0(S_1,\theta)$ projects  to $\n \in \N_{(u,\theta)}$, and  $z_1(S_1,\theta,S_2)$ maps to the hyperplane $\N^{\bC^{n}}_{(u,\theta,\tilde{n})}$.  Lemma \ref{lem:diffeo_end_near_codim_3}  implies that there is a unique such  curve in the neighbourhood of $\cod{3}{L}$ in $\cP(L;0)$, which is moreover the image of  $\left( (u,z_0(S_1,\theta),[\tilde{v}]), S_2 \right)$ under the composition  of gluing maps \eqref{eq:compose_inverse_gluing_codim_3_gluing}.  We conclude

\begin{lem} \label{lem:gluing_codim_3_to_codim_2_commutes}
The gluing map \eqref{eq:glue_compact_region_codim_2_near_codim_3} agrees with the restriction of the composition \eqref{eq:compose_inverse_gluing_codim_3_gluing} to $ U_{S}^{S'}( \codt{2}{L}) \times [S, +\infty)$ .  \noproof
\end{lem}

There is also a smooth embedding
\begin{equation} \label{eq:glue_codim_3_to_codim_1} \Gcod{3}{1} \co \codt{3}{L} \times [S, +\infty) \to \cod{1}{L}  \end{equation}
for $S$ large enough, defined by the choices of right inverses and strip-like ends as in the previous section.  Indeed, starting with a parametrized sphere bubble $[\tilde{v}]$ passing through $u(\theta)$, a choice of a sufficiently large gluing parameter $ S_1 $ defines a pre-glued curve which is an approximately holomorphic disc $ \preGext_{S} \left( [\tilde{v}], \gh(u(\theta))\right)$.  The pre-glued disc is equipped with two distinguished interior points $z_0(S_1)$ and $z_1(S_1)$ coming from $\{0,1\} \in \bC \bP^1$, and, as in the previous section, the restriction of the linearisation of the $\dbar$ operator becomes an isomorphism with uniformly bounded inverse once its domain is restricted to vector fields which vanish at $-1$ and $z_0(S_1)$  and take value along some hyperplane $T \N_{(u,\theta)}$ at $z_1(S_1)$.  Using this isomorphism to define the right inverse to $D_{\dbar}$, we apply Floer's Picard Lemma to construct $ \Gcod{3}{1}$, and the existence of a smooth inverse to $\Gcod{3}{1}$ in order to conclude that $\Gcod{3}{1}$ is a diffeomorphism onto a neighbourhood of $\codt{3}{L} $ in $\cod{1}{L}$.  In particular, the gluing map \eqref{eq:glue_codim_3_to_codim_1}  is compatible with the projection map to the $\bC \bP^{n-1}$ bundle $\N(\cP_{0,1}(L;-\beta))$ over $\cP_{0,1}(L;-\beta)$ in the sense that we have a commutative diagram
\begin{equation} \label{eq:glue_codim_3_to_codim_1_evaluate_point} \xymatrix{ \codt{3}{L} \times [S, +\infty) \ar[r] \ar[rd] &  \cod{1}{L} \ar[d] \\
& \N(\cP_{0,1}(L;-\beta)) .}   \end{equation}
\begin{lem} \label{lem:smooth_structure_compactification_codim_1_stratum} The gluing map $\Gcod{3}{1}$ determines a smooth structure on the manifold with boundary
\begin{equation} \codb{1}{L} = \cod{1}{L} \cup \codt{3}{L}. \end{equation}
 \noproofe
\end{lem}

\subsection{The codimension $2$ stratum} \label{sec:codim-2-strat}

In this Section, we discuss Lemma \ref{lem:gluing_commute_codim_3_codim_2_top_stratum}.  The proof is essentially the same as that of Lemma \ref{lem:codim_3_to_parametrized_space_diffeo_gluing}, so our discussion shall be brief.  The key step is a careful construction of the gluing map.  First, we consider a family of hyperplanes $\N_{(u,z)} \subset \bC \bP^{n-1}$ smoothly parametrized by $z \in \inte D^2$ and $u \in \cP(L;-\beta)$, which agrees with $\N_{(u,\theta)}$ near the boundary. In addition, we require that 
\begin{equation} u(z) \notin \N_{(u,z)}\end{equation}
whenever $z \in D^2$, and $u \in \cP(L; -\beta)$.  As in the discussion following \eqref{eq:condition_hyperplanes}, one can easily construct such a family starting with the family of hyperplanes ``orthogonal" to each point, then modifying this family by a cutoff function near the boundary to ensure that the surfaces are independent of the radial direction.   We obtain a bundle
\begin{equation}\xymatrix{  \N( \cP_{1,0}(L;-\beta)) \ar[r] \ar[dr] &  \bC \bP^{n-1} \times \cP_{1,0}(L;-\beta) \ar[d]  \\
&  \cP_{1,0}(L;-\beta)  \cong \cP(L; -\beta) \times D^2 }
\end{equation}
whose fibres are the hyperplanes $\N_{(u,z)}$.  We denote the normal bundle of $\N( \cP_{1,0}(L;-\beta))$ in the product by $ \tilde{\N}( \cP_{1,0}(L;-\beta)) $, and associate, as before, hyperplanes
\begin{equation} \N_{(u,z, \tilde{\n})} \subset \bC \bP^{n-1} \text{ and } \N^{\bC^{n}}_{(u,z, \tilde{\n})} \subset \bC^{n} \times \bC \bP^{n-1}  \end{equation}
to every $\tilde{\n} \in \tilde{\N}_{(u,z)}$.  Moreover, we can associate to each such $\tilde{\n}$ a unique holomorphic sphere $[\tilde{v}]$ satisfying the analogue of Equation \eqref{eq:condition_parametrization_sphere_passing_hyperplane}, yielding a diffeomorphism
\begin{equation} \tilde{\N}( \cP_{1,0}(L;-\beta)) \to \codt{2}{L}. \end{equation}

We fix the choice of cylindrical end near $\infty \in \bC \bP^{1}$ as in \eqref{eq:cylindrical_end_infinity_sphere}, so the next step is to choose a cylindrical end for each $z \in D^2$ which we require to be compatible with the choice of cylindrical end determined by pre-gluing whenever $z$ is sufficiently close to the boundary.  Note that setting $S_1=0$ in \eqref{eq:varying_cylindrical_end} defines a cylindrical end which is independent of $\theta$.  In particular, the previously chosen family of cylindrical ends extends from a neighbourhood of $D^2$ to the entire disc.

We have now made all the necessary choices to carry through the same procedure as in Section \ref{sec:codim_3_to_0_gluing}, and define a gluing map $\Gcod{2}{0}$.
\begin{rem}
Since $\cod{2}{L}$ is not compact, there is no a priori reason to believe that there is a uniform constant for which the gluing map is defined.  However, Lemma \ref{lem:gluing_codim_3_to_codim_2_commutes} implies that if we restrict to a neighbourhood of $\codt{3}{L}$ in $\codt{2}{L}$, we can define gluing using the composition with the gluing map $\Gcod{3}{0}$ with inverse of the gluing map $\Gcod{3}{2}$, as in \eqref{eq:compose_inverse_gluing_codim_3_gluing}.  The existence of a global constant $S$ for which the gluing map  $\Gcod{2}{0}$ is defined follows immediately.
\end{rem}

It remains to prove that the gluing map \eqref{eq:codimension_2_gluing_map} is a diffeomorphism.   Unlike all our other arguments, the gluing procedure in this setting takes place in the region where the $1$-form $\gamma$ does not vanish, which could a priori affect the analytical parts of the argument.  However, the reader may check that the estimates in Appendix \ref{ap:pointwise_estimates} and Section \ref{sec:codim1_gluing} relied on the fact that $\gamma$ was uniformly bounded in the $L^{p}$-norm, rather than its vanishing.  It is not hard to see  that this is still true even with the weighted $L^p$-norms we would use when gluing at interior nodes.  With this in mind, the proof that $\Gcod{2}{0}$ is a diffeomorphism onto a neighbourhood of $\codt{2}{L}$ is entirely analogous to the strategy deployed in Section \ref{sec:codim_3_to_0_gluing} and is left to the reader.   We simply note that the analogue of Equation \eqref{eq:extended_gluing_codim_3}  is a map
\begin{equation} \label{eq:codim_2_gluing_add_parameter}  \codt{2}{L} \times [S, +\infty)  \times B_{2 \eta} \to \cP(L;0) \times \bC \end{equation}
where $B_{2 \eta}$ is a disc of radius $2 \eta$, and the map to the second component is given by the addition:
\begin{equation}  (u,z,[\tilde{v}],S_2, z_0) \mapsto z + z_0.\end{equation}

\subsection{Compactifying the capping moduli spaces to manifolds with corners} \label{sec:compactify_capping_mod_space}
In this Section, we prove Lemma \ref{lem:compactification_capping_moduli_spaces_corners}.
\begin{lem} \label{lem:gluing_capping_spheres_to_discs}
There exists a smooth embedding 
\begin{equation} \label{eq:gluing_capping_spheres_to_discs} \G^{\cC} \co \codtC{2}{L} \times [S, +\infty) \to \cC(L) \end{equation}
onto a neighbourhood of $\codC{2}{L}$ in $ \cC(L)$.  Moreover, on the boundary $ \codt{3}{L}$ of $\codtC{2}{L}$, $ \G^{\cC}$ restricts to a gluing map with target $\partial \cC(L) \cong \cod{1}{L}$ which satisfies the commutative diagram given in Equation \eqref{eq:glue_codim_3_to_codim_1_evaluate_point}.
\end{lem}

To define $ \G^{\cC}$, we proceed in exactly the same way as before, and consider a smooth family $\N_{(c_u,z)}$ of hyperplanes in $\bC \bP^{n-1}$ parametrized by $z \in D^2$ and $c_u$ a capping disc for $u \in \cP(L;-\beta)$ as in Equation \eqref{eq:choose_capping_disc}. We require that the projection of $c_{u}(z)$ to $\bC \bP^{n-1}$ not lie in $\N_{(c_u,z)}$, and that $\N_{(c_u,re^{i\theta})} \equiv \N_{(u,\theta)}$ whenever $r$ is sufficiently close to $1$.   The union of these hyperplanes defines a smooth submanifold
\begin{equation} \N( \tilde{\partial}^{2} \cC(L) )  \subset \bC \bP^{n-1} \times \cP(L;0) \times D^2 \end{equation}
Following the strategy used in the previous sections, we consider the unit normal bundle $\tilde{\N}( \tilde{\partial}^{2} \cC(L) ) $ of $\N( \tilde{\partial}^{2} \cC(L) )$, which projects submersively onto $ \cP(L;0) \times D^2$ with fibre at $(u,z)$ the unit normal bundle $\tilde{\N}_{(c_u,z)}$.  Again, each point $\tilde{\n} \in  \tilde{\N}_{(c_u,z)}$  determines a hyperplane $ \N^{\bC^{n}}_{(c_u,z,\tilde{\n})}$.  As $z$ and $c_u$ vary, we conclude

\begin{lem} There exists a diffeomorphism
\begin{align} \tilde{\N}( \tilde{\partial}^{2} \cC(L) ) & \to  \codtC{2}{L} 
\\ (u,z, \tilde{\n}) & \mapsto (u,z, [\tilde{v}]) 
\end{align}
where the class of $[\tilde{v}]$ is determined by a parametrization mapping $\infty$ to $c_u(z)$, $0$ to $\n$, and $1$ to the hyperplane $ \N^{\bC^{n}}_{(c_u,z,\tilde{\n})}$. \noproof
\end{lem}

Using this identification, we can construct the gluing map of Lemma \ref{lem:gluing_capping_spheres_to_discs}.  Since the gluing map is defined by first pre-gluing the sphere bubble to a ghost bubble, the only remaining choice is that of a cylindrical end at the origin of the disc.  Because the choice \eqref{eq:varying_cylindrical_end_ghost_bubble} of cylindrical ends on ghost bubbles passing through the boundary of $c_u$  does not extend to the disc, we shall work with a constant the cylindrical end $\xi_{0} \equiv \xi_{0,0}$ which is independent of $z$.  With this in mind, the rest of the construction is left to the reader.

We can now construct a manifold with corners $\Chat(L)$ as asserted in Lemma \ref{lem:compactification_capping_moduli_spaces_corners}
 \begin{proof}[Proof of Lemma \ref{lem:compactification_capping_moduli_spaces_corners}]
 We define $\Chat(L)$  to be the union of $\cC(L)$ and $\codtC{2}{L} \times (S, +\infty]$ along $\codtC{2}{L} \times (S, +\infty)$ mapping to $\cC(L)$ by the gluing map  \eqref{eq:gluing_capping_spheres_to_discs}.  Any choice of homeomorphism from $ (S, +\infty] $ to $(0,1]$ which is differentiable in the interior determines a smooth corner structure on $\codtC{2}{L} \times (S, +\infty]$, and hence on $\Chat(L)$. \end{proof}

Note that this construction induces the structure of a smooth manifold with boundary on 
\begin{equation} \codbC{1}{L} = \cod{1}{L} \cup \codt{3}{L}.\end{equation}
We constructed a smooth structure on this very space in Lemma \ref{lem:smooth_structure_compactification_codim_1_stratum}, and denoted the resulting manifold with boundary $\codb{1}{L}$.  Since the choice of cylindrical ends on the ghost discs is constant on one side and depends on the angular parameter in the other, the gluing maps are a priori different, so the smooth structure may also be different.  The defect is resolved by the following result:
\begin{lem} \label{lem:boundary_smooth_structure_diffeo} There exists a diffeomorphism 
\begin{equation} \codbC{1}{L}  \to\codb{1}{L} \end{equation}
which agrees with the identity on $\cod{3}{L}$ and away from a neighbourhood of $\cod{3}{L}$.
\end{lem}
\begin{proof}
Let us write $\G^{\cC}_{S}$ for the gluing map 
\begin{equation} \codt{3}{L} \to \cod{1}{L} \end{equation}
which is the restriction of Equation \eqref{eq:codimension_2_gluing_map} to the boundary of $\codtC{2}{L}$ and to a gluing parameter $S$.  Since $\codb{1}{L}$ and  $\codbC{1}{L}$ are homeomorphic, and diffeomorphic away from the boundary, the image of $\G^{\cC}_{S}$ is a smoothly embedded submanifold $\codb{1}{L}$ which is cobordant to the boundary.  We write $X^{S}$ for this cobordism.

Note that both components of the boundary of $X^{S}$ are equipped with diffeomorphisms from $\codt{3}{L}$ (one is $\G^{\cC}_{S}$, and the other is essentially the identity).  Our construction would be complete if we were to know the existence of a diffeomorphism
\begin{equation} \label{eq:cobordism_trivialization} \codt{3}{L} \times [0,1] \to  X^{S}, \end{equation}
whose restriction to $\codt{3}{L}  \times \{ 1 \}$ is $\G^{\cC}_{S}$, and whose restriction to $\codt{3}{L}  \times \{ 0 \}$ is the identity.  Indeed, $\codt{3}{L} \times [0,1]$ is also diffeomorphic (upon choosing an identification of $[0,1]$ with $[S,+\infty]$ with the cobordism in $\codbC{1}{L} $ between the image of $\G^{\cC}_{S}$ and $\codt{3}{L}$, so \eqref{eq:cobordism_trivialization} defines a diffeomorphism from a neighbourhood of the boundary in $\codb{1}{L}$ to a similar neighbourhood in $\bar{\partial} \cC(L) $, which is the identity on the boundary component $\codt{3}{L}  \times \{ 0 \}$.  The condition that this map restricts to $\G^{\cC}_{S}$ on the slice $\codt{3}{L}  \times \{ 1 \}$ implies, as desired, that this map extends to the identity outside this neighbourhood.

It remains therefore to prove \eqref{eq:cobordism_trivialization}.  To do this, we consider the manifold $Y^S$ obtained by gluing the two boundaries of $X^S$ together along the map $\G^{\cC}_{S}$.  Note that the diagram \eqref{eq:glue_codim_3_to_codim_1_evaluate_point} is commutative for the gluing maps used to define both $\codb{1}{L}$ and  $\bar{\partial} \cC(L) $.  This implies that we have a fibration
\begin{equation} Y^S \to \N(\cP_{0,1}(L;-\beta))  \end{equation}
whose fibres are diffeomorphic to the $2$-torus.  We would like to check that this bundle is trivial.  Since Gromov compactness implies that we have a  trivialization of $X^{S}$ as a topological cobordism, we know that $Y^{S}$ is a trivial fibration in the topological category. A classical result, see \cite{EE}, asserts that the inclusion of the group of  diffeomorphisms into  homeomorphisms of $T^2$ is a homotopy equivalence, so we conclude the desired result.
\end{proof}
 
\section{Construction of a smooth manifold with corners} \label{sec:construction_corners}
In this section, we prove Lemma \ref{lem:existence_manifold_corner} by constructing the desired manifold with corners $\Phat(L;0)$. The only missing technical point for that Lemma is the proof that the gluing map $\G_{S}$ of Equation \eqref{eq:gluing_map} has degree $1$; this shall be facilitated by choosing the right inverse to $D_{\cod{1}{L}}$ in such a way that the following property holds:
\begin{equation} \label{eq:right_inverse_constraint_basepoints}
\parbox{35em}{there exists an open subset $U_{\star}$ of $\cod{1}{L}$ intersecting each component non-trivially, and a (possibly disconnected) complex hypersurface $\N_\star$ such that  $v$ intersects $\N_ \star $ orthogonally at $0$ whenever $(u,\theta,v)$ lies in $U_{\star}$, and the image of right inverse to $D_{\cod{1}{L}}$, restricted to curves in $U_{\star}$, is contained in the subspace of vector fields whose value at the origin $0$ lies in $T \N_\star$.} 
\end{equation}

The notion of degree makes sense because of the properness asserted in the next Proposition 
\begin{prop} \label{prop:good_approximate_inverse_gluing}
There exists a proper smooth map 
\begin{equation}  F \co \cP(L;0) \cup \codt{2}{L}  \to [0,+ \infty) \end{equation}
and a sequence $S^i \to \infty$ such that each $S^i$ is a regular value of $F$ with
\begin{equation} F^{-1}(S^i) =  \Im \G_{S^i}.  \end{equation}
Moreover, there exists a subset  $V_{\star} \subset U_{\star}$ of $\cod{1}{L}$ intersecting each component non-trivially such that the diagram
\begin{equation}
 \xymatrix{ V_\star \times [S, +\infty) \ar[r] \ar[d] & \cP(L;0) \cup \codt{2}{L}   \ar[d] \\
[S, +\infty) \ar[r] & [0,+\infty] } \end{equation}
commutes.
\end{prop}

We can now define a manifold with corners $\Phat^{S^i}(L;0)$ for each sufficiently large $S^i$ as the inverse image under $F$
\begin{equation}
\Phat^{S^i}(L;0) = F^{-1}(-\infty,S^i]
\end{equation}
The reader may easily verify that such a manifold satisfies all the desired conditions stated in Lemma \ref{lem:existence_manifold_corner}.  For example, the inclusions of $\cod{1}{L}$ and $\codt{3}{L}$ are given by the gluing maps $\G_{S^i}$ and $\G_{S_i,+\infty}$.

\begin{rem}
It seems more than plausible that the diffeomorphism type of $\Phat^{S^i}(L;0)$ is independent of $S^i$.  Since we do not need this property, we have not attempted to prove it.
\end{rem}
\subsection{The degree of the gluing map} \label{sec:degree_gluing}
Recall that the inclusion $\cod{1}{L}$ into $\cF_{\cod{1}{L}}(L)$ requires the choice of a section of an $\Aut(D^2,-1)$ bundle.  The gluing theorem of Lemma \ref{lem:smooth_structure_compactification_codim_1_stratum}  determines such a section away from a compact subset in $\cod{1}{L}$, which we extend globally  (there are no obstructions since   $\Aut(D^2,-1)$ is contractible).  Moreover, we choose the compact subset $K \subset \cod{1}{L}$ to be sufficiently large so that its interior together with the image of $\Gcod{3}{1}$ form a cover of $\cod{1}{L}$.  Finally, we assume that the right inverse chosen in  \eqref{eq:fibre_product_right_inverse} satisfies the following property:
\begin{equation} \label{eq:constrain_right_inverse}
\parbox{35em}{On the image of $\Gcod{3}{1}$, $Q_{\cod{1}{L}}$ agrees with the right inverse induced by the gluing map.} 
\end{equation}

Let us briefly explain the construction of this induced right inverse: for each pair $(u, \theta, v)$ in this neighbourhood, $v$ comes equipped with two marked points obtained as the images of the two marked point which stabilize a sphere bubble.  The image of $Q_{\cod{1}{L}}$ is the set of tangent vectors whose projection to $\bC \bP^{n-1}$ vanishes at one of the two points, and lies on a prescribed hyperplane at the other.

\begin{lem} \label{lem:commutative_gluing_codim_3_to_codim_1}
There exists a continuous gluing map
\begin{equation} \label{eq:codimension_1_gluing_map}\G \co \cod{1}{L} \times [S, +\infty) \to \cP(L;0) \end{equation}
whose restriction to each slice $ \cod{1}{L} \times \{ S_1 \}$  is a smooth immersion  and such that the diagram
\begin{equation} \label{eq:commutative_gluing_codim_1} \xymatrix{ \codt{3}{L} \times [S,+\infty)^2 \ar[dr] \ar[r] & \cod{1}{L} \times [S,+\infty) \ar[d] \\ & \cP(L;0) } \end{equation}
commutes. \noproof
\end{lem}
Note that there is a slight abuse of notation since we are writing $\G$ in Equation \eqref{eq:codimension_1_gluing_map} for what should more consistently be written as $\Gcod{1}{0}$.  The proof of this result is essentially the same as that of Lemma \ref{lem:gluing_commute_codim_3_codim_2_top_stratum}:  In a neighbourhood of $\codt{3}{L}$, we define $\G $ as the composition of the  inverse to $\Gcod{3}{1}$, with the gluing map $\Gcod{3}{0}$.  On $K \times [S_1, +\infty)$, the gluing map $\G$  is defined as in Equation \ref{eq:gluing_map}.  The condition we imposed in Equation \eqref{eq:constrain_right_inverse} guarantees that these two maps agree whenever  both are defined, which implies the commutativity of the diagram \eqref{eq:commutative_gluing_codim_1}.

We shall now prove that this map has degree $1$ onto a neighbourhood of $\cod{1}{L}$.  Our starting point is the construction of a function on pairs of points in $\cF_{\cP}(L)$ which will play the role of a distance even though  it does not satisfy the triangle inequality.  Given two maps $w$ and $w'$ in $\cF(L)$ we define
\begin{equation}  \tilde{d}\left((R,w),(R',w')\right) =  \left| R-R' \right| + \min_{\theta} \left(|\theta| + \max_{z} d(w(ze^{i \theta}), w'(z)) \right),
\end{equation}
where $d$ on the right hand side is the ordinary distance with respect to a fixed metric which is smaller than all $g_{z,R}$ metrics.  
\begin{lem} The map $\tilde{d}$ is continuous and  vanishes only on the diagonal.  \noproof \end{lem}

We shall compare this function to the product metric
\begin{equation} d((R,u,\theta,v, S), (R',u',\theta', v', S' )) =|R-R'| +|S -S'| + d(\theta, \theta') + d(u,u') + d(v,v')  \end{equation}
of $\cod{1}{L}$.  The following Lemma is a quantitative version of the injectivity property of $\preG$ implied by Lemma \ref{lem:pre-gluing_injective}.

\begin{lem} \label{lem:diameter_inverse_image_decays}
For each $\rho \geq 0$, there exists a constant $\tilde{\epsilon}(\rho) \geq 0$ such that
\begin{equation*}  \tilde{d}( \preG_{S}(u,\theta,v), \preG_{S'}(u',\theta',v') ) \leq \tilde{\epsilon}(\rho)   \implies  d((R,u,\theta,v, S), (R',u',\theta', v', S' )) \leq \rho \end{equation*}
whenever $(R,u,\theta,v ,S)$ and $(R',u',\theta', v', S' )$ are elements of $K$ with $u(1) = v(-1)$ and $u'(1) = v'(-1)$.
\end{lem}
\begin{proof}
Fix $\rho \geq 0$, and assume, by contradiction, that there exists a sequence of pairs $(R_i,u_i,\theta_i,v_i,S_i)$ and $(R_i,u'_i,\theta'_i,v'_i,S'_i)$ such that
\begin{align} d ((R_i,u_i,\theta_i,v_i,S_i),(R'_i, u'_i,\theta'_i,v'_i,S'_i)) & \geq \rho \\ \label{eq:limit_strange_distance_0}
\lim_{i \to \infty}  \tilde{d}( \preG_{S_i}(u_i,\theta_i,v_i) , \preG_{S'_i}(u'_i,\theta'_i,v'_i) ) & = 0
 \end{align}
We write $(R_i,w_i)$ and $(R'_i,w'_i)$ for $\preG_{S_i}(u_i,\theta_i,v_i) $ and $\preG_{S'_i}(u'_i,\theta'_i,v'_i) $ respectively.  First, we observe that the finiteness of the moduli space $\cP(L; -\beta)$ and Lemma \ref{lem:one_exceptional_solution_at_a_time} imply that we may assume that for sufficiently large $i$, we have an identity
\begin{equation}  (R_i, u \circ r_{-\theta_i}) =  (R'_i, u' \circ r_{-\theta'_i}) = (R,u) \end{equation}
for some fixed $(R,u)$ that is independent of $i$.

Using the compactness of $K$, we may assume that the sequences $(\theta_i,v_i)$ and $(\theta'_i,v'_i)$ converge respectively to $(\theta,v)$ and $(\theta',v')$.  In addition, if $S_i$ converges to $S$ and $S'_i$ to $S'$, then the fact, proved  in Lemma \ref{lem:pre-gluing_injective},  that $\preG$ is an embedding,  contradicts our assumptions.   

One of $S_i$ and $S'_i$ must therefore converge to $+\infty$.  To see that the other must satisfy the same property, we observe that the fact that $v$ is not in the trivial homotopy class implies that for each $(u,\theta,v) \in K$ there is a point $z_{(u,\theta,v)}$  in the interior of $D^2$  such that
\begin{equation} d\left(v(z_{(u,\theta,v)}), L\right) \geq 2 \rho_0 \end{equation}
for a uniform constant $\rho_0$; we assume that $\rho \leq \rho_0$.  The point $z_{(u,\theta'_i,v')}$ determines a marked point $z_{S_i,w_i}$ in the domain of $w_i$, which, for an unbounded sequence $S_i$, converges to the boundary.  Assuming that $S_i$ converges to $+\infty$ but $S'_i$ does not, we find that there is a neighbourhood of $\partial D^2$ which is mapped to the $\rho_0$-sized neighbourhood of $L$ by $w'_i$, but whose image under $w_i$ lies away from this set.  We conclude that $S_i$ and $S'_i$ both converge to $+\infty$.

A similar analysis shows that $\tilde{d}(w_i,w'_i)$ is necessarily bounded above by $\theta - \theta'/2$ as $S_i$ and $S'_i$ converge to  infinity, so we must assume $\theta=\theta'$ in order for \eqref{eq:limit_strange_distance_0} to be satisfied.   The final step of showing that $v=v'$ is left to the reader.
\end{proof}

Next, we consider a map $\G_{\epsilon}$ given as the composition of the gluing map $\Gext_{\epsilon}$ with the projection to $\cP(L;0)$.  Pick $\rho$ much smaller than the injectivity radius of $\cod{1}{L}$, and let the constant $\epsilon$ in Proposition \ref{prop:statement_implicit_function_theorem} be smaller  than $\frac{\tilde{\epsilon}(\rho)}{8c_p}$, where $\tilde{\epsilon}(\rho)$ is given by the previous Lemma and $c_p$ is  a uniform Sobolev constant (see Lemma \ref{lem:uniform_sobolev}):
 \begin{lem} \label{cor:inverse_image_extended_gluing_small} The inverse image of a point under $\G_{\epsilon}$ has radius smaller than $\rho$. 
 \end{lem}
 \begin{proof} Assume that $\G_{\epsilon,S}(u,\theta,v,Y)$ and  $\G_{\epsilon,S'}(u',\theta',v',Y')$ are equal, and note that the projections of $\Gext_{\epsilon,S}(u,\theta,v,Y)$ and $\Gext_{\epsilon,S'}(u',\theta',v',Y')$ to $S^1$ are respectively $\frac{\tilde{\epsilon}}{8c_p}$-close to $\theta$ and $\theta'$.

Expressing $\G_{\epsilon,S}(u,\theta,v,Y)$ and $\G_{\epsilon,S'}(u',\theta',v',Y')$ as the composition of this small rotation with the exponential of a vector field along the pre-gluing map, we find that  
\begin{equation}  \tilde{d} ( \exp_{w}(Y \circ r_{-\theta} + \sol(Y)\circ r_{-\theta}), \exp_{w'}(Y'\circ r_{-\theta'} + \sol(Y')\circ r_{-\theta'} )) \leq \frac{\tilde{\epsilon}}{4 c_p} \leq \frac{\tilde{\epsilon}}{4}   \end{equation}
where  $w=u \#_{S} v  \circ \phi_{S} \circ r_{-\theta}$ and $w' = u' \#_{S'} v' \circ \phi_{S'} \circ r_{-\theta'}$.  Since $|Y+ \sol(Y)|_{C^0} \leq  \frac{\tilde{\epsilon}}{4}$ (similarly for $Y'$), we conclude that 
\begin{equation}  \tilde{d}( \preG_{S}(u,\theta,v),  \preG_{S'}(u',\theta',v')) \leq \tilde{\epsilon}(\rho) ,\end{equation}
which implies the desired result by the previous Lemma.
\end{proof}

Consider the smooth manifolds with boundary
\begin{equation}  \cP(L;0) \cup \codt{2}{L} \textrm{ and } \codb{1}{L} \equiv \cod{1}{L} \cup \codt{3}{L}  \end{equation}
obtained respectively  by gluing $\codt{2}{L} \times (S, + \infty]$ to $\cP(L;0)$ along $\Gcod{2}{0}$, and $\codt{3}{L} \times (S, + \infty]$ to $\cod{1}{L}$ along $\Gcod{3}{1}$.   Combining Lemmas \ref{lem:codim_3_to_parametrized_space_diffeo_gluing},  \ref{lem:gluing_commute_codim_3_codim_2_top_stratum}, and \ref{lem:commutative_gluing_codim_3_to_codim_1} we obtain a map
\begin{equation}\codb{1}{L} \times [S, +\infty)  \to   \cP(L;0) \cup \codt{2}{L}  \end{equation}
whose restriction to each slice $\codb{1}{L} \times \{ S_1 \} $ is an immersion which is an embedding near the boundary of $\codb{1}{L}$, and is moreover transverse to the boundary of $ \cP(L;0) \cup \codt{2}{L} $.

\begin{lem} \label{lem:existence_proper_map}
There exists a proper continuous map 
\begin{equation}  F^{\rho} \co  \cP(L;0) \cup \codt{2}{L}  \to [0,+ \infty) \end{equation} and a constant $S$ sufficiently large such that the diagram
\begin{equation}
 \xymatrix{ \preGext^{*} \ker_{\epsilon} D_{\cP_{0,1}} \times [S, +\infty) \ar[r] \ar[d] &\cP(L;0) \ar[d] \\
[S, +\infty) \ar[r] & [0,+\infty] } \end{equation}
commutes up to the universal constant $2 \rho$ for $S$ sufficiently large.
\end{lem}
\begin{proof}
By Lemmas \ref{lem:codim_3_to_parametrized_space_diffeo_gluing} and \ref{lem:commutative_gluing_codim_3_to_codim_1}, the restriction of $\G$ to the image of $\Gcod{3}{1}$ is  a diffeomorphism onto a neighbourhood of $\cod{3}{L}$ in $\cP(L;0)$.  We define $F^{\rho}$ in this neighbourhood to agree with the gluing parameter, and we let it vanish identically away from the image of $\G_{\epsilon}$.  Our main task is to extend this map to the remainder of $\cP(L;0)$.

Given $w \in \Im(\G_{\epsilon})$ we consider a sequence of compact neighbourhoods with non-empty interior whose intersection is $w$.  The inverse images are compact subsets of $\preGext^{*} \ker_{\epsilon} \times [S, +\infty)$ whose intersection is $\G_{\epsilon}^{-1}(w)$.  By Lemma \ref{cor:inverse_image_extended_gluing_small}, the diameter of $\G_{\epsilon}^{-1}(w)$ is bounded by $\rho$, so we may choose a neighbourhood of $w$ whose inverse image has diameter  bounded by $2\rho$.   By fixing a triangulation of $\cP(L;0)$ which is subordinate to this cover, it is easy to construct the desired function (simply by averaging over simplices).  We omit the details but the reader may consult the proof of the next Lemma, and adapt it here. 
\end{proof}

In particular, the map $\G$ has a well-defined degree at infinity; a path between points in ${F^{\rho}}^{-1}([S_1,+\infty)$ does not intersect the image of the boundary of $ \cod{1}{L} \times [S, +\infty)$ if $S_1 \geq S + 2 \rho$.

\begin{lem}
The gluing map $\G$ has degree $1$ at infinity, in particular it is a surjection onto a neighbourhood of $\cod{1}{L}$ in $\cP(L;0)$. 
\end{lem}
\begin{proof}
We shall construct a map from a neighbourhood of $\cod{1}{L}$ in $\cP(L;0)$ to the product $\cod{1}{L} \times [S, +\infty)$ which will serve as a homotopy left inverse to the gluing map $\G$.  Pick a triangulation of $\cP(L;0)$ as in the previous Lemma such all simplices have inverse images under $ \G_{\epsilon} $  whose diameter is smaller than $2 \rho$.

Construct a map $I$ from the image of $\Gext_{\epsilon}$ to $\cod{1}{L}  \times [S_1, +\infty)$ as follows:  In a neighbourhood of $\cod{3}{L}$, this map is the honest inverse to $\G$.  Away from $\cod{3}{L}$, pick an arbitrary inverse image in $\preGext^{*}  \ker_{\epsilon} D_{\cP_{0,1}}$ for each vertex of the chosen triangulation, and let the image of a vertex be the projection of this pre-image to $\cod{1}{L} \times [S, +\infty)$.   We extend this map to a cell of the triangulation by observing that the image of all vertices is contained in a geodesically convex ball, and using geodesics to define an analogue of linear interpolation.

To prove that $I$ is indeed a homotopy left inverse to $\G$, note that the composition  $I \circ \G$ maps every point to a point which is at most $ 2 \rho$ away. In particular, every point and its image are connected by a unique shortest geodesic.  Following this geodesic defines the desired isotopy. 
\end{proof}

\subsection{Construction of $\Phat(L;0)$} \label{sec:construction_manifold_corner}
Returning to the proof of Lemma \ref{cor:inverse_image_extended_gluing_small} in the case where $Y$ vanishes, we find that since $\sol_{S}(0)$ is bounded above in norm by a constant multiple of $ e^{-2(1-\delta)S}$, the assumption that \begin{equation} \label{eq:images_gluing_agree}  \G_{S}(u,\theta,v) = \G_{S'}(u,\theta',v'),  \end{equation}
implies that the $\tilde{d}$-distance between $\preGext(u,\theta,v,S)$ and $\preGext(u,\theta',v',S')$ is also bounded  by a constant multiple of $ e^{-2(1-\delta)S} +  e^{-2(1-\delta)S'}$.  Applying Lemma \ref{lem:diameter_inverse_image_decays}, we conclude: 
\begin{lem} \label{lem:gluing_better at preserving_distances} For any $\rho$, we may choose $S$ large enough so that Equation \eqref{eq:images_gluing_agree} implies that  $(u,\theta,v,S) $ and $(u,\theta',v',S')$ are within $\rho$ of each other. \noproof
\end{lem} 
Since the distance in $\cod{1}{L} \times [S,+\infty)$ is bounded below by the difference between the gluing parameters, we conclude
\begin{cor} \label{cor:surfaces_disjoint_large_S}
If a sequence $S_i$ converges to infinity, and
\begin{equation} \Im \G_{S_i} \cap \G_{S'_i} \neq \emptyset \end{equation}
then $\lim_{i \to +\infty}|S_i - S'_i| = 0$. \noproof
\end{cor}
We can also specialize Lemma \ref{lem:gluing_better at preserving_distances} to a fixed $S$.  For $S$ large enough, it implies that no two points that are a bounded distance away from each other can have the same image.  On the other hand, Proposition \ref{prop:local_injectivity_gluing} implies that $\G_{S}$ is injective in sufficiently small balls in $\cod{1}{L}$.  We conclude:
\begin{lem}  \label{lem:gluing_map_S_embedding}
If $S$ is large enough, the gluing map $\G_{S}$ is a smooth codimension $1$ embedding. Moreover, for each $\rho$, we may choose $S$ large enough so that the inverse images of points under $\G| [S,+\infty)$ are bounded in diameter by $\rho$.\noproof
\end{lem}

The image of $\G_{S}$ bounds a compact submanifold with corners included as a codimension $0$ submanifold of $\cP(L;0) \cup \codt{2}{L}$ and can therefore serve as a model of $\Phat(L;0)$.  However, since we have not shown the gluing map to be smooth in the direction of the gluing parameter, we shall perform a slightly more complicated construction in order to obtain appropriate control of the normal bundle of the boundary of  $\Phat(L;0)$.  Using Assumption \eqref{eq:right_inverse_constraint_basepoints}, we can recover the gluing parameter from the intersection point with  $\N_\star$, as in Section \ref{sec:codim_3_to_0_gluing}:
\begin{lem}
There exists a smooth map
\begin{equation} \cP(L;0) \to \bR \end{equation}
whose restriction to the image of $U_{\star} \times [S,+\infty)$ is the inverse to the gluing parameter. \noproof
\end{lem}

In particular, the images of $U_{\star}$ under the gluing map for different values of the gluing parameter are disjoint.  In addition:

\begin{lem} \label{lem:injectivity_near_basepoint} There exists an open subset $V_\star \subset U_\star $ which intersects each component of $\cod{1}{L}$ non-trivially such that
\begin{equation} \G^{-1}( \G_{S}(u,\theta,v)) = (u, \theta,v,S) \end{equation}
whenever $S$ is large enough and $ (u, \theta,v) \in V_\star $. \end{lem}
\begin{proof}
Since each map $\G_{S}$ is injective, it suffices to show that upon shrinking $U_\star$, we can avoid the images of points in the complement of the region where Condition \eqref{eq:right_inverse_constraint_basepoints} holds.  This is an immediate consequence of the second part of Lemma \ref{lem:gluing_map_S_embedding}.
\end{proof}

This allows us to refine the function $F^{\rho}$, and hence establish the main result of this section:
\begin{proof}[Proof of Proposition \ref{prop:good_approximate_inverse_gluing}]
Corollary \ref{cor:surfaces_disjoint_large_S}, implies that we may pick a sequence $S^i$ of positive real numbers such that 
\begin{equation} \lim_{i \to +\infty} S^i = + \infty  \textrm{ and } \Im \G_{S} \cap \Im \G_{S^i} = \emptyset  \text{ whenever $ S \geq S^{i+1}$.} \end{equation}
 This implies that $\Im \G_{S^{i+1}}$ separates $\Im \G_{S^{i}}$ from the end.  In particular, a neighbourhood of $\codb{1}{L}$ can be written as union of cobordisms between $\Im \G_{S^{i}}$ and $\Im \G_{S^{i+1}}$. We set the function $F$ to agree with $S^{i}$ on the image of the map $\G_{S_i}$, and with the value of the gluing parameter on the image of $\G| V_\star \times [S, +\infty) $.  Lemma \ref{lem:injectivity_near_basepoint} implies that these two conditions are compatible.   Trivializing a neighbourhood of  $\Im \G_{S^{i}}$ using a map to $\bR$ which agrees with the $[S,+\infty)$ coordinate near $\G_{S_i}\left( V_\star \right)$, we obtain collar neighbourhoods of $\Im \G_{S^{i}}$ in the two cobordism for which it appears as a boundary, so we can extend $F$ linearly with respect to the normal direction.  Any extension of $F$ to a function on the interior of the cobordism which agrees with this linear function near the boundary, agrees with the gluing parameter on $\G \left(  V_\star \times [S^i, S^{i+1}) \right)$, and takes values in the open interval $(S^i, S^{i+1})$ will satisfy the desired conclusion. 
\end{proof}

\section{Triviality of the tangent space of the cobordism} \label{sec:triviality_tangent_space}
In this section, we prove Lemma \ref{lem:iso_stabilizations_bundles}.  We first explain the construction of a compact CW pair $\left( \cX(L), \partial \cX(L) \right)$ with an inclusion of $(\Chat(L), \codbC{1}{L})$, and an extension of the gluing map $\G^{\cC}$ to $ \partial \cX(L) $.  In the last part of this section, we construct the CW complex $\cY(L)$.

\subsection{Constructing a finite-dimensional replacement for $\cF_{\cC}(L)$ } 

We begin by recalling that the projection map
\begin{equation} \label{eq:project_to_exceptional_solutions} \pi \co \partial \cF_{\cC}(L) \to \cP_{0,1}(L;-\beta) \end{equation} is a weak homotopy equivalence. The choice of a section of 
\begin{equation} \codt{3}{L} \to \cP_{0,1}(L;-\beta) \cong  \cP(L;-\beta) \times S^1  \end{equation}
together with the gluing map $\Gcod{3}{1}_S$ for a sufficiently large $S$ gives an inclusion
\begin{equation} \label{eq:gluing_include_exceptional_solutions} \iota \co \cP_{0,1}(L;-\beta) \to   \partial \cC(L) \to \partial \cF_{\cC}(L) \end{equation}
such that the composite with the projection map \eqref{eq:project_to_exceptional_solutions} is the identity.  Moreover, if we consider the closure of the codimension $1$ stratum $ \codbC{1}{L} \subset  \Chat(L) $ (see the paragraph following Lemma \ref{lem:compactification_capping_moduli_spaces_corners}), then the choice of such a gluing parameter also determines up to a contractible choice an inclusion
\begin{equation} \codbC{1}{L} \to \partial \cC(L)  \end{equation}
which is the identity away from a collar of the boundary.

We now fix a triangulation of $ \codbC{1}{L} $.  Given a vertex $v \in  \codbC{1}{L}$,  pick a path in 
\begin{equation}  \pi^{-1}(\pi(v)) \subset \partial \cF_{\cC}(L) \end{equation}
 connecting  $v$ to $\iota \pi(v)$. Note that this is possible because the fibres of $\pi$ are connected (in fact, they have the weak homotopy type of a point).  Next, we observe that every edge in our triangulation of $\codbC{1}{L}$ gives rise to a loop in $ \partial \cF_{\cC}(L) $ given as the concatenation of the given edge with the paths running from its endpoints to $ \iota \left( \cP_{0,1}(L;-\beta) \right)$, and with the image of the edge under the composition of the projection \eqref{eq:project_to_exceptional_solutions}  with the section \eqref{eq:gluing_include_exceptional_solutions}.  The image of this loop in $\cP_{0,1}(L;-\beta)$  is the composition of a path with its inverse, hence obviously contractible, so we can extend the loop in $ \partial \cF_{\cC}(L) $  from $S^1 = \partial D^2$ to the interior of the disc.  Note that we've now constructed a $2$-dimensional CW complex which contains the $1$-skeleton of $ \codbC{1}{L}$ and which retracts to $\cP_{0,1}(L;-\beta)$.

Proceeding inductively, and using at every step the fact that $\pi$ is a weak homotopy equivalence, we obtain a CW complex that we will denote $\partial \cX(L; \beta)$ which still  retracts to $\cP_{0,1}(L;-\beta)$, and contains all of $\codbC{1}{L} $ More precisely, 
\begin{lem} \label{lem:homeo_type_X}
$\partial \cX(L; \beta)$ is homeomorphic to the cone on the projection map
\begin{equation} \pi \co \codbC{1}{L}  \to \cP_{0,1}(L;-\beta) ,\end{equation} 
and is equipped with a map
\begin{equation} \iota_{\partial \cX} \co \partial \cX(L; \beta) \to \partial \cF_{\cC}(L) \end{equation}
which is a weak homotopy equivalence.
\noproof
\end{lem}

Using the fact that the projection map
\begin{equation} \cF_{\cC}(L) \to \cP(L;-\beta) \times D^2  \end{equation}
is also a weak homotopy equivalence, the same procedure constructs a CW complex $\cX(L; \beta)$ containing $\partial \cX(L; \beta)$ as a subcomplex, together with a diagram
 \begin{equation} \xymatrix{\left( \cP(L;-\beta)\times D^{2} ,  \cP(L;-\beta)\times  S^{1} \right) \ar@{^(->}[rd]^{\cong} \ar@{^(->}[r] & \left(\Chat(L), \codbC{1}{L} \right) \ar[r]   & \left( \cX(L) ,  \partial \cX(L) \right) \ar[ld]^{\cong}_{i_{\cX}} \\
& \left( \cF_{\cC}(L),  \partial \cF_{\cC}(L) \right) &   }  \end{equation}
where the top row consists of inclusions of pairs of CW complexes.  Recall that $\cF_{\cC}(L)$  carries a Banach bundles $\cE_{\cC}$ whose fibres are smooth anti-holomorphic $1$-forms on the disc with values in $TM$. 

\begin{lem}
There exists a finite dimensional trivial vector bundle $V_{\cX}$ over $\cX(L)$ and a map of vector bundles
\begin{equation} \label{eq:preturbations_to_surject_CW} V_{\cX} \to i_{\cX}^{*} \cE_{\cC} \end{equation}
whose image consists of $1$-forms supported away from $-1$, and whose direct sums with $D_{\cC}$ and $D_{\partial \cC}$ give surjective maps of bundles
\begin{align} \label{eq:surjection_dbar_CW_smooth} D_{\cC}^{\cX} \co V_{\cX} \oplus i_{\cX}^{*} T \cF_{\cC}(L) & \to  i_{\cX}^{*} \cE_{\cC} \\
D_{\partial \cC}^{\cX} \co V_{\cX} \oplus i_{\partial \cX}^{*} T \partial  \cF_{\cC}(L) & \to  i_{\cX}^{*} \cE_{\cC}.
 \end{align}
 Moreover, we may assume that the map \eqref{eq:preturbations_to_surject_CW} vanishes on $\Chat(L) \times \{0 \}  \subset \cX(L) $ .
\end{lem}
\begin{proof}
By a standard patching argument, it suffices to prove the result in a neighbourhood of a single point.  Given a smooth map $v \co (D^2, S^1) \to (M,L)$, the fact that the Cauchy-Riemann operator is elliptic implies that the cokernel of
\begin{equation}  D_{\dbar} \co C^{\infty} ( (D^2, S^1), (v^{*} TM, v^{*} TL)) \to C^{\infty}(v^{*} TM \otimes \Omega^{0,1} D^2) \end{equation}
is finite dimensional.  Starting with any finite-dimensional subspace $V_\infty$ of $C^{\infty}(v^{*} TM \otimes \Omega^{0,1} D^2) $ which is transverse to the image of $D_{\dbar}$, we pass to our favourite Sobolev space completion and consider the surjective map with bounded inverse
\begin{equation} \label{eq:surjection_using_smooth_functions} V_\infty  \oplus  W^{1,p,\delta}( (D^2, S^1), (v^{*} TM, v^{*} TL))  \to L^{p,\delta} ( v^{*} TM \otimes \Omega^{0,1} D^2) . \end{equation}
Since elements of $V_\infty$ are smooth on $D^2$, they decay like $e^{-s}$ on the strip, so, for sufficiently large $S$, the difference between an element of $V_0$ and its product with a cutoff function that vanishes when $s \geq S$ has arbitrarily small $| \_|_{p,\delta}$-norm.   In particular, starting with a right inverse to \eqref{eq:surjection_using_smooth_functions}, we can construct a right inverse to  
\begin{equation} V_S \oplus  W^{1,p,\delta}( (D^2, S^1), (v^{*} TM, v^{*} TL))  \to L^{p,\delta} ( v^{*} TM \otimes \Omega^{0,1} D^2),  \end{equation}
where $V_{S}$ stands for those $1$-forms obtained by cutting off $V_\infty$ at $S$.  It follows immediately that $V_{S}$ surjects onto the cokernel of $D_{\dbar}$ in a neighbourhood of $v$ as well.
\end{proof}

\subsection{Pre-gluing smooth maps} \label{sec:pre-gluing_smooth}

Recall that for each $S$ sufficiently large, we constructed a manifold with corners $ \Phat^{S}(L;0)$ in the previous section together with an inclusion 
\begin{equation} \codbC{1}{L} \cong \codb{1}{L} \to \partial \Phat^{S}(L;0)\end{equation}
as a union of the closure of some top boundary strata.  Moreover, this map factors through  an inclusion
\begin{equation} \label{eq:gluing_map_into_smooth_functions} \codbC{1}{L}  \to \cF_{\cP_{0,1}}(L:0) \cong S^1 \times \cF_{\cP}(L;0),  \end{equation}
which we shall extend to $\partial \cX(L; \beta)$.

We first define a map 
\begin{equation} \preGext^{\partial \cX}_{S} \co \partial \cX(L) \to  \cF_{\cP_{0,1}}(L;0)\end{equation}
for each $S$ sufficiently large by pre-gluing, as in Section \ref{sec:pre-gluing_maps}, the chosen map in $\partial \cX(L)$ with the exceptional solution whose boundary it intersects.  For convenience, we require that the restriction of the map $\iota_{\cX}$ defined in Lemma \ref{lem:homeo_type_X} to \begin{equation} \codbC{1}{L} \times [0,1/4)  \subset \partial \cX(L) .\end{equation} factor through the projection to $\codbC{1}{L}$.  In particular, the same property holds for  $ \preGext^{\partial \cX}_{S}$.

Next, we interpret the vector field $\sol_{S}$ introduced in Proposition \ref{prop:statement_implicit_function_theorem} as a vector field valued in 
\begin{equation} {\preGext^{\partial \cX}_{S}}^{*} T \cF_{\cP_{0,1}}(L;0)| \codbC{1}{L} \times \{0\}  .\end{equation} 
Note that this is possible since we've chosen an inclusion of $\codbC{1}{L}$ as a compact subset of $\partial \cC(L)$, which we might as well assume is contained in the subset $K$ discussed in Proposition \ref{prop:statement_implicit_function_theorem} since this subset was an arbitrary compact subset.  By linear interpolation, we extend this vector field to $\codbC{1}{L} \times [0,1/4] $ so that it vanishes on  $\codbC{1}{L}  \times \{ 1/4 \} $, which allows us to further extend $\sol_{S}$  to a section 
\begin{equation} \sol_{S}^{\partial \cX} \co \partial \cX(L) \to  \preGext^{\partial \cX}_{S} T \cF_{\cP_{0,1}}(L;0)   \end{equation}
which vanishes away from the aforementioned neighbourhood of $\codbC{1}{L} $.  Exponentiating $\sol_{S}^{\partial \cX}$ defines the desired map
\begin{equation} \Gext_{S}^{\partial \cX} \equiv \exp_{\preGext^{\partial \cX}_{S}}\left(\sol_{S}^{\partial \cX}\right) \co \partial \cX(L) \to \cF_{\cP_{0,1}}(L;0).\end{equation}

Note that the pre-gluing map of vector fields defined in Equation \eqref{eq:definition_pre-gluing_vector_fields} extends to this setting to give a map
\begin{equation} \predGext^{\partial \cX}_{S} \co \iota_{\partial \cX}^{*} T^{ext}  \partial \cF^{1,p,\delta}_{\cC}(L;\beta) \to \preGext^{\partial \cX}_{S} T \cF^{1,p}_{\cP_{0,1}}(L;0),  \end{equation} 
which is uniformly bounded by the argument given in Lemma \ref{lem:pre-gluing_tangent_vectors_bounded}.

As in Section \ref{sec:right_inverse}, we can construct a right inverse to the $\dbar$-operator
\begin{lem}
If the gluing parameter $S$ is sufficiently large, then there exists a map
\begin{equation} \label{eq:extend_perturbations_pre-gluing} V_{\cX}| \partial \cX \to   { \Gext^{\partial \cX} }^{*} \cE^{p,\delta}_{\cP_{0,1}} ,\end{equation} 
whose restriction to $\cod{1}{L}$ vanishes, and such that its sum with $D_{\cP_{0,1}}$ is a surjective operator 
\begin{equation} \label{eq:surjective_X_dbar_operator_gluing} D_{\cP_{0,1}}^{\partial \cX} \co V_{\cX} \oplus  {\Gext^{\partial \cX}}^{*} T \cF^{1,p,\delta}_{\cP_{0,1}}(L;0) \to {\Gext^{\partial \cX}}^{*} \cE^{p,\delta}_{\cP_{0,1}} .\end{equation}
Moreover, as the gluing parameter $S$ converges to  $+ \infty$, we may choose a uniformly bounded right inverse  $Q_{\Gext^{\partial \cX}}$ to this operator.
\end{lem}
\begin{proof}
For simplicity, we only explain the construction in $\partial \cX(L)  - \left( \partial \cC(L) \times [0,1/4) \right)$, where $\Gext^{\partial \cX}$ agrees with the pre-gluing map $\preGext^{\partial \cX}$.  To extend the construction to $ \codbC{1}{L} \times [0,1/4)$, we repeat all arguments, inserting parallel transport from $\preGext^{\partial \cX}$ to $\Gext^{\partial \cX}$.  Since parallel transport will take place only in a neighbourhood of the compact set $\codbC{1}{L}$, Proposition \ref{prop:quadratic_inequality} can be applied to bound all error terms that it produces.

Let us fix a right inverse $Q_{\partial \cX}$ to the extension  of \eqref{eq:surjection_dbar_CW_smooth} to the Banach spaces
\begin{equation} V_{\cX} \oplus \iota_{\partial \cX}^{*} T^{ext}  \partial \cF^{1,p,\delta}_{\cC}(L;\beta) \to  i_{\partial \cX}^{*} \cE^{ext, p,\delta}_{\cC},\end{equation}
where we're using the notation introduced in Equations \eqref{eq:extended_CR} and \eqref{eq:extended_CR-boundary}.

First, we define the map \eqref{eq:extend_perturbations_pre-gluing} as a composition
\begin{equation}V_{\cX} \to  i_{\partial \cX}^{*} \cE_{\cC} \to {\Gext^{\partial \cX}}^{*} \cE^{p,\delta}_{\cP_{0,1}}\end{equation}
where the second arrow is defined by pre-gluing of $1$-forms.  This is made particularly easy by the fact that the image of $V_{\cX} $ in $ i_{\partial \cX}^{*} \cE_{\cC}$ consists of compactly supported forms, so the pre-gluing map is simply an extension by $0$ with no cutoff needed.

Next, we define an approximate right inverse $\tilde{Q}_{\Gext^{\partial \cX}}$ again as a composition
\begin{multline} {\Gext^{\partial \cX}}^{*} \cE^{p,\delta}_{[0,+\infty)} \to i_{\partial \cX}^{*} \cE^{ext, p,\delta}_{\cC} \stackrel{Q_{\partial \cX}}{\longrightarrow} \\
V_{\cX} \oplus i_{\partial \cX}^{*} T^{ext}  \partial \cF^{1,p,\delta}_{\cC}(L;\beta) \to V_{\cX} \oplus {\Gext^{\partial \cX}}^{*} T \cF^{1,p,\delta}_{\cP_{0,1}}(L;0) \end{multline}
where the first map is the breaking map of \eqref{eq:breaking_1_forms_holomorphic}, and the last map is the direct sum of the identity on $V_{\cX}$ with the pre-gluing map of vector fields $\predGext^{\partial \cX}_{S} $.

Lemma \ref{lem:approximate_right_inverse} didn't require holomorphicity of the curves (just smoothness), so that the norm of the operator 
\begin{equation}\lim_{S \to +\infty} \left\| \id - D^{\partial \cX}_{\cP_{0,1}} \circ \tilde{Q}_{\Gext^{\partial \cX}_{S}}\right\|  \to 0 \end{equation}
converges to $0$ as $S$ grows.  We've made things particularly simple by requiring that  the image of the map $V_{\cX} \to  i_{\partial \cX}^{*} \cE_{\cC}$  consist of compactly supported $1$-forms.  In particular, choosing $S$ large enough so that their support is disjoint from $\xi_{1}(\I{S}{+\infty})$, we know that $\tilde{Q}_{\Gext^{\partial \cX}}$ is an honest inverse to $D^{\partial \cX}_{\cP_{0,1}}$ on the image of $V_{\cX}$.  We can then construct a right inverse as in Corollary \ref{cor:uniformly_bounded_inverse}.
\end{proof}

The next step in the proof of Lemma \ref{lem:iso_stabilizations_bundles} is the construction of the map between kernels of operators asserted in Equation \eqref{eq:map_kernels}.    

\begin{cor} \label{cor:injective_gluing_kernel}
If $S$ is sufficiently large, the composition of pre-gluing of tangent vector fields, parallel transport from $\preGext_{S}^{\partial \cX}$ to $\Gext_{S}^{\partial \cX}$, and the projection to the kernel of $D_{\cP_{0,1}}^{\partial \cX}$ along the inverse $Q_{\Gext^{\partial \cX}}$
\begin{equation} \label{eq:iso_on_kernels} \left( \id  - Q_{\Gext^{\partial \cX}} \circ D_{\cP_{0,1}}^{\partial \cX} \right) \circ \Pi_{\preGext^{\partial \cX}}^{\Gext^{\partial \cX}} \circ \predGext^{\partial \cX}  \end{equation}
restricts to an isomorphism between the kernels
\begin{equation} \label{eq:iso_on_kernels_expression_extended} \mathring{\Psi}_{\partial \cX} \co \ker D_{\partial \cC}^{\partial \cX} \to  \ker D_{\cP_{0,1}}^{\partial \cX}  \end{equation}
which is uniformly bi-lipschitz.
\end{cor}
\begin{proof}
It is not hard to check, using Lemma \ref{lem:pre-gluing_uniform_bilipschitz} and the estimates stated in Appendix \ref{ap:pointwise_estimates}, that the restriction of $\Pi_{\preGext_{S}^{\partial \cX}}^{\Gext_{S}^{\partial \cX}}  \circ \predGext_{S}^{\partial \cX} $ to $\ker D_{\partial \cC}^{\partial \cX}$  is an injection which distorts the norm by an arbitrarily small amount if $S$ is sufficiently large.  The arguments given to justify Lemma \ref{lem:tangent_pre-gluing_almost_holomorphic} and the quadratic inequality of Proposition \ref{prop:quadratic_inequality} show that the norm of 
\begin{equation}  D_{\cP_{0,1}}^{\partial \cX} \circ \Pi_{\preGext^{\partial \cX}}^{\Gext^{\partial \cX}} \circ \predGext^{\partial \cX}  \end{equation} decays with $S$. Using in addition the fact that the norm of $ Q_{\Gext^{\partial \cX}} $ is uniformly bounded,  we conclude that \eqref{eq:iso_on_kernels} is arbitrarily close to $\Pi_{\preGext^{\partial \cX}}^{\Gext^{\partial \cX}} \circ \predGext^{\partial \cX}$  as $S \to +\infty$, and hence it is also an injection.  Comparing ranks implies that it is an isomorphism.
\end{proof}
\begin{rem}
From now on, we shall not have to change the gluing parameter $S$ anymore in our construction of $\Phat^{S}(L;0)$, and hence of the manifold $\hat{W}(L)$.
\end{rem}

The projection map \eqref{eq:iso_on_kernels} is the key step in the construction of the stabilization of $T\hat{W}(L)$ which we shall use to prove triviality of the tangent space of $\hat{W}(L)$.  Upon restricting \eqref{eq:iso_on_kernels} to $\codbC{1}{L}$, we need to control the image of the subspace $\aut(D^2,-1)$.  Let $(u,\theta,v)$ denote an element of $ \partial \cC(L) $ in the image of the previously fixed inclusion $\codbC{1}{L} \to \partial \cC(L)$, and let $w^{\flat}_{S}$ denote its image under $\Gext_{S}$.

Assume that $(u,\theta,v)$ lies in the open set $U_{\star}$ of Equation \eqref{eq:right_inverse_constraint_basepoints}.  Even though we do not know that the path $\G_{S}(u,\theta,v)$ is smooth as $S$ varies, Lemma \ref{prop:good_approximate_inverse_gluing} implies that the derivative of $F$ with respect to $S$ along the path $\G_{S}(u,\theta,v)$ is equal to $1$.   In particular, this path is obviously outwards pointing; we shall show that it has, in an appropriate sense,  an approximate derivative in directions normal to $N_{\star}$, and that $\mathring{\Psi}_{\partial \cX} \left(dv \left( \partial_{s}  \right)\right)$  is close to this derivative, which will imply that it is  also outwards pointing.

Fix a map $\pi_{\star}$ from a neighbourhood of $v(0)$  in $\bC^{n} \times \bC \bP^{n-1}$ to $D^2$ whose fibre at $0$ is the intersection of this neighbourhood with $\N_{\star}$.  Since the tangent space of $F^{-1}(S) = \Im \G_{S}$ consists of vector fields whose value at $z_{0}(S)$ lies in $\N_{\star}$, the kernel of $dF$ and $d \pi_{\star}$ agree:
\begin{lem} \label{eq:factor_inverse_gluing_parameter_evaluation_at_hyperplane} The differential of $F$ factors through the composition of evaluation at $z_{0}(S)$ with the differential of $ \pi_{\star}$. \noproof \end{lem}
We must now identify a tangent vector to $D^2$ at the origin which yields an outwards pointing vector upon gluing.  The reader may easily check that the map
\begin{equation} \label{eq:derivative_gluing_at_a_point_candidate} d \pi_{\star} \circ d  \Gext_{S}(u,\theta,v)\left( \partial_{s} \right)|_{z_{0}(S)} \end{equation}
where $\partial_{s} $ is the vector field along $D^2$ which is the infinitesimal translation on  $D^2 - \{ \pm 1\}$ under the identification with strip $\I{-\infty}{+\infty}$ coming from our choices of strip-like ends, serves as an approximate derivative at $s=0$ to 
\begin{align}  U_{\star} \times [0,+\infty) \times (-1,1) & \to D^2 \\
\left( (u,\theta,v), S, s \right) & \mapsto \pi_{\star} \left(   \Gext_{S+s}(u,\theta,v) \left( z_{0}(S) \right) \right) \end{align}
if $S$ is sufficiently large.  The precise statement which can be proved using Gromov compactness is that
\begin{equation} \label{eq:derivative_gluing_at_a_point}  \lim_{S \to +\infty} \limsup_{s \to 0}  \frac{\left| \pi_{\star} \left(   \Gext_{S+s}(u,\theta,v) \left( z_{0}(S) \right) \right) - d \pi_{\star} \circ d  \Gext_{S}(u,\theta,v)\left( \partial_{s} \right)|_{z_{0}(S)} \right|}{s} = 0 .\end{equation}
We conclude that \eqref{eq:derivative_gluing_at_a_point_candidate} is arbitrarily close to an outwards pointing vector for $S$ sufficiently large.

\begin{lem} \label{lem:gluing_extra_vectors_close_to_where_they_should_be}
If $S$ is sufficiently large, the image of $dv \left( \partial_{s} \right)$ under the composition of $ \mathring{\Psi}_{\partial \cX} $ with the projection
\begin{equation} \label{eq:projection_forget_marked_point_moduli_space} T \cP_{0,1}(L;0) \to T \cP(L;0) \end{equation}
is a vector which points outwards along the boundary of $\Phat^{S}(L;0)$.  The image of $dv\left(\partial_{p}\right)$ under $ \mathring{\Psi}_{\partial \cX} $ is  bounded below in norm by a uniform constant independent of $S$ and is arbitrarily close in the $L^p$-norm to a positive multiple of the generator $dw^{\flat}_{S}\left(\partial_{\theta}\right)$ of the kernel of \eqref{eq:projection_forget_marked_point_moduli_space}.
\end{lem}
\begin{proof}
The statement about the image of $dv \left( \partial_p \right)$ follows from the results proved in Section \ref{sec:injectivity_derivative}, in particular, Lemmas \ref{lem:pre-gluing_uniform_bilipschitz} and \ref{lem:pre-gluing_parabolic_close_to_rotation}, so it suffices to prove the first statement.  We also know from Section \ref{sec:injectivity_derivative}, in particular Lemmas \ref{lem:pre-gluing_uniform_bilipschitz}, and \ref{lem:image_differential_gluing_transverse_circle}, that the projection of the image of $dv \left( \partial_{s} \right)$ under the compositions of \eqref{eq:iso_on_kernels} and \eqref{eq:projection_forget_marked_point_moduli_space}  is transverse to the boundary of $\Phat^{S}(L;0)$ at every point.   Therefore, it suffices to prove that on each component  of $\codbC{1}{L}$, there is some point for which it is outwards pointing; we choose this point to lie in $U_\star$.  By Equation \eqref{eq:derivative_gluing_at_a_point} and Lemma \ref{eq:factor_inverse_gluing_parameter_evaluation_at_hyperplane}, we simply have to prove that the image of the vector field $ \mathring{\Psi}_{\partial \cX}  \left( dv \left( \partial_{s} \right) \right)$ evaluated at $z_0(S)$ under the differential of $\pi_{\star} $ is sufficiently close to $d \pi_{\star} \circ d  \Gext_{S}(u,\theta,v)\left( \partial_{s} \right)|_{z_{0}(S)} $.

Gromov compactness implies that the norm of \eqref{eq:derivative_gluing_at_a_point} is uniformly bounded from below.  More precisely,
\begin{equation}  \lim_{S \to +\infty} \Big{|} d  \pi_{\star} \circ d \Gext_{S}(u,\theta,v) \left( \partial_{s} \right)|{z_{0}(S)} -   d \pi_{\star} \circ   \mathring{\Psi}_{\partial \cX}  \left( dv \left( \partial_{s} \right) \right) |{z_{0}(S)}  \Big{|} = 0, \end{equation}
and the norm of $d\pi_{\star} \circ \Pi_{\preGext}^{\Gext} \circ \predGext_{S}\left(dv \left( \partial_{s} \right) \right) $ is independent of $S$ (a priori, there might be an error term coming from parallel transport in the $\theta$ and $R$ directions, but our choice of metric is independent of these two variables in this region).  It suffices therefore to prove that
\begin{equation}  d \pi_{\star} \circ \mathring{\Psi}_{\partial \cX} \left( dv \left( \partial_{s} \right) \right)  | z_{0}(S) \end{equation}
is close to $ d\pi_{\star} \circ \Pi_{\preGext}^{\Gext} \circ \predGext_{S} \left(dv \left( \partial_{s} \right)  \right) |{z_{0}(S)} $, which is the same as bounding the norm of
\begin{equation} d\pi_{\star} \circ Q_{\Gext^{\partial \cX}} \circ D_{\cP_{0,1}}^{\partial \cX} \circ \Pi_{\preGext^{\partial \cX}}^{\Gext^{\partial \cX}} \circ \predGext^{\partial \cX} \end{equation} 
applied to $ dv \left( \partial_{s} \right)  $, and evaluated at $z_0(S)$.  However, the composition
\begin{equation} d\pi_{\star} \circ Q_{\Gext^{\partial \cX}} \end{equation}
vanishes since $Q_{\Gext^{\partial \cX}}$ takes values in vector fields which, at $z_0(S)$, lie in the tangent space of $\N_{\star}$.  
\end{proof}

\subsection{Construction of the isomorphism of kernels}
The isomorphism asserted in Equation \eqref{eq:map_kernels} is constructed starting with the isomorphism defined in Equation \eqref{eq:iso_on_kernels}.  Since   \eqref{eq:map_kernels} only refers to  $  \cF_{\cP}(L;0) $, we start with the projection
\begin{align} \label{eq:project_rotation_theta} \cF_{\cP_{0,1}}(L;0) & \to  \cF_{\cP}(L;0) \\
(R,\theta,w) & \mapsto (R, w \circ r_{-\theta}).
\end{align}
which gives a short exact sequence on tangent spaces
\begin{equation} \label{eq:short_exact_sequence_forget_marked} 0 \to TS^1 \to  T\cF_{\cP_{0,1}}(L;0)  \to  T \cF_{\cP}(L;0) \to 0
\end{equation}
Note this decomposition differs, by $\theta$-rotation, from the decomposition of $ T\cF_{\cP_{0,1}}(L;0) $ induced by the definition of the space $\cF_{\cP_{0,1}}(L;0)$ as a product $S^1 \times \cF_{\cP}(L;0)$.

Our goal is to describe the operator $D_{\cP_{0,1}}$, and hence also $D_{\cP_{0,1}}^{\partial \cX}$ using only the space $\cF_{\cP}(L;0)$.  To do this, we fist consider a rank $1$ vector space spanned by a vector $\partial_{\theta}$ which we think of as the infinitesimal generator of rotations on the disc.  By compositing the inclusion of  $\cF_{\cP}(L;0)$ into $\cF_{\cP_{0,1}}(L;0)$ at angle $\phi$ with rotation by $\phi$, we consider the composition of maps
\begin{equation} \label{eq:composition_extend_differential_to_rotation} 
\xymatrix{ \langle  \partial_{\theta}  \rangle \ar[rr]^(.35){(\partial_{\theta}, dw(\partial_{\theta})) } & & T_{\phi} S^1 \oplus   T_{(R,w)} \cF_{\cP}(L;0) \ar[r]^(.45){(\id,r_{\phi})} &  T_{\phi} S^1 \oplus T_{(R,w \circ r_{\phi})} \cF_{\cP}(L;0) \ar[d]^{\cong} \\  
\cE_{\cP,(R,  w)} &  & \cE_{\cP,(R,  w \circ r_{\phi} )} \ar[ll]^{r_{-\phi}} &  T_{(R, \phi, w \circ r_{\phi})} \cF_{\cP_{0,1}}(L;0)  \ar[l]^{D_{\cP_{0,1} } }   }\end{equation}
By inspecting our construction of $D_{\cP_{0,1}}$, we find that it is $S^1$-equivariant since all our choices (of almost complex structures $J_{\theta,R}$, metrics $g_{\theta,R}$, and $1$-form $\gamma_{\theta,R}$) are pulled back from a fixed choice at $\theta=0$ by the rotation $r_{\theta}$.  We conclude
\begin{lem} \label{lem:add_rotation_direction_to_P} The map 
\begin{equation}  \partial_{\theta} \to \cE_{\cP,(R,  w)} \end{equation}
obtained by composing all the arrows in Equation \eqref{eq:composition_extend_differential_to_rotation} is independent of $\phi$, and vanishes if $w$ lies in $\cP(L;0)$. \noproof  \end{lem}

We denote the  direct sum of Equation \eqref{eq:composition_extend_differential_to_rotation} with $D_{\cP}$ by:
\begin{equation} \label{eq:add_rotation_to_linearisation_dbar_operator_parametrized} D_{\cP}^{\langle \partial_{\theta} \rangle} \co \langle \partial_{\theta} \rangle \oplus T \cF_{\cP}(L;0) \to \cE_{\cP}. \end{equation}
\begin{lem}
Given $(R,\theta,w) \in  \cF_{\cP_{0,1}}(L;0)$, and $\theta \in S^1$, rotation by $\theta$ defines a commutative diagram
\begin{equation} \label{eq:recover_circle_from_base}  \xymatrix{ \langle \partial_{\theta} \rangle \oplus T_{(R,w)}\cF_{\cP}(L;0) \ar[d] \ar[r] &  T_{(R,\theta, w \circ r_{\theta})}  \cF_{\cP_{0,1}}(L;0) \ar[d] \\
\cE_{\cP,(R,w)}  \ar[r]   &  \cE_{\cP_{0,1},(R,\theta, w \circ r_{\theta})} } \end{equation}
whose horizontal arrows are isomorphism, and whose top arrow, upon identifying $\partial_\theta$ with the generator of $TS^1$ gives a splitting of the short exact sequence \eqref{eq:short_exact_sequence_forget_marked}. \noproof
\end{lem}

We define the map
\begin{equation}  \G^{\partial \cX}_{S} \co \partial \cX \to \cF_{\cP}(L;0) \end{equation}
to be composition $\Gext^{\partial \cX}_{S}$  with the projection \eqref{eq:project_rotation_theta}.    Returning to Equation \eqref{eq:surjective_X_dbar_operator_gluing}, and using the previous Lemma, we find that we have a surjective operator
\begin{equation} D_{\cP}^{\langle \partial_{\theta} \rangle, \cX} \co V_{\cX} \oplus  \langle \partial_{\theta} \rangle \oplus {\G_{S}^{\partial \cX}}^{*} T\cF_{\cP}(L;0)  \to {\G^{\partial \cX}_{S}}^{*} \cE_{\cP} ,\end{equation}
where we can pass to spaces of smooth maps using the fact that all curves in the image of $\G^{\partial \cX}$ are smooth, as well as elliptic regularity.  Moreover, we have an isomorphism of bundles
\begin{equation} \ker D_{\cP}^{\langle \partial_{\theta} \rangle, \cX} \cong  \ker D_{\cP_{0,1}}^{\cX} , \end{equation}
which, when composed with the isomorphism $ \mathring{\Psi}_{\partial \cX}$ of Equation \eqref{eq:iso_on_kernels_expression_extended}, gives an isomorphism
\begin{equation} \label{eq:iso_on_kernels_correct} \Psi_{\partial \cX} \co \ker D_{\partial \cC}^{\cX} \to \ker D_{\cP}^{\langle \partial_{\theta} \rangle, \cX} .  \end{equation}
Finally, using Lemma \ref{lem:gluing_extra_vectors_close_to_where_they_should_be}, we conclude:
\begin{lem}
The restriction of \eqref{eq:iso_on_kernels_correct} to $\codbC{1}{L}$ decomposes as a direct sum of the identity on $V_{\cX}$ with a map
\begin{equation}  \aut(D^2,-1) \oplus T\codbC{1}{L} \to T \cP(L;0) \end{equation} 
satisfying the conditions listed in Lemma \ref{lem:iso_stabilizations_bundles}.  
\end{lem} 
\begin{proof}
That we get $V_{\cX}$ as a summand follows from the fact that $D_{\cP}^{\langle \partial_{\theta} \rangle, \cX} $ vanishes on $V_{\cX}$ when restricted to $\codbC{1}{L}$.  The only isotopy we need to perform is constant on $ \langle \partial_{s} \rangle \oplus T\codbC{1}{L} $ and deforms the image of $\partial_{p}$ to be a positive multiple of $\partial_{\theta}$.  Lemma \ref{lem:gluing_extra_vectors_close_to_where_they_should_be}, and the fact that $ \mathring{\Psi}_{\partial \cX}$ is uniformly bi-lipschitz implies one can perform this deformation through a family of isomorphism.
\end{proof}

\subsection{Completion of the proof of Lemma \ref{lem:iso_stabilizations_bundles}} \label{sec:construct_Y}

We now construct the CW complex $\cY(L)$ by choosing a triangulation of $\Phat^{S}(L;0)$ extending the triangulation of $\codbC{1}{L}$.  We start with the union
\begin{equation} \Phat^{S}(L;0) \cup_{\codbC{1}{L}} \partial \cX(L) \end{equation}
as a CW complex equipped with a map to $\cF_{\cP} (L)$.  Consider the weak homotopy equivalences
\begin{equation} L \ \to \cF_{\cP} (L) \to L ,\end{equation}
where the first map is the inclusion of constant maps, and the second is the evaluation map to $L$.  The argument presented in the previous section to construct $\cX$ and prove Lemma \ref{lem:homeo_type_X} implies:
\begin{lem}
There exists a CW complex $\cY(L)$ including $ \partial \cX(L)$, $\Phat^{S}(L;0)$, and $L$ as subcomplexes, with a map
\begin{equation} \iota_{\cY}  \co \cY(L) \to \cF_{\cP} (L)  \end{equation} 
which is a homotopy equivalence. \noproof
\end{lem}

We extend $V_{\cX}$ to a trivial vector bundle over $\cY(L)$, equipped with a map 
 \begin{equation} V_{\cX} \to \iota_{\cY}^{*}  \cE^{p,\delta}_{\cP}  \end{equation} 
extending \eqref{eq:extend_perturbations_pre-gluing}.   The surjectivity property of Equation \eqref{eq:surjective_X_dbar_operator_gluing} may not hold globally, but it holds in some neighbourhood of $ \partial \cX(L)$ in $\cY(L)$.   Moreover, using the same local construction as in the proof of Lemma \ref{eq:preturbations_to_surject_CW}, we can also produce a trivial vector bundle $V_{\cY}$ over $\cY(L)$ equipped with a map
 \begin{equation} V_{\cY} \to \iota_{\cY}^{*}  \cE^{p,\delta}_{\cP}  \end{equation} 
which vanishes near $ \partial \cX(L)$, and such that the direct sum
\begin{equation} V_{\cY} \oplus V_{\cX} \oplus \langle \partial_{\theta} \rangle \oplus \iota_{\cY}^{*} T \cF_{\cP}(L;0) \to \iota_{\cY}^{*} \cE_{\cP} \end{equation}
is surjective.  Extending $V_{\cY}$ trivially to $\cX(L)$ and defining the map
\begin{equation} V_{\cY} \to \iota_{\cX}^{*}  \cE_{\cC}  \end{equation}
to vanish, we let $V_{\cZ} = V_{\cX}  \oplus  V_{\cY}$, and conclude: 
\begin{lem} \label{lem:construct_Z}
There exists a trivial vector bundle $V_{\cZ}$ over $\cZ(L)$ satisfying equation \eqref{eq:dbar_capping_surjective} and \eqref{eq:dbar_para_surjective}, and equipped with an isomorphism map as in \eqref{eq:map_kernels} \noproof
\end{lem}

 \section{Analytic estimates for gluing}  \label{sec:analyt-results-gluin}

In this final Section, we collect the proof of some results which are needed for the gluing Theorem.

\subsection{Pointwise estimates} \label{ap:pointwise_estimates}

We start by collecting a series of pointwise estimates which are used for gluing.  All the results are elementary in nature.  The only reason for including them is that our chosen metric $| \_ |_{1,p,S}$ treats the norm of a vector field and of its covariant derivative differently as $S$ varies.  In order to obtain uniform estimates that are independent of $S$, it does not therefore seem to be sufficient, at least if we work only with our family of metrics $g_{S}$, to bound an expression in terms of the $C^1$-norm of a vector field $X$.  Rather, one needs to know its separate behaviour with respect to $|X|$ and $|\nabla X|$, and ensure that the terms multiplying $|X|$ can themselves be bounded in the $L^p$-norm.

Let $\Sigma$ be a Riemann surface (in our example, either a strip or a disc), $w$ a map from $\Sigma$ to $M$, and $Z$ a vector field along $w$.  Let $w_{Z} = \exp_{w} Z$, and let $\Pi_{w}^{w_Z}$ denote the parallel transport map along the image of the exponential map from $w$ to $w_Z$.  Our first estimate is

\begin{lem} \label{lem:jacobi_field_estimate}
There exists a constant $C$ depending only the metric on $M$ such that
\begin{equation} \left| dw_Z \right| \leq (|dw| +|\nabla Z| )  e^{C|Z|}.  \end{equation}
In particular, if $|Z|$  is sufficiently small,
\begin{equation} \label{eq:bound_exponential} \left| dw_Z \right| \leq  (|dw|  +|\nabla Z| )(1+ C|Z|).  \end{equation}
Moreover,
\begin{equation} \label{eq:bound_first_derivative_exponential} \left| \nabla_{\Pi_{w}^{w_{Z}} Z}  dw_{Z} \right| \leq  C|Z| (|dw|  +|\nabla Z| )  .  \end{equation}
\end{lem}
\begin{proof}
Fix a point $z \in D^2$ and a tangent vector $\partial_{s}$.  By definition, $Y(y) = dw_{yZ}(\partial_{s})$ is a Jacobi vector field with initial conditions $Y(0) = dw(\partial_{s})$ and
\begin{equation} Y'(0) \equiv \nabla_{Z}  dw(\partial_{s}) = \nabla_{dw (\partial_{s})} Z .\end{equation}
The equation for a Jacobi field is given by
\begin{equation} Y''(y) \equiv  \nabla^{2}_{ \Pi_{w}^{w_{yZ}} Z}  dw(\partial_{s})  = R\left(\Pi_{w}^{w_{yZ}} Z, Y(y)\right) \Pi_{w}^{w_{yZ}} Z ,\end{equation}
(see e.g. Equation 6.11 of \cite{berger}), so that
\begin{equation}|Y''(y)|  \leq C| Z|^{2}|Y(y)|\end{equation}
which readily implies that
\begin{equation}| Y(y)| \leq a e^{y b|Z|} \end{equation}
for $(a,b)$ independent of $y$.  Using the initial conditions, we obtain the desired results.
\end{proof}
From now on, we shall assume that $|Z|$ is bounded (say by $1$).  Next, let $w_{y} \co \Sigma \to M$ be a family of maps smoothly parametrized by $y \in [0,1]$.  Let $X$ be a vector field along the image of $w$, and write $\Pi_{y} X$ for the image of $X$ by  parallel transport along the path $w_{y}$.  We shall be interested in comparing the norm of $\nabla X$ to that of $\nabla \Pi_{y} X$:
\begin{lem} \label{lem:error_parallel_transport_derivative}
There exists a constant $C$ depending only on the metric on $M$ such that\begin{equation} \label{eq:parallel_transport_bound_error} \left|\Pi_{y}   \nabla X - \nabla \Pi_{y} X \right| \leq C|X|    \max_{y}|d w_{y}|   \int_{y}|\partial_{y} w_{y}| . \end{equation}
In particular, if $w_{y} = \exp_{w} yZ$ for some vector field $Z$, then
\begin{equation} \label{eq:parallel_transport_bound_exponential}  \left|\Pi_{w}^{w_1}   \nabla X - \nabla \Pi_{w}^{w_1} X\right|  \leq C|X||Z| (|dw| +|\nabla Z|) \end{equation}
\end{lem}
\begin{proof}
Fixing a chart on $\Sigma$ with coordinates $s$ and $t$, we bound the $y$-derivative of the norm of the left hand side of Equation \eqref{eq:parallel_transport_bound_error}:
\begin{align*} \frac{d}{dy}   \left| \Pi_{y}   \nabla_{dw_{y}(\partial_s)} X - \nabla_{dw_{y}(\partial_s)} \Pi_{y} X \right| & \leq \left| \nabla_{\partial_{y} w_y}  \Pi_{y}   \nabla_{dw_{y}(\partial_s)} X \right|   +  \left| \nabla_{\partial_{y} w_y} \nabla_{dw_{y}(\partial_s)} \Pi_{y} X \right| \\
& \leq  C | \partial_{y} w_{y}| |dw_{y}(\partial_s)||X|,\end{align*}
where $C$ is the norm of the curvature of the metric.  This yields Equation \eqref{eq:parallel_transport_bound_error} by integration.  Using Lemma \ref{lem:jacobi_field_estimate}, we conclude the bound \eqref{eq:parallel_transport_bound_exponential}.
\end{proof}
Given a two-parameter family of maps $w_{x,y}$, with $(x,y) \in [0,1]^2$, we shall consider the difference between parallel transport along the two sides of the square.
\begin{lem}  \label{lem:bound_derivative_parallel_transport}
There exists a constant $C$ depending only on the metric on $M$ such that 
\begin{align} \label{eq:parallel_transport_bound_error_family} \left|\Pi_{w_{0,1}}^{w_{1,1}}  \Pi_{w_{0,0}}^{w_{0,1}}   X - \Pi_{w_{1,0}}^{w_{1,1}}  \Pi_{w_{0,0}}^{w_{1,0}}  X \right| & \leq C  \int_{x,y}|\partial_{y} w_{x,y}||\partial_{x} w_{x,y}| |X| \end{align}
\begin{multline} \label{eq:parallel_transport_bound_error_family_derivative} 
\left| \nabla \left( \Pi_{w_{0,1}}^{w_{1,1}}  \Pi_{w_{0,0}}^{w_{0,1}}   X - \Pi_{w_{1,0}}^{w_{1,1}}  \Pi_{w_{0,0}}^{w_{1,0}}  X \right) \right|  \leq C| \nabla X|  \int_{x,y}|\partial_{y} w_{x,y}||\partial_{x} w_{x,y}| \\
+ C|X| \big( \max_{x}|d w_{x,0}|\int_{x}|\partial_{x} w_{x,0}| + \max_{x}|d w_{x,1}| \int_{x}|\partial_{y} w_{x,1}| + \\ \max_{y}|d w_{0,y}|  \int_{y}|\partial_{y} w_{0,y}|   + \max_{y}|d w_{1,y}|  \int_{y}|\partial_{y} w_{1,y}| \big) 
\end{multline}
\end{lem}
\begin{proof}[Idea of proof:]
The first result follows from a straightforward computation along the lines of Lemma \ref{lem:error_parallel_transport_derivative}.  To prove the second estimate, we commute $\nabla$ past the parallel transport map twice in each term,  introducing  error terms bounded as in Equation \eqref{eq:parallel_transport_bound_error} which account for the first term of \eqref{eq:parallel_transport_bound_error_family_derivative}.  The bound then follows from the first estimate applied to $\nabla {X}$.
\end{proof}

Next, we prove a pointwise analogue of the quadratic inequality.  Let $\gamma$ be a $1$-form on $\Sigma$, $X_H$ a Hamiltonian vector field on $M$ and $J_z$ an almost complex structure on $M$ depending on $z \in \Sigma$.  Given a map $w \co \Sigma \to M$, and a vector field $X$ along $w$ we define
\begin{equation} D_{\gamma} (X)  = \left( \nabla X -  \gamma \otimes \nabla_{X} X_{H}  \right)^{0,1} - \frac{1}{2} J_{z} \left( \nabla_{X}J_{z} \right) \partial_{\gamma} w \end{equation}
where $ \partial_{\gamma} w = \left( dw - \gamma \otimes X_{H} \right)^{1,0}$, and $\nabla$ stands for the Levi-Civita connection of a $z$-dependent metric $g_z$ on $M$ which is almost hermitian for $J_{z}$.  Let $Z$ be another vector field along $w$.  Consider the exponential of the vector fields $yZ$ for $y \in [0,1]$, and write $w_y$ for $w_{yZ}$ and $\Pi_{y}$ for $\Pi_{w}^{w_{y}}$.  The first estimate we shall need concerns the failure of the complex linear connection $\tilde{\Pi}$ to be an isometry
\begin{lem} \label{lem:bound_difference_parallel_transports}
There is a uniform constant $C$ such that
\begin{equation}  \left| \tilde{\Pi}_{w}^{w_{Z}} Y  -{\Pi}_{w}^{w_{Z}} Y \right| \leq  C|Y||Z|.\end{equation}
\end{lem}
\begin{proof}
We take the derivative with respect to $y$ of the of the norm of the above expression replacing $w_{Z}$ by $w_{yZ}$:
\begin{align*}  \frac{d}{dy}   \left| \tilde{\Pi}_{y} Y  -{\Pi}_{y} Y \right|  & \leq \left| J \left( \nabla_{\Pi_{y} Z} J \right) \tilde{\Pi}_{y} Y \right| \\
& \leq C|Z|  \left| \tilde{\Pi}_{y} Y \right|.
\end{align*}
We then use the initial condition that $ \tilde{\Pi}_{y}$ and $ \Pi_{y}$ agree when $y=0$.
\end{proof}

\begin{lem} \label{eq:pointwise_quadratic_inequality}
There exists a constant $C$ depending only on $M$ such that
\begin{equation} \left| \tilde{\Pi}_{w}^{w_{Z}} D_{\gamma} X -  D_{\gamma} (\Pi_{w}^{w_{Z}} X) \right| \leq C|Z| \left(|\nabla X|  +|X| \left(|\gamma|  +|dw| +|\nabla Z| \right) \right) \end{equation}
\end{lem}
\begin{proof}
Using the definition of the connection $\tilde{\Pi}$, we obtain
\begin{equation*}  
\frac{d}{dy} \left| \tilde{\Pi}_{y} D_{\gamma} X -  D_{\gamma} (\Pi_{y} X) \right| \leq \left| \nabla_{\Pi_{y} Z} D_{\gamma} (\Pi_{y} X) \right| + \left|  \frac{1}{2} J_{z}  \left(\nabla_{\Pi_{y} Z}J_{z} \right) D_{\gamma} X \right| \end{equation*}
The lowest order terms of the right hand side (where we take a derivative of $J$ with respect to the parallel transport of $Z$ in either term) are bounded by a constant multiple of 
\begin{equation} \label{eq:lowest_order_terms_vector_field}|Z| \left(|\nabla \left(\Pi_{y} X \right)| +|X||\gamma| +|X| | \partial_{\gamma} w_{y}|  +|\nabla X| \right) .\end{equation}
Using the expression for $\partial_{\gamma} w_{y} $, and Equation \eqref{eq:parallel_transport_bound_exponential}, we simplify this to a constant multiple of
\begin{equation} \label{eq:better_bound_lower_order}  |Z| \left(|\nabla X| +  |X|  (|Z| |\nabla Z| + |\gamma| +|dw|) \right) \end{equation} To bound the remaining terms, we must bound
\begin{multline} \label{eq:higher_order_terms_vector_field}
\left| \nabla_{\Pi_{y} Z}  \nabla \Pi_{y} X \right| +|\gamma| \left| \nabla_{\Pi_{y} Z}  \nabla_{\Pi_{y} X} X_H \right|  + \left|  \left( \nabla_{\Pi_{y} Z}  \left( \nabla_{\Pi_{y} X} J_{z} \right) \right)\partial_{\cP_{0,1}} w   \right| \\ + \left|    \nabla_{\Pi_{y} X} J_{z} \right| \left|   \nabla_{\Pi_{y} Z}  \partial_{\cP_{0,1}} w  \right| 
\end{multline}
We observe that the defining property of the Levi-Civita connection implies that
\begin{equation} \nabla_{\Pi_{y} Z}  \nabla_{dw_{y}(\partial_{s} )} \Pi_{y} X  = R\left(\Pi_{y} Z, d w_{y}(\partial_s) \right) \Pi_{y} X , \end{equation}
so that the first term is bounded by a constant multiple of
\begin{equation} \label{eq:bound_second_derivative} |Z||X||d  w_{y}| \leq|Z||X| (|dw|  +|\nabla Z| ) ,\end{equation}
using Equation \eqref{eq:bound_exponential} to bound $|d  w_{y}| $.  Since we can write $\nabla_{\Pi_{y} X} J_{z}  = \langle \Pi_{y} X, \nabla  J_{z} \rangle $, and similarly for $\nabla_{\Pi_{y} X} X_H $, we find  that the second and third term are bounded by a constant multiple of 
\begin{equation} \label{eq:bound_derivative_J_X} |Z||X|(|\gamma| +|dw_{y}|) \leq|Z||X| (|\gamma| +|dw| +|\nabla Z|) ,\end{equation}
Finally, since
\begin{equation}  \left|   \nabla_{\Pi_{y} Z}  \partial_{\gamma} w  \right|  \leq C \left( |Z|  \left( \left| dw_{y} - \gamma \otimes X_H  \right| +|\gamma| \left| \nabla X_H \right| \right) + \left| \nabla_{  \Pi_{y} Z} dw_{y} \right|  \right) \end{equation}
we can apply Equation \eqref{eq:bound_first_derivative_exponential} to conclude that the last term is bounded by a constant multiple of
\begin{equation} \label{eq:bound_derivative_dw} |Z||X| (|\gamma| +|dw| +|\nabla Z|) .\end{equation}
The sum of equations \eqref{eq:better_bound_lower_order}, \eqref{eq:bound_second_derivative}, \eqref{eq:bound_derivative_J_X}, and \eqref{eq:bound_derivative_dw}  gives the desired bound.
\end{proof}
Let $\kappa_{\Sigma}$ denote any function on $\Sigma$, and consider the a norm $| \_|_{p,\Sigma}$, weighted by $\kappa_{\Sigma}$, on $L^p$ functions:
\begin{cor} \label{cor:quadratic_inequality_basic_case}
There exists a constant $C$ depending only on the metric on $M$ and the Hamiltonian $H$ such that
\begin{multline} \left| \tilde{\Pi}_{w}^{w_{Z}} D_{\gamma} X -  D_{\gamma} (\Pi_{w}^{w_{Z}} X) \right|_{p,\Sigma} \leq \\ C|Z|_{C^0}  \left(  |\nabla X|_{p,\Sigma}  +|X|_{C^0} \left(|\gamma|_{p,\Sigma} +|dw|_{p,\Sigma} +|\nabla Z|_{p,\Sigma} \right)  \right) .\end{multline}
\noproof
\end{cor}

Next, we prove a Lemma that shall be used in Section \ref{sec:derivative_gluing}.  Consider the case $\Sigma=D^2$, equipped with the $L^p$-metric $| \_|_{p,S}$ on vector fields introduced in Remark \ref{rem:L_p_norm_tangent_vectors}, and let $X$ and $Z$ be a pair of vector fields along a map $w \co D^2 \to M$.
\begin{lem} \label{lem:error_add_vector_fields_L_p_S}
There exists a constant $C$ depending only the diameter of $w$ and the metric on $M$ such that
\begin{equation} \left| \exp^{-1}_{w}\left( \exp_{w_{X}}  \Pi_{w}^{w_X} Z \right)  - X -Z \right|_{p,S} \leq C|X|_{C^0}|Z|_{p,S}\end{equation}
\end{lem}
\begin{proof}
Given a point $q$ in $M$ and its image $q_X$ under the exponential map applied to a vector $X$,  we have a bound
\begin{equation} \left| \exp^{-1}_{q}\left( \exp_{q_{X}}  \Pi_{q}^{q_X} Z \right)  - X -Z \right| \leq C|X| |Z|,  \end{equation}
since the left handside is a smooth function that vanishes whenever $X$ or $Z$ vanish.  More generally, given a path $\gamma$ with end points $q$ and $r$, the same argument shows that there is a constant $C_{\ell(\gamma)}$ depending only on the length of $\gamma$ and the metric such that:
\begin{equation} \left|  \exp^{-1}_{r}\left( \exp_{r_{X}}  \Pi_{r}^{r_{X}} Z \right)  - X -Z  - \Pi_{q}^{r} \left(\exp^{-1}_{q}\left( \exp_{q_{X}}  \Pi_{q}^{q_{X}} Z \right)  - X -Z  \right) \right| \leq C_{\ell(\gamma)}  |X| |Z|.\end{equation}
The result follows immediately by using the first inequality to bound the $L^p$ norm of $\exp^{-1}_{w}\left( \exp_{w_{X}}  \Pi_{w}^{w_X} Z \right)  - X -Z$ in the complement of the neck, and the second inequality to bound the $L^p$ norm, on the image of the neck, of the difference between this vector field and the parallel transport of its value at $\xi_{S,\neck}(0,0)$.
\end{proof}

\subsection{Proof of results from Section \ref{sec:pre-gluing_maps}}  \label{sec:proof-results-pre-gluing}

\begin{proof}[Proof of Lemma \ref{lem:weighted_1_p_implies_c_0_exponential_decay}]
For specificity, we assume that $r$ is a boundary marked point.  A straightforward computation shows that there is a constant $c_0$ depending only on $u$ such that for any vector field $X_{Z}$ in $W^{1,p}_{\I{S}{+\infty}} \left(  (u\circ \xi_{r})^*TM, ( u\circ \xi_{r})^*TL) \right) $, 
\begin{equation} \label{eq:equivalence_two_weighted_norms} \left| e^{\delta|s|} X_Z \right|_{1,p} \leq c_0  \left| X_Z \right|_{1,p,\delta}   . \end{equation}
Note that all strips $\I{S}{+\infty}$ are isometric, and moreover, the curvature of the Levi-Civita connection $\nabla$ on $u \circ \xi_{r}(\I{S}{+\infty})$ is uniformly bounded.  In particular, the Sobolev embeddings
\begin{equation} W^{1,p}_{\I{S}{+\infty}} \left(  (u\circ \xi_{r})^*TM, ( u\circ \xi_{r})^*TL) \right) \to C^0_{\I{S}{+\infty}} \left(  (u\circ \xi_{r})^*TM, ( u\circ \xi_{r})^*TL) \right) \end{equation}
admit a uniform Sobolev constant $c_p$ that is independent of $S$.  Combining this with \eqref{eq:equivalence_two_weighted_norms}, we conclude that
\begin{equation}  \left| e^{\delta|s|} X_{Z} \right|_{\infty} \leq c c_{p} \left| X_Z \right|_{1,p,\delta} \end{equation}
which proves the Lemma upon setting
\begin{equation} X_{Z} = X \circ \xi_{r} - \Pi_{u(r)}^{u \circ \xi_{r}} X(r).  \end{equation}
\end{proof}

\begin{proof}[Sketch of the proof of Lemma \ref{lem:1_p_delta_maps_Banach_manifold}]
We omit the elementary proof that $W^{1,p,\delta}_{\left(\Sigma, \{p_i \}, \{ q_j\}\right)}(u^*TM, u^*TL)$ is a Banach space.  To prove the lemma, we embed $(M,L) \subset (\bC^N,\bR^N)$ for a sufficiently large $N$ such that $M \cap \bR^N = L$.  Since Condition \eqref{eq:W_1_p-bounded_function} is independent on the target's metric, and $W^{1,p,\delta}_{\left(\Sigma, \{p_i \}, \{ q_j\}\right)}(u^*TM, u^*TL)$ embeds into $C^0$ by the Sobolev Lemma, $\cF^{1,p,\delta}_{\left(\Sigma, \{p_i \}, \{ q_j\}\right)}(L)$ is the closed subset of the Banach space $W^{1,p,\delta}_{\left(\Sigma, \{p_i \}, \{ q_j\}\right)} (\bC^N,\bR^N)$ consisting of maps whose image lies in $M$.

To build the chart associated to a map $u \in \cF^{1,p,\delta}_{\left(\Sigma, \{p_i \}, \{ q_j\}\right)}(L)$ we pick a metric on $\bC^N$, which is totally flat near the image of all marked points,  and such that $M$, and $L$ are both totally geodesic. Consider the exponential map
\begin{align}
W^{1,p,\delta}_{\left(\Sigma, \{p_i \}, \{ q_j\}\right)}(u^*TM, u^*TL) & \to \cF^{1,p,\delta}_{\left(\Sigma, \{p_i \}, \{ q_j\}\right)}(L)\subset W^{1,p,\delta}_{\left(\Sigma, \{p_i \}, \{ q_j\}\right)} (\bC^N,\bR^N) \\
X & \mapsto u_{X} \equiv \exp_{u}(X).
\end{align}
It is easy to check using the flatness of the metric near the image of all marked points that Condition \eqref{eq:W_1_p-bounded_function} holds for $u_{X}$ if the norm of $X$ is sufficiently small; i.e. that the map is indeed well defined.  To prove that the restriction of the exponential map to vector fields of sufficiently small norm is a smooth bijection onto a neighbourhood of $u$, one can easily combine the estimates that go into proving the analogous results for compact manifolds  with elementary computations near $u(p_i)$ and $u(q_j)$ for interior and boundary marked points.
\end{proof}

\subsection{Sobolev bounds from Section \ref{sec:codim1_gluing}} \label{sec:proof-results-gluing}
We begin by proving that the spaces we work with are indeed Banach spaces:
\begin{proof}[Proof of Lemma \ref{lem:Sobolev_complete_norm}]
It suffices to compare $| \_|_{1,p,S}$ to a more familiar Sobolev norm. The most convenient such norm comes from the induced metric $g_{S}$ on $\Sigma_{S}$; we denote it $|\_|_{1,p,g_{S}}$.

One direction follows from the Sobolev embedding theorem.  Indeed, the first term of Equation \eqref{eq:weighted_norms_tangent_vectors_S} is obviously bounded by the Sobolev norm $|X|_{1,p,g_{S}}$.  For the third term, we have
\begin{align*}  \left| X \circ  \xi_{S,\neck} - \Pi(X(\xi_{S,\neck}(0,0)) \right|_{1,p,S, \delta} & \leq  \left| X \circ  \xi_{S,\neck} \right|_{1,p,S,\delta}  +  \left| \Pi(X(\xi_{S,\neck}(0,0)) \right|_{1,p,S, \delta} \\
& \leq e^{2\delta S}  \big{(} \left| X \circ  \xi_{S,\neck} \right|_{1,p} +  \left| \Pi(X(\xi_{S,\neck}(0,0)) \right|_{p} + \\
& \qquad +   \left| \nabla \Pi(X(\xi_{S,\neck}(0,0)) \right|_{p} \big{)} \\
& \leq e^{2\delta S}  \left| X \circ  \xi_{S,\neck} \right|_{1,p}  + e^{2\delta S} C_{u \#_{S} v} \left|X(\xi_{S,\neck}(0,0) \right|
\end{align*} where $C_{u \#_{S} v}$ is some constant depending on $S$ but not on $X$ (see Lemma \ref{lem:derivative_p-norm_bounded_by_1,p}).  By the Sobolev embedding theorem, $| X(\xi_{S,\neck}(0,0))|$ is itself bounded by some multiple of $|X|_{1,p,g_{S}}$, proving the existence of a constant $C$ such that
\begin{equation}| \_|_{1,p,S} \leq  C| \_|_{1,p,g_{S}} . \end{equation}

In the other direction, we compute that the usual $W^{1,p}$ norm of a vector field along the neck is bounded as follows
\begin{align*} \left| X \circ  \xi_{S,\neck}\right|_{1,p} & \leq \left| X \circ  \xi_{S,\neck} - \Pi(X(\xi_{S,\neck}(0,0)) \right|_{1,p} + \left|\Pi(X(\xi_{S,\neck}(0,0)) \right|_{1,p} \\
& \leq  \left| X \circ  \xi_{S,\neck} - \Pi(X(\xi_{S,\neck}(0,0)) \right|_{1,p, \delta} + C_{u \#_{S} v} \left|X(\xi_{S,\neck}(0,0)\right|. \end{align*}
The result follows immediately.
\end{proof}

In order to proceed with the proof of the assumptions of Floer's Picard Lemma, we must establish that our chosen Sobolev norms $| \_|_{1,p,S}$ and $| \_|_{1,p}$ satisfy reasonable properties with respect to each other, and the $C^0$-norm. 
\begin{lem} \label{lem:uniform_sobolev} If the Sobolev space $W^{1,p}$ is equipped with the norm $| \_|_{1,p,S}$, then the inclusion
\begin{equation} \label{eq:sobolev_emebdding} W^{1,p} \left( (u \#_{S} v)^{*} TM,  (u \#_{S} v)^{*} TL \right) \to C^{0} \left( (u \#_{S} v)^{*} TM,  (u \#_{S} v)^{*} TL \right) \end{equation}
admits a Sobolev constant $c_p$ that is independent of $S$  (and hence of $u$ and $v$ since $K$ is compact).   \end{lem}
\begin{proof}
Note that surfaces obtained by removing the neck from $\Sigma_{S}$ are independent of $S$.   Moreover, the standard Sobolev norm on functions $\I{-2S}{2S}$ has a Sobolev constant that is independent of $S$ (the strips satisfy a uniform cone condition, see the discussion preceding Theorem 5.4 in \cite{adams}).  Since the curvature of the metric is uniformly bounded on the image of $u \#_{S} v \circ \xi_{S,\neck}$, the same property holds for $TM$-valued vector fields along $u \#_{S} v \circ \xi_{S,\neck}$.  Let $C$ denote a constant, larger than $1$, and larger than the Sobolev constants for the strip $\I{-2S}{2S}$ and for the complement of the neck in $\Sigma_{S}$.  We compute
\begin{align*}|X|_{\infty} & \leq\left| X|_{\Sigma_{S} - \Im(\xi_{S,\neck})} \right|_{\infty}  +|X \circ \xi_{S,\neck}|_{\infty}   \\
& \leq C  \left|X|_{\Sigma_{S} - \Im(\xi_{S,\neck})} \right|_{1,p} +|X \circ \xi_{S,\neck} - \Pi(X(\xi_{S,\neck}(0,0))|_{\infty}   +| \Pi(X(\xi_{S,\neck}(0,0))|_{\infty}  \\
& \leq  C \left| X|_{\Sigma_{S} - \Im(\xi_{S,\neck})} \right|_{1,p} + C |X \circ \xi_{S,\neck} - \Pi(X(\xi_{S,\neck}(0,0))|_{1,p,g_{S}} + | X(\xi_{S,\neck}(0,0)| \\
& \leq 3C|X|_{1,p,S}  \end{align*}
\end{proof}

We will use this estimate for $u \#_{S} v$, and exploit the fact that $\partial_{s} u$ and $\partial_{s} v$ decay exponentially along the end. 
\begin{lem} \label{lem:derivative_p-norm_bounded_by_1,p}
There is a universal constant $C$, independent of $S$ such that whenever $X$ is a vector field along $u \#_{S} v$, 
\begin{equation}| \nabla X|_{p,S} \leq C \cdot|X|_{1,p,S}. \end{equation}
\end{lem}

\begin{proof}
This would be an obvious result if we were using the usual Sobolev norm.  Comparing the definitions of the norms $| \_|_{p,S}$ and $|\_|_{1,p,S}$, we see that the only interesting term to bound is
\begin{equation} 
 \left| \nabla \Pi_{u(1)}^{u \#_{S} v} (X(\xi_{S,\neck}(0,0)) \right|_{p,S,\delta}.
\end{equation}  
Writing 
\begin{equation} u \#_{S} v \circ \xi_{S,\neck} = \exp_{u(1)}(Z(s,t)), \end{equation} 
we consider 
\begin{equation} w_{y}(s,t) = \exp_{u(1)} \left( Z(ys, t) \right) \end{equation}
and note that the parallel transport along the images of horizontal lines we used in equation \eqref{eq:weighted_norms_tangent_vectors_S} agrees with parallel transport along $w_{y}$.  The exponential decay estimates of Lemma \ref{lem:exponential_decay_bounds}  imply that $|\nabla  \partial_{y} w_{y}|$ is uniformly bounded, and that
\begin{equation} \int_{0}^{1}| \partial_{y} w_{y} (s,t)| = \int_{0}^{s}|\partial_{s} \left( u \#_{S} v \circ \xi_{S,\neck} \right)| \leq c_{u,\theta, v} e^{-2S+|s|} .\end{equation}
for a constant $ c_{u,\theta, v}$ that is independent of $S$.  Plugging these estimates into Equation \eqref{eq:parallel_transport_bound_error}, and using our assumption that $\delta < 1$, we obtain the desired result.
\end{proof}
\begin{cor} \label{cor:linearisation_dbar_bounded}
The operator 
\begin{equation} D_{\cP_{0,1}}  \co \preGext^{*}  T\cF^{1,p}_{\cP_{0,1}} (L) \to \preGext^{*}   \cE^{p}_{\cP_{0,1}}(M) \end{equation} is uniformly bounded.
\end{cor}
\begin{proof}
Note that applying $D_{\cP_{0,1}}$ to a deformation of the parameters $\theta$ and $R$ gives an element of $ \cE^{p}_{\cP_{0,1}}(M) $ which is supported away from the neck, and it is elementary in this case to show that the bound is uniform in $S$.  We may therefore restrict to vectors $X$ coming from $T\cF^{1,p} (L)$ for which the result follows from the previous Lemma, the expression for $D_{\cP_{0,1}}$ given in Equation \eqref{eq:expression_linearisation_circle_varying_parametrized}, and the fact that $\partial_{\cP_{0,1}} \left( \preGext_{S}(u,\theta,v)\right)$ is uniformly bounded in the $| \_|_{p,S}$ norm by the exponential decay estimate of Lemma \ref{lem:exponential_decay_bounds}.
\end{proof}

\subsection{Approximate and  right inverses}
We begin by proving a result from Section \ref{sec:right_inverse} about the approximate inverse:
\begin{proof}[Proof of Lemma \ref{lem:pre-gluing_tangent_vectors_bounded}]
The continuity and smoothness properties of $\predGext$ may be easily checked by the reader; we shall focus on fiberwise boundedness. By comparing the expressions for the two Sobolev norms \eqref{eq:definition_weighted_sobolev_ends} and \eqref{eq:weighted_norms_tangent_vectors_S}, we see that we need to show that the sum
\begin{equation} \label{eq:bound_terms_pregluing}  \left|X_{u} \#_{S} X_{v}  ( \xi_{S,\neck}(0,0))\right|  + \left| X_{u} \#_{S} X_{v}  \circ  \xi_{S,\neck} - \Pi^{u  \#_{S} v \circ \xi_{S,\neck} }_{u(1)} (X_{u} \#_{S} X_{v}(\xi_{S,\neck}(0,0))\right|_{1,p,S}  \end{equation}
is bounded by
\begin{equation*}  |X_{u}(1)|  + \left| X_{u}  \circ  \xi_{1} - \Pi^{u\circ  \xi_{1} }_{u(1)}  X_{u} (1)\right|_{1,p,\delta} +\left|X_{v}  \circ  \xi_{-1} - \Pi^{v \circ  \xi_{-1}}_{u(1)} X_{u} (1)\right|_{1,p,\delta} . \end{equation*}
For the purpose of proving Equation \eqref{eq:pre-gluing_converges_toisometry}, we shall in fact show that the difference between these two expressions decays exponentially, except for an error term which is bounded by a constant in the general case, but decays whenever $X_{u}$ and $X_{v}$ are smooth.

We shall repeatedly use the fact that $u(1) = v(-1) =( u  \#_{S} v ) \xi_{S,\neck}(0,t)$.   To estimate the first term, we use Lemma \ref{lem:weighted_1_p_implies_c_0_exponential_decay} to obtain a bound
\begin{align}\left|X_{u} \#_{S} X_{v}  ( \xi_{S,\neck}(0,0)) - X_{u}(1)\right| & \leq \left|X_{u} ( \xi_{1} (2S,t))  - \Pi^{u( \xi_{1} (2S,t)) }_{u(1)}  X_{u} (1) \right| +  \notag \\
& \qquad \left|X_{v} ( \xi_{-1} (-2S,t))  - \Pi^{v( \xi_{-1} (-2S,t)) }_{u(1)}  X_{u} (1) \right|  \notag \\
& = \rO\left(e^{-2 \delta S}\right) \left|X_{u}(1) \right|  \label{eq:value_glued_vector_field_middle_converges}  \end{align} 
 It remains therefore to bound the second term of equation  \eqref{eq:bound_terms_pregluing} independently of $S$.  Note that this term is an integral performed over $\I{-2S}{2S}$; by symmetry, we shall restrict our attention to the contribution of $\I{0}{2S}$. Using the fact that \eqref{eq:formula_gluing_vector_fields} simplifies in this region, we find that it would be sufficient to bound the following four terms
\begin{align} 
& \left| \Pi_{v \circ \xi_{-1} \circ \tau_{-2S} }^{u \#_{S} v \circ \xi_{S,\neck} }  X_{v} \circ  \xi_{-1} \circ \tau_{-2S}  - \Pi_{u(1)}^{u \#_{S} v \circ \xi_{S,\neck} } X_{u}(1) |_{\I{0}{2S}}   \right|_{1,p,\delta} \label{eq:first_term_bounded_pre-gluing_vector_fields}   \\
&\left|\Pi_{u(1)}^{u \#_{S} v \circ \xi_{S,\neck} } X_{u}(1)- \Pi_{u(1)}^{u \#_{S} v \circ \xi_{S,\neck} } X_{u} \#_{S} X_{v}  ( \xi_{S,\neck}(0,0))  |_{\I{0}{2S}} \right|_{1,p,\delta}
\label{eq:second_term_bounded_pre-gluing_vector_fields} 
\\ &\left|\Pi_{u \circ \xi_{1} \circ \tau_{2S} }^{u \#_{S} v \circ \xi_{S,\neck} } X_{u} \circ  \xi_{1}\circ \tau_{2S} - \Pi_{u(1)}^{u \#_{S} v \circ \xi_{S,\neck} } X_{u}(1)  |_{\I{0}{S+1}} \right|_{1,p,\delta} \label{eq:third_term_bounded_pre-gluing_vector_fields}   \\ \label{eq:fourth_term_bounded_pre-gluing_vector_fields} 
&|\nabla \chi| \cdot\left|\Pi_{u \circ \xi_{1} \circ \tau_{2S} }^{u \#_{S} v \circ \xi_{S,\neck} } X_{u} \circ  \xi_{1}\circ \tau_{2S} - \Pi_{u(1)}^{u \#_{S} v \circ \xi_{S,\neck} } X_{u}(1)  |_{\I{S-1}{S+1}}  \right|_{p,\delta} \end{align}
Note that the last two terms are integrals performed on the proper subsets of $\I{0}{2S}$  where $\chi$ and $\nabla \chi$ respectively do not vanish.  

To bound the expression \eqref{eq:first_term_bounded_pre-gluing_vector_fields}, we use Lemma \ref{lem:associative_up_to_exponential_error} to replace $\Pi_{u(1)}^{u \#_{S} v \circ \xi_{S,\neck} } X_{u}(1) $ by the composition of two parallel transport maps $\Pi_{v \circ \xi_{-1} \circ \tau_{-2S}}^{u \#_{S} v \circ \xi_{S,\neck} } \Pi_{u(1)}^{v \circ \xi_{-1} \circ \tau_{-2S} } X_{u}(1)$, introducing an error term which is bounded by a constant multiple of $ e^{-2(2-\delta)S} |X_{u}(1)|$, so that it suffices to bound
\begin{equation*} \left| \Pi_{v \circ \xi_{-1} \circ \tau_{-2S} }^{u \#_{S} v \circ \xi_{S,\neck} }  \left( X_{v} \circ  \xi_{-1} \circ \tau_{-2S}  -  \Pi_{u(1)}^{v \circ \xi_{-1} \circ \tau_{-2S} } X_{u}(1) X_{u}(1) \right) |_{\I{0}{2S}}   \right|_{1,p,\delta} . \end{equation*}
We apply the estimate \eqref{eq:parallel_transport_bound_error} to $\Sigma = \I{0}{2S}$,  $w = v \circ  \xi_{-1}\circ \tau_{2S}$ and $w_1 = u \#_{S} v \circ \xi_{S,\neck}$, with the family $w_y$ defined by interpolating between these two maps using geodesics of minimal length.  Using the exponential decay results of Lemma \ref{lem:exponential_decay_bounds}, we see that this term is bounded by  
\begin{equation}  \left(1+ \rO\left(e^{-S}\right)\right) \left|X_{v} \circ  \xi_{-1} - \Pi_{v(-1)}^{v \circ  \xi_{-1} }  X_{v}(-1) |_{\I{-2S}{0}} \right|_{1,p,\delta}   .\end{equation}

Concerning the second term (Equation \eqref{eq:second_term_bounded_pre-gluing_vector_fields}), we use Equation \eqref{eq:value_glued_vector_field_middle_converges}, together with the decay estimates of Lemma \ref{lem:exponential_decay_bounds}, and apply \eqref{eq:parallel_transport_bound_error} setting $w$ to be the constant map at $v(-1)$ and $w_{1}$ to be $u \#_{S} v \circ \xi_{S,\neck}$, with $w_y$ again given by minimal geodesics, and conclude that the $p$th power of \eqref{eq:second_term_bounded_pre-gluing_vector_fields} is bounded by
 \begin{equation}  \label{eq:parallel_transport_which_should_decay} \rO(e^{-2 p \delta S})  |X_{v}(-1)|^p \int_{0}^{2S} 2 e^{\delta p (2S-s)} ds  =  \rO(1)|X_{v}(-1)|^p .\end{equation}

The third term is bounded in the same way as the first, applying \eqref{eq:parallel_transport_bound_error}  to $\Sigma = \I{0}{S+1}$,  $w = u \circ  \xi_{1}\circ \tau_{-2S}$ and $w_1 = u \#_{S} v \circ \xi_{S,\neck}$, so that \eqref{eq:third_term_bounded_pre-gluing_vector_fields}  is smaller than
\begin{equation}  \left(1+ \rO\left(e^{-S} \right) \right) \left|X_{u} \circ  \xi_{1} - \Pi_{u(1)}^{u \circ  \xi_{1} }  X_{u}(1) |_{\I{2S}{3S}} \right|_{1,p,\delta} .\end{equation}

Using the estimate \eqref{eq:parallel_transport_bound_error} and exponential decay yet again, we see that \eqref{eq:fourth_term_bounded_pre-gluing_vector_fields}  is bounded by
\begin{equation} \rO \left(e^{-2 \delta S}\right)   \left(1 +\rO \left(e^{-S} \right) \right)   \left| X_{u}  \circ  \xi_{1} -  \Pi_{v(-1)}^{u  \circ  \xi_{1} } (X_{v}(-1))  |_{\I{3S-1}{3S+1}}  \right|_{p,\delta} .\end{equation}
Note the appearance of the factor $e^{- 2\delta S}$ in the above expression.  This is due to the fact that the weight in $\I{S-1}{S+1}$ is approximately $e^{\delta S}$; this is the weight used to compute the norm in Equation \eqref{eq:weight_S_varying_sobolev_strip}.  We are trying to bound this with respect to a norm computed on $\xi_{1}$, and the corresponding domain is shifted by $2S$, and becomes $\I{3S-1}{3S+1}$ as indicated in the notation.  In particular, the weight is approximately given by $e^{3\delta S}$, and the ratio between the two weight is, up a multiplicative constant, equal to $e^{-2\delta S}$

Finally, we briefly sketch of the proof of \eqref{eq:pre-gluing_converges_toisometry} :  First we note that the right hand side of the bound \eqref{eq:parallel_transport_which_should_decay} can be replaced by $\rO(e^{-2p(1-\delta)S})|X_{v}(-1)|^p$ whenever the vector field $X_u$ and $X_v$ are smooth. With this in mind, we observe that for a fixed $S$, the estimates proved above show that the difference between the norms before and after gluing are either given by an exponentially decaying term, or by a constant multiple of the contribution to $|X_u|_{1,p,\delta}$ and  $|X_v|_{1,p,\delta}$ coming from $\xi_{1} \I{S}{+\infty}$ and $\xi_{-1} \I{-\infty}{-S}$; these contribution decays exponentially with $S$.
\end{proof}

Next, we turn to the right inverse. We must first be able to bound $| \_ |_{1,p,S}$-norm of $X_{u} \#_{S} X_{v}$.  To do this, consider the two vector fields along the image of $ u \#_{S} v$ obtained by taking the parallel transport of $ X_{u}(1) = X_{v}(-1)$ along 
\begin{equation} \label{eq:parallel_transport_two_paths} \parbox{35em}{(i) the image of horizontal paths under $ u \#_{S} v$ or (ii) the concatenation of the image of  horizontal paths under $v$ (or $u$) with the minimal geodesic from $v$ (or $u$) to $ u \#_{S} v$.} \end{equation}
 If we restrict attention to $ \xi_{S,\neck} \I{0}{1}$, Lemma \ref{lem:bound_derivative_parallel_transport} implies that the difference between these two vector fields is bounded in $C^1$-norm by 
\begin{equation} \label{eq:bound_parallel_transport_size_1_strip} \rO\left(e^{-4S}\right) |X|. \end{equation}
Since $ u \#_{S} v  \circ \xi_{S,\neck} $ and $v \circ \xi_{-1} \circ \tau_{-2S}$ agree on $\I{1}{2S}$, we use the symmetric nature of our construction to conclude:
\begin{lem} 
The $| \_|_{1,p,S}$-norm of the difference between the vector fields obtained by parallel transport of $X_{u}(1)$ along the two paths described in  Equation \eqref{eq:parallel_transport_two_paths} decays exponentially in $S$:
\begin{equation} \label{lem:associative_up_to_exponential_error} \left| \Pi_{u(1)}^{u \#_{S} v \circ \xi_{S,\neck} } X_{u}(1) - \Pi_{v\circ \xi_{-1} \circ \tau_{-2S}}^{u \#_{S} v \circ \xi_{S,\neck} } \Pi^{v\circ \xi_{-1} \circ \tau_{-2S}}_{u(1) } X_{u}(1)   \right|_{1,p,S} = \rO(e^{-2(2-\delta)S})|X_{u}(1)| \end{equation} \noproofe
\end{lem}
\begin{proof}
Fix $(u,\theta,v) \in K$ and $S \in [0,+\infty)$.  Given $Y \in L^{p}(u \#_{S} v^{*} (TM) \otimes \Omega_{0,1})$, note that $D_{\cP_{0,1}} \circ \tilde{Q} (Y) - Y$  is supported on  $\xi_{S,\neck}(\I{-S-1}{S+1})$, and, by the symmetry of our construction, it suffices to estimate
\begin{equation} \int_{\I{-S-1}{0}} \left| D_{\cP_{0,1}} \circ \tilde{Q}  (Y) (\xi_{S,\neck}(s,t)) - Y(\xi_{S,\neck}(s,t))\right|^p e^{\delta (2S+s)} ds dt .\end{equation} 
Let us write
\begin{equation} Q_{\cod{1}{L}} \circ B(Y)  =  (\lambda_{\theta} \partial_{\theta}, \lambda_{R} \partial_R, X_u, X_v) \end{equation}
Since the cut off function preceding the third term of \eqref{eq:formula_gluing_vector_fields} is identically equal to $1$ on $\I{-S-1}{0}$,  the formula for $\tilde{Q} $ simplifies to
\begin{align} \notag \tilde{Q} (Y)(\xi_{S,\neck}(s,t)) &   =  \Pi_{u \circ \xi_{1} \circ \tau_{2S} }^{u \#_{S} v \circ \xi_{S,\neck} } X_{u} \circ  \xi_{1}( s+2S , t)    \\    \label{eq:expression_right_inverse} 
 &  + (1-\chi_{-S}) \left(  \Pi_{v \circ \xi_{-1} \circ \tau_{-2S} }^{u \#_{S} v \circ \xi_{S,\neck} }  X_{v} \circ  \xi_{-1} (s-2S,t)  - \Pi_{u(1)}^{u \#_{S} v \circ \xi_{S,\neck} } X_{u}(1) \right) \end{align}
Turning our attention to the first term, the fact that $Q_{\cod{1}{L}}$ is a right inverse to $D_{\cod{1}{L}}$ implies the vanishing of
\begin{equation}  D_{\cP_{0,1}}  X_{u} - Y_{u} .\end{equation}
Since $u \circ \xi_{1} \circ \tau_{2S}$ and $u \#_{S} v \circ \xi_{S,\neck}$  agree on $\I{-S-1}{-2}$, the restriction to this strip vanishes, so it suffices to bound the restriction of the first term in \eqref{eq:expression_right_inverse}  to $
\I{-2}{0}$.  Lemma \ref{lem:exponential_decay_bounds}, implies that  $u \circ \xi_{1} \circ \tau_{2S}$ can be obtained by exponentiating a vector field $Z$ along  $u \#_{S} v \circ \xi_{S,\neck}$ whose $C^1$-norm on $\I{-2}{0}$ is bounded by a constant multiple of $e^{-2S}$.  In particular, the $| \_|_{1,p,S}$-norm of this vector field is bounded by  $\rO \left( e^{-2(1-\delta)S} \right)$.  Applying the basic form of the quadratic inequality stated in Corollary \ref{cor:quadratic_inequality_basic_case}, we conclude that
\begin{equation} \left|    D_{\cP_{0,1}}  \Pi_{v \circ \xi_{-1} \circ \tau_{-2S} }^{u \#_{S} v \circ \xi_{S,\neck} }  X_{v} \circ  \xi_{-1} \circ \tau_{-2S} - Y \circ \xi_{S,\neck}|_{\I{-S-1}{0}} \right|_{p,S} = \rO\left(e^{-2(1-\delta)S}\right)|Y|_{p,S} . \end{equation}

It remains to show that the result of applying $D_{\cP_{0,1}}$ to the last term of \eqref{eq:expression_right_inverse} has norm that decays exponentially with $S$.  Inspecting the definition of $D_{\cP_{0,1}}$, we see that we can bound the contribution of this term by separately bounding
\begin{equation} D_{\cP_{0,1}} \left(  \Pi_{v \circ \xi_{-1} \circ \tau_{-2S} }^{u \#_{S} v \circ \xi_{S,\neck} }  X_{v} \circ  \xi_{-1} \circ \tau_{-2S}  - \Pi_{u(1)}^{u \#_{S} v \circ \xi_{S,\neck} } X_{u}(1) \right) \end{equation}
and the result of differentiating $\chi_{-S}$ 
\begin{equation}  \nabla \chi_{-S} \cdot \left(  \Pi_{v \circ \xi_{-1} \circ \tau_{-2S} }^{u \#_{S} v \circ \xi_{S,\neck} }  X_{v} \circ  \xi_{-1} \circ \tau_{-2S}  - \Pi_{u(1)}^{u \#_{S} v \circ \xi_{S,\neck} } X_{u}(1) \right) . \end{equation}
Since $X_v$ is holomorphic on $\xi_{-1} (\I{-2S}{-\infty})$, and the $C^{1}$-distance between $u \#_{S} v \circ \xi_{S,\neck}|\I{-S-1}{0} $ and $v \circ \xi_{-1}|\I{-3S-1}{-2S}  $ is bounded by $\rO\left( e^{-S} \right)$ by the decay estimate of Lemma  \ref{lem:exponential_decay_bounds}, applying Corollary \ref{cor:quadratic_inequality_basic_case} to the vector field along  $v \circ \xi_{-1} \circ \tau_{-2S}$ whose image is $u \#_{S} v \circ \xi_{S,\neck}$ yields
\begin{equation} \left|  D_{\cP_{0,1}}  \Pi_{v \circ \xi_{-1} \circ \tau_{-2S} }^{u \#_{S} v \circ \xi_{S,\neck} }  X_{v} \circ  \xi_{-1} \circ \tau_{-2S}|_{\I{-S-1}{0}} \right| = \rO( e^{-(1-\delta)S})|Y|_{p,S} .  \end{equation}
Essentially the same argument applied to the constant map with target $u(1)$, equipped with the constant vector field $  X_{u}(1)$ shows that 
\begin{equation} \left|D_{\cP_{0,1}}  \Pi_{u(1)}^{u \#_{S} v \circ \xi_{S,\neck} } X_{u}(1)|_{\I{-S-1}{0}} \right| = \rO(e^{-(1-\delta)S})|Y|_{p,S}  .\end{equation}

Finally, note that $|\nabla \chi_{-S}|$ is bounded independently of $S$, and that it is supported in the interval $s \in [-S-1,-S+1]$, so it suffices to bound
\begin{equation}\left| \Pi_{v \circ \xi_{-1} \circ \tau_{-2S} }^{u \#_{S} v \circ \xi_{S,\neck} }  X_{v} \circ  \xi_{-1} \circ \tau_{-2S}  - \Pi_{u(1)}^{u \#_{S} v \circ \xi_{S,\neck} } X_{u}(1)| \I{-S-1}{-S+1} \right|_{1,p,S}. \end{equation}
Recall that all our parallel transport maps take place along paths of length less than $e^{-S}$, so \eqref{eq:parallel_transport_bound_error} implies that this expression is bounded by
\begin{equation} \label{eq:2_delta_S_decay} e^{- 2 \delta S} (1 + \rO(e^{-S})) |X_v|_{1,p,\delta} =  \rO(  e^{ -2 \delta S} )|Y|_{p,S} .\end{equation} 
Once again, the weight $ e^{- 2 \delta S}$ appears because we're using the gluing and weight conventions of \cite{FOOO}.

Considering all the error terms above, we see that all terms decay at least as fast as $ e^{- 2 \delta S}$, since $\delta$ is assumed to be smaller than $1/4$.  \end{proof}

\subsection{Proof of surjectivity (Corollary \ref{cor:surjectivity_extended_map})} \label{sec:proof_surjectivity_extended_gluing}

Given a sequence $w_i \in \cP(L;0)$ converging $(u,\theta,v) \in K$, we will prove that 
\begin{equation}  (w_i, \theta) \in \cP(L;0) \times S^1 \cong \cP_{0,1}(L;0) \subset \cF^{1,p}_{\cP_{0,1}}(L)\end{equation} lies in the image of $\preGext_{\epsilon}$, whenever $i$ is sufficiently large.  The statement of Proposition \ref{prop:statement_implicit_function_theorem} implies that $\preGext_{\epsilon}$ is surjective onto a neighbourhood of pre-glued curves which has uniform size with respect to the norms $| \_|_{1,p, S}$, so it suffices to prove that for sufficiently large $i$, $(w_i, \theta)$ lies arbitrarily close (in the $| \_|_{1,p, S}$-norm)  to some pre-glued curve.

From Gromov convergence, we extract the existence of a sequence $S_i \to +\infty$ such that we have uniform convergence of all derivatives on compact subsets:
\begin{align}
w_{i} \circ \tau_{-4S_i} \circ r_{-\theta} | D^{2} - \{ 1 \} & \to v  \\
w_i \circ r_{-\theta}  | D^{2} -\{ -1 \} & \to u  .
\end{align}

We conclude that  for each $S \geq 0$, $w_i$ is $C^1$-close  to $u \#_{S_i} v \circ \phi_{S_i}^{-1}$ on the complement of the image of $\xi_{S_{i},\neck}\left( \I{-2(S_i-S)}{2(S_i-S)}\right)$ for $S_i$ sufficiently large.   Section \ref{sec:proof-results-gluing} ends with the proof of this result:
\begin{lem} \label{lem:surjectivity_codim_1_gluing}
There exists a sequence  of positive real numbers $S_i$ and vector fields $X_i$ along $u \#_{S_i} v$ such that
\begin{equation} \exp_{u \#_{S_i} v}(X_i) = w_i \circ r_{-\theta}\end{equation}
whenever $i$ is sufficiently large.  Moreover, the $| \_|_{1,p,S_i}$-norm of $X_i$ converges to $0$.
\end{lem}
\begin{proof}
The preceding discussion   implies that for each $S$, $X_i$ is well-defined away from the finite strips $ \xi_{S_i,\neck}\left(\I{-2(S_i-S)}{2(S_i-S)}\right)$ as long as $i$ is sufficiently large. First, we make sure that $X_i$ is defined on the entire domain. The key point is that the energy of the strip satisfies
\begin{equation} \label{eq:limit_energy_strip_zero} \lim_{S \to +\infty}  \lim_{i \to +\infty}  E\left(w_i \circ  \xi_{S_i,\neck}| \I{-2(S_i-S)}{2(S_i-S)}  \right) = 0 \end{equation}
since the complement of this region converges in $C^{1}$ to $u$ or $v$, and hence carries most of the energy.    For each $S$, we assume that $i$ is large enough that
\begin{equation} \label{eq:exponential_decay_sequence}|X_{i}|_{C^{1}} \leq e^{-S}\end{equation}
in this region, where the $C^{1}$-norm is computed with respect to the metric $g_{S_{i}}$.

Fix a ball of some radius $\epsilon$ about $v(-1)$ on which there exists an involution which preserves the complex structure $J_{\alg}$ and fixes $L$ pointwise.  We will first prove using the monotonicity Lemma that for each $S$, $w_i \circ \xi_{S_{i},\neck}| \I{-2(S_i-S)}{2(S_i-S)}$ takes values in this ball for $i$ sufficiently large.  To see this, note that the restriction of $w_i \circ \xi_{S_{i},\neck}$ to the union of the two strips of length $S$
\begin{equation}    \I{-(2S_i-S)}{2S_i-S} - \I{-2(S_i-S)}{2(S_i-S)}\end{equation} 
takes values inside the ball of radius $\epsilon$ about $v(-1) = u(1)$ for $i$ sufficiently large by $C^1$ convergence to $u$ and $v$ in the complement of $  \I{-2(S_i-S)}{2(S_i-S)}  $ .  In particular, assuming by contradiction that there is a point in $\I{-2(S_i-S)}{2(S_i-S)}$ whose image under $w_i \circ \xi_{S_{i},\neck}$ lands outside $B_{\epsilon}(v(-1))$ then there is a component of $w_i^{-1}(M - B_{\epsilon/2}(v(-1)))$ which is contained in $\I{-2(S_i-S)}{2(S_i-S)}$.  Since we've assumed this component to have diameter larger than a fixed $\epsilon/2$, the monotonicity Lemma for holomorphic curves with Lagrangian boundary conditions (see Proposition 4.7.2 of \cite{sikorav}) gives a lower bound for the energy of $w_i$ restricted to the strip, contradicting \eqref{eq:limit_energy_strip_zero} if $S$ and $S_i$ are sufficiently large.

We conclude that $w_i \circ \xi_{S_{i},\neck}| \I{-(2S_i-S)}{(2S_i-S)}$ can be doubled to a holomorphic map from a cylinder 
\begin{equation} S^1 \times [-2(S_i-S), 2(S_i-S)] \to M. \end{equation}  A standard exponential decay estimate, as proved for example in Equation (4.7.13) of \cite{MS}, implies that for any constant $\mu <1$, there is a constant $C_\mu$ independent of $w_i$, $S$ and $S_{i}$ such that
\begin{equation}\label{eq:exponential_decay_long_strip_small_energy} | \partial_{s} w_i (s,t)| \leq C_{\mu} e^{- \mu ((2S_i -S) -|s|)}   E\left(w_i \circ \xi_{S_{i},\neck}\left( \I{-(2S_i-S)}{(2S_i-S)} \right) \right) \end{equation}
whenever $|s| \leq 2S_i-2S$.

Starting with the $C^0$ analogue of the estimate \eqref{eq:exponential_decay_sequence} as a boundary value when $|s| = 2S_i -S$, we integrate the exponential decay estimate \eqref{eq:exponential_decay_long_strip_small_energy} to conclude that  $w_i \circ \xi_{S_i,\neck}| \I{-(2S_i-S)}{2S_i-S}$ is $C^0$-close to $u(1) = v(-1)$, which means that $X_i$ is well-defined and satisfies
\begin{align} |X_{i} \circ \xi_{S_{i},\neck}(s,t)| & = \rO\left(e^{-S}+ \left( (2S_i -S) -|s| \right) e^{- \mu ((2S_i -S) -|s|)} \right) \\
\left|\nabla X_{i} \circ \xi_{S_{i},\neck}(s,t)\right| & = \rO\left(e^{- \mu ((2S_i -S) -|s|)}\right).
\end{align}
In particular, the path from $w_{i} \circ \xi_{S_{i},\neck}(0,t)$ to $w_{i} \circ \xi_{S_{i},\neck}(s,t)$ is bounded in length by $|s| \rO\left(e^{- \mu ((2S_i -S) -|s|)}\right)$.   Applying the bound \eqref{eq:parallel_transport_bound_error}, we see that if $|s| \leq 2S_i-S$ 
\begin{equation} \left| \nabla \Pi_{w_{i}\circ \xi_{S_{i},\neck}(0,0)}^{w_{i}\circ \xi_{S_{i},\neck}(s,0)}X_i \circ \xi_{S_{i},\neck}(0,0)  \right| = \rO\left(e^{-\mu (4S_{i} -2S) -|s|}\right) \end{equation}
Choosing $\mu$ to be larger than $\delta$, and letting $S$ go to infinity, we conclude that
\begin{equation} \lim_{S \to +\infty}  \lim_{i \to +\infty} \left| X_i \circ \xi_{S_i,\neck}| \I{-(S_i+S)}{S_i-S} \right|_{1,p,S_i,\delta} = 0. \end{equation}

The discussion preceding the statement of the Lemma, implies a similar estimate away from this finite strip.  The only part that doesn't follow trivially from $C^1$-convergence is the bound on the derivative of the parallel transport of $X_i \circ \xi_{S_{i},\neck}(0,0)$.  However, we have just shown that the norm of $X_i \circ \xi_{S_{i},\neck}(0,0)$ decays $\mu$-exponentially as $S_{i}$ grows. Using the bound \eqref{eq:parallel_transport_bound_error} to control the error term coming from parallel transporting this vector, we conclude that for each $S$, 
\begin{equation} \lim_{i \to \infty}\left| X_i| D^2 - \xi_{S_i,\neck}\left(\I{-(S_i-S)}{S_i-S}\right) \right|_{1,p,S_i} = 0 ,\end{equation}
thereby proving the Lemma.
\end{proof}

\subsection{Proof of the Quadratic Inequality} \label{sec:proof_quadratic}
In this section, we shall prove Proposition \ref{prop:quadratic_inequality}.  Let us  return to the notation introduced in the discussion preceding Lemma \ref{lem:bounds_for_implicit_function_theorem}, and write $w^{\natural}_{S}$ for $\preGext_{S}(u,\theta,v)$.   The Proposition asserts the existence of a constant $c$ such that given a tangent vector $X^{\natural}$ to $\cF_{\cP_{0,1}}^{1,p}(L)$ at $w^{\natural}_{S}$ we have
\begin{equation} \left| \tilde{\Pi}_{w^{\natural}_{S}}^{w_{S,Z^{\natural}}^{\natural}} D_{\cP_{0,1}} X^{\natural} -  D_{\cP_{0,1}} \left(\Pi_{w^{\natural}_{S}}^{w_{S,Z^{\natural}}^{\natural}} X^{\natural}\right) \right|_{p,\Sigma} \leq c \left|Z^{\natural}\right|_{1,p,S}\left|X^{\natural}\right|_{1,p,S}\end{equation}
whenever the norm of $Z^{\natural}$ is itself bounded as in \eqref{eq:bound_C_0-for_quadratic}.

The first ingredient we need before proving the quadratic inequality is a uniform estimate on the $L^p$-norm of $d(u \#_{S} v)$.  Of course, we need to use the weighted norm
\begin{equation}|d(u \#_{S} v)|^{p}_{p,S} = \int \kappa_{S,\delta} |d(u \#_{S} v)|^{p} . \end{equation}
Applying the exponential decay estimates  of Lemma \ref{lem:exponential_decay_bounds}, we conclude that there is a constant $c_0$, independent of $S$, such that
\begin{equation} \label{eq:uniform_L_p} |d(u \#_{S} v)|_{p,S} \leq c_{0} . \end{equation}
We shall assume that $c_0$ is also larger than the uniform Sobolev constant of Lemma \ref{lem:uniform_sobolev}.

The case of Proposition \ref{prop:quadratic_inequality}  where $\lambda$ and $r$ both vanish is covered by Corollary \ref{cor:quadratic_inequality_basic_case} as long as we restrict to the codimension $2$ subspace $T \cF(L)$ of $T \cF_{\cP_{0,1}}(L)$.  Indeed, considering such a vector field $X$, Lemma \ref{lem:uniform_sobolev} implies that $|Z|_{C^0}$ is bounded by a multiple of $|Z|_{1,p,S} $  that is independent of $S$, while Lemma \ref{lem:derivative_p-norm_bounded_by_1,p} implies that $| \nabla X|_{p,S}$ and $| \nabla Z|_{p,S}$ are respectively bounded by constant multiples of $|X|_{1,p,S}$. Since $\gamma$ vanishes in the neck we obtain a uniform bound on $|\gamma_{\theta,R}|_{p,S}$, which, together with Equation \eqref{eq:uniform_L_p} gives the desired bound on $T \cF(L)$.

To extend this to the subspace $ T[0,+\infty)$ of  $T \cF_{\cP_{0,1}}(L)$, we shall prove that 
\begin{equation} \label{eq:bound_derivative_R_direction_transport_Z} \left|  \tilde{\Pi}_{w^{\natural}_{S}}^{w_{S,Z}^{\natural}} D_{\cP_{0,1}}|_{w^{\natural}_{S}} \partial_{R} -  D_{\cP_{0,1}} |_{w_{S,Z}^{\natural}}  \partial_{R}\right|_{p,S} \leq  C|Z|_{1,p,S}\end{equation}
for some constant $C$ that does not depend on $Z$.  Following the strategy repeatedly adopted in Section \ref{ap:pointwise_estimates}, this follows from a bound on the derivative with respect to parallel transport in the direction of $Z$. Dropping all subscripts from $J$, and writing $\tilde{\Pi}_y$ for $ \tilde{\Pi}_{w^{\natural}_{S}}^{w_{S,yZ}^{\natural}} $, we have to bound two terms;  the first is a simple estimate
\begin{equation*} \left| J \left( \nabla_{ \tilde{\Pi}_{y} Z} J \right)   D_{\cP_{0,1}}|_{w^{\natural}_{S}} \partial_{R}  \right|  \leq C|Z| \left(\left|\gamma_{\theta,R} \right| + \left|  \frac{d \gamma_{\theta,R}}{dR} \right| +|dw_{S}| \right),\end{equation*}
while the second relies on results from Section \ref{ap:pointwise_estimates}:
\begin{align*}\left| \nabla_{\Pi_{y} Z} D_{\cP_{0,1}} \partial_{R} \right| & \leq \left| \nabla_{\Pi_{y} Z}  \left( -\frac{1}{2} J \frac{\partial J}{\partial R} \left( dw_{S,yZ} -  \gamma_{\theta,R} \otimes X_{H}  \right) - \frac{d \gamma_{\theta,R}}{dR}    \otimes X_{H} \right)^{0,1} \right| \\
& \leq C|Z| \left(\left| \gamma_{\theta,R} \right| + \left|  \frac{d \gamma_{\theta,R}}{dR} \right|  +|dw_{S,yZ}| \right) \\
& \leq C|Z| \left(\left|\gamma_{\theta,R} \right| + \left|  \frac{d \gamma_{\theta,R}}{dR} \right| +|dw_{S}| +|\nabla Z|  \right) 
. \end{align*}
In particular,  the last inequality follows from Equation \eqref{eq:bound_exponential}.  With this in mind, it is easy to prove the bound \eqref{eq:bound_derivative_R_direction_transport_Z}.   The proof for the additional parameter $\theta$ is essentially indistinguishable, and hence omitted.  We conclude
\begin{lem} \label{lem:quadratic_inequality_change_Z}
There exists a constant $c$ independent of $S$ such that
\begin{equation} \left| \tilde{\Pi}_{w^{\natural}_{S}}^{w^{\natural}_{S,Z}} D_{\cP_{0,1}} X^{\natural} -  D_{\cP_{0,1}} (\Pi_{w^{\natural}_{S}}^{w^{\natural}_{S,Z}} X^{\natural}) \right|_{p,\Sigma} \leq c |Z|_{1,p,S}|X^{\natural}|_{1,p,S}\end{equation} \noproof
\end{lem}

We shall now extend this result to the case where the additional parameters $\lambda$ and $r$ do not vanish.   Our strategy is to replace the parallel transport from $w^{\natural}_{S} = (R,\theta,w_S)$ to $w^{\natural}_{S,(R+r,\lambda,Z)} = (R+r, \theta + \lambda, w_{S,Z})$ along the image of the exponential map, with parallel transport along a broken geodesic which passes through $w^{\natural}_{S,(0,0,Z)}$.  
\begin{lem} \label{lem:bound_triangle_parallel_transport_error}
There is a constant $C$ that is independent of $S$ such that we have an bound on the commutativity of parallel transport on the fibers of $\cE^{p}(M)$ equipped with the $| \_ |_{p,S}$ norm
\begin{align} \label{eq:error_complex_parallel_transport_commute} \left\| \tilde{\Pi}_{w^{\natural}_{S}}^{ w^{\natural}_{S,(r,\lambda,Z)}} - \tilde{\Pi}_{w^{\natural}_{S,(0,0,Z)}}^{ w^{\natural}_{S,(r,\lambda,Z)}}  \tilde{\Pi}_{w^{\natural}_{S}}^{ w^{\natural}_{S,(0,0,Z)}}  \right\| & \leq C|Z|^{\natural}_{C^0}. \end{align}
In addition, the operator
\begin{equation} T_{w^{\natural}_{S}} \cF^{1,p}(L) \to L^{p}\left( {w^{\natural}_{S,Z^{\natural}}}^{*} \left(TM\right) \otimes \Omega^{*} D^2 \right) \end{equation} 
obtained by composing $\nabla$ with this difference of parallel transport maps is also uniformly bounded:
\begin{align} \left\| \nabla \left( {\Pi}_{w^{\natural}_{S}}^{ w^{\natural}_{S,(r,\lambda,Z)}} - {\Pi}_{w^{\natural}_{S,(0,0,Z)}}^{ w^{\natural}_{S,(r,\lambda,Z)}}  {\Pi}_{w^{\natural}_{S}}^{ w^{\natural}_{S,(0,0,Z)}} \right) \right\| & \leq C|Z^{\natural}|_{C^0}
 \end{align}
\end{lem}
\begin{proof}
Note that the proof of Lemma \ref{lem:bound_difference_parallel_transports} directly applies to our setting, and shows that the difference between $\tilde{\Pi}$ and $\Pi$ is an operator bounded by a constant multiple of $|Z^{\natural}|$. For the Levi-Civita connection, Equation \eqref{eq:parallel_transport_bound_error_family} implies that
\begin{equation} \left\| {\Pi}_{w^{\natural}_{S}}^{ w^{\natural}_{S,(r,\lambda,Z)}} -  {\Pi}_{w^{\natural}_{S,(0,0,Z)}}^{ w^{\natural}_{S,(r,\lambda,Z)}}  {\Pi}_{w^{\natural}_{S}}^{ w^{\natural}_{S,(0,0,Z)}} \right\| \leq C|Z^{\natural}|_{C^{0}} (|r| +|\lambda|) ,\end{equation}
from which the desired estimate \eqref{eq:error_complex_parallel_transport_commute} follows for fibers of $\cE^{p}(M)$.

To prove the second estimate we apply Equation \eqref{eq:parallel_transport_bound_error_family_derivative} to obtain a pointwise bound:
\begin{multline} \left| \nabla \left( {\Pi}_{w^{\natural}_{S}}^{ w^{\natural}_{S,(r,\lambda,Z)}} X - {\Pi}_{w^{\natural}_{S,(0,0,Z)}}^{ w^{\natural}_{S,(r,\lambda,Z)}}  {\Pi}_{w^{\natural}_{S}}^{ w^{\natural}_{S,(0,0,Z)}} X \right) \right| \\ \leq C|Z^{\natural}| \left(|X| \left(|dw^{\natural}_{S}| +|\nabla Z| \right) +| \nabla X| \right)
\end{multline}
Using the Sobolev embedding for $|X|_{C^0}$, and the boundedness of the $| \_ |_{p,S}$ norm of  $dw^{\natural}_{S}$, we obtain the desired result. 
\end{proof}

This reduces the proof of Proposition \ref{prop:quadratic_inequality} to proving the result in the special case where $Z$ vanishes.   Simplifying the notation further by dropping the parameters $R$ and $r$, we write $w^{\natural}_{S,\lambda}$ for the pair $(\theta+\lambda,w)$.  We first observe that the arguments given in the previous section give the following bounds, which respectively come from Equation \eqref{eq:bound_exponential}, \eqref{eq:bound_first_derivative_exponential}, and \eqref{eq:parallel_transport_bound_exponential}:
\begin{align} 
\label{eq:change_norm_differential_flat}\left| dw^{\natural}_{S,\lambda} \right| & \leq |dw^{\natural}_{S}| (1+ C|\lambda|) \\
\left| \nabla_{\partial_{\theta}}  dw^{\natural}_{S,\lambda} \right| & \leq  C|dw^{\natural}_{S}| \\
\label{eq:change_norm_parallel_transport_flat}
 \left| \nabla  {\Pi}_{w^{\natural}_{S}}^{w^{\natural}_{S,\lambda}}  X  -   {\Pi}_{w^{\natural}_{S}}^{w^{\natural}_{S,\lambda}}  \nabla X \right|  & \leq   C|X||\lambda||dw^{\natural}_{S}|  
\end{align}

Using this, we conclude
\begin{align} \label{eq:change_norm_del_operator_flat}
 \left| \partial_{\cP_{0,1}} w^{\natural}_{S,\lambda} \right|  & \leq C \left( 1 +|\lambda| \right)   \left(|dw^{\natural}_{S}|  + \left|  \gamma_{\theta, R} \right| \right)
\end{align}

We can use \eqref{eq:bound_derivative_R_direction_transport_Z} to prove the desired pointwise result for a tangent vector $X^{\natural} = (\rho, \eta, X)$:
\begin{lem} \label{lem:pointwise_bound_dbar_op_parallel_transport_parameter}
There exists a constant $C$ independent of $S$ such that
\begin{multline}  
\left| \tilde{\Pi}_{w^{\natural}_{S}}^{w^{\natural}_{S,(r,\lambda,0)}} D_{\cP_{0,1}} X^{\natural} - D_{\cP_{0,1}} {\Pi}_{w^{\natural}_{S}}^{w^{\natural}_{S,(r,\lambda,0)}} X^{\natural} \right| \\ \leq C (|r| +|\lambda|) (|\nabla X| +|X| (|\gamma_{\theta,R}| +|\nabla \gamma_{\theta,R}| +|dw^{\natural}_{S}|) )
\end{multline}
\end{lem}
\begin{proof}
We prove this only in the special case where $\rho=\eta=0$, and return to ignoring the variable $R$ (the case where only $\rho$ and $\eta$ are non-zero is left to the reader).  The first step is to control the derivative with respect to $\lambda$ using:
\begin{align*}  
\frac{d}{d\lambda} \left| \tilde{\Pi}_{w^{\natural}_{S}}^{w^{\natural}_{S,\lambda}} D_{\cP_{0,1}} X - D_{\cP_{0,1}} {\Pi}_{w^{\natural}_{S}}^{w^{\natural}_{S,\lambda}} X \right| & \leq  \left| J \left( \nabla_{\partial_{\theta}} J \right) \tilde{\Pi}_{w^{\natural}_{S}}^{w^{\natural}_{S,\lambda}} D_{\cP_{0,1}} X \right| + \left| \nabla_{\partial_{\theta}} D_{\cP_{0,1}}  {\Pi}_{w^{\natural}_{S}}^{w^{\natural}_{S,\lambda}}  X  \right|  \end{align*}
Using the expression for $  D_{\cP_{0,1}} X $ and Lemma \ref{lem:bound_difference_parallel_transports} we find that the first term is bounded by a constant multiple of 
\begin{equation}|\nabla X| +|X|(|\gamma_{\theta,R}| +|\nabla \gamma_{\theta,R}|) \end{equation} 
To bound the second term,  we return to the expression \eqref{eq:expression_linearisation_circle_varying_parametrized}, and first observe that in analogy with \eqref{eq:lowest_order_terms_vector_field}, the terms where a derivative of $J$ or $\gamma$ is taken are bounded by a constant multiple of
\begin{equation} \left| \nabla  {\Pi}_{w^{\natural}_{S}}^{w^{\natural}_{S,\lambda}}  X \right| + |X| \left(|\gamma_{\theta,R}| + \left| \frac{d \gamma_{\theta,R}}{d\theta} \right| + \left| \partial_{\cP_{0,1}} w^{\natural}_{S,\lambda} \right| \right) \end{equation}
Using Equations \eqref{eq:change_norm_parallel_transport_flat} and \eqref{eq:change_norm_del_operator_flat}, we see that this is in turn bounded by a constant multiple of
\begin{equation}|\nabla X| +|X|  \left(\left|\gamma_{\theta,R}\right| + \left| \nabla \gamma_{\theta,R} \right| +|dw^{\natural}_{S}| \right) . \end{equation}
The higher order terms are bounded by
\begin{multline} \left| \nabla_{\partial_{\theta}} \nabla  {\Pi}_{w^{\natural}_{S}}^{w^{\natural}_{S,\lambda}}  X   \right| + \left|\gamma_{\theta,R}\right| \left|\nabla_{\partial_{\theta}}  \nabla_{ {\Pi}_{w^{\natural}_{S}}^{w^{\natural}_{S,\lambda}}  X } X_{H} \right|  \\ + \left| \nabla_{\partial_{\theta}}  \nabla_{ {\Pi}_{w^{\natural}_{S}}^{w^{\natural}_{S,\lambda}}  X } J_{\theta,R} \right| \left| \partial_{\cP_{0,1}} w^{\natural}_{S,\lambda} \right|  +|X| \left|  \nabla_{\partial_{\theta}}  \partial_{\cP_{0,1}} w^{\natural}_{S,\lambda} \right| . \end{multline}
Note that this is essentially the same expression as \eqref{eq:higher_order_terms_vector_field}.  Using \eqref{eq:change_norm_differential_flat}-\eqref{eq:change_norm_parallel_transport_flat}, we adapt the argument we used to bound  \eqref{eq:higher_order_terms_vector_field}, and conclude that the higher order terms are bounded by a constant multiple of
\begin{equation}|X| \left(|dw^{\natural}_{S}| +|\gamma_{\theta,R}| +| \nabla \gamma_{\theta,R}|\right)  \end{equation}
\end{proof}

Finally, we can prove the main result of this section

\begin{proof}[Proof of Proposition \ref{prop:quadratic_inequality}]
We compute that
\begin{multline} \label{eq:bound_L_p_norm_quadratic_proof} \left| \tilde{\Pi}_{w^{\natural}_{S}}^{w_{S,Z^{\natural}}^{\natural}} D_{\cP_{0,1}} X^{\natural} -  D_{\cP_{0,1}} \Pi_{w^{\natural}_{S}}^{w_{S,Z^{\natural}}^{\natural}} X^{\natural} \right|_{p,S} \leq  \\ \left|  \left( \tilde{\Pi}_{w^{\natural}_{S}}^{w_{S,Z^{\natural}}^{\natural}} - \tilde{\Pi}_{w^{\natural}_{S,(0,0,Z)}}^{w_{S,Z^{\natural}}^{\natural}}  \tilde{\Pi}_{w^{\natural}_{S}}^{w^{\natural}_{S,(0,0,Z)}} \right)  D_{\cP_{0,1}} X^{\natural}  \right|  +  
\left| D_{\cP_{0,1}} \left(\Pi_{w^{\natural}_{S,(0,0,Z)}}^{w_{S,Z^{\natural}}^{\natural}} \Pi_{w^{\natural}_{S}}^{w^{\natural}_{S,(0,0,Z)}} - \Pi_{w^{\natural}_{S}}^{w_{S,Z^{\natural}}^{\natural}} \right) X^{\natural} \right| + \\
\left|\tilde{\Pi}_{w^{\natural}_{S,(0,0,Z)}}^{w_{S,Z^{\natural}}^{\natural}}  \tilde{\Pi}_{w^{\natural}_{S}}^{w^{\natural}_{S,(0,0,Z)}} D_{\cP_{0,1}} X^{\natural} - D_{\cP_{0,1}} (\Pi_{w^{\natural}_{S,(0,0,Z)}}^{w_{S,Z^{\natural}}^{\natural}} \Pi_{w^{\natural}_{S}}^{w^{\natural}_{S,(0,0,Z)}} X^{\natural})  \right|  \end{multline}
By Equation \eqref{eq:error_complex_parallel_transport_commute}, and using the fact that the norm of the operator $D_{\cP_{0,1}}$ is bounded independently of $S$ (see Corollary \ref{cor:linearisation_dbar_bounded}), we find that the first term is bounded by a constant multiple of 
\begin{equation}|Z^{\natural}|_{1,p,S}|X^{\natural}|_{1,p,S}  . \end{equation}
Pointwise, the second term is bounded by a constant multiple of
\begin{multline} 
\left| \nabla \left(\Pi_{w^{\natural}_{S,(0,0,Z)}}^{w_{S,Z^{\natural}}^{\natural}} \Pi_{w^{\natural}_{S}}^{w^{\natural}_{S,(0,0,Z)}} - \Pi_{w^{\natural}_{S}}^{w_{S,Z^{\natural}}^{\natural}} \right) X^{\natural} \right| \\ + \left(|\partial_{\cP_{0,1}} w^{\natural}_{S,Z}| +|\gamma| \right) \left|\left(\Pi_{w^{\natural}_{S,(0,0,Z)}}^{w_{S,Z^{\natural}}^{\natural}} \Pi_{w^{\natural}_{S}}^{w^{\natural}_{S,(0,0,Z)}} - \Pi_{w^{\natural}_{S}}^{w_{S,Z^{\natural}}^{\natural}} \right) X^{\natural} \right| \end{multline} 
Applying  Lemma \eqref{lem:bound_triangle_parallel_transport_error} to the first term, and Equations \eqref{eq:parallel_transport_bound_error_family} and \eqref{eq:change_norm_del_operator_flat} to the second, we conclude that the $L^p$-norm of this error term is bounded by a constant multiple of
\begin{equation}|Z^{\natural}|_{1,p,S}|X^{\natural}|_{1,p,S} + (|dw_{S}|_{p} +|\gamma|_{p})|Z^{\natural}|_{C^0} |X^{\natural}|_{C^0} ,\end{equation}
which by the Sobolev lemma implies that the second term of \eqref{eq:bound_L_p_norm_quadratic_proof} is indeed uniformly bounded by a constant multiple of
\begin{equation}|Z^{\natural}|_{1,p,S}|X^{\natural}|_{1,p,S} . \end{equation}

It remains to bound the third term of \eqref{eq:bound_L_p_norm_quadratic_proof}.  We shall leave this step to the reader, and simply note that one should use Lemmas \ref{lem:quadratic_inequality_change_Z}, \ref{lem:bound_triangle_parallel_transport_error}, and \ref{lem:pointwise_bound_dbar_op_parallel_transport_parameter}. 
\end{proof}

We complete this appendix with the proof of a different version of the quadratic inequality:
\begin{proof}[Proof of Equation \eqref{eq:quadratic_inequality}]
Since the norm of a convex combination of $X_1^{\natural}$ and $X_2^{\natural}$ is bounded above by $|X_1^{\natural} + X_2^{\natural}|$, it suffices to combine the estimate \eqref{eq:quadratic_inequality_first_statement} with the expression
\begin{multline}
\sF_{\preGext_{S} ( u , \theta ,v)}\left(X_1^{\natural}\right) - \sF_{  \preGext_{S} ( u , \theta ,v)} \left(X_2^{\natural} \right)  \\ = \int_{0}^{1} d \sF_{\preGext_{S} ( u , \theta ,v)} \left(x X_1^{\natural} + (1-x) X_2^{\natural} \right) \left( X_1^{\natural} - X_2^{\natural} \right) dx.
\end{multline}
\end{proof}


\begin{bibdiv}
\begin{biblist}

\bib{adams}{book}{
   author={Adams, Robert A.},
   title={Sobolev spaces},
   note={Pure and Applied Mathematics, Vol. 65},
   publisher={Academic Press [A subsidiary of Harcourt Brace Jovanovich,
   Publishers], New York-London},
   date={1975},
   pages={xviii+268},
   review={\MR{0450957 (56 \#9247)}},
}

\bib{ASIV}{article}{
   author={Atiyah, M. F.},
   author={Singer, I. M.},
   title={The index of elliptic operators. IV},
   journal={Ann. of Math. (2)},
   volume={93},
   date={1971},
   pages={119--138},
   issn={0003-486X},
   review={\MR{0279833 (43 \#5554)}},
}

\bib{ALP}{article}{
   author={Audin, Mich{\`e}le},
   author={Lalonde, Fran{\c{c}}ois},
   author={Polterovich, Leonid},
   title={Symplectic rigidity: Lagrangian submanifolds},
   conference={
      title={Holomorphic curves in symplectic geometry},
   },
   book={
      series={Progr. Math.},
      volume={117},
      publisher={Birkh\"auser},
      place={Basel},
   },
   date={1994},
   pages={271--321},
   review={\MR{1274934}},
}



\bib{berger}{book}{
   author={Berger, Marcel},
   title={A panoramic view of Riemannian geometry},
   publisher={Springer-Verlag},
   place={Berlin},
   date={2003},
   pages={xxiv+824},
   isbn={3-540-65317-1},
   review={\MR{2002701 (2004h:53001)}},
}


\bib{BC}{article}{
   author={Biran, Paul},
   author={Cieliebak, Kai},
   title={Lagrangian embeddings into subcritical Stein manifolds},
   journal={Israel J. Math.},
   volume={127},
   date={2002},
   pages={221--244},
   issn={0021-2172},
   review={\MR{1900700 (2003g:53151)}},
}

\bib{buh}{article}{
   author={Buhovsky, Lev},
   title={Homology of Lagrangian submanifolds in cotangent bundles},
   journal={Israel J. Math.},
   volume={143},
   date={2004},
   pages={181--187},
   issn={0021-2172},
   review={\MR{2106982 (2005m:53167)}},
   doi={10.1007/BF02803498},
}


\bib{EE}{article}{
   author={Earle, Clifford J.},
   author={Eells, James},
   title={A fibre bundle description of Teichm\"uller theory},
   journal={J. Differential Geometry},
   volume={3},
   date={1969},
   pages={19--43},
   issn={0022-040X},
   review={\MR{0276999 (43 \#2737a)}},
}


\bib{floer}{article}{
   author={Floer, Andreas},
   title={The unregularized gradient flow of the symplectic action},
   journal={Comm. Pure Appl. Math.},
   volume={41},
   date={1988},
   number={6},
   pages={775--813},
   issn={0010-3640},
   review={\MR{948771 (89g:58065)}},
}

\bib{floer-monopole}{article}{
   author={Floer, A.},
   title={Monopoles on asymptotically flat manifolds},
   conference={
      title={The Floer memorial volume},
   },
   book={
      series={Progr. Math.},
      volume={133},
      publisher={Birkh\"auser},
      place={Basel},
   },
   date={1995},
   pages={3--41},
   review={\MR{1362821 (96j:58024)}},
}

\bib{FHS}{article}{
   author={Floer, Andreas},
   author={Hofer, Helmut},
   author={Salamon, Dietmar},
   title={Transversality in elliptic Morse theory for the symplectic action},
   journal={Duke Math. J.},
   volume={80},
   date={1995},
   number={1},
   pages={251--292},
   issn={0012-7094},
   review={\MR{1360618 (96h:58024)}},
}


\bib{fuk}{article}{
   author={Fukaya, Kenji},
   title={Application of Floer homology of Langrangian submanifolds to
   symplectic topology},
   conference={
      title={Morse theoretic methods in nonlinear analysis and in symplectic
      topology},
   },
   book={
      series={NATO Sci. Ser. II Math. Phys. Chem.},
      volume={217},
      publisher={Springer},
      place={Dordrecht},
   },
   date={2006},
   pages={231--276},
   review={\MR{2276953 (2007j:53112)}},
}


\bib{FOOO}{book}{
   author={Fukaya, Kenji},
   author={Oh, Yong-Geun},
   author={Ohta, Hiroshi},
   author={Ono, Kaoru},
   title={Lagrangian intersection Floer theory: anomaly and obstruction.
   Part I},
   series={AMS/IP Studies in Advanced Mathematics},
   volume={46},
   publisher={American Mathematical Society},
   place={Providence, RI},
   date={2009},
   pages={xii+396},
   isbn={978-0-8218-4836-4},
   review={\MR{2553465}},
}

\bib{gromov}{article}{
   author={Gromov, M.},
   title={Pseudoholomorphic curves in symplectic manifolds},
   journal={Invent. Math.},
   volume={82},
   date={1985},
   number={2},
   pages={307--347},
   issn={0020-9910},
   review={\MR{809718 (87j:53053)}},
}

\bib{HZW}{article}{
   author={Hofer, H.},
   author={Wysocki, K.},
   author={Zehnder, E.},
   title={A general Fredholm theory. I. A splicing-based differential
   geometry},
   journal={J. Eur. Math. Soc. (JEMS)},
   volume={9},
   date={2007},
   number={4},
   pages={841--876},
   issn={1435-9855},
   review={\MR{2341834 (2008m:53202)}},
}


\bib{KM}{article}{
   author={Kervaire, Michel A.},
   author={Milnor, John W.},
   title={Groups of homotopy spheres. I},
   journal={Ann. of Math. (2)},
   volume={77},
   date={1963},
   pages={504--537},
   issn={0003-486X},
   review={\MR{0148075 (26 \#5584)}},
}

\bib{OK}{article}{
   author={Kwon, Daesung},
   author={Oh, Yong-Geun},
   title={Structure of the image of (pseudo)-holomorphic discs with totally
   real boundary condition},
   note={Appendix 1 by Jean-Pierre Rosay},
   journal={Comm. Anal. Geom.},
   volume={8},
   date={2000},
   number={1},
   pages={31--82},
   issn={1019-8385},
   review={\MR{1730896 (2001b:32050)}},
}



\bib{lazzarini}{article}{
   author={Lazzarini, L.},
   title={Existence of a somewhere injective pseudo-holomorphic disc},
   journal={Geom. Funct. Anal.},
   volume={10},
   date={2000},
   number={4},
   pages={829--862},
   issn={1016-443X},
   review={\MR{1791142 (2003a:32044)}},
}


\bib{MS}{book}{
   author={McDuff, Dusa},
   author={Salamon, Dietmar},
   title={$J$-holomorphic curves and symplectic topology},
   series={American Mathematical Society Colloquium Publications},
   volume={52},
   publisher={American Mathematical Society},
   place={Providence, RI},
   date={2004},
   pages={xii+669},
   isbn={0-8218-3485-1},
   review={\MR{2045629 (2004m:53154)}},
}


\bib{oh-removal}{article}{
   author={Oh, Yong-Geun},
   title={Removal of boundary singularities of pseudo-holomorphic curves
   with Lagrangian boundary conditions},
   journal={Comm. Pure Appl. Math.},
   volume={45},
   date={1992},
   number={1},
   pages={121--139},
   issn={0010-3640},
   review={\MR{1135926 (92k:58065)}},
}

\bib{oh-perturb-boundary}{article}{
   author={Oh, Yong-Geun},
   title={Fredholm theory of holomorphic discs under the perturbation of
   boundary conditions},
   journal={Math. Z.},
   volume={222},
   date={1996},
   number={3},
   pages={505--520},
   issn={0025-5874},
   review={\MR{1400206 (97g:58024)}},
}


\bib{oh-disjunction}{article}{
   author={Oh, Yong-Geun},
   title={Gromov-Floer theory and disjunction energy of compact Lagrangian
   embeddings},
   journal={Math. Res. Lett.},
   volume={4},
   date={1997},
   number={6},
   pages={895--905},
   issn={1073-2780},
   review={\MR{1492128 (98k:58048)}},
}

\bib{polterovich-monotone}{article}{
   author={Polterovich, Leonid},
   title={Monotone Lagrange submanifolds of linear spaces and the Maslov
   class in cotangent bundles},
   journal={Math. Z.},
   volume={207},
   date={1991},
   number={2},
   pages={217--222},
   issn={0025-5874},
   review={\MR{1109663 (93c:58075)}},
}


\bib{seidel-kronecker}{article}{
   author={Seidel, Paul},
   title={Exact Lagrangian submanifolds in $T\sp *S\sp n$ and the graded
   Kronecker quiver},
   conference={
      title={Different faces of geometry},
   },
   book={
      series={Int. Math. Ser. (N. Y.)},
      volume={3},
      publisher={Kluwer/Plenum, New York},
   },
   date={2004},
   pages={349--364},
   review={\MR{2103000 (2005h:53153)}},
}

\bib{seidel-LOS}{article}{
   author={Seidel, Paul},
   title={A long exact sequence for symplectic Floer cohomology},
   journal={Topology},
   volume={42},
   date={2003},
   number={5},
   pages={1003--1063},
   issn={0040-9383},
   review={\MR{1978046 (2004d:53105)}},
}

\bib{sikorav}{article}{
   author={Sikorav, Jean-Claude},
   title={Some properties of holomorphic curves in almost complex manifolds},
   conference={
      title={Holomorphic curves in symplectic geometry},
   },
   book={
      series={Progr. Math.},
      volume={117},
      publisher={Birkh\"auser},
      place={Basel},
   },
   date={1994},
   pages={165--189},
   review={\MR{1274929}},
}


\bib{wells}{book}{
   author={Wells, R. O., Jr.},
   title={Differential analysis on complex manifolds},
   series={Graduate Texts in Mathematics},
   volume={65},
   edition={2},
   publisher={Springer-Verlag},
   place={New York},
   date={1980},
   pages={x+260},
   isbn={0-387-90419-0},
   review={\MR{608414 (83f:58001)}},
}

\bib{whitehead}{book}{
   author={Whitehead, George W.},
   title={Elements of homotopy theory},
   series={Graduate Texts in Mathematics},
   volume={61},
   publisher={Springer-Verlag},
   place={New York},
   date={1978},
   pages={xxi+744},
   isbn={0-387-90336-4},
   review={\MR{516508 (80b:55001)}},
}


\end{biblist}
\end{bibdiv}
\end{document}